\documentclass[a4paper,11pt]{article}%
\usepackage{amsmath}%
\usepackage{amsfonts}%
\usepackage{amssymb}%
\usepackage{amsthm}
\usepackage{stmaryrd}
\usepackage{graphicx}
\usepackage[utf8]{inputenc}
\usepackage[T1]{fontenc}
\usepackage[french, english]{babel}
\usepackage{enumitem}
\usepackage{tikz}
\usetikzlibrary{decorations.markings}
\usepackage[colorlinks=true]{hyperref}

\newtheorem{thm}{Theorem}

\newtheorem{theorem}{Theorem}[section]

\newtheorem{claim}[theorem]{Claim}

\newtheorem{corollary}[theorem]{Corollary}

\newtheorem{lemma}[theorem]{Lemma}

\newtheorem{proposition}[theorem]{Proposition}
\newtheorem{prop}[thm]{Proposition}

\theoremstyle{definition}
\newtheorem{definition}[theorem]{Definition}

\newtheorem{remark}[theorem]{Remark}

\newlength{\espaceavantspecialthm}
\newlength{\espaceapresspecialthm}
\newcommand{\R}{\mathbb{R}}
\newcommand{\N}{\mathbb{N}}

\newcommand{\Q}{\mathbb{Q}}
\newcommand{\Z}{\mathbb{Z}}
\newcommand{\T}{\mathbb{T}}
\newcommand{\Sp}{\mathbb{S}}
\newcommand{\Hy}{\mathbb{H}}
\newcommand{\F}{\mathcal{F}}
\newcommand{\I}{\mathcal{I}}
\newcommand{\varep}{\varepsilon}
\newcommand{\diam}{\operatorname{diam}}
\newcommand{\conv}{\operatorname{conv}}
\newcommand{\Homeo}{\operatorname{Homeo}}
\newcommand{\card}{\operatorname{Card}}
\newcommand{\dom}{\operatorname{dom}}
\newcommand{\fix}{\operatorname{Fix}}
\newcommand{\Id}{\operatorname{Id}}
\newcommand{\wt}{\widetilde}
\newcommand{\wh}{\widehat}
\newcommand{\X}{\mathcal{X}}
\newcommand{\x}{\mathbf{x}}
\newcommand{\len}{\operatorname{length}}

{\\}

\newenvironment{defi*}[1][]{
\vskip \espaceavantspecialthm \noindent \textbf{D\'efinition.} }%
{\vskip \espaceapresspecialthm}

\counterwithin{equation}{section}

\tikzset{->-/.style={decoration={
  markings,
  mark=at position .5 with {\arrow{latex}}},postaction={decorate}}}

\newcommand\test[1]{
\pgfmathsetmacro{\var}{#1}
\pgfmathparse{ifthenelse(\var>=0,"positif","négatif")} \pgfmathresult}%

\newcommand{\hgline}[3]{
\pgfmathsetmacro{\thetaone}{#1}
\pgfmathsetmacro{\thetatwo}{#2}
\pgfmathsetmacro{\theta}{(\thetaone+\thetatwo)/2}
\pgfmathsetmacro{\phi}{abs(\thetaone-\thetatwo)/2}
\pgfmathsetmacro{\close}{less(abs(\phi-90),0.0001)}
\ifdim \close pt = 1pt
    \draw[->-, color=#3] (\thetaone:1) -- (\thetatwo:1);
\else
	\pgfmathsetmacro{\R}{tan(\phi)}
	\pgfmathsetmacro{\test}{(\thetaone-\thetatwo)/abs(\thetaone-\thetatwo)}
	\draw[->-, color=#3] (\thetaone:1) arc (\thetaone+\test*90:\thetaone+\test*(270-2*\phi):\R);
\fi
}

\newcommand{\hglinefill}[3]{
\pgfmathsetmacro{\thetaone}{#1}
\pgfmathsetmacro{\thetatwo}{#2}
\pgfmathsetmacro{\theta}{(\thetaone+\thetatwo)/2}
\pgfmathsetmacro{\phi}{abs(\thetaone-\thetatwo)/2}
\pgfmathsetmacro{\close}{less(abs(\phi-90),0.0001)}
\ifdim \close pt = 1pt
    \filldraw[->-, color=#3] (\thetaone:1) -- (\thetatwo:1) ;
\else
	\pgfmathsetmacro{\R}{tan(\phi)}
	\pgfmathsetmacro{\test}{(\thetaone-\thetatwo)/abs(\thetaone-\thetatwo)}
	\filldraw[color=#3, opacity=.2] (\thetaone:1) arc (\thetaone+\test*90:\thetaone+\test*(270-2*\phi):\R) arc (\thetaone+\test*(270-2*\phi)-90:\thetaone+\test*90+90:1);
\fi
}

\begin{document}

\sloppy

\title{Homotopic rotation sets for higher genus surfaces}
\author{Pierre-Antoine Guih\'eneuf, Emmanuel Militon}
\maketitle

\begin{abstract}
This paper states a definition of homotopic rotation set for higher genus surface homeomorphisms, as well as a collection of results that justify this definition. We first prove elementary results: we prove that this rotation set is star-shaped, we discuss the realisation of rotation vectors by orbits or periodic orbits and we prove the creation of new rotation vectors for some configurations.

Then we use the theory developped by Le Calvez and Tal in \cite{MR3787834} to obtain two deeper results: \\
-- If the homotopical rotation set contains the direction of a closed geodesic which has a self-intersection, then there exists a rotational horseshoe and hence infinitely many periodic orbits in many directions. \\
-- If the homotopical rotation set contains the directions of two closed geodesics that meet, there exists infinitely many periodic orbits in many directions.

\end{abstract}

\selectlanguage{english}

\setcounter{tocdepth}{2}
\tableofcontents

\section{Introduction}

The key invariant in the study of circle homeomorphisms dynamics is Poincaré's rotation number, which measures the orbits' asymptotic mean speed of rotation around the circle. It leads to the celebrated Poincaré classification, which asserts -- among others -- that the rotation number is rational if and only if the homeomorphism possesses a periodic orbit.

The generalisation of this invariant to the two dimensional torus leads to the definition of rotation set, as the accumulation set of asymptotic mean speeds of orbits rotation around the torus. More formally, given a homeomorphism $f$ of the torus $\T^2$ homotopic to identity, and $\tilde f : \R^2\to\R^2$ one of its lifts, the rotation set $\rho(\tilde f)$ is the set of all possible limits of sequences $\frac{\tilde f^{n_k}(x_k) - x_k}{n_k}$, for $n_k$ going to $+\infty$ and $x_k\in\R^2$. It is a compact convex subset of $\R^2$ \cite{MR1053617}, invariant under conjugation by isotopically trivial homeomorphisms. As for the circle case, its shape is strongly related to the dynamics: for example, any point with rational coordinates in the interior of the rotation set is associated to a periodic point of the homeomorphism \cite{MR967632}.

The literature exploring the properties of this set is now quite consequent and makes use of a wide range of different techniques, from Brouwer theory and its improvements by Le Calvez \cite{MR2217051}, culminating to the Le Calvez-Tal recent works \cite{MR3787834, 1803.04557, guiheneuf2020theorie}, to Nielsen-Thurston classification \cite{MR1101087}, prime ends or Pesin theories (e.g. \cite{MR4092853})\dots{} Even if there are still quite a lot of open questions (e.g. whether there exists a torus homeomorphism having a rotation set with nonempty interior and smooth boundary), the subject is now mature and rich enough to attempt tackling similar issues in more complex situations.

A natural extension is to study rotation properties of homeomorphisms of higher genus (closed) surfaces. Let us point out that for the torus, first homotopy and homology groups coincide, while this is not the case for higher genus surfaces. Hence in this new setting one expects to get two distinct definitions of rotation sets, both in homotopical and homological senses. The latter, defined formally a long time ago \cite{MR88720}, regained attention in the last years.

\subsection*{Homological rotation sets}

Let us recall one possible definition of the homological rotation set rotation set (see \cite{MR1094554,MR88720}).

Let $S$ be a closed surface, and fix $f \in \mathrm{Homeo}_0(S)$. 
Let us denote by $D$ the diameter of $S$. For any two points $x$ and $y$ of $S$, we choose a geodesic path $g_{x,y}$ of length lower than or equal to $D$ which joins the point $x$ to the point $y$. Fix an isotopy $(f_t)_{ t\in[0,1]}$ between $f_0=\Id_{S}$ and $f_1=f$. As usual, we extend it to an isotopy $(f_t)_{ t\in \R}$ by setting $f_{t}=f_{t-\lfloor t \rfloor}\circ f^{\lfloor t \rfloor}$, where $\lfloor t \rfloor$ denotes the lower integer part of $t$.

For any point $x$ in $S$, we define $l_{n,x}$ as the loop obtained by concatenating the path $(f_{t}(x))_{ t\in[0,n]}$ with the geodesic path $g_{f^{n}(x),x}$. This loop defines a cycle and we denote by $[l_{n,x}]_{H_{1}(S)}$ the class of this cycle in $H_{1}(S)=H_{1}(S,\R)$.

\begin{definition}\label{DefHomolo}
The \emph{homological rotation set} of $f\in\Homeo_0(S)$ is the set $\rho_{H_1}(f)$ of points $\rho\in H_1(S,\R)\simeq \R^{2g}$ such that there exists $(x_k)_k\in S$ and $(n_k)_k$ going to $+\infty$ such that $[l_{n_k,x_k}]_{H_1}/n_k$ tends to $\rho$. 
\end{definition}

As we divide by $n_k$ in this definition, this set does not depend on the chosen geodesic paths $g_{x,y}$.

When $g(S) \geq 2$, this set does not depend on the chosen isotopy, as two such isotopies are homotopic with fixed endpoints. Indeed, the topological space $\mathrm{Homeo}_0(S)$ is contractible. If $g(S)=1$, this set depends on the chosen isotopy but two such sets differ by an integral translation. Indeed, two isotopies between the identity and $f$ are homotopic up to composition with an integral translation. Also, it is possible to associate a homological rotation vector to any $f$-invariant measure (see \cite{lellouch} for more details).
\medskip

Let us describe a few known results about this homological rotation set:
\begin{itemize}
\item Entropy: Is it possible to get sufficient conditions on the homological rotation set for the homeomorphism to have positive topological entropy? Here are some known conditions:
\begin{itemize}
\item If there exist $2g+1$ periodic points whose homological rotation vectors do not lie on a hyperplane of $H_1(S,\R)\simeq \R^{2g}$ \cite{MR1094554}.
\item If $f$ is a $C^1$-diffeomorphism, and if there exist $g+2$ periodic points whose homological rotation vectors form a $g+1$-nondegenerate simplex \cite{MR1444450}.
\item If there exists two invariant probability measures whose homological rotation vectors have a nontrivial intersection \cite{lellouch}. This result implies the two above results.
\end{itemize}
\item Realisation of periodic points: in which cases some vectors of the homological rotation set are realised by periodic points? Such results were obtained under similar hypotheses to the ones for positiveness of entropy:
\begin{itemize}
\item If there exist $2g+1$ periodic points whose homological rotation vectors do not lie on a hyperplane of $H_1(S,\R)\simeq \R^{2g}$, then any rational point in the interior of the simplex spanned by these rotation vectors is the rotation vector of some periodic point \cite{MR1334719}.
\item If $f$ is a $C^1$-diffeomorphism, and if there exist $g+2$ periodic points whose homological rotation vectors form a $g+1$-nondegenerate simplex, then any rational point in the interior of this simplex is the rotation vector of some periodic point \cite{MR1444450}.
\item Under the set of hypotheses called fully essential system of curves by the authors\footnote{That is satisfied under some hypotheses on stable/unstable manifolds of periodic points if $f$ is a $C^{1+\varepsilon}$ diffeomorphism; in particular it implies that the homological rotation set has nonempty interior.}, any rational point in the interior of the rotation set is the rotation vector of some periodic point \cite{MR4190050}. In this case, the authors also get convexity of the rotation set, uniform bounds on displacements, etc.
\item If there exist two invariant probability measures $\mu$ and $\nu$ whose homological rotation vectors $\rho_\mu$ and $\rho_\nu$ have a nontrivial intersection, then any point of the simplex spanned by $0$, $\rho_\mu$ and $\rho_\nu$ is accumulated by rotation vectors of periodic points \cite{lellouch}. In this thesis the author also gets uniform bounds on displacements if $0$ lies in the interior of the rotation set. 
\end{itemize}
\item Generic shape: for a generic homeomorphism, the rotation set is given by a union of at most $2^{5g-3}$ convex sets \cite{alonso2020generic}.
\end{itemize}

Note also the work \cite{MR3820002} which (among others) gives conditions under which the dynamics of an area preserving homeomorphism of $S$ can be decomposed into dynamics of lower genus surface homeomorphisms.

\subsection*{Homotopical rotation sets}

Note that some rotational information is lost when using the homological rotation set: for instance it does not see the difference between the trivial loop and a commutator (see the path $\alpha$ in Figure~\ref{FigCommut}). This incites finding a practical definition of homotopical rotation set in the higher genus context.

\begin{figure}[ht]
\begin{center}
\includegraphics[trim=0 20 0 10, scale=1.1]{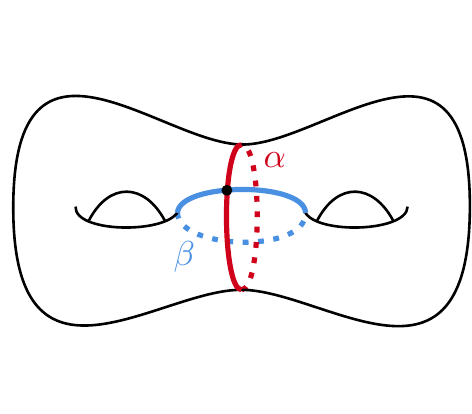}
\caption{\label{FigCommut}The rotation around the path $\alpha$, which is homologically trivial (it is a commutator in the $\pi_1$) is not detected by the homological rotation set. In this paper, we get (among others) the existence of infinitely periodic orbits when there are rotation vectors in both directions $\alpha$ and $\beta$.}
\end{center}
\end{figure}

Unlike what we have seen in the homological context, there is no such commonly accepted definition of a homotopical rotation set. To our knowledge, the only known result is the one of Lessa \cite{MR2846925}, but it has no consequence on the initial surface homeomorphism dynamics. 

In this paper, we propose a new notion of homotopical rotation set for higher genus surfaces homeomomorphisms. Let $S$ be a closed surface of genus $\ge 2$, and $f$  a homeomorphism of $S$ which is homotopic to the identity. The universal cover of $S$ is the hyperbolic plane $\Hy^2$, that we equip with its canonical metric. Let $\tilde f$ be the unique lift of $f$ to $\Hy^2$ that extends to identity to $\partial \Hy^2$. Remark that the set of geodesics of $\Hy^2$ can be parametrized by the set of couples of distinct points of $\partial \Hy^2$. For any $(\alpha,\beta)\in (\partial \Hy^2)^2$ with $\alpha\neq\beta$ and any $v\in \R_+^*$, we will say that the triple $(\alpha,\beta,v)$ is a \emph{rotation vector of $f$} if there exists a sequence $(x_k)_{k\in \N}$ of points of $\Hy^2$, and an increasing sequence $(n_k)_{k\in\N}$ of natural numbers tending to infinity such that, if we denote by $\pi_{\alpha,\beta}$ the orthogonal projection\footnote{In other words, the projection to the closest point of the geodesic $(\alpha,\beta)$.} on the geodesic linking $\alpha$ to $\beta$,
\begin{equation}\label{EqDefRot}
\left(x_k\ ,\ \tilde f^{n_k}(x_k)\ ,\ \frac{d\Big(\pi_{\alpha,\beta}(x_k)\, ,\ \pi_{\alpha,\beta}\big(\tilde f^{n_k}(x_k)\big)\Big)}{n_k}\right) \underset{k\to +\infty}{\longrightarrow} \big(\alpha,\beta,v\big).
\end{equation}
The (homotopical) \emph{rotation set $\rho(f)$ of $f$} is then defined as the collection of rotation vectors of $f$, together with all the singletons $\{(\alpha,\beta,0)\}$ for all the geodesics $(\alpha,\beta)$ of $\Hy^2$ (to emphasize the fact that $f$ has a contractible fixed point, by Lefschetz formula).

Note that this definition is a bit different from the one of Lessa \cite{MR2846925}.

\subsection*{Review of the results}

In this whole paragraph, we consider an orientable closed surface\footnote{In the core of the article, we will mention when our results trivially generalize to the case of a non compact orientable surface of finite type; in this introduction we will simply give the statements in the compact case.} $S$ of genus $g\ge 2$, and a homeomorphism $f$ of $S$ homotopic to the identity.

We will state quite a lot of different results, that in some sense give grounds for our definition of homotopical rotation set, some of them rather elementary, others more difficult. They are mainly of two different types: realisation of ``rational'' vectors by periodic orbits, and convexity-like results (the presence of some kinds of orbits forces the presence of others, which are ``convex combinations'' of the initial orbits). We will get our results from three different techniques. The first one consists in using the property of quasi-convexity of fundamental domains. The second one is also elementary, it uses geodesics in the universal cover and their images by the lift of the dynamics to get separating sets of the hyperbolic plane; it allows to get simple convexity-like results. The third and last main tool we use is the forcing theory of Le Calvez and Tal \cite{MR3787834}; it gives much stronger results at the cost of longer and more difficult proofs.

To start with, we prove quasi-convexity of fundamental domains. This result was already known for the torus \cite{MR1053617}, we extend the proof to the higher genus case:
there exists $R = R(S)>0$ such that for any path connected fundamental domain $D$ of $S$ in its universal cover $\wt{S}$, and any point $x$ in the convex hull $\conv(D)$ of $D$, we have $B(x,R)\cap D\neq\emptyset$ (Proposition~\ref{PropQuasiConvex}).

As a consequence, we get that the rotation set $\rho(f)$ is star-shaped (Theorem~\ref{TheoStar}).

\begin{thm}\label{TheoStarIntro}
For any $(\alpha,\beta,v)\in \rho(f)$, and any $v'\in [0,v]$, one has $(\alpha,\beta,v')\in\rho(f)$.
\end{thm}

As a byproduct of this theorem's proof, we get that rotation vectors are realised by segments of orbits whose endpoints stay at a bounded distance to the corresponding geodesic (Proposition~\ref{ProprealisDist}).

Section~\ref{Secrealis} is devoted to other realisation results for rotation vectors associated to closed geodesics. These rotation vectors are directly related to the rotation set of the annulus homeomorphism obtained by quotienting the universal cover $\wt{S}$ of $S$ by this closed geodesic. Note that this is where the fact that in our definition of rotation set, speeds are measured by means of projections on geodesics, is crucial.
As a direct consequence, an application of already known results for rotation sets of annulus homeomorphisms leads to realisation of rotation vectors by periodic orbits, under ``classical'' conditions (Proposition~\ref{PropRealPtExtRat}).
As an application of \cite{MR2846925}, we also get that (still in the closed geodesic case) the extremal rotation vector is realised by an orbit whose lift to $\wt{S}$ stays at sublinear distance from the geodesic (Proposition~\ref{sublindistance}).

We then get to forcing results. The ones of Section~\ref{SecCreat} use only elementary arguments. To begin with, we consider geodesics of the surface with auto-intersection (see Proposition~\ref{newrot} for a more formal statement).

\begin{prop} \label{newrotIntro}
Let $\tilde \gamma$ be a geodesic of $\wt{S}$ projecting to a geodesic $\gamma$ of $S$ which auto-intersects. Let $\gamma'$ be the geodesic of $S$ obtained as a ``shortcut'' of the geodesic $\gamma$. 

If $(\tilde\gamma,v) \in \rho(f)$, then $(\tilde\gamma',v) \in \rho(f)$.
\end{prop}

The general case of two geodesics intersecting is treated in Proposition~\ref{newrot2}, with weaker conclusions.

\begin{prop}\label{newrot2Intro}
Let $(\alpha_1,\beta_1,v_1) \in\rho(f)$, with $v_1>0$. Let also $(\alpha_2,\beta_2)$ be a geodesic of $\Hy^2$ that intersects $(\alpha_1,\beta_1)$, and such that there exists $(y_k)\in\Hy^2$ and $u_k \in\N$ such that $y_k\to \alpha_2$ and $\tilde f^{u_k}(y_k)\to \beta_2$. Then,  there exist $v',v''\ge 0$ satisfying $v'+v'' = v_1$ such that:
\begin{enumerate}[label=(\roman*)]
\item either $(\alpha_1,\beta_2,v')\in \rho(f)$ or $(\alpha_1,\alpha_2,v')\in \rho(f)$;
\item either $(\beta_2,\beta_1,v'')\in \rho(f)$, or $(\alpha_2,\beta_1,v'')\in \rho(f)$.
\end{enumerate}
\end{prop}

The proofs of these results are heavily inspired by the forcing theory \cite{MR3787834}, where geodesic play the role of leaves of Brouwer-Le Calvez foliations.

This last proposition is used in Section~\ref{SecPseudo} to study what we call \emph{almost annular} homeomorphisms (Proposition~\ref{notransverseint}).

\begin{prop} \label{notransverseintIntro}
Suppose that the only nonzero rotation vectors of $f$ are associated to the lifts of a single geodesic $\gamma$ of $S$. Then $\gamma$ has no self-intersection.
\end{prop}

After exposing some examples in Section~\ref{SecEx}, we study in more detail the creation of new rotation vectors for closed geodesics. 

As a first step, in Section~\ref{SecHomo}, we get weak consequences when the homotopical rotation set contains two vectors associated to two closed geodesics which intersect, as in Figure~\ref{FigCommut} (Proposition~\ref{lifttorus} and Corollary~\ref{entropy}; in particular we get positive entropy). These statements rely on the notion of \emph{covering map associated to two distinct closed geodesics} (Definition~\ref{DefCovering}). In our case, the covering surface is a single punctured torus, and the homological consequences we mentioned are stated in terms of rotation set of the lift of the initial homeomorphism to this torus.

This formalism is used in Section~\ref{SecClosed} to get the existence of a rotational horseshoe (see Definition~\ref{DefRotHorse}) when $f$ has a rotation vector associated to a closed geodesic with auto-intersection (Theorem~\ref{ExistSuperCheval}).

\begin{thm}\label{ExistSuperChevalIntro}
Let $\gamma$ be a closed geodesic with a geometric auto-intersection (as in Figure~\ref{FigCommut2}) associated to the deck transformation $T_1$ (in the sense of Definition \ref{DefTransverseInter}). Denote $T_2$ the deck transformation such that $T=T_1T_2$ is the deck transformation associated to the closed geodesic $\gamma$. 

Suppose that $(\gamma,\ell(\gamma))\in \rho(f)$.
Then, $f^7$ has a topological horseshoe associated to the deck transformations $T_1$, $T_1^2$, $T_2$, $T_1T_2$, $T_2T_1$ and $T_1T_2T_1$.
\end{thm}

\begin{figure}[ht]
\begin{center}
\includegraphics[trim=35 20 0 30, scale=1.3]{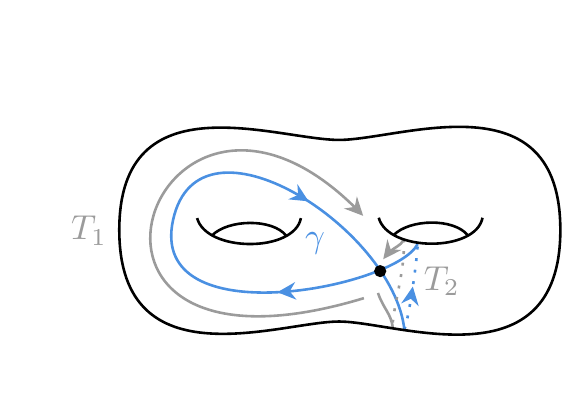}
\caption{\label{FigCommut2}A possible configuration for Theorem~\ref{ExistSuperChevalIntro}: the geodesic $\gamma$ on the surface has a geometric auto-intersection.}
\end{center}
\end{figure}

Note that our definition of rotational horseshoe is different from the one of \cite{MR3784518} and \cite{1803.04557}, as it is stated in terms of Markovian intersection and not of semi-conjugation to a shift (we get this semi-conjugacy as a consequence in Proposition~\ref{PropConjug}). In fact, this kind of rotational horseshoes appears as soon as the homeomorphism has a periodic trajectory under the isotopy to identity which auto-intersects geometrically (see also Proposition~\ref{GeomImpliqFeuill}).

The proof of this theorem is much longer than the previous ones and based on the recent forcing theory \cite{MR3787834,1803.04557}. This is also the case for our last result (Theorem~\ref{Th2transverse} and Corollary~\ref{Cor2transverse}). For any deck transformation $T$ of the universal cover $\wt{S} \rightarrow S$, we denote by $\tilde{\gamma}(T)$ its axis. 

\begin{thm} \label{Th2transverseIntro}
Let ${\gamma}_1$ and ${\gamma}_2$ be two closed geodesics of ${S}$, that lift to $\wt{S}$ to geodesics $\tilde{\gamma}_1$ and $\tilde{\gamma}_2$ that cross (they can be for example the curves $\alpha$ and $\beta$ of Figure~\ref{FigCommut}, note also that they can have auto-intersections). Let $T_1$ and $T_2$ be the deck transformations associated to the respective closed geodesics $\gamma_1$ and $\gamma_2$ and which respectively preserve $\tilde{\gamma}_1$ and $\tilde{\gamma}_2$.

Suppose that there exist nonzero rotation vectors of directions $\tilde{\gamma}_1$ and $\tilde{\gamma}_2$ in $\rho(f)$. Then, for any element $w$ in $\langle T_1,T_2\rangle _+\setminus(\langle T_1\rangle _+ \cup \langle T_2\rangle _+)$, there are nonzero vectors of direction $\tilde{\gamma}(w)$ in $\rho(f)$ which are realised by periodic orbits.

\end{thm}

The proof of this theorem is quite long and divided in numerous sub-cases. One of the difficulties is that the transverse paths associated to the trajectories realising the rotation vectors do not need to have an $\F$-transverse intersection (see Figure~\ref{Figlellouch} for such an example, due to Lellouch \cite{lellouch}).

\subsection*{Some open questions}

We state here some questions about homotopical rotation sets in higher genus that are still open.

\begin{enumerate}[label=\arabic*)]
\item Clarify the links between homotopical and homological rotation sets. In particular, when does a homological rotation vector gives birth to a homotopical rotation vector? This would certainly bring into play hyperbolic geometry as in \cite{MR2846925}.
\item Get more realisation results: is every rotation vector realised by a single orbit of the homeomorphism? What can be the sets of times $n_k$ appearing in \eqref{EqDefRot}?
\item Get more forcing results, for example: if $(\alpha_1,\beta_1,v_1), (\alpha_2,\beta_2,v_2)\in\rho(f)$, with $v_1,v_2>0$, and if the geodesics $(\alpha_1,\beta_1)$ and $(\alpha_2,\beta_2)$ cross, do we have $(\alpha_1,\beta_2,v')\in\rho(f)$ for some $v'>0$? 
A first step may be to get such results under some recurrence hypotheses about the geodesics, or to get it for a single (non closed) geodesic with auto-intersection.
\item Obtain a wider collection of examples to illustrate the diversity of possible behaviours.
\item Explore more the notion of almost annular homeomorphisms.
\item If $\tilde f$ is transitive, what can be said about the $\rho(f)$ (see \cite{MR2929025})?
\item What is the shape of the rotation set of a generic homeomorphism?
\end{enumerate}

\subsection*{Acknowledgements}

We warmly thank Erwann Aubry for insightful discussions about Proposition~\ref{PropQuasiConvex}, as well as Maxime Wolff and Frédéric Le Roux.
We are indebted to Alejo Garc\'ia who has read a first version of this article in detail and suggested numerous corrections. 
The reference \cite{MR295352} was suggested to us by Indira Chatterji, and the references \cite{PabloUnpublished,MR611385,MR2003742} by Pablo Lessa. 

P.-A. Guih\'eneuf was supported by a PEPS/CNRS grant. E. Militon was supported by the ANR project Gromeov ANR-19-CE40-0007.

\section{Quasi convexity of fundamental domains}

A now well known result is the quasi-convexity of 2-torus fundamental domains, whose first proof was given in \cite{MR1053617}, with an argument due to Douady. Based on the index of a curve, it can be replaced by a very elementary one. Here, we adapt this elementary proof to higher genus surfaces\footnote{A. Passeggi informed us in a private communication that he had a proof of this result, but it stayed unpublished.}.

\begin{definition}
We say that a set $X\subset \Hy^2$ is \emph{$R$-quasi convex} if for any point $x$ of the hyperbolic convex hull $\conv(X)$ of $X$, one has $B(x,R)\cap X \neq\emptyset$.
\end{definition}

In what follows, we identify $\overline{\Hy^{2}}$ with the unit closed disk in the complex plane $\mathbb{C}$. In particular, the complex numbers $i$ and $-i$ are identified with points of $\partial \Hy^2$. We endow $\partial \Hy^2$ with the distance induced by the euclidean distance on $\mathbb{C}$. For any two distinct points of the boundary $\alpha,\beta\in\partial \Hy^2$, we denote $(\alpha,\beta)$ the oriented geodesic of $\Hy^2$ having $\alpha$ as $\alpha$-limit and $\beta$ as $\omega$-limit.

\begin{proposition}\label{PropQuasiConvex}
For any orientable closed surface $S$ of genus $g\ge 2$ there exists $R = R(S)>0$ such that any path connected fundamental domain $D\subset \Hy^2$ of $S$ is $R$-quasi convex.
\end{proposition}

\begin{lemma}\label{LemEnvoieDroite}
Let $S$ be a closed surface of genus $g\ge 2$ and $K$ a compact subset of $\Hy^2$. Then, there exists a finite set $F\subset \pi_1(S)$ such that for any oriented geodesic $\tilde \gamma$ of $\Hy^2$ passing through $K$, there exists $T\in F$ such that the right of $T\tilde \gamma$ is a strict subset of the right of the segment $(-i,i)$, where $(-i,i)$ is oriented from $-i$ to $i$.
\end{lemma}

\begin{figure}[ht]
\begin{center}
\begin{tikzpicture}[scale=1.2]
\clip (-1.4,-1.32) rectangle (1.6,1.4);

\draw (0,0) circle (1);
\draw[color=blue!50!black](-90:1) node[below]{$-i$};
\draw[color=blue!50!black](90:1) node[above]{$i$};
\draw[color=blue!50!black](130:1) node[left]{$b$};
\draw[color=blue!50!black](40:1) node[right]{$a$};
\draw[color=blue!50!black](5:1) node[right]{$T\gamma$};
\draw[color=red!50!black](60:1) node[right]{$T$};

\hgline{-90}{90}{blue}
\hgline{130}{40}{blue}
\hgline{-10}{20}{blue}
\hgline{50}{10}{red}

\end{tikzpicture}
\end{center}
\caption{Configuration of the proof of Lemma~\ref{LemEnvoieDroite}.}
\end{figure}
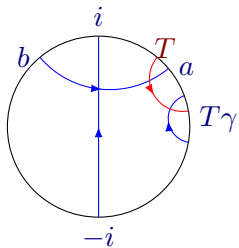

\begin{proof}[Proof of Lemma \ref{LemEnvoieDroite}]
We will use the following classical fact:
\emph{For any point $x\in\partial\Hy^2$, there exists a sequence of geodesic axis of deck transformations whose endpoints both tend to $x$ (and $x$ is between these two endpoints).}
It comes from the following: there exists a constant $\theta_0$ such that any geodesic $\tilde \gamma$ of $\Hy^2$ whose projection to $S$ is non-closed, crosses axis of deck transformations, with an angle $\ge \theta_0$, syndetically. To see this, consider a fundamental domain of $S$ with boundary made of deck transformations axis.
\medskip

First, by applying some iterate of some $T_0\in \pi_1(S)$ with one axis endpoint on the right of $(-i,i)$ if necessary, one can suppose that $K$ is contained in the right of $(-i,i)$ and moreover that the Hausdorff distance between these two sets is at least 1.

Now, take a geodesic $\tilde \gamma$ that crosses $K$. It has to have an endpoint $a$ at the right of $(-i,i)$, and moreover, if we denote by $b$ the other endpoint of this geodesic, the distances\footnote{We have endowed the Poincar\'e circle $\partial \Hy^2$ with its canonical distance.} $d(a,b)$, $d(a,-i)$ and $d(a,i)$ are bigger than some $d_0>0$ which only depends on $K$ and $S$. Take $F_0$ a finite subset of $\pi_1(S)$ such that for any $c\in \partial\Hy^2$, there exists an element of $F_0$ whose axis endpoints belong to respectively $]c-d_0,c[$ and $]c,c+d_0[$. It exists by the above fact. Moreover, we can take $F_{0}$ finite by compactness of $\partial \Hy^2$.

Then, still by compactness of $\partial \Hy^2$, there exists $N\in\N$, depending only on $F_0$ and $d_0$, some $n\in \Z$ with $|n|\le N$ and $T \in F_0$, such that both points $T^n(a)$ and $T^n(b)$ lie on the right of $(-i,i)$ (and thus $T^n\tilde \gamma$ is entirely contained in the right of $(-i,i)$), and that the right of $T^n\tilde \gamma$ is contained in the right of $(-i,i)$. This proves the lemma for
\[F = \big\{T^n \mid T\in F_0,|n|\le N\big\}.\]
\end{proof}

\begin{figure}[ht]
\begin{center}
\includegraphics[scale=1]{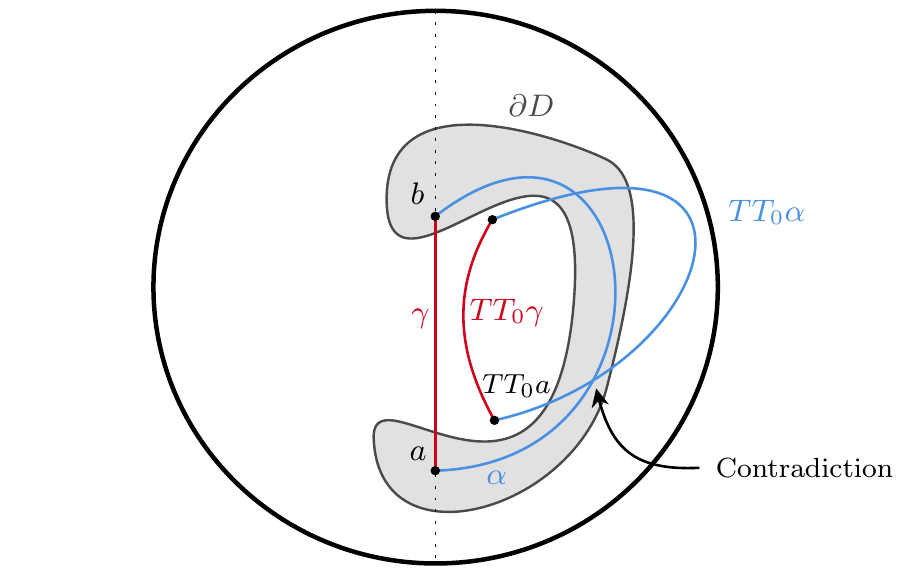}
\caption{Proof of Proposition \ref{PropQuasiConvex}. Here, $\beta=\alpha\cup\gamma$.}
\end{center}
\end{figure}

\begin{proof}[Proof of Proposition \ref{PropQuasiConvex}]
In this proof, $d$ denotes the distance in $\Hy^2$.
Denote by $r_0$ a positive number such that any half-ball of radius $r_0$ in $\Hy^2$ contains some fundamental domain of $S$. Let $B$ be the closure of the connected component of $B(0,r_0)\setminus (-i,i)$ on the right of $(-i,i)$. Let $F\subset \pi_1(S)$ be given by Lemma \ref{LemEnvoieDroite} applied to $K=B$, and set
\[R = \max \big\{d(0,Tx) \mid T\in F, x\in B\big\}.\]

Now, take $D$ a path connected fundamental domain of $S$, and $x\in \conv(D)$. Suppose for a contradiction that $B(x,R) \cap D = \emptyset$.

As $x\in \conv(D)$, there exists $a,b\in D$ such that $x\in [a,b]$. The property of quasi-convexity being invariant under isometry, one can suppose that $x=0$ is the center of the Poincar\'e disk, and that the geodesic line $(a,b)$ is $(-i,i)$. 

As $D$ is path connected, there exists a path $\alpha$ contained in $D$ whose endpoints are $a$ and $b$. By taking a subpath and changing $a$ and $b$ in $\gamma$ if necessary, one can suppose that $\alpha$ does not meet $\gamma$ on its interior. Moreover, by applying a symmetry with respect to $\gamma$ if necessary, one can suppose that the interior of $\alpha$ is included in the right of $\gamma$. Let $\beta$ be the Jordan curve formed by the union of $\alpha$ with the geodesic segment $[a,b]$.

As $B$ contains a fundamental domain, there exists $T_0\in\pi_1(S)$ such that $T_0(a)\in B$. By Lemma \ref{LemEnvoieDroite}, there exists $T\in F$ such that the right of $TT_0(-i,i)$ is included in the right of $(-i,i)$.

By the definition of $R$, one has $TT_0(a)\in B(0,R)$. Then, the hypothesis $B(0,R) \cap D = \emptyset$ implies that $TT_0(a)$ belongs to the Jordan domain bounded by $\beta$. Let $m\in\alpha$ be such that $d(m,(-i,i)) = \max_{y\in\alpha} d(y,(-i,i))$. As $TT_0(-i,i)$ lies on the right of $(-i,i)$, and as $TT_0$ is an isometry, one has that
\[d(TT_0(m),(-i,i)) = d(m,(TT_0)^{-1}(-i,i)) > d(m,(-i,i));\]
so $TT_0(m)$ does not belong to the Jordan domain defined by $\beta$. By continuity, there exists a point $x_0\in \alpha$ such that $TT_0(x_0)\in \beta$. But $TT_0(x_0)$ belongs to the right of $TT_0(-i,i)$, which is included in the right of $(-i,i)$, so $TT_0(x_0)\notin (-i,i)$. This implies that $TT_0(x_0)\in \alpha$. We have found a point $x_0$ and a deck transformation $TT_0\neq\Id$ such that $x_0$ and $TT_0(x_0)$ both lie in $D$. This is a contradiction, thus $B(0,R) \cap D \neq \emptyset$.
\end{proof}

\section{Rotation sets: definition and star shape}

We fix a distance on $\partial\Hy^2$, given by the Euclidean distance on the circle in the Poincar\'e disk model. We denote by $\Delta$ the diagonal in $(\partial\Hy^2)^2$, that is to say
$$\Delta= \left\{ (x,x) \mid x \in \partial \Hy^2 \right\}.$$

In the whole paper, $S$ will be an orientable surface of negative Euler characteristic of finite type. We will specify in each statement when the additional assumption of compactness of $S$ is necessary.

Let $f\in\Homeo_0(S)$ (where $\Homeo_0(S)$ denotes the set of homeomorphisms of $S$ that are homotopic to identity).
We denote by $\tilde f$ the lift of $f$ to $\Hy^2$ which is isotopic to the identity (and thus it extends to the circle $\partial \Hy^2$ by the identity by Lemma 3.8 p.53 in \cite{MR964685}). It is well-known that the homeomorphism $\tilde{f}$ has a fixed point: otherwise, by associating to each point $\tilde{x} \in \Hy^2$ the vector at $\tilde{x}$ pointing towards $\tilde{f}(\tilde{x})$, we would obtain a nowhere vanishing vector field on our surface $S$, a contradiction.

\begin{definition}
A point $(\alpha,\beta,v)\in ((\partial \Hy^2)^2 \setminus \Delta)\times\R_+^*$  is a \emph{rotation vector of $f$} if there exists a sequence $(x_k)_{k\in \N}$ of points of $\Hy^2$, and an increasing sequence $(n_k)_{k\in\N}$ of natural numbers tending to infinity such that, if we denote by $\pi_{\alpha,\beta}$ the orthogonal projection\footnote{In other words, the projection to the closest point of the geodesic $(\alpha,\beta)$.} on the geodesic linking $\alpha$ to $\beta$,
\begin{equation}\tag{\ref{EqDefRot}}
\left(x_k\ ,\ \tilde f^{n_k}(x_k)\ ,\ \frac{d\Big(\pi_{\alpha,\beta}(x_k)\, ,\ \pi_{\alpha,\beta}\big(\tilde f^{n_k}(x_k)\big)\Big)}{n_k}\right) \underset{k\to +\infty}{\longrightarrow} \big(\alpha,\beta,v\big).
\end{equation}

The \emph{rotation set of $f$} is the union of rotation vectors of $f$, together with the singleton $\{(\alpha,\beta,0)\}$, and quotiented by the relation
\[(\alpha,\beta,0) \sim (\alpha',\beta',0).\]
\end{definition}

We add the point $\{(\alpha,\beta,0)\}$ to the rotation set to stress out the fact that the homeomorphism $\tilde{f}$ has a fixed point (by Lefschetz formula), hence an orbit with speed $0$. 

\begin{figure}
\begin{center}
\begin{tikzpicture}[scale=2.5]

\draw[line width=1.1] (0,0) circle (1);
\hgline{-180}{0}{green!50!black}
\hgline{165}{195}{gray}
\hgline{-20}{20}{gray}

\draw[color=green!50!black] (-1,0) node{$\times$} node[left]{$\alpha$};
\draw[color=green!30!black] (1,0) node{$\times$} node[right]{$\beta$};

\draw (-0.8,0.12) node{$\times$} node[above right]{$x_k$};
\draw (0.76,-0.2) node{$\times$} node[below left]{$\tilde f^{n_k}(x_k)$};

\draw[color=gray, dashed] (-0.3,0.5)  node[above]{$\pi_{\alpha,\beta}(x_k)$} to[bend left] (-0.77,0)node{$\times$};
\draw[color=gray, dashed] (0.3,0.3) node[above]{$\pi_{\alpha,\beta}(\tilde f^{n_k}(x_k))$} to[bend right] (0.7,0) node{$\times$};

\draw[<->, >=latex, color=red!50!black] (-0.77,-.05) -- (0.7,-.05);
\draw[color=red!50!black] (0,-.05) node[below]{$\simeq n_k v$};

\end{tikzpicture}
\caption{Definition of the rotation set.}
\end{center}
\end{figure}
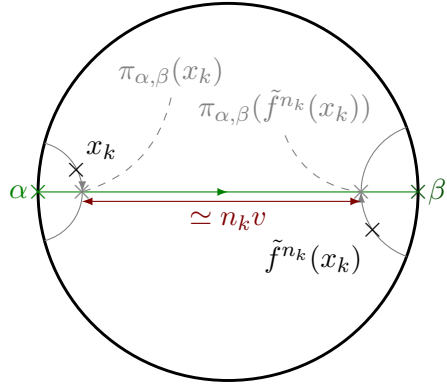

\bigskip

Note that the first two elements of Equation \eqref{EqDefRot} define a geodesic of $\Hy^2$, and the last element corresponds to a speed. Hence, a rotation vector is made of an asymptotic direction and an asymptotic speed.

Be careful, this set is not necessarily closed (see Subsection~\ref{SecExEmmanuel}).

Remark also that, as isotopically trivial homeomorphisms of higher genus surfaces have a canonical lift (the one extending to the identity on $\partial\Hy^2$), the rotation set is uniquely defined for the homeomorphism. It contrasts with rotation sets of torus or annulus homeomorphisms, which depend on the choice of the lift to the universal cover.

\begin{proposition}\label{PropRhoInv}
If $f\in\Homeo_0(S)$, then
\[\rho(f^{-1}) = \big\{(\beta,\alpha,v) \mid (\alpha,\beta,v)\in\rho(f)\},\]
and for any $n\ge 1$,
\[\rho(f^n) = n\rho(f) \doteq \big\{(\alpha,\beta, nv)\mid(\alpha,\beta,v)\in\rho(f)\big\}.\]
For any deck transformation $T$ of $S$, if $(\alpha,\beta,v)\in\rho(f)$, then $(T\alpha,T\beta,v)\in\rho(f)$.
\end{proposition}

\begin{proof}
The first part is immediate.

For the second part, the inclusion $\rho(f^n)\subset n\rho(f)$ is trivial. For the other inclusion, it suffices to remark that any $k\in\N$ can be written as $k = nq+r$, with $0\le r<n$, and that there exists $C>0$ such that $d(\tilde f^r,\Id)\le C$ for any $0\le r <n$. Hence, for any $x\in\Hy^2$, one has $d(\tilde f^k(x), (\tilde f^n)^q(x))\le C$.

The last property comes from the fact that $\tilde{f}$ commutes with deck transformations.
\end{proof}

\begin{remark}
One can translate our definition of rotation set in the case of the torus $\T^2$~; it gives a rotation set which is formed of couples (direction, speed), where the direction is a line of $\R^2$ passing through the origin (which is the same as fixing an angle in $\R/2\pi\Z$) and the speed of orbits of a lift of our homeomorphism is measured \emph{via} an orthogonal projection on this line. From this definition of rotation set, it is possible to recover the classical rotation set for torus homeomorphisms (and reciprocally): fixing one direction $D$ (a line through the origin) in $\R^2$, one can consider the set of points whose projection on $D$ is $v$, where $(D,v)$ is in the rotation set. This set of points is a band of $\R^2$, and the intersection of such bands over all directions $D$ of $\R^2$ gives us the classical rotation set. 
\end{remark}

\begin{remark} [Dependence of the rotation set on the hyperbolic metric]
Fix $f \in \Homeo_0(S)$. Denote by $\tilde{S}$ the universal cover of $S$. We call set of directions of the rotation set of $f$ the set
$$ \mathcal{D}(f)= \left\{ (\alpha, \beta) \in \partial \tilde{S} \times \partial \tilde{S} \ | \ \exists v>0, \ (\alpha,\beta,v) \in \rho(f) \right\}.$$
This set does not depend on the hyperbolic metric we chose on the surface $S$. To understand this better, we describe below another way to see $\mathcal{D}(f)$. We see the fundamental group $\pi_1(S)$  of $S$ as the group of deck transformations of $\tilde{S}$. We embed $\pi_1(S)$ in $\tilde{S}$ by taking one orbit of its action on $\tilde{S}$. The Svar\v{c}-Milnor lemma (see Lemma \ref{svarcmilnor}) ensures that this embedding is a quasi-isometry so that it defines an isomorphism between the Gromov boundary $\partial \pi_1(S)$ of the fundamental group of $S$ and the boundary $\partial \tilde{S}$ of $\tilde{S}$. Denote by $T:\tilde{S} \rightarrow \pi_1(S)$ a quasi-inverse of this quasi-isometry and endow $\pi_1(S)$ with a distance induced by wordlength. Then, via this isomorphism, the set of directions of $\rho(f)$ is
the set of $(\alpha,\beta) \in \partial \pi_1(S) \times \partial \pi_1(S) \setminus \Delta$ such that there exists a sequence $(x_k)_{k \geq 0}$ of points of $\tilde{S}$ a sequence of intergers $n_k \rightarrow +\infty$ such that the sequence $(T(x_k))_k$ converges to $\alpha$, the sequence $T(\tilde{f}^{n_k}(x_k))$ converges to $\beta$ and 
$$\liminf_{k \rightarrow +\infty} \frac{1}{n_k} d(T(\tilde{f}^{n_k}(x_k)),T(x_k)) >0.$$
Changing the hyperbolic metric on $S$ changes the map $T$ into a map $T'$ such that $d(T(x),T'(x))$ is bounded.
Hence it does not change the above set.

Moreover, for any direction $(\alpha,\beta) \in \mathcal{D}(f)$ which corresponds to a closed geodesic on $S$, we understand how the associated speeds change when we change the hyperbolic metric on $S$ thanks to Lemma \ref{LemCompareAnneau}. Indeed, take two hyperbolic metrics $g_1$ and $g_2$ on $S$ and, for $i=1,2$, denote by $\rho_i(f)$ the rotation set of $f$ with respect to the hyperbolic metric $g_i$. Fix an isotopy class of loops on $S$. For $i=1,2$, denote by $\gamma_i$ the unique closed geodesic in this homotopy class and by $(\alpha_i,\beta_i)$ a lift of $\gamma_i$ to $\tilde{S}$. Then, by Lemma \ref{LemCompareAnneau},
$$\left\{ \frac{v}{\ell(\gamma_1)} \ | \ (\alpha_1,\beta_1,v) \in \rho_1(f) \right\}= \left\{ \frac{v}{\ell(\gamma_2)} \ | \ (\alpha_2,\beta_2,v) \in \rho_2(f) \right\}.$$
Note that some our results of Section~\ref{SecHomo} and the next ones only concern closed geodesics: how these results change after a metric modification is completely explicit.

For nonclosed geodesics, how speeds change in directions which correspond to nonclosed geodesics after a modification of the hyperbolic metric depends on the restriction of this metric to the geodesic. However, the general picture of these changes remains quite mysterious to the authors.
\end{remark}

For the torus case, a consequence of the quasi-convexity of fundamental domains is the convexity of rotation sets (see \cite{MR1053617}). In the case of negatively curved surfaces, the outcome is weaker: the rotation set is star-shaped with respect to 0 (Theorem~\ref{TheoStarIntro} of the introduction).

\begin{theorem}\label{TheoStar}
If $S$ is closed, then the rotation set of any $f\in\Homeo_0(S)$ is star-shaped: for any $(\alpha,\beta,v)\in \rho(f)$, and any $v'\in [0,v]$, one has $(\alpha,\beta,v')\in\rho(f)$.
\end{theorem}

As a byproduct of the proof of Theorem~\ref{TheoStar}, we obtain the realisation of rotation vectors by pieces of orbits whose extremities stay at a finite distance to the geodesic.

\begin{proposition}\label{ProprealisDist}
Suppose that $S$ is closed. Let $f\in\Homeo_0(S)$ and $(\alpha,\beta,v)\in \rho(f)$ such that $v>0$. Then, there exists a sequence $(\x_k)_{k\in \N}$ of points of $\Hy^2$, and an increasing sequence $(n_k)_{k\in\N}$ of natural numbers tending to infinity such that \eqref{EqDefRot} holds, and moreover, 
\[\max\Big( d\big(\x_k, (\alpha,\beta)\big),\, d\big(\tilde f^{n_k}(\x_k), (\alpha,\beta)\big)\Big) \le R+1+\delta,\]
where $R$ is the constant of Proposition~\ref{PropQuasiConvex}, and $\delta$ the smallest diameter of fundamental domains of $S$ in $\Hy^2$.
\end{proposition}

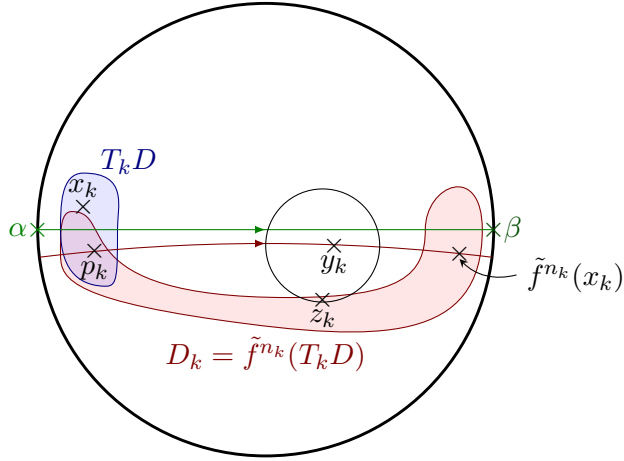
\begin{figure}
\begin{center}
\begin{tikzpicture}[scale=3.]

\draw[fill=red, draw=red!50!black, fill opacity=0.1] (-.9,-.1) .. controls +(0,-0.1) and +(-0.7,0.1).. (-.2,-.4) .. controls +(0.7,-0.1) and +(0,-0.5).. (0.95,0) .. controls +(0,0.3) and +(-0,0.2).. (0.7,0) .. controls +(0.,-0.2) and +(0.5,0).. (0.1,-.3) .. controls +(-1,0) and +(0.2,0).. (-.85,.08) .. controls +(-0.05,0) and +(0,.1).. (-.9,-.1);
\draw[fill=blue, draw=blue!50!black, fill opacity=0.1] (-.9,0) .. controls +(0,-0.2) and +(-0.1,0).. (-.7,-.25) .. controls +(0.05,0) and +(0,-.2).. (-.65,0) .. controls +(0,.1) and +(.2,0).. (-.8,.25) .. controls +(-.1,0) and +(0,.1).. (-.9,0);

\draw[color=blue!50!black] (-.6,.3) node{$T_k D$};
\draw[color=red!50!black] (0,-.55) node{$D_k = \tilde f^{n_k}(T_kD)$};

\draw[line width=1.1] (0,0) circle (1);
\hgline{-180}{0}{green!50!black}
\draw[color=green!50!black] (-1,0) node{$\times$} node[left]{$\alpha$};
\draw[color=green!30!black] (1,0) node{$\times$} node[right]{$\beta$};
\hgline{-173}{-7}{red!50!black}

\draw (-0.8,0.1) node{$\times$} node[above]{$x_k$};
\draw (-0.75,-0.095) node{$\times$} node[below]{$p_k$};
\draw (0.85,-0.105) node{$\times$};
\draw[-stealth] (1.1,-0.2) node[right]{$\tilde f^{n_k}(x_k)$} to[bend left] (0.85,-0.13);
\draw (.3,-.07) node{$\times$} node[below]{$y_k$};
\draw (.25,-.07) circle (.25);
\draw (.25,-.31) node{$\times$} node[below]{$z_k$};

\end{tikzpicture}
\caption{Proof of Theorem \ref{TheoStar}.}
\end{center}
\end{figure}

\begin{proof}[Proof of Theorem \ref{TheoStar}]
We fix a fundamental domain $D$ of $S$ in $\Hy^2$, and choose $(\alpha,\beta,v)\in \rho(f)$, with $v>0$. Let $v'\in (v/2,v)$.

By definition, there exists two sequences $x_k\in \Hy^2$ and $n_k\in\N$, with $\lim n_k = +\infty$, such that Equation \eqref{EqDefRot} holds.

Fix one point $a_0$ on the geodesic defined by $\alpha$ and $\beta$, and $\varep>0$ such that $v/2+\varep<v'$. Then, by taking a subsequence if necessary, at least one of the two sequences $d(a_0,\pi_{\alpha,\beta}(x_k))/n_k$ and  $d(a_0,\pi_{\alpha,\beta}(\tilde f^{n_k}(x_k)))/n_k$ is eventually bigger than or equal to $v/2-\varep$. By taking $f^{-1}$ instead of $f$ and applying Proposition \ref{PropRhoInv} if necessary, one can suppose that it is the second one. Moreover, by taking a subsequence again, one can suppose that $d(a_0,\pi_{\alpha,\beta}(x_k))/n_k$ is eventually smaller than or equal to $v/2+\varep$.

Let $T_k$ be a deck transformation such that $x_k \in T_k(D)$. Then, there exists $p_k\in T_k(D)$ which is a fixed point of $\tilde f$. As deck transformations are isometries, the distance $d(x_k,p_k)$ is uniformly bounded (by $\diam(D)$). In particular $p_k$ tends to $\alpha$, and the distance between $\pi_{\alpha,\beta}(p_k)$ and $\pi_{\alpha,\beta}(x_k)$ is  bounded by $\diam(D)$.

Hence, the fundamental domain $D_k=\tilde f^{n_k}(T_k D)$ contains both points $p_k = \tilde f^{n_k}(p_k)$ and $\tilde f^{n_k}(x_k)$. By Proposition \ref{PropQuasiConvex}, this fundamental domain is $R$-quasi convex for some fixed $R>0$: for any $y\in [p_k,\tilde f^{n_k}(x_k)]$, one has $B(y,R) \cap D_k \neq\emptyset$.

Let us choose $y_k\in [p_k,\tilde f^{n_k}(x_k)]$ such that $d(\pi_{\alpha,\beta}(p_k),\pi_{\alpha,\beta}(y_k)) = n_k v'$, and $z_k\in B(y_k,R) \cap D_k$. This implies that $\lim d(\pi_{\alpha,\beta}(p_k),\pi_{\alpha,\beta}(z_k))/n_k = v'$. As $\lim d(\pi_{\alpha,\beta}(p_k),a_0)/n_k \le v/2+\varep$ (because $p_k$ is at a bounded distance of $x_k$), we have $\lim d(a_0,\pi_{\alpha,\beta}(z_k))/n_k \ge v'-v/2-\varep>0$, so the sequence $z_k$ tends to $\beta$ (this is here where we need $v' > \frac{v}{2}$). 

Moreover, $\tilde f^{-n_k}(z_k)\in T_k D$ is at a bounded distance of $x_k$, so it tends to $\alpha$, and $\lim d(\pi_{\alpha,\beta}(\tilde f^{-n_k}(z_k)),\pi_{\alpha,\beta}(z_k))/n_k = v'$.

We have proved that for any $v'\in (v/2,v]$, one has $(\alpha,\beta,v')\in\rho(f)$. The theorem follows easily from an induction using this property (in fact, one only has to use this property twice: once for $f$ and once for $f^{-1}$).
\end{proof}

\begin{proof}[Proof of Proposition~\ref{ProprealisDist}]
We fix a fundamental domain $D$ of $S$ in $\Hy^2$, with minimal diameter $\delta$, and choose $(\alpha,\beta,v)\in \rho(f)$, with $v>0$. Let $(x_k)$ and $(n_k)$ be such that \eqref{EqDefRot} holds. Fix one point $a_0$ on the geodesic defined by $\alpha$ and $\beta$.

Let $T_k$ be a deck transformation such that $x_k \in T_k(D)$. Then, there exists $p_k\in T_k(D)$ which is a fixed point of $\tilde f$. The distance $d(x_k,p_k)$ is uniformly bounded (by $\diam(D)=\delta$). In particular $p_k$ tends to $\alpha$, and the distance between $\pi_{\alpha,\beta}(p_k)$ and $\pi_{\alpha,\beta}(x_k)$ is bounded by $\delta$.

Hence, the fundamental domain $D_k=\tilde f^{n_k}(T_k D)$ contains both points $p_k = \tilde f^{n_k}(p_k)$ and $\tilde f^{n_k}(x_k)$. By Proposition \ref{PropQuasiConvex}, this fundamental domain is $R$-quasi convex for some fixed $R>0$: for any $y\in [p_k,\tilde f^{n_k}(x_k)]$, one has $B(y,R) \cap D_k \neq\emptyset$.

Let us choose $y_k\in [p_k,\tilde f^{n_k}(x_k)]$ such that 
\begin{equation}\label{EqDefYk}
d\big(a_0,\pi_{\alpha,\beta}(y_k)\big) = d\big(a_0,\pi_{\alpha,\beta}(\tilde f^{n_k}(x_k))\big) - \sqrt{d\big(a_0,\pi_{\alpha,\beta}(\tilde f^{n_k}(x_k))\big)}.
\end{equation}

The following claim is a consequence of basic hyperbolic geometry.

\begin{claim}\label{ClaimHypGeom}
If $k$ is large enough, then $d\big(y_k,\pi_{\alpha,\beta}(y_k)\big) \le 1/2$.
\end{claim}

\begin{proof}[Proof of the claim]
Let us consider the half-plane model of $\Hy^2$, such that $\alpha = \infty$ and $\beta=0$ (in this case, $(\alpha, \beta)$ is the positive imaginary axis). The set of points at distance $1/2$ of $(\alpha, \beta)$ is the union of two Euclidean half-lines starting at $0=\beta$, making angle $\theta$ with the imaginary axis, with $|\sin\theta| = \operatorname{tanh}( 1/2)$. As the choice of $a_0$ was arbitrary, one can choose $a_0 = i$.

\begin{figure}[ht]
\begin{center}
\includegraphics[scale=1]{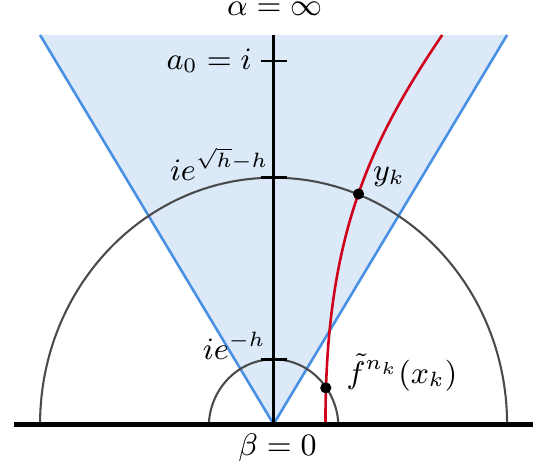}
\caption{Proof of Claim \ref{ClaimHypGeom}.}
\end{center}
\end{figure}

Denote $h=d\big(a_0,\pi_{\alpha,\beta}(\tilde f^{n_k}(x_k))\big)$. Then, $\pi_{a,b}(y_k)$ can be computed in terms of $d$: the hyperbolic distance satisfies, for $p,q>0$, $d(ip, iq) = |\log p - \log q|$. Hence, by \eqref{EqDefYk}
\[\pi_{\alpha,\beta}(\tilde f^{n_k}(x_k)) = i e^{-h} \quad \text{and} \quad \pi_{\alpha,\beta}(y_k) = i e^{-h+\sqrt h} = i e^{\sqrt h}e^{-h}.\]
Hence, the point $y_k$ lies on the Euclidean circle centred at 0 with radius $e^{\sqrt h}e^{-h}$.

Fix $A>0$, then for any $k$ large enough the geodesic passing through $p_k$ and $\tilde f^{n_k}(x_k)$ is either a Euclidean circle with one extremity inside $[-e^{-h}, e^{-h}]$ and the other outside $[-A,A]$ or a line which is orthogonal to the real axis and with one extremity inside $[-e^{-h},e^{-h}]$. Simple Euclidean geometry shows that if $A$ and $h$ are large enough, then the intersection $y_k$ of this geodesic with the Euclidean circle centred at 0 with radius $e^{\sqrt h}e^{-h}$ is inside the Euclidean cone made of the points at distance at most $1/2$ of $(\alpha,\beta)$. Indeed, the angle between the imaginary axis and the straight line between $0$ and $y_k$ tends to $0$ when $A \rightarrow +\infty$ and $h \rightarrow +\infty$. 

Hence, for any $k$ large enough, $d\big(y_k,\pi_{\alpha,\beta}(y_k)\big) \le 1/2$.
\end{proof}

Pick some $z_k\in B(y_k,R) \cap D_k$. This implies that $d(z_k, (\alpha,\beta)) \le R +1/2$. By the definition of $y_k$, one has $\lim_{k\to +\infty} y_k = \beta$, so we also have $\lim_{k\to +\infty} z_k = \beta$. Moreover, $\tilde f^{-n_k}(z_k) \in T_kD$, hence $\lim_{k\to +\infty}\tilde f^{-n_k}(z_k) = \alpha$, and 
\[ \frac{d\Big(\pi_{\alpha,\beta}\big(f^{-n_k}(z_k)\big)\, ,\ \pi_{\alpha,\beta}( z_k)\Big)}{n_k} = v.\]

At this point, we have no information about the distance between $f^{-n_k}(z_k)$ and the geodesic $(\alpha,\beta)$, but we know that $d(f^{-n_k}(z_k), x_k)\le\delta$. To solve this issue, it suffices to make the same argument a second time, with $f^{-1}$ instead of $f$ and $z_k$ instead of $x_k$. We then get a point $\x_k$ with the desired properties.
\end{proof}

\section{Realisation of rotation vectors for closed geodesics}\label{Secrealis}

We now state a realisation result, which is a direct consequence of already known realisation results for annulus homeomorphisms \cite{MR967632,MR2217051}.

\begin{proposition}\label{PropRealPtExtRat}
Let $f\in \Homeo_0(S)$ and take $\alpha \neq \beta \in\partial\Hy^2$ which define a closed geodesic in $S$, of length $\ell>0$. Suppose that $(\alpha,\beta,v_0) \in\rho(f)$, with $v_0 = \ell p/q $ ($p/q$ irreducible), and that one of the following conditions holds:
\begin{enumerate}[label=(\roman*)]
\item $v_0$ is maximal, i.e. $v_0 = \max\big\{v\in\R_+ \mid (\alpha,\beta,v)\in\rho(f)\big\}$;
\item $f$ is chain transitive;
\item there exists $(\alpha_1,\beta_1,v_1), (\alpha_2,\beta_2,v_2) \in \rho(f)$, with $v_1,v_2>0$, such that $(\alpha_1,\beta_1)$ crosses $(\alpha,\beta)$ positively and $(\alpha_2,\beta_2)$ crosses $(\alpha,\beta)$ negatively.
\end{enumerate}
Then $f$ possesses a periodic point of period $q$ and rotation vector $(\alpha,\beta,v_0)$.
\end{proposition}

We will need the notion of rotation set of an annulus homeomorphism, which we recall now. Let $g : \mathbb{S}^1 \times [-1,1]= \mathbb{R} / \mathbb{Z} \times [-1,1] \rightarrow \mathbb{R} / \mathbb{Z} \times [-1,1]$ be a homeomorphism which is isotopic to the identity. Let us denote by $\tilde{g} : \mathbb{R} \times [-1,1] \rightarrow \mathbb{R} \times [-1,1]$ one of its lifts and by $p_1 : \mathbb{R} \times [-1,1] \rightarrow \mathbb{R}$ the projection. The rotation set $\rho(\tilde{g})$ of $\tilde{g}$ is the set of limits of sequences of the form
$$\left( \frac{p_1(\tilde{g}^{n_k}(x_k))-p_1(x_k)}{n_k} \right)$$
with $n_{k} \rightarrow +\infty$ and $x_{k} \in \mathbb{R} \times [-1,1]$.

We will need another way to see this rotation set of $\tilde{g}$. Denote by $\mathcal{M}(g)$ the set of $g$-invariant probability measures. Then the rotation set of $\tilde{g}$ is also
$$ \left\{ \int_{\mathbb{S}^1 \times [-1,1]} (p_1(\tilde{g}(\tilde{x}))-p_1(\tilde{x})) d\mu(x) \ | \ \mu \in \mathcal{M}(g) \right\}.$$
Indeed, denote by $\rho_{\mu}(\tilde{g})$ this second rotation set. As sequences of probability measures of the form
$$ \left( \frac{1}{n_k} \sum_{i=0}^{n_k-1} \delta_{g^{i}(x_k)} \right)_{k \geq 0}$$
have a limit point, we obtain that $\rho(\tilde{g}) \subset \rho_{\mu}(\tilde{g})$. The other inclusion is the consequence of the connexity of $\rho(\tilde{g})$ and of the following facts.
\begin{enumerate}
\item Extremal points of $\rho_{\mu}(\tilde{g})$ are realised by ergodic measures.
\item The Birkhoff ergodic theorem applied to $x \mapsto p_1(\tilde{g}(\tilde{x}))-p_1(\tilde{x})$ with those ergodic measures implies that those extremal points belong to $\rho(\tilde{g})$.
\end{enumerate} 
We use notation from the proposition. Let $T$ be the deck transformation associated to the closed geodesic $(\alpha,\beta)$. Let $\check f : \Hy^2/T \to \Hy^2/T$ be the quotient map of $\tilde f$; as $\tilde f$ extends by identity to $\partial\Hy^2$, this map $\check f$ can be seen as a map of the closed annulus $\Sp^1\times [-1,1]$, homotopic to the identity. Hence, it has a well defined rotation set $\rho(\check f)\subset\R$.

\begin{lemma}\label{LemCompareAnneau}
Under the above hypotheses,
\[\ell \rho(\check f) = \big\{v\in\R \mid (\alpha,\beta,v)\in\rho(f) \text{ or }  (\beta,\alpha,-v)\in\rho(f)\big\}.\]
\end{lemma}

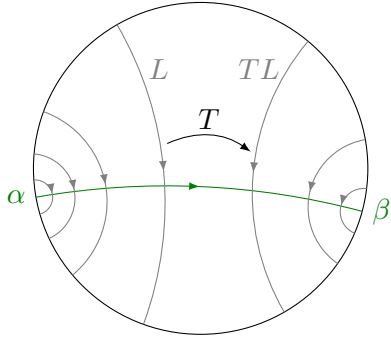
\begin{figure}
\begin{center}
\begin{tikzpicture}[scale=2.2]

\draw (0,0) circle (1);
\hgline{-170}{-15}{green!50!black}
\hgline{-240}{-110}{gray}
\hgline{-200}{-140}{gray}
\hgline{-185}{-155}{gray}
\hgline{-176}{-164}{gray}
\hgline{50}{-60}{gray}
\hgline{10}{-35}{gray}
\hgline{-7}{-23}{gray}

\draw[->,>=latex] (-.2,.15) to[bend left] (.3,.1);
\draw (0.05,.3) node{$T$};
\draw[color=green!50!black] (-170:1) node[left]{$\alpha$};
\draw[color=green!50!black] (-15:1) node[right]{$\beta$};
\draw[color=gray] (-.25,.6) node{$L$};
\draw[color=gray] (.35,.6) node{$TL$};

\end{tikzpicture}
\caption{\label{FigReal}Realisation of periodic points.}
\end{center}
\end{figure}

\begin{proof}[Proof of Lemma \ref{LemCompareAnneau}]
Let $L$ be a geodesic of $\Hy^2$ orthogonal to $(\alpha,\beta)$, and $\check D\subset \Hy^2$ the fundamental domain of the open annulus $\Hy^2/T$, consisting of the points that are between $L$ and $TL$ (see Figure~\ref{FigReal}). For any $x\in\Hy^2$, we denote by $i_x$ the integer satisfying $x \in T^{i_x}(\check D)$. As $L$ is orthogonal to $(\alpha,\beta)$, for any $x,y\in\Hy^2$,
\begin{equation}\label{EqProjAnneau}
\ell(|i_x-i_y| - 1) \le d\big(\pi_{\alpha,\beta}(x)\, ,\ \pi_{\alpha,\beta}(y)\big) \le \ell(|i_x-i_y| + 1).
\end{equation}

Suppose that $(\alpha,\beta,v)\in\rho(f)$. Then, there exists $(x_k)$ and $(n_k)$ such that Equation~\eqref{EqDefRot} holds. Applying Equation \eqref{EqProjAnneau} to $x=x_k$ and $y=\tilde f^{n_k}(x_k)$, one gets that $v / \ell\in\rho(\check f)$. Conversely, any sequence $\check x_k \in\Hy^2/T$ realising a rotation number $v'\in \rho(\check f)$ lifts to a sequence $x_k\in\Hy^2$ with rotation vector of the form $(\alpha,\beta,\ell v')$ if $v'>0$, or $(\beta,\alpha,-\ell v')$ if $v'<0$.
\end{proof}

\begin{proof}[Proof of Proposition \ref{PropRealPtExtRat}]
By applying Lemma \ref{LemCompareAnneau}, and lifting the points to $\Hy^2$, we are reduced to prove the proposition in the case of the closed annulus. Point (i) comes from the fact that any extremal rational rotation number is realised by some periodic orbit (\cite[Theorem 3.5]{MR967632}), point (ii) from \cite[Theorem 2.2]{MR967632}, and point (iii) is a direct consequence of a generalization of Poincar\'e-Birkhoff theorem \cite[Theorem 9.1]{MR2217051}.
\end{proof}

The following proposition means that any extremal point of $\rho(f)$ in a closed geodesic direction is realised by an orbit which stays at sublinear distance from the geodesic line. It uses results from \cite{MR2846925} and \cite{MR1037109}.

\begin{proposition} \label{sublindistance}
Let $(\alpha,\beta)$ be a geodesic line of $\Hy^2$ which projects to a closed geodesic $\gamma$. Suppose $(\alpha,\beta,v)$ is an extremal point of $\rho(\tilde{f})$, with $v \neq 0$. Then there exists a point $\tilde{x}$ in $\tilde{S}=\Hy^2$ such that
$$\left\{ \begin{array}{rcl}
\displaystyle \lim_{n \rightarrow +\infty} \frac{1}{n}d(\tilde{f}^n(\tilde{x}),\tilde{x}) & = & v \\
\displaystyle \lim_{n \rightarrow +\infty}\tilde{f}^n (\tilde{x}) & = & \beta \\
\displaystyle \lim_{n \rightarrow +\infty} \frac{1}{n} d(\tilde{f}^{n}(\tilde{x}),\pi_{\alpha,\beta}(\tilde{f}^{n}(\tilde{x}))) & = & 0
\end{array}
\right. $$

Moreover, either the orbit under $\tilde{f}$ of $\tilde{x}$ stays within a bounded distance of the geodesic line $(\alpha,\beta)$ and its closure (in $\overline{\Hy^2}$) does not contain any fixed point of $\tilde{f}$ or, for any rational number $0<r<\frac{v}{\ell(\gamma)}$, there exist  a periodic orbit realising the rotation vector $(\alpha,\beta,r \ell(\gamma))$.
\end{proposition}

To explain why the vector $(\alpha,\beta,v) \in \rho(f)$ is realised by the orbit of $\tilde{x}$, let $T$ be the deck transformation associated to $(\alpha,\beta)$. Fix any sequence of integers $(k_{n})_{ n \geq 0}$ such that 
$$\left\{ \begin{array}{rcl}
\displaystyle \lim_{n \rightarrow +\infty} k_{n} & = & +\infty \\
\displaystyle \lim_{n \rightarrow +\infty} \frac{k_{n}}{n} & = & 0.
\end{array} \right.$$

Then
$$ \lim_{n \rightarrow +\infty} \frac{1}{n}d\Big(\pi_{\alpha,\beta}\big(T^{-k_n}(\tilde{x})\big),\pi_{\alpha,\beta}\big(\tilde{f}^n(T^{-k_{n}}(\tilde{x}))\big)\Big)
= \lim_{n \rightarrow +\infty} \frac{1}{n}d(\pi_{\alpha,\beta}(\tilde{x}),\pi_{\alpha,\beta}(\tilde{f}^n(\tilde{x})))=v$$
and
$$\left\{ \begin{array}{rcl}
\displaystyle \lim_{n \rightarrow +\infty} T^{-k_{n}}(\tilde{x})=\alpha \\
\displaystyle \lim_{n \rightarrow +\infty} \tilde{f}^{n}(T^{-k_{n}}(\tilde{x}))=\beta.
\end{array}
\right.$$
This means that the rotation vector $(\alpha,\beta,v)$ is realised by the orbit of $\tilde{x}$.

\begin{proof}[Proof of Proposition \ref{sublindistance}]
During this proof, we will need the following elementary result of hyperbolic geometry.

\begin{claim} \label{compareproj2}
For any $\alpha_1,\alpha_2,\beta\in\partial\Hy^2$ such that $\alpha_1\neq\beta$ and $\alpha_2\neq\beta$, we have :
$$\lim_{\substack{ y \in \Hy^2\\ y \rightarrow \beta}} d\big(\pi_{\alpha_{2},\beta}(y),\pi_{\alpha_{1},\beta}(y)\big)=0.$$
\end{claim}

\begin{proof}
We see $\Hy^2$ as the upper half-plane in $\R^2$ so that $\partial \Hy^2$ is the union of the line $\R \times \left\{ 0 \right\}$ with the point $\infty$ at infinity. As, for any isometry $\sigma$ of $\Hy^2$,
$$\begin{array}{rcl}
d\big(\pi_{\alpha_{2},\beta}(y),\pi_{\alpha_{1},\beta}(y)\big) & = & d\big(\sigma(\pi_{\alpha_{2},\beta}(y)),\sigma(\pi_{\alpha_{1},\beta}(y))\big) \\
 & = & d\big(\pi_{\sigma(\alpha_{2}),\sigma(\beta)}(i(y)),\pi_{\sigma(\alpha_{1}),\sigma(\beta)}(i(y))\big)
\end{array}$$
and as the group of isometries of $\Hy^2$ acts transitively on the boundary $\partial \Hy^2$, we can suppose that $\beta=\infty$ and that $\alpha_{1}$ and $\alpha_{2}$ are two points of $\R \times \left\{ 0 \right\}$. To carry out this proof, we will use the distance $d_{Euc}$ on $\Hy^2$ which is induced by the Euclidean distance on $\R^2$. In this model, the geodesic lines $(\alpha_{1},\beta)$ and $(\alpha_{2},\beta)$ are respectively the sets $\left\{ \alpha_{1} \right\} \times \R^{*}_{+}$ and $\left\{ \alpha_{2} \right\} \times \R^{*}_{+}$.

\begin{figure}[ht]
\begin{center}
\includegraphics[scale=1]{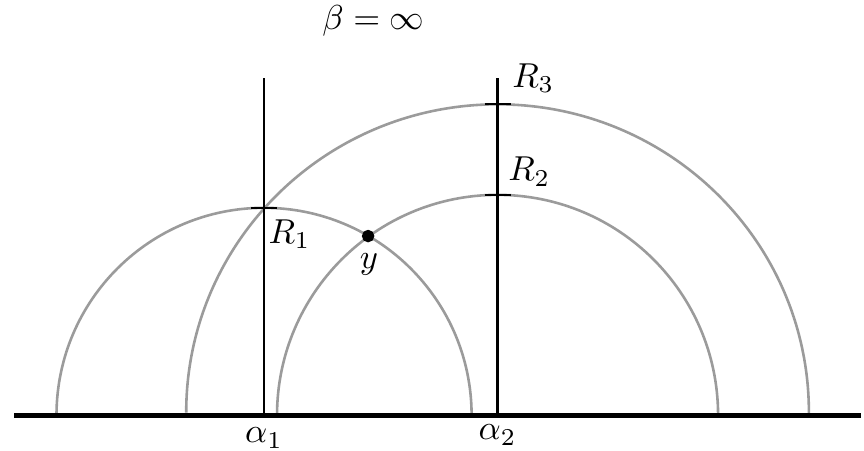}
\caption{Proof of Claim \ref{compareproj2}.}
\end{center}
\end{figure}

Fix $y \in \Hy^{2}$. The geodesic line passing through $y$ and orthogonal to $(\alpha_{1},\beta)$ is the intersection of $\Hy^2$ with the Euclidean circle of center $\alpha_{1}$ and of radius $R_{1}=d_{Euc}(y,\alpha_{1})$. Hence $\pi_{\alpha_{1},\beta}(y)=(\alpha_{1},R_{1})$. In the same way, if $R_{2}=d_{Euc}(\alpha_{2},y)$ and $R_{3}=d_{Euc}(\pi_{\alpha_{1},\beta}(y),\alpha_{2})$, then $\pi_{\alpha_{2},\beta}(y)=(\alpha_{2},R_{2})$ and $\pi_{\alpha_{2},\beta}\circ \pi_{\alpha_{1},\beta}(y)=(\alpha_{2},R_{3})$.

We have
$$d\big(\pi_{\alpha_{1},\beta}(y),\pi_{\alpha_{2},\beta}(y)\big)
\leq d\big(\pi_{\alpha_{1},\beta}(y),\pi_{\alpha_{2},\beta}(\pi_{\alpha_{1},\beta}(y))\big)
+d\big(\pi_{\alpha_{2},\beta}(\pi_{\alpha_{1},\beta}(y)), \pi_{\alpha_{2},\beta}(y)\big).$$
When the point $y$ tends to $\beta=\infty$, the distance $d(\pi_{\alpha_{1},\beta}(y),\pi_{\alpha_{2},\beta}(\pi_{\alpha_{1},\beta}(y)))$, which is the (hyperbolic) distance between the point $\pi_{\alpha_{1},\beta}(y)$ and the geodesic line $(\alpha_{2},\beta)$, tends to $0$ as the point $\pi_{\alpha_{1},\beta}(y)$ remains on the geodesic line $(\alpha_{1},\beta)$.

To prove the claim, it suffices to prove that, when the point $y$ tends to $\infty$, the distance $d(\pi_{\alpha_{2},\beta}(\pi_{\alpha_{1},\beta}(y)), \pi_{\alpha_{2},\beta}(y))$ tends to $0$. To do this, it suffices to prove that the Euclidean distance $d_{Euc}(\pi_{\alpha_{2},\beta}(\pi_{\alpha_{1},\beta}(y)), \pi_{\alpha_{2},\beta}(y))$ remains bounded. We have
$$\begin{array}{rcl}
d_{Euc}\big(\pi_{\alpha_{2},\beta}(\pi_{\alpha_{1},\beta}(y)), \pi_{\alpha_{2},\beta}(y)\big) & = & | R_{2}-R_{3}| \\
 & \leq & |R_{2}-R_{1}| + |R_{1}-R_{3}|.
 \end{array}
$$
However,
$$\left\{ \begin{array}{l}
|R_{2}-R_{1}| =|d_{Euc}(y,\alpha_{2})-d_{Euc}(y,\alpha_{1})| \leq d_{Euc}(\alpha_{2},\alpha_{1}) \\
|R_{3}-R_{1}| = |d_{Euc}(\pi_{\alpha_{1},\beta}(y),\alpha_{2})-d_{Euc}(\pi_{\alpha_{1},\beta}(y),\alpha_{1})| \leq d_{Euc}(\alpha_{1},\alpha_{2})
\end{array}
\right.
$$
and
$$d_{Euc}\big(\pi_{\alpha_{2},\beta}(\pi_{\alpha_{1},\beta}(y)), \pi_{\alpha_{2},\beta}(y)\big) \leq 2 d_{Euc}(\alpha_{1},\alpha_{2}).$$
\end{proof}

Let $T$ be the deck transformation associated to $(\alpha,\beta)$. Let $A$ be the open annulus $\tilde{S} / \langle T\rangle =\Hy^2 /\langle T\rangle $ and $\overline{A}$ be the closed annulus $\overline{\Hy^2}-\left\{ \alpha,\beta \right\} / \langle  T \rangle $. Denote by $\check{f}$ the lift of $f$ to $\overline{A}$ induced by $\tilde{f}$ and recall that, as $f$ is isotopic to the identity, the homeomorphism $\check{f}$ pointwise fixes the boundary of the closed annulus $\overline{A}$. Fix coordinates on $\overline{A}$ so that we can make the following identifications:
$$\left\{ \begin{array}{rcl}
\overline{\Hy^2}-\left\{ \alpha, \beta \right\}& = & [-1,1] \times \R \\
(\alpha,\beta) & = & \left\{ 0 \right\} \times \R \\
\overline{A} & = & [-1,1] \times \R/\Z
\end{array} \right.$$
and the projection $p_{2}: \overline{\Hy^2}- \left\{ \alpha,\beta \right\}=[-1,1] \times \R \rightarrow \R$ on the second coordinate is equal to $\pi_{\alpha,\beta}$. 

By Lemma \ref{LemCompareAnneau}, the point $\frac{v}{\ell(\gamma)}$ is an extremal point of $\rho(\check{f})$. Hence there exists an $\check{f}$-invariant ergodic probability measure $\check{\mu}$ on $\overline{A}$ such that
$$ \int_{\overline{A}} p_{2}\big(\tilde{f}(\tilde{x})-\tilde{x}\big)\,\mathrm{d}\check{\mu}(\check{x})= \frac{v}{\ell(\gamma)}.$$
As the homeomorphism $\check{f}$ pointwise fixes the boundary $\partial \overline{A}$ of $\overline{A}$, observe that $\check{\mu}(\partial A)=0$. Indeed, otherwise, we would have $\check{\mu}=\check{\mu}(A) \check{\mu}_1+ \check{\mu}(\partial A) \check{\mu}_2$, where $\check{\mu}_1= \frac{1}{\check{\mu}(A)} \check{\mu}(A \cap .)$ and $\check{\mu}_2= \frac{1}{\check{\mu}(\partial \overline{A})} \check{\mu}(\partial \overline{A} \cap .)$, which contradicts the ergodicity of $\check{\mu}$ (observe that the definition of $\check{\mu}$ imposes that $\check{\mu}(A) >0$).

By Birkhoff ergodic theorem, the subset $C$ of $A$ consisting of points $\check{x}$ such that 
$$ \lim_{n \rightarrow +\infty} \frac{1}{n}p_{2}(\tilde{f}^{n}(\tilde{x})-\tilde{x})=\frac{v}{\ell(\gamma)},$$
where the point $\tilde{x} \in \Hy^2$ is any lift of $\check{x}$, has full $\check{\mu}$-measure.

Going back to the hyperbolic distance on $A$, this means that, for any point $\tilde{x}$ of $\tilde{S}$ which projects to a point of $C$,
$$\left\{
\begin{array}{rcl}
\displaystyle \lim_{n \rightarrow +\infty} \tilde{f}^n(\tilde{x})& = & \beta \\
 \displaystyle \lim_{n \rightarrow +\infty} \frac{1}{n}d\big(\pi_{\alpha,\beta}(\tilde{f}^{n}(\tilde{x})),\tilde{x}\big) & = & v.
 \end{array}
 \right.$$
Let $\check{\pi}$ be the covering map $A \rightarrow S$ and $\mu=\check{\pi}_{*} \check{\mu}$. The probability measure $\mu$ is $f$-invariant as $\check{\mu}$ is $\check{f}$-invariant. The following lemma is a consequence of Corollary 21 of \cite{MR2846925}. This statement is actually valid for any $f$-invariant ergodic probability measure.

\begin{lemma} \label{Lessa}
There exists a full $\mu$-measure subset $B$ of $S$ such that, for any point $\tilde{x}$ in $\tilde{\pi}^{-1}(B)$, there exists a geodesic line $(\alpha_{\tilde{x}},\beta_{\tilde{x}})$ such that
$$\lim_{n \rightarrow +\infty} \frac{1}{n}d\big(\tilde{f}^n(\tilde{x}),(\alpha_{\tilde{x}},\beta_{\tilde{x}})\big)=0.$$
\end{lemma}

Now take any point $\tilde{x}$ in $\Hy^2$ which is a lift of a point in $\check{\pi}^{-1}(B) \cap C$ (this set has full $\check{\mu}$-measure and is hence nonempty). As the point $\tilde{x}$ is a lift of a point of $C$, then
$$ \left\{ \begin{array}{rcl}
\displaystyle \lim_{n \rightarrow +\infty} \frac{1}{n} d\big(\pi_{\alpha,\beta}(\tilde{f}^{n}(\tilde{x})),\pi_{\alpha,\beta}(\tilde{x})\big) & = & v \\
\displaystyle \lim_{n \rightarrow +\infty} \tilde{f}^{n}(\tilde{x})=\beta
\end{array} \right.
$$
so that $\beta_{\tilde{x}}=\beta$. Moreover
$$
d\big(\tilde{f}^{n}(\tilde{x}),\pi_{\alpha,\beta}(\tilde{f}^{n}(\tilde{x}))\big)  \leq  d\big(\tilde{f}^{n}(\tilde{x}),\pi_{\alpha_{\tilde{x}},\beta}(\tilde{f}^{n}(\tilde{x}))\big) + d\big(\pi_{\alpha_{\tilde{x}},\beta}(\tilde{f}^{n}(\tilde{x})),\pi_{\alpha,\beta}(\tilde{f}^{n}(\tilde{x}))\big).$$
However, by Lemma \ref{Lessa},
$$ \lim_{n \rightarrow +\infty} \frac{1}{n} d\big(\tilde{f}^{n}(\tilde{x}),\pi_{\alpha_{\tilde{x}},\beta}(\tilde{f}^{n}(\tilde{x}))\big)=0$$
and, by Claim \ref{compareproj2},
$$ \lim_{n \rightarrow +\infty} d\big(\pi_{\alpha_{\tilde{x}},\beta}(\tilde{f}^{n}(\tilde{x})),\pi_{\alpha,\beta}(\tilde{f}^{n}(\tilde{x}))\big)=0.$$
Hence
$$ \lim_{n \rightarrow +\infty} \frac{1}{n}d\big(\tilde{f}^{n}(\tilde{x}),\pi_{\alpha,\beta}(\tilde{f}^{n}(\tilde{x}))\big)=0.$$ 

The last part of the proposition is a consequence of a result by Handel (see \cite{MR1037109}). If the orbit of $\tilde{x}$ does not stay within a bounded distance of the geodesic $(\alpha,\beta)$, then the closure of this orbit meets the boundary of the closed annulus $\overline{A}$. However, recall that this boundary is fixed under $\check{f}$, so that a fixed point of $\tilde{f}$ lies in the closure of the orbit of $\tilde{x}$, which in particular does not have the same rotation number as $\tilde{x}$. Then, by the proof of Lemma 2.1 p.343 in \cite{MR1037109} and by Lemma \ref{LemCompareAnneau}, for any rational number $0<r < \ell(\gamma)$, there exist periodic orbits for $f$ with rotation number $(\alpha,\beta, r \ell(\gamma))$.
\end{proof}

See Example~\ref{ExNonDenombrable} for an example of a homeomorphism of the genus 2 closed surface with an ergodic probability measure $\mu$ for which an uncountable set of geodesics is necessary to describe the rotation set of $\mu$ almost every point.

\section{Creation of new rotation vectors: elementary results}\label{SecCreat}

In this section, we state forcing results about rotation vectors: the existence of orbits with nontrivial rotation vectors, whose associated geodesics of $\Hy^2$ cross, force the existence of other rotation vectors (and hence other orbits, with different rotation vectors). The two results we get, Propositions~\ref{newrot} and \ref{newrot2}, are heavily inspired by Le Calvez-Tal's fundamental proposition \cite[Proposition 20]{MR3787834}, although they use only basic plane topology (and in particular, no Brouwer-Le Calvez plane dynamical foliation, see Section~\ref{SubSecForcing} for some results of this theory). Note that similar arguments without the use of Brouwer-Le Calvez plane dynamical foliation appeared recently in \cite{PatriceRecent}.

The first proposition concerns geodesics of the surface with auto-intersection (Proposition~\ref{newrotIntro} of the introduction), and the second one treats the general case (with weaker conclusions).

\begin{proposition} \label{newrot}
Let $(\alpha,\beta)$ be a geodesic line of $\Hy^{2}$ and $T$ a nontrivial deck transformation such that $(\alpha,\beta) \cap T^{-1}(\alpha, \beta) \neq \emptyset$. Denote $\{p_0\} = (\alpha,\beta) \cap T^{-1}(\alpha,\beta)$, and suppose that $Tp_{0}\in(\alpha,\beta)$ is such that for the natural order on $(\alpha,\beta)$, the sequence $(\alpha,p_{0},Tp_{0},\beta)$ is increasing.

Suppose there exists $v >0$ such that $(\alpha,\beta,v) \in \rho(f)$. Then $(T\alpha,\beta,v) \in \rho(f)$.
\end{proposition}

Of course, under the hypothesis of this proposition we also deduce that $(\alpha,T^{-1}(\beta), v) \in \rho(f)$.

If the sequence $(\alpha,Tp_0,p_0,\beta)$ is increasing instead of $(\alpha,p_{0},Tp_{0},\beta)$, then we can apply Proposition \ref{newrot} to the homeomorphism $f^{-1}$ and use Proposition \ref{PropRhoInv} to obtain that $(\alpha,T\beta,v) \in \rho(f)$ and $(T^{-1} \alpha, \beta,v) \in \rho(f)$ when $(\alpha,\beta,v) \in \rho(f)$.

\begin{figure}[ht]
\begin{center}
\includegraphics[scale=1]{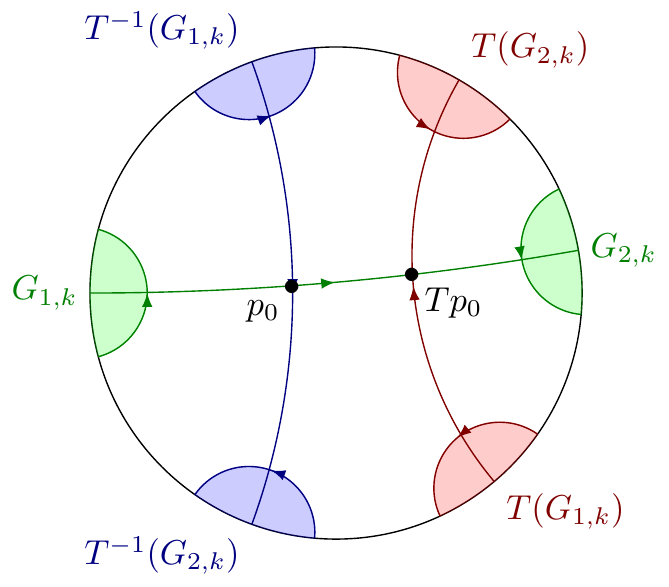}
\caption{\label{FigPropNotransverseint3}Configuration of Proposition \ref{newrot}.}
\end{center}
\end{figure}

\begin{proof}
By definition, there exist a sequence $(x_{k})_{ k \in \N}$ of points in $\Hy^{2}$ and a sequence $(n_{k})_{ k \in \N}$ of integers such that
$$ \left\{
\begin{array}{l}
\displaystyle \lim_{k \rightarrow +\infty} x_{k}  =  \alpha \\
\displaystyle \lim_{ k \rightarrow +\infty} \tilde{f}^{n_{k}}(x_{k})  =  \beta \\
\displaystyle \lim_{ k \rightarrow +\infty} \frac{d(\pi_{\alpha,\beta}(\tilde{f}^{n_{k}}(x_{k})),\pi_{\alpha,\beta}(x_{k}))}{n_{k}}  =  v.
\end{array}
\right.
$$ 
For any $k \geq 0$, denote by $G_{1,k}$ (respectively $G_{2,k}$) the unique geodesic line passing through $x_{k}$ (respectively $\tilde{f}^{n_{k}}(x_{k})$) which is orthogonal to $(\alpha,\beta)$ (see Figure \ref{FigPropNotransverseint3}). Extracting a subsequence if necessary, we can suppose that both sequences
\[\left(\frac{d(\pi_{\alpha,\beta}(x_{k}),\,\pi_{\alpha,\beta}(p_0))}{n_{k}}\right)_{k}
\qquad \text{and}\qquad
\left(\frac{d(\pi_{\alpha,\beta}(p_0),\, \pi_{\alpha,\beta}(\tilde{f}^{n_{k}}(x_{k})))}{n_{k}}\right)_{k}\]
converge with respective limits $v'_{1}$ and $v'_{2}$. Observe that $v'_{1}+v'_{2}=v$.

\begin{lemma} \label{intersection}
For any large enough $k$,
$$\tilde{f}^{n_{k}}(T(G_{1,k})) \cap G_{2,k} \neq \emptyset.$$
\end{lemma}

\begin{proof}
If $k$ is large enough, the sets $T^i(G_{1,k})$, for $i \in \Z$, are pairwise disjoint  and so are the sets $T^j(G_{2,k})$, for $j \in \Z$. Each set $G_{1,k}$ defines an interval $I_{1,k}$ on $\partial\Hy^2$, as the connected component of $\partial\Hy^2 \setminus \overline{G_{1,k}}$ containing $\alpha$. Similarly, the set $G_{2,k}$ defines an interval $I_{2,k}$ on $\partial\Hy^2$, as the connected component of $\partial\Hy^2 \setminus \overline{G_{2,k}}$ containing $\beta$. There is an orientation on $\partial\Hy^2$ such that the following intervals are all ordered positively:
\begin{equation}\label{OrdreInterval}
I_{1,k},\ TI_{1,k},\ I_{2,k},\ TI_{2,k}.
\end{equation}

These orientations of the $I_{j,k}$'s induce orientations of the $G_{j,k}$'s. For now we fix a large enough $k$ so that the above properties hold.

Suppose for a contradiction that $\tilde{f}^{n_{k}}(G_{1,k}) \cap T^{-1}G_{2,k}$ is empty.

Let us parametrize the oriented geodesic $G_{1,k}$ by the arc length. Let $t\in\R$ be such that 
\[\begin{cases}
\tilde f^{n_k}(G_{1,k}|_{(-\infty,t)}) \cap T^i G_{2,k} = \emptyset\quad &
\text{for all } i\ge 0\\
\tilde f^{n_k}(G_{1,k}|_{(-\infty,t]}) \cap T^{i_0} G_{2,k} \neq\emptyset \quad &
\text{for some } i_0\ge 0.
\end{cases}\]
Such a $t$ exists as the intersection $\tilde{f}^{n_{k}}(G_{1,k}) \cap G_{2,k}$ is nonempty and as $\tilde f^{n_k}(G_{1,k})\cap T^iG_{2,k}=\emptyset$ for any $i$ large enough. Remark that the number $i_0$ satisfying this property is unique.

Let $t'\in\R$ be such that $\tilde f^{n_k}(G_{1,k}|_{(-\infty,t]}) \cap T^{i_0} G_{2,k}|_{(-\infty,t']}$ is reduced to the point $\tilde f^{n_k}(G_{1,k}(t))=T^{i_0}G_{2,k}(t')$ , and denote by $\delta$ the path which is the concatenation of $\tilde f^{n_k}(G_{1,k}|_{(-\infty,t]})$ and $ T^{i_0} G_{2,k}|_{(-\infty,t']}$. It is a path linking $\partial\Hy^2$ to $\partial\Hy^2$. We now prove that this path is disjoint from its translate by $T$. Indeed, the fact that the $T^iG_{1,k}$ and $T^j G_{2,k}$ are pairwise disjoint reduces the possible intersections to
\[\tilde f^{n_k}(G_{1,k}|_{(-\infty,t]}) \cap  T^{i_0+1} G_{2,k}|_{(-\infty,t']} \]
or 
\[ \tilde f^{n_k}(TG_{1,k}|_{(-\infty,t]}) \cap  T^{i_0} G_{2,k}|_{(-\infty,t']}= T \Big( \tilde f^{n_k}(G_{1,k}|_{(-\infty,t]}) \cap  T^{i_0-1} G_{2,k}|_{(-\infty,t']}\Big),\]
but these intersections are empty by uniqueness of $i_0$, and the fact that by contradiction hypothesis, $\tilde f^{n_k}(G_{1,k}) \cap  T^{-1} G_{2,k} = \emptyset$ (to treat the case $i_0=0$).

But, by the ordering of the intervals given by \eqref{OrdreInterval} (this is where we use the order on points $\alpha,p_0,Tp_0,\beta$), if $i_0\ge 0$, then the two endpoints of $T\delta$ lie in different connected components of $\Hy^2\setminus \delta$, a contradiction.
\end{proof}

We now prove that $(T(\alpha),\beta,v) \in \rho(f)$.
By Lemma \ref{intersection}, for any large enough $k$,
$$\tilde{f}^{n_{k}}(TG_{1,k}) \cap G_{2,k} \neq \emptyset.$$

Hence there exists a sequence $(y_{k})_{k}$ of points of $\Hy^{2}$ such that, for any $k$
\begin{enumerate}
\item $y_{k} \in TG_{1,k}$.
\item $\tilde{f}^{n_{k}}(y_{k}) \in G_{2,k}$.
\end{enumerate}
Observe that the sequence of points $(y_{k})$ converges to the point $T(\alpha)$ and that the sequence of points $\tilde{f}^{n_{k}}(y_{k})$ converges to $\beta$.

To prove Proposition \ref{newrot}, we need the following hyperbolic geometry lemma.

\begin{lemma} \label{compareproj}
Let $\alpha_{1}, \alpha_{2}, \beta_{1}, \beta_{2}$ be pairwise distinct points of $\partial \Hy^{2}$. For any point $p_{0}$ of $\Hy^{2}$,
$$ \lim_{\substack{y \in \Hy^2\\y \rightarrow \beta_{1}}}\  \frac{d\big(\pi_{\alpha_{1},\beta_{1}}(y),\pi_{\alpha_{1},\beta_{1}}(p_{0})\big)}{d\big(\pi_{\alpha_{2},\beta_{1}}(y),\pi_{\alpha_{2},\beta_{1}}(p_{0})\big)}
\quad = \quad
\lim_{\substack{y \in \Hy^2\\y \rightarrow \alpha_{1}}}\ \frac{d\big(\pi_{\alpha_{1},\beta_{1}}(y),\pi_{\alpha_{1},\beta_{1}}(p_{0})\big)}{d\big(\pi_{\alpha_{1},\beta_{2}}(y),\pi_{\alpha_{1},\beta_{2}}(p_{0})\big)}\quad = \quad 1.$$
\end{lemma}

Before proving the lemma, we prove Proposition \ref{newrot}. Recall that we fixed a point $p \in \Hy^2$. For any sufficiently large index $k$,
\begin{multline*}
\frac{d\big(\pi_{T(\alpha),\beta}(\tilde{f}^{n_{k}}(y_{k})),\pi_{T(\alpha),\beta}(y_{k})\big)}{n_{k}}\\
  =  \frac{d\big(\pi_{T(\alpha),\beta}(\tilde{f}^{n_{k}}(y_{k})),\pi_{T(\alpha),\beta}(p)\big)}{n_{k}}+\frac{d\big(\pi_{T(\alpha),\beta}(p),\pi_{T(\alpha),\beta}(y_{k})\big)}{n_{k}} 
\end{multline*}
as the point $\pi_{T(\alpha),\beta}(p)$ lies between the points $\pi_{T(\alpha),\beta}(\tilde{f}^{n_{k}}(y_{k}))$ and $\pi_{T(\alpha),\beta}(y_{k}))$ on the geodesic $(T(\alpha),\beta)$. Hence

\begin{multline*}
\frac{d\big(\pi_{T(\alpha),\beta}(\tilde{f}^{n_{k}}(y_{k})),\pi_{T(\alpha),\beta}(y_{k})\big)}{n_{k}}
  \\
 =  \frac{d\big(\pi_{T(\alpha),\beta}(\tilde{f}^{n_{k}}(y_{k})),\pi_{T(\alpha),\beta}(p)\big)}{d\big(\pi_{\alpha,\beta}(\tilde{f}^{n_{k}}(y_{k})),\pi_{\alpha,\beta}(p)\big)} \  \frac{d\big(\pi_{\alpha,\beta}(\tilde{f}^{n_{k}}(y_{k})),\pi_{\alpha,\beta}(p)\big)}{n_{k}} \\
   + \frac{d\big(\pi_{T(\alpha),\beta}(p),\pi_{T(\alpha),\beta}(y_{k})\big)}{d\big(\pi_{T(\alpha),T(\beta)}(p),\pi_{T(\alpha),T(\beta)}(y_{k})\big)} \ \frac{d\big(\pi_{T(\alpha),T(\beta)}(p),\pi_{T(\alpha),T(\beta)}(y_{k})\big)}{n_{k}}.
\end{multline*}

Now, the points $y_{k}$ and $T(x_{k})$ both belong to the geodesic line $T(G_{1,k})$ which is orthogonal to $(T(\alpha), T(\beta))$. Hence
$$ \begin{array}{rcl} d\big(\pi_{T(\alpha),T(\beta)}(y_{k}),\pi_{T(\alpha),T(\beta)}(p)\big) & = &  d\big(\pi_{T(\alpha),T(\beta)}(T(x_{k})),\pi_{T(\alpha),T(\beta)}(p)\big) \\
 & =  & d\big(\pi_{\alpha,\beta}(x_{k}),\pi_{\alpha,\beta}(T^{-1}(p))\big)
\end{array}$$
so that
$$ \lim_{ k \rightarrow +\infty} \frac{d\big(\pi_{T(\alpha),T(\beta)}(y_{k}),\pi_{T(\alpha),T(\beta)}(p)\big)}{n_{k}}=v'_{1}.$$
In the same way,
$$ \lim_{ k \rightarrow +\infty} \frac{d\big(\pi_{\alpha,\beta}(p),\pi_{\alpha,\beta}(\tilde f^{n_{k}}(y_{k}))\big)}{n_{k}}=v'_{2}.$$
Hence, by Lemma \ref{compareproj},
$$ \lim_{ k \rightarrow +\infty} \frac{d\big(\pi_{T(\alpha), \beta}(\tilde{f}^{n_{k}}(y_{k})), \pi_{T(\alpha),\beta}(y_{k})\big)}{n_{k}}=v'_{1}+v'_{2}=v >0.$$
Recall that the sequence of points $(\tilde{f}^{n_{k}}(y_{k}))_{k}$ converges to the point $\beta$ and that the sequence of points $(y_{k})_{k}$ converges to the point $T(\alpha)$. Therefore
$$(T(\alpha),\beta,v) \in \rho(f)$$
and Proposition \ref{newrot} is proved.
\end{proof}

\begin{proof}[Proof of Lemma \ref{compareproj}]
As the proof of both items are similar, we will only prove the first one.

Let $p_{1}= \pi_{\alpha_{1},\beta_{1}}(p_{0})$ and $p_{2}= \pi_{\alpha_{2},\beta_{1}}(p_{0})$. For any point $y$ of $\Hy^2$, we have
$$d(\pi_{\alpha_{1},\beta_{1}}(y),p_{1}) \geq d(\pi_{\alpha_{2},\beta_{1}}(y),p_{2})-d(p_{1},p_{2})- d(\pi_{\alpha_{1},\beta_{1}}(y),\pi_{\alpha_{2},\beta_{1}}(y)).$$
Similarly,
$$d(\pi_{\alpha_{2},\beta_{1}}(y),p_{2}) \geq d(\pi_{\alpha_{1},\beta_{1}}(y),p_{1})-d(p_{1},p_{2})- d(\pi_{\alpha_{1},\beta_{1}}(y),\pi_{\alpha_{2},\beta_{1}}(y)).$$
Hence Lemma \ref{compareproj} is a consequence of Claim~\ref{compareproj2}.
\end{proof}
\bigskip

We now come to the general case concerning two intersecting geodesics. Our results are weaker than when the second geodesic is the image of the first one by a deck transformation (Proposition~\ref{newrot2Intro} of the introduction).

\begin{proposition}\label{newrot2}
Let $f\in \Homeo_0(S)$, and $(\alpha_1,\beta_1,v_1) \in\rho(f)$, with $v_1>0$. Let also $(\alpha_2,\beta_2)$ be a geodesic of $\Hy^2$ that intersects $(\alpha_1,\beta_1)$, and such that there exists $(y_k)\in\Hy^2$ and $u_k \in\N$ such that $y_k\to \alpha_2$ and $\tilde f^{u_k}(y_k)\to \beta_2$. Then,  there exist $v',v''\ge 0$ satisfying $v'+v'' = v_1$ such that
\begin{enumerate}[label=(\roman*)]
\item either $(\alpha_1,\beta_2,v')\in \rho(f)$ or $(\alpha_1,\alpha_2,v')\in \rho(f)$;
\item either $(\beta_2,\beta_1,v'')\in \rho(f)$, or $(\alpha_2,\beta_1,v'')\in \rho(f)$.
\end{enumerate}
\end{proposition}

Remark that this result can be applied when moreover $(\alpha_2,\beta_2,v_2)\in\rho(f)$, with $v_2>0$.

\begin{proof}
By definition, there exist a sequence $(x_{k})_{ k \in \N}$ of points in $\Hy^{2}$, a sequence $(t_{k})_{ k \in \N}$ of integers, a sequence $(y_{k})_{ k \in \N}$ of points in $\Hy^{2}$ and a sequence $(u_{k})_{ k \in \N}$ of integers such that
$$ \left\{
\begin{array}{ll}
\displaystyle \lim_{k \rightarrow +\infty} x_{k}  =  \alpha_1
& \displaystyle \lim_{k \rightarrow +\infty} y_{k}  =  \alpha_2 \\
\displaystyle \lim_{ k \rightarrow +\infty} \tilde{f}^{t_{k}}(x_{k})  =  \beta_1 
& \displaystyle \lim_{ k \rightarrow +\infty} \tilde{f}^{u_{k}}(y_{k})  =  \beta_2.\\
\displaystyle \lim_{ k \rightarrow +\infty} \frac{d\big(\pi_{\alpha_1,\beta_1}(x_{k}),\, \pi_{\alpha_1,\beta_1}(\tilde{f}^{t_{k}}(x_{k}))\big)}{t_{k}}  =  v_1\ 
\end{array}
\right.
$$ 
For any $k \geq 0$ and $i=1,2$, denote by $L_{i,k}$ (respectively $R_{i,k}$) the left of the unique geodesic line passing through $x_{k}$ if $i=1$, $y_k$ if $i=2$ (respectively the right of the unique geodesic line passing through $\tilde{f}^{t_{k}}(x_{k})$ if $i=1$ and through $\tilde{f}^{u_{k}}(y_{k})$ if $i=2$) which is orthogonal to $(\alpha_i,\beta_i)$), the left/right being defined with the help of the oriented geodesics $(\alpha_i,\beta_i)$. Let $p\in\Hy^{2}$ be the intersection between the geodesics $(\alpha_1,\beta_1)$ and $(\alpha_2,\beta_2)$. Extracting a subsequence if necessary, we can suppose that the sequences
\[\left(\frac{d(\pi_{\alpha_1,\beta_1}(x_{k}),\, \pi_{\alpha_1,\beta_1}(p))}{t_{k}}\right)_{k}
\quad \text{and}\quad
\left(\frac{d(\pi_{\alpha_1,\beta_1}(p),\, \pi_{\alpha_1,\beta_1}(\tilde{f}^{t_{k}}(x_{k})))}{t_{k}}\right)_{k}\]
converge with respective limits $v'$ and $v''$. Observe that those limits do not depend on the chosen point $p$ and that $v'+v'' = v_1$.

The reader can refer to Figure \ref{FignoInter2}. Fix $k\in\N$, and set $L_b = L_{2,k}$ and $R_d = R_{2,k}$. If $k$ is large enough, the closure of these sets in the disk $\overline\Hy^2$ contains neither $\alpha_1$ nor $\beta_1$. Then there exists $L\in\N$ such that for any $\ell\ge L$, the sets $L_a = L_{1,\ell}$ and $R_c = R_{1,\ell}$ are disjoint from the set
\[X = \tilde f^{u_k}(L_b) \cup R_d.\]
Note that the set $X$ separates the sets $L_a$ and $R_c$ (that is, these sets lie in different connected components of $X^\complement$). Moreover, the trace of the closure of $X$ on $\partial\Hy^2$ is the same as the trace of the closure of $L_b\cup R_d$. 

\begin{figure}[ht]
\begin{center}

\includegraphics[scale=1]{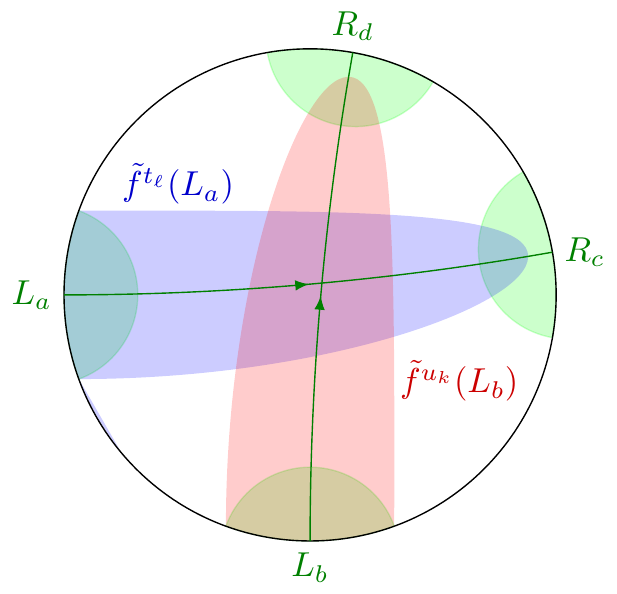}
\caption{\label{FignoInter2} Proof of Proposition \ref{newrot2}.}
\end{center}
\end{figure}

Take $\ell\ge L$ such that $t_\ell \gg u_k$. As $x_\ell \in L_a$ and $\tilde f^{t_\ell}(x_\ell) \in R_c$, the set $Y = \tilde f^{t_\ell}(L_a) \cup R_c$ is (path) connected. This implies that $X\cap \tilde f ^{t_\ell}(L_a)\neq\emptyset$, hence one of the following intersections is nonempty:
\[\tilde f^{t_\ell}(L_a) \cap R_d \qquad\text{or}\qquad \tilde f^{t_\ell-u_k}(L_a) \cap L_b.\]

Remark that this conclusion is similar to the one of Lemma \ref{intersection}. Reasoning as in the proof of Proposition \ref{newrot}, and in particular using Lemma \ref{compareproj}, we deduce that there exist a sequence of points $(z_i)$ in $\Hy^2$, and a sequence of times $n_i$ going to infinity such that $\lim z_i = \alpha_1$ and
\[\begin{split}
\text{either}\quad \tilde f^{n_i}(z_i) \underset{i\to+\infty}{\longrightarrow} \beta_2
\qquad \text{and} \qquad
\limsup_{i\to +\infty} \frac{d(\pi_{\alpha_1,\beta_2}(z_i), \pi_{\alpha_1,\beta_2}(\tilde f^{n_i}(z_i)))}{n_i}\ge v',\\
\text{or}\quad \tilde f^{n_i}(z_i) \underset{i\to+\infty}{\longrightarrow} \alpha_2
\qquad \text{and} \qquad
\limsup_{i\to +\infty} \frac{d(\pi_{\alpha_1,\alpha_2}(z_i), \pi_{\alpha_1,\alpha_2}(\tilde f^{n_i}(z_i)))}{n_i}\ge v',
\end{split}\]
which implies, together with Theorem \ref{TheoStar}, that either $(\alpha_1,\beta_2,v')\in \rho(f)$, or $(\alpha_1,\alpha_2,v')\in \rho(f)$.

Considering $\tilde f^{-t_\ell}(R_c)$ instead of $\tilde f^{t_\ell}(L_a)$ gives the following similar conclusion: one of the following intersections is nonempty:
\[\tilde f^{-t_\ell}(R_c) \cap R_d
\qquad\text{or}\qquad
\tilde f^{-t_\ell-u_k}(R_c) \cap L_b.\]
As before, this implies that either $(\beta_2,\beta_1,v'')\in \rho(f)$, or $(\alpha_2,\beta_1,v'')\in \rho(f)$.
In conclusion 
\begin{equation*}
\begin{split}
\text{either }(\alpha_1,\beta_2,v')\in \rho(f),&\quad \text{or } (\alpha_1,\alpha_2,v')\in \rho(f)\text{, and}\\
\text{either }(\beta_2,\beta_1,v'')\in \rho(f),&\quad \text{or } (\alpha_2,\beta_1,v'')\in \rho(f).
\end{split}
\end{equation*}
\end{proof}

\section{Almost annular homeomorphisms}\label{SecPseudo}

In this section, we study the situation where the rotation vectors are all associated to a single geodesic of the surface. We will prove that this implies that the geodesic has no auto-intersection.

Let $\alpha$ and $\beta$ be points of $\partial \Hy^2$ and $p: \tilde{S}_{g}= \Hy^{2} \rightarrow S$ be a covering map. For $f \in \mathrm{Homeo}_{0}(S)$, we say that the rotation set of $f$ is \emph{contained in the geodesic} $p(\alpha,\beta)$ if
$$ \rho(f) \subset \bigcup_{T \in \pi_{1}(S)} T(\alpha,\beta) \times \R_{+}.$$
Next proposition is Proposition~\ref{notransverseintIntro} of the introduction.

\begin{proposition} \label{notransverseint}
Let $f \in \mathrm{Homeo}_{0}(S)$ and $\alpha,\beta \in \partial \Hy^{2}$. Suppose that the rotation set of $f$ is not reduced to $\left\{ 0 \right\}$ and is contained in $p(\alpha,\beta)$. Then, for any $T\in\pi_1(S)$, the geodesics $(\alpha,\beta)$ and $T(\alpha,\beta)$ have no common point in $\Hy^2$.
\end{proposition}

In particular, this proposition implies that the geodesic $p(\alpha,\beta)$ has no transverse auto-intersection.

This proposition is still true (with the same proof) if we suppose that

$$ \rho(f) \subset \left(\bigcup_{T \in \pi_{1}(S)} T(\alpha,\beta) \times \R_{+}\right)
\cup
\left(\bigcup_{T \in \pi_{1}(S)} T(\beta,\alpha) \times \R_{+}\right).$$

\begin{remark}
The only examples we know of homeomorphisms $f$ such that the rotation set of $f$ is not reduced to $\left\{ 0 \right\}$ and is contained in $p(\alpha,\beta)$ are those for which $p(\alpha,\beta)$ is a closed geodesic. We can wonder whether there are other examples.
\end{remark}

\begin{proof}
Take $v >0$ such that $(\alpha,\beta,v) \in \rho(f)$. We take the same notation as in the proof of Proposition \ref{newrot}: there exist a sequence $(x_{k})_{ k \in \N}$ of points in $\Hy^{2}$ and a sequence $(n_{k})_{ k \in \N}$ of integers such that
$$ \left\{
\begin{array}{l}
\displaystyle \lim_{k \rightarrow +\infty} x_{k}  =  \alpha \\
\displaystyle \lim_{ k \rightarrow +\infty} \tilde{f}^{n_{k}}(x_{k})  =  \beta \\
\displaystyle \lim_{ k \rightarrow +\infty} \frac{d\big(\pi_{\alpha,\beta}(\tilde{f}^{n_{k}}(x_{k})\big),\,\pi_{\alpha,\beta}(x_{k}))}{n_{k}}  =  v.
\end{array}
\right.
$$ 
For any $k \geq 0$, denote by $G_{1,k}$ (respectively $G_{2,k}$) the unique geodesic line passing through $x_{k}$ (respectively $\tilde{f}^{n_{k}}(x_{k})$) which is orthogonal to $(\alpha,\beta)$. Fix a point $p$ in $\Hy^{2}$. Extracting a subsequence if necessary, we can suppose that both sequences $\left(\frac{d(\pi_{\alpha,\beta}(x_{k}),\pi_{\alpha,\beta}(p))}{n_{k}}\right)_{k}$ and $\left(\frac{d(\pi_{\alpha,\beta}(p),\pi_{\alpha,\beta}(\tilde{f}^{n_{k}}(x_{k})))}{n_{k}}\right)_{k}$ converge with respective limits $v'_{1}$, $v'_{2}$. Observe that those limits do not depend on the chosen point $p$ and that $v'_{1}+v'_{2}=v$.

Suppose for a contradiction that the geodesic line $p(\alpha,\beta)$ has a transverse autointersection: there exists $T \in \pi_{1}(S)$ such that 
\begin{equation}\label{EqCard}
\card\big(T(\alpha,\beta) \cap (\alpha,\beta) \big) = 1.
\end{equation}
By Proposition \ref{newrot}, either $(\alpha,T(\beta),v) \in \rho(f)$ or $(T(\alpha), \beta,v) \in \rho(f)$.

Let us finish the proof of Proposition \ref{notransverseint} in the first case $(\alpha,T(\beta),v) \in \rho(f)$. The second case is similar. We will use the following classical lemma.

\begin{lemma} \label{commonendpoint}
Let $\eta_{1}$ and $\eta_{2}$ be nontrivial transformations in $\pi_{1}(S)$. If the respective axis $A_{1}$ and $A_{2}$ of $\eta_{1}$ and $\eta_{2}$ have a common endpoint, then $A_{1}=A_{2}$ and there exists nonzero integers $n_{1}$ and $n_{2}$ such that
$$ \eta_{1}^{n_{1}}=\eta_{2}^{n_{2}}.$$
\end{lemma}

To prove this lemma, observe that, otherwise, for any point $p\in \Hy^2$, there would be infinitely many points of the form $\eta_{1}^{n}\eta_{2}^{m}(p)$ in a compact subset of $\Hy^2$. This is not possible as the group $\pi_{1}(S)$ acts properly on $\Hy^2$. 
\medskip

By hypothesis on $\rho(f)$, there exists a deck transformation $T_{1}$ such that $T_{1}(\alpha,\beta)=(\alpha,T(\beta))$. Then 
\begin{enumerate}
\item either $T_{1}(\alpha)=\alpha$ and $T_{1}(\beta)=T(\beta)$;
\item or $T_{1}(\alpha)=T(\beta)$ and $T_{1}(\beta)=\alpha$.
\end{enumerate}

Suppose first that $T_{1}(\alpha)=\alpha$ and $T_{1}(\beta)=T(\beta)$. Then $\alpha$ is an endpoint of the axis of $T_{1}$. Let $T_{2}=T^{-1} T_{1}$. As the deck transformation $T_{2}$ fixes the point $\beta$, either $T_{2}$ is trivial or $\beta$ is an endpoint of the axis of $T_{2}$. If $T_{2}$ was trivial, we would have $T_{1}=T$ but this is not possible as $T(\alpha) \neq \alpha = T_{1}(\alpha)$.

Hence $T_{2} \neq 1$. Now, let us prove that $p(\alpha,\beta)$ is a closed geodesic. This will lead to a contradiction: by Lemma \ref{commonendpoint}, the axis of $T_{1}$ would be $(\alpha,\beta)$, which is not possible as, by \eqref{EqCard}, $T_{1}(\beta)=T(\beta) \neq \beta$.

Denote by $(\alpha,\alpha_{1})$ the axis of $T_{1}$ and by $(\beta,\beta_{1})$ the axis of $T_{2}$. Let $A_{1}$ (respectively $A_{2}$) be the closed annulus obtained by quotienting the closed band $\overline{\Hy}^{2} \setminus \left\{ \alpha,\alpha_{1} \right\}$ (respectively $\overline{\Hy}^{2} \setminus \left\{ \beta,\beta_{1} \right\}$) by the action of the group generated by $T_{1}$ (respectively $T_{2}$). For $i=1,2$, denote by $\overline{f}_{i}$ the homeomorphism induced by $\tilde{f}$ on $A_{i}$. Denote by $\rho(\overline{f}_{i})$ the rotation set of $\overline{f}_{i}$ on $A_{i}$.

Recall that both sequences $\left(\frac{d(\pi_{\alpha,\beta}(x_{k}),\pi_{\alpha,\beta}(p))}{n_{k}}\right)_{k}$ and $\left(\frac{d(\pi_{\alpha,\beta}(p),\pi_{\alpha,\beta}(\tilde{f}^{n_{k}}(x_{k})))}{n_{k}}\right)_{k}$  converge with respective limits $v'_{1}$, $v'_{2}$ with $v'_{1}+v'_{2}=v>0$. Therefore, by Lemma \ref{LemCompareAnneau} and Lemma \ref{compareproj}, $v'_{1} \in \rho(\overline{f}_{1})$ and $v'_{2} \in \rho(\overline{f}_{2})$. Hence, one of those two rotation sets contains a nonzero rotation number. But this implies that there exists $v'>0$ such that either $(\alpha,\alpha_{1},v') \in \rho(f)$ or $(\beta_{1},\beta,v') \in \rho(f)$. As we supposed that $\rho(f)$ was contained in a geodesic line, this means that either $(\alpha,\alpha_{1})$ or $(\beta_{1},\beta)$ is the image of $(\alpha,\beta)$ under a deck transformation. Hence, as $p(\alpha,\alpha_{1})$ and $p(\beta_{1},\beta)$ are closed geodesics, so is $p(\alpha,\beta)$, what we wanted to prove.
\medskip

Now, let us suppose that $T_{1}(\alpha)=T(\beta)$ and $T_{1}(\beta)=\alpha$. Then $T^{-1}T_{1}^2 (\beta)=\beta$ and $T_{1}T^{-1} T_{1}(\alpha)=\alpha$. If one of the deck transformations $T^{-1} T_{1}^2$ or $ T_{1} T^{-1} T_{1}$ is trivial, then $T=T_{1}^{2}$. Then there exists an orientation of the circle $\partial \Hy^2$ such that, when one follows this orientation of the circle, one successively meets the points $\beta$, $T_{1}^{2}(\beta)=T(\beta)$, $T_{1}(\beta)=\alpha$ and $T_{1}^{3}(\beta)=T(\alpha)$. This is not possible as the dynamics of $T_{1}$ on $\partial \Hy^{2}$ is a north-south dynamic (see Figure~\ref{Fignotransverseint}, left).

\begin{figure}[ht]
\begin{center}
\begin{tikzpicture}[scale=1]
\draw (0,0) circle (1);
\hgline{-180}{10}{green!50!black}
\hgline{-90}{80}{blue!50!black}
\draw[color=green!50!black] (-180:1) node[left]{$\alpha = T_1(\beta)$};
\draw[color=green!50!black] (10:1) node[right]{$\beta$};
\draw[color=blue!80!black] (-90:1) node[below]{$T(\alpha) = T_1^3(\beta)$};
\draw[color=blue!80!black] (80:1) node[above]{$T(\beta) = T_1^2(\beta) = T_1(\alpha)$};
\end{tikzpicture}
\hspace{70pt}
\begin{tikzpicture}[scale=1]
\draw (0,0) circle (1);
\hgline{-180}{10}{green!50!black}
\hgline{-90}{80}{blue!50!black}
\hgline{25}{65}{gray!80!black}
\draw[color=green!50!black] (-180:1) node[left]{$\alpha = T_1(\beta)$};
\draw[color=green!50!black] (10:1) node[right]{$\beta$};
\draw[color=gray!80!black] (50:1) node[right]{$T_1$ axis};
\draw[color=blue!80!black] (-90:1) node[below]{$T(\alpha)$};
\draw[color=blue!80!black] (80:1) node[above]{$T(\beta) = T_1(\alpha)$};
\end{tikzpicture}
\caption{\label{Fignotransverseint} End of proof of Proposition \ref{notransverseint}.}
\end{center}
\end{figure}
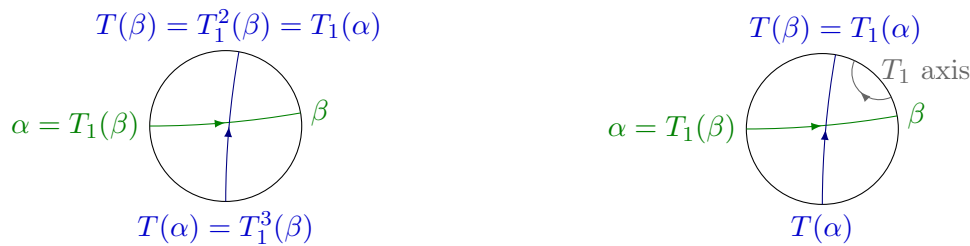

So both deck transformations $T^{-1} T_{1}^2$ and $ T_{1} T^{-1} T_{1}$ are nontrivial. Remark that the point $\beta$ (respectively $\alpha$) is an endpoint of the axis of $T^{-1} T_{1}^{2}$ (respectively $T_{1} T^{-1} T_{1}$). As in the above case, this implies that the geodesic line $p(\alpha,\beta)$ is closed and that $(\alpha,\beta)$ is the axis of both deck transformations $T^{-1} T_{1}^2$ and $T_{1} T^{-1} T_{1}$. Then $T^{-1}T_{1}^{2}(\alpha)=\alpha$. However, we will see that this relation is not possible. This will complete the proof of Proposition \ref{notransverseint}.

Indeed, orient the circle $\partial \Hy^2$ in such a way that, following this orientation, we successively meet the points $\alpha$, $T(\alpha)$, $\beta$, $T(\beta)$ (see Figure~\ref{Fignotransverseint}, right). As $T_{1}(\beta)=\alpha$ and $T_{1}(\alpha)= T(\beta)$, both endpoints of the axis of $T_{1}$ have to belong to the oriented open interval $(\beta, T(\beta))$ of $\partial \Hy^{2}$ and the point $T_{1}^{2}(\alpha)= T_{1}(T(\beta))$ belongs to the open interval $(\beta, T(\beta)) \subset (T(\alpha), T(\beta))\subset\partial \Hy^2$. Hence the point $T^{-1}(T_{1}^{2}(\alpha))$ belongs to the open interval $(\beta,\alpha)$ and, in particular, $T^{-1}(T_{1}^{2}(\alpha)) \neq \alpha$.
\end{proof}

\section{Examples}\label{SecEx}

See also Section~10 of \cite{alonso2020generic} for other interesting examples, based on a technique due to Kwapisz \cite{MR1176627}.

\subsection{A uniquely ergodic diffeomorphism with uncountable rotational directions}
\label{ExNonDenombrable}

Let us give an example of a homeomorphism of the punctured torus with an ergodic probability measure $\mu$ for which an uncountable set of geodesics is necessary to describe the rotation set of $\mu$ almost every point.

Let $\alpha\in\R\setminus\Q$ and $X\equiv(1,\alpha)$ be the constant vector field on $\T^2 = \R^2/\Z^2$. Let $\kappa : \T^2\to\R_+$ be a continuous nonnegative function such that $\kappa(x)=0 \iff x = 0$, and that $\kappa(x)\sim \|x\|^a$ for some $0<a<2$ ($\|\cdot\|$ is the Euclidean norm). Let $(\phi_t)$ be the flow associated to the vector field $\kappa X$ (Cauchy-Lipschitz theorem applies on $\T^2\setminus\{0\}$ as $\kappa X$ is locally Lipschitz on it; we set $\phi^t(0)=0$ for any $t$). It is of class $C^a$, and its flow curves are straight lines with slope $\alpha$ (apart from two half-lines). By Section 2 of \cite{MR2670926}, this flow has two ergodic probability invariant measures: $\delta_0$ and an absolutely continuous one $\mu$ with total support. The rotation number of this measure $\mu$ is nonzero, hence the rotation set (as a homeomorphism of $\T^2$) of this flow is a nontrivial segment. 
As 0 is a fixed point for it, it induces a flow on the noncompact manifold $\T^2\setminus\{0\}$; we denote by $f$ the time 1 of this flow. Then, seen as a homeomorphism of $\T^2\setminus\{0\}$, $f$ is uniquely ergodic.

By Lemma \ref{Lessa} (from \cite{MR2846925}), to $\mu$-almost every point $x$ and every lift $\tilde x$ of $x$ to the universal cover of $\T^2\setminus\{0\}$, there exists a geodesic line $(\alpha_{\tilde{x}},\beta_{\tilde{x}})$ such that
$d(\tilde{f}^n(\tilde{x}),(\alpha_{\tilde{x}},\beta_{\tilde{x}}))=o(n)$.

Let $(x_i)_{i\in I}$ an uncountable set of $\mu$-typical points, such that any two of them are on a different flow line. Using Svar\v{c}-Milnor lemma (Lemma \ref{svarcmilnor}), one can easily see that for $i\neq j$, the orbits of $x_i$ and $x_j$ move away one from the other at least at linear speed. This implies in particular that that $(\alpha_{\tilde{x_i}},\beta_{\tilde{x_i}}) \neq (\alpha_{\tilde{x_j}},\beta_{\tilde{x_k}})$. Indeed, any two lifts of the flow curves to $\mathbb{R}^2$ are separated by a lift of the point $0 \in \mathbb{T}^2$.

If $1\le a<2$, it is possible to blow up the fixed point of the homeomorphism to a circle and gluing two such examples along the two obtained circles; it gives a homeomorphism of the compact surface of genus 2. The ergodic invariant measures of this homeomorphism are either one of the two absolutely continuous invariant probability measures, each one supported in a domain homeomorphic to a punctured torus, or supported in the fixed point set which is a circle. Then there is no positive Lebesgue-measure set $A$ such that any two points $x$ and $y$ of $A$ have the same rotational direction (i.e. the geodesic defined by the flow line passing through $x$ or $y$).

\subsection{A diffeomorphism with trivial rotation set but unbounded displacements in all directions}\label{ExUnbounded0}

This example is the higher genus counterpart of the torus example of Koropecki and Tal \cite{MR3238423}. In this paper, the authors build an open topological disk embedded in $\T^2$ whose lift meets each fundamental domain of $\T^2$ in $\R^2$, and whose complement has zero measure. They then define a smooth Lebesgue measure preserving Bernoulli\footnote{And hence ergodic.} diffeomorphism of the torus which is equal to the identity outside this disk, using a method due to Katok \cite{MR554383}: there exists a smooth Bernoulli diffeomorphism of the unit disk, equal to identity on the boundary of the disk, which gives a smooth Bernoulli diffeomorphism of the torus via the embedding of the disk.

The torus example of Koropecki and Tal \cite{MR3238423} can be generalized to any positive Euler characteristic connected surface, with an (almost) identical proof. The construction of the example from the existence of the embedded disk follows from Katok \cite{MR554383} and is not specific to the torus. For the construction of the embedded disk, the only part of the proof that is specific to the torus is Lemma 4, that can be replaced by the following.

\begin{lemma}
If $\wt x$ and $\wt y$ are two points of $\wt S \simeq \Hy^2$ that do not project on the same point of $S$, then there is an arc $\alpha$ joining $\wt x$ to $\wt y$ whose projection to $S$ is injective.
\end{lemma}

\begin{proof}
Fix a point $\wt x\in \wt S$, and consider the set 
\[A_{\wt x} = \big\{\wt y\in\wt S \mid \text{$\exists \wt \alpha$ arc joining $\wt x$ to $\wt y$ s.t. $\alpha$ is injective}\big\}.\]
Equip $S$ with a metric that lifts to the canonical metric on $\Hy^2$. For $\epsilon>0$, let
\[E_\epsilon = \big\{\wt y\in\wt S \mid \text{the projection of $B(\wt y,\epsilon)$ to $S$ is injective}\big\}.\]

Let $\wt y\in A_{\wt x} \cap E_\epsilon$. Then there exists a path $\wt \alpha$ joining $\wt x$ to $\wt y$ whose projection to $S$ is injective. Remark that it is possible to make this path smooth if necessary. Then, consider a diffeomorphism $h$ of $S$, equal to identity out of $B(y,\epsilon)$, such that any point in $B(y,\epsilon) \setminus\{y\}$ has its $\omega$-limit included in $\partial B(y,\epsilon)$. Then, for $n$ large enough, $h^n (\alpha) \cap B(y,\epsilon/2)$ is a connected path. In particular, for any point $\wt z\in B(\wt y,\epsilon/2)\setminus \wt h^n (\wt \alpha) $, it is easy to build a path joining $\wt y$ to $\wt z$ and included in $B(\wt y,\epsilon/2)\setminus \wt h^n (\wt \alpha) $. This shows that $B(\wt y,\epsilon/2) \subset A_{\wt x}$.

The uniformity of $\epsilon$ in this property shows that the connected component of $\wt x$ in $E_\epsilon$ is included in $A_{\wt x}$. But for any $\wt x,\wt y\in \wt S$, there exists $\epsilon_0>0$ such that for any $0<\epsilon<\epsilon_0$, $\wt x$ and $\wt y$ lie in the same (path) connected component of $E_\epsilon$. Hence, $A_{\wt x} = \Hy^2$.
\end{proof}

Finally, we get the following result.

\begin{proposition}\label{KoroTal}
For any surface $S$ with negative Euler characteristic with finite measure, for any fixed compact connected fundamental domain $D\subset\Hy^2$ of $S$, there is a $C^\infty$ area-preserving diffeomorphism $f : S\to S$ homotopic to the identity, with a lift $\wt f : \Hy^2\to\Hy^2$ such that
\begin{itemize}
\item $\rho(f)$ is reduced to a single rotation vector (with zero speed):
\item $f$ is metrically isomorphic to a Bernoulli shift (in particular, $f$ is ergodic)
with Lebesgue measure;
\item For Lebesgue almost every point $\wt x\in\Hy^2$, the forward and backward orbits
of $\wt x$ accumulates in every direction at infinity, i.e. 
\[\partial_\infty\{\wt f^n (\wt x)\mid n\in \N\} = \Sp^1 = \partial_\infty\{\wt f^{-n} (\wt x)\mid n\in \N\}.\]
Moreover, the forward and backward orbits of $\wt x$ visit every fundamental domain $TD$, with $T\in\pi_1(S)$.
\end{itemize}
\end{proposition}

This example can be modified in a simple way to get other rotational behaviours. For example, consider a simple essential closed annulus $A\subset S$, such that $S\setminus A$ is connected. Then, one can apply Proposition~\ref{KoroTal} to $S\setminus A$ to get a homeomorphism $h$ of $S\setminus A$ which is equal to identity on $\partial A$, and extend $h$ to $A$ such that $h|_A$ has a nontrivial rotation set. This gives an example of an almost annular (in the sense of Section~\ref{SecPseudo}) that has unbounded displacements in all directions not intersecting the direction of $A$.

\subsection{An example of homeomorphism with non closed rotation set}\label{SecExEmmanuel}

Fix a closed surface $S$ with negative Euler characteristic which is endowed with a hyperbolic metric.
Take two disjoint closed simple geodesics $\alpha$ and $\beta$ of $S$ and a simple geodesic $\gamma$ whose $\alpha$-limit is $\alpha$ and whose $\omega$-limit is $\beta$ and which is disjoint from both $\alpha$ and $\beta$. Let $A_\beta$ be a tubular neighbourhood of $\beta$ which is disjoint from $\alpha$ and such that $A_{\beta} \cap \gamma$ is connected. Let $B_\gamma$ be a tubular neighbourhood of $\gamma$ which is disjoint from both $\alpha$ and $\beta$ such that $B_{\gamma} \cap A_{\beta}$ is connected. We denote by $\tilde{\alpha}=(\alpha_1,\alpha_2)$, $\tilde{\beta}=(\beta_1,\beta_2)$ and $\tilde{\gamma}$ the respective lifts of $\alpha$, $\beta$ and $\gamma$ to the universal cover $\tilde{S} \equiv \Hy^2$ of $S$ such that $\tilde{\gamma}$ joins the points $\alpha_1$ and $\beta_2$ of $\partial \Hy^2$.

Let $f_1$ be a homeomorphism in $\mathrm{Homeo}_0(S)$, which is supported on $A_{\beta}$, with the following properties.
\begin{enumerate}
\item The canonical lift $\tilde{f}_1$ of $f_1$ preserves $\tilde{\beta}$ and acts as a translation of $\theta \in \mathbb{R}$ on $\tilde{\beta}$.
\item There exists an open neighbourhood $U_{\beta}$ of $\beta$ such that
$$\left\{ \begin{array}{l}
\overline{f_1(U_{\beta})} \subset U_{\beta} \\
\bigcap_{n \geq 0} f_1^{n}(U_{\beta})=\beta \\
\bigcup {n \in \Z} f_1^{n}(U_{\beta})= A_{\beta}
\end{array} \right.
$$
\item For any $n \in \mathbb{Z}$, the set $f_1^{n}(U_{\beta}) \cap B_{\gamma}$ is connected.
\end{enumerate}

\begin{figure}[ht]
\begin{center}
\includegraphics[width=.5\linewidth]{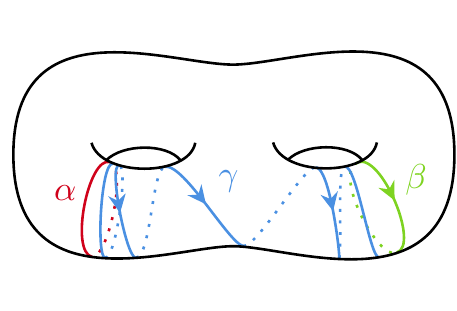}
\hfill
\includegraphics[width=.33\linewidth]{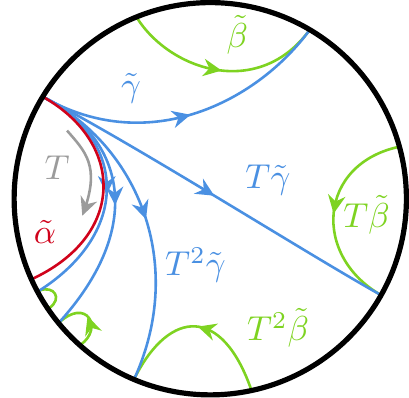}
\caption{The right drawing is the universal cover of the left one. On the left, the red closed curve $\alpha$ (left) is made of fixed points, and the green closed curve $\beta$ (right) rotates. The blue geodesic $\gamma$ is heteroclinic to $\alpha$ and $\beta$. On the right, $T$ is the deck transformation associated to a lift $\tilde \alpha$ of $\alpha$.}\label{FigExEmmanuel}
\end{center}
\end{figure}

Let $f_2$ be a homeomorphism which is supported in $\overline{B_\gamma}$ with the following properties. 
\begin{enumerate}
\item For any point $x \in B_\gamma$, the sequence $(f_2^{n}(x))_{n \geq 0}$ accumulates to $\beta$ and the sequence $(f_2^{-n}(x))_{n \geq 0}$ accumulates to $\alpha$.
\item For any $n \in \mathbb{Z}$, $f_2(f_1^{n}(U_{\beta})) \subset f_1^{n}(U_{\beta})$. 
\end{enumerate}
Observe that the homeomorphism $f_2$ pointwise fixes $\alpha$ and $\beta$.

Finally, let $f_3=f_2 \circ f_1$. The dynamics of $f_3$ is described on Figure \ref{FigExEmmanuel}. Observe that the recurrent orbits of $f_3$ consist of its fixed points and the points of $\beta$. Observe also that $(\beta_1,\beta_2,\theta) \in \rho(f_3)$ and, because of the orbits on $\gamma$ with asymptotic speed $\theta$, $(\alpha_1,\beta_2,\theta) \in \rho(f_3)$. Let $T$ be the deck transformation associated to $\tilde{\alpha}$. As the set $\rho(f_3)$ is invariant under deck transformations, we obtain that, for any $n \geq 0$,
$$ (T^n \alpha_1, T^{n} \beta_2,\theta)=(\alpha_1, T^{n} \beta_2,\theta) \in \rho(f_3).$$
Observe also that 
$$ \lim_{n \rightarrow +\infty} T^{n} \beta_2=\alpha_2.$$
However, $(\alpha_1,\alpha_2,\theta) \notin \rho(f_3)$. Indeed, otherwise, by Proposition \ref{sublindistance}, there would exist a recurrent orbit of $\tilde{f}_3$ with a nontrivial rotation vector which stays at a bounded distance from the geodesic $\tilde{\alpha}$. But there exists no such orbit as the only nontrivial recurrent orbits of $f_3$ are contained in $\beta$. Hence the rotation set of $f_3$ is not closed.

\section{Intersections of closed geodesics: consequences on the homological rotation set}\label{SecHomo}

In this section, we get the first consequences of the fact that two closed geodesics that cross and are associated to a nontrivial rotation vector. In this case, we get the existence of a toral covering in which the induced rotation set has nonempty interior (Proposition~\ref{lifttorus}). This implies positive topological entropy and the existence of an infinite number of periodic orbits (Corollary~\ref{entropy}). This is somehow a first step towards the results of Section~\ref{LastSection}, in which we will get stronger conclusions under weaker hypotheses. It will be the occasion to introduce and study the notion of associated covering map (Definition~\ref{DefCovering}) that will be used in the three last sections.

\subsection{Background on covering maps}

In the sequel, $S$ will be chosen as either the closed surface $S$ of genus $g\ge 2$, or the domain of an isotopy $\dom I\subset S$ of $f$ (see Subsection~\ref{SubSecForcing}).
The surface $S$ is endowed with a complete hyperbolic metric so that its universal cover $\tilde{S}$ is identified with the hyperbolic plane $\Hy^2$.

Let $x_0\in S$, and $\alpha_{1}$ and $\alpha_{2}$ two loops of $S$. Denote by $\tilde x_0$, $\tilde \alpha_{1}$ and $\tilde \alpha_{2}$ lifts of respectively $x_0$, $\alpha_1$ and $\alpha_2$ to the universal cover $\tilde{S}$ of $S$.
Suppose that $\tilde x_0\in \tilde\alpha_{1} \cap \tilde\alpha_{2}$; we take $\tilde x_0$ as a basepoint for both those loops. 

Note that each of the paths $\tilde \alpha_{1}$ and $\tilde \alpha_{2}$ stay at a finite distance to a closed geodesic of $\tilde S$ (and this geodesic is determined by the deck transformation associated to the loops $\alpha_1$ and $\alpha_2$).

\begin{definition}\label{DefTransverseInter}
We say that $\alpha_{1}$ and $\alpha_{2}$ have \emph{a geometric transverse intersection at $x_0$} if some geodesics associated to their lifts to $\tilde S$ intersect in $\Hy^2$.

We say that a loop $\alpha$ of $S$ has a \emph{geometric transverse autointersection at $x_0$} if there exists a deck transformation $T$ of $\tilde S \to S$ such that $\tilde\alpha$ and $T\tilde\alpha$ intersect transversally.
\end{definition}

Note that by definition, a transverse intersection stays transverse in any covering space.

\begin{lemma}\label{LemTransverseRevet}
Let $F\subset S$ be a closed set, and $\alpha$ a loop of $S\setminus F$ which has a transverse geometric transverse auto-intersection at $x_0$ for a deck transformation $T$ of $\tilde S \to S$. Then, there exists a deck transformation $T_1$ of $\wt{S\setminus F} \to S\setminus F$, projecting to $T$ in $\tilde S \to S$, such that $\alpha$ has a geometric transverse auto-intersection at $x_0$ for $T_1$.
\end{lemma}

\begin{proof}
Let $\gamma$ be a geodesic of $S\setminus F$ (for a hyperbolic metric on $S\setminus F$), and $\check\gamma$ a lift of $\gamma$ to $\tilde S$. By hypothesis, the lift to $\tilde S$ of the geodesic of $S$ associated to $\gamma$ intersects its translate by $T$. Hence, $\check\gamma$ and $T\check\gamma$ intersect in $\tilde S$. Hence, there is a deck transformation $T_1$ of $\wt{S\setminus F} \to S\setminus F$, projecting to $T$, such that if $\tilde\gamma$ is a lift of $\gamma$ to $\wt{S\setminus F}$, then $\tilde\gamma$ and $T_1\tilde\gamma$ intersect.
\end{proof}

Recall that there is a bijective correspondence between subgroups of the fundamental group $\pi_{1}(S,x_0)$
of $S$ at $x_0$ and the covering maps of $S$: to any subgroup $G$ of $\pi_{1}(S,x_0)$ can be associated the covering map $\hat{S}=\tilde{S}/G \rightarrow S$, where $G$ is seen as a subgroup of the group of deck transformations of $\pi: \tilde{S} \rightarrow S$. Moreover, $G=\pi_{1}(\hat{S},\hat{x_0})$ for some lift $\hat x_0$ of $x_0$..

Denote by $[\alpha_{1}]_{x_0}$ and $[\alpha_{2}]_{x_0}$ the respective classes of the loops $\alpha_{1}$ and $\alpha_{2}$ in $\pi_{1}(S,x_0)$. Recall that any nontrivial class of $\pi_1(S)$ contains a unique closed geodesic.

\begin{definition}\label{DefCovering}
We call \emph{covering map associated to $(\alpha_{1},\alpha_{2},x_0)$} a covering map $\hat{S} \rightarrow S$ associated to the subgroup $\langle[\alpha_{1}]_{x_0},[\alpha_{2}]_{x_0}\rangle$ of $\pi_{1}(S,x_0)$ generated by $[\alpha_{1}]_{x_0}$ and $[\alpha_{2}]_{x_0}$. The loops $\hat{\alpha}_{1}$ and $\hat{\alpha}_{2}$ which represent respectively $[\alpha_{1}]_{x_0}$ and $[\alpha_{2}]_{x_0}$ in $G=\pi_{1}(\hat{S},\hat{x_0})$ and respectively lift the loops $\alpha_{1}$ and $\alpha_{2}$ are called \emph{canonical lifts} of $\alpha_{1}$ and $\alpha_{2}$ in $\hat{S}$. 
\end{definition}

A covering map associated to $(\alpha_{1},\alpha_{2},x_0)$ depends on the choice of the lift $\hat{x_0}$ of the point $x_0$. However, two such covering maps are isomorphic one to each other.

A closed geodesic $\gamma:[0,1] \rightarrow S$ of $S$ is called \emph{primitive} if no strict restriction of $\gamma$ defines a closed geodesic of $S$\footnote{Note that it does not force the geodesic to be simple.}. Observe that it amounts to saying that the element of $\pi_1(S)$ induced by $\gamma$ is not of the form $a^n$, with $n \geq 2$ and $a \in \pi_1(S)$.  A loop of a surface is called \emph{essential} if its free homotopy class is non-trivial and if it is not freely homotopic to a small loop around a puncture.

\begin{proposition} \label{coveringmap1}
Suppose $\gamma_{1}$ and $\gamma_{2}$ are two primitive closed geodesics which meet transversely at the point $x_0$. Denote by $\hat{S} \rightarrow S$ a covering map associated to $(\gamma_{1},\gamma_{2},x_0)$. Then
\begin{enumerate}
\item The surface $\hat{S}$ is homeomorphic to either the torus with one puncture or the sphere with three punctures.
\item If one of the loops $\gamma_{1}$ or $\gamma_{2}$ is simple, then the surface $\hat{S}$ is homeomorphic to the torus with one puncture.
\item If the surface $\hat{S}$ is homeomorphic to the torus with one puncture, then the canonical lifts of $\gamma_{1}$ and $\gamma_{2}$ in $\hat{S}$ are essential loops of $\hat{S}$ generating the homology of $\hat{S}$.
\item If $\hat{S}$ is homeomorphic to the sphere with three punctures, then none of the lifts of the loops $\gamma_{1}$ or $\gamma_{2}$ is simple.
\end{enumerate}
\end{proposition}

Figure~\ref{FigExemSph} displays an example showing that case 4. of this proposition is nonempty: there exists two based loops of the genus 2 closed surface whose associated covering map is the sphere with three punctures. Indeed, Proposition~\ref{coveringmap2} shows that the covering map associated to the two loops formed by the red loop with auto-intersection is the sphere with three punctures. Hence, if we call $a$ and $b$ some generators of the fundamental group of the sphere with three punctures such that the lift of the red curve is homotopic to $ab$, then the lift of the blue curve to the sphere with three punctures is homotopic to $ab^2$. In this case, the lifts of the curves generate the fundamental group of the sphere with 3 punctures. This proves that the associated covering map is the sphere with three punctures.

\begin{figure}[ht]
\includegraphics[width=.47\linewidth]{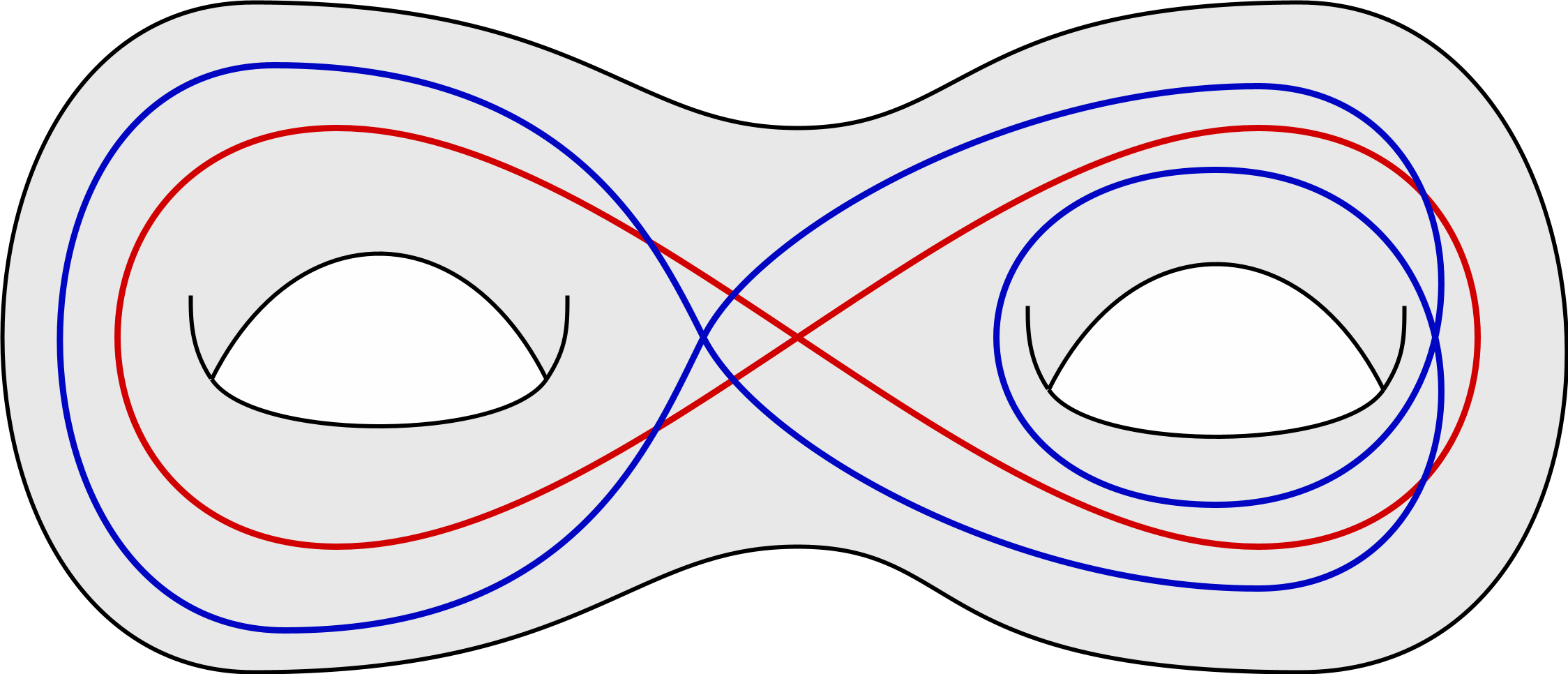}
\hfill
\includegraphics[width=.47\linewidth]{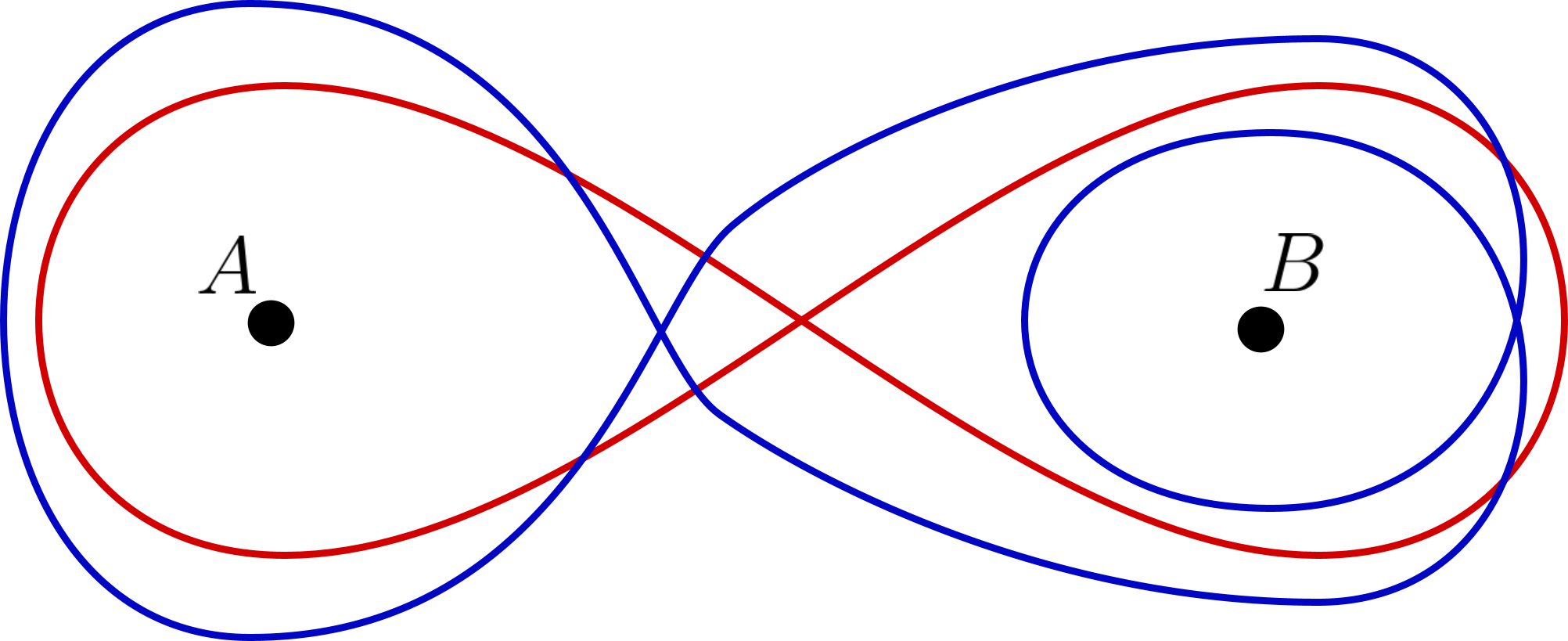}
\caption{\label{FigExemSph}Left: for some generators $a,b,c,d$ of $\pi_1(S)$, where $S$ is the genus 2 closed surface, the left red curve is homotopic to $ab$ and the left blue curve is homotopic to $ab^2$. On the right, the associated covering map is the sphere with three punctures, and for some generators $a,b$ of the fundamental group of the three punctured sphere, the red loop is homotopic to $ab$ and the blue one to $ab^2$.
}
\end{figure}

\begin{proposition} \label{coveringmap2}
Suppose $\gamma_{1}$ and $\gamma_{2}$ are loops based at a point $x_0$ and that the concatenation of $\gamma_{1}$ and $\gamma_{2}$ is a closed primitive geodesic of $S$ with a geometric transverse autointersection at the point $x_0$. Denote by $\hat{S} \rightarrow S$ the covering map associated to $(\gamma_{1},\gamma_{2},x_0)$. Then
\begin{enumerate}
\item The surface $\hat{S}$ is homeomorphic to the sphere with three punctures.
\item Suppose that the lifts $\hat \gamma_{1}$ and $\hat \gamma_{2}$ of $\gamma_{1}$ and $\gamma_{2}$ to $\hat{S}$ relatively to $\hat x_0$ are homotopic to simple curves. Denote $a,b,c$ the canonical classes of the curves, each one winding once around one of the three punctures $A$, $B$ and $C$ of $\hat S$, and not winding around the others, and such that $ab=c$. Then there exists a homeomorphism of $\hat S$ sending $([\hat\gamma_1]_{\hat x_0},\, [\hat\gamma_2]_{\hat x_0})$ to $(a,b^{-1})$ (see Figure \ref{FigModifQ}, left).
\end{enumerate}
\end{proposition}

\begin{figure}[ht]
\begin{center}
\begin{minipage}{.3\linewidth}
\begin{center}
\begin{tikzpicture}[scale=1]
\draw[color=blue!70!black, thick, ->] (-.5,.5) .. controls +(-1,1) and +(-1,-1) .. (-.5,-.5) -- (0,0) -- (-.5,.5);
\draw[color=red!70!black, thick, ->] (.5,.5) .. controls +(1,1) and +(1,-1) .. (.5,-.5) -- (0,0) -- (.5,.5);
\draw[color=blue!70!black] (-.2,.9) node[left]{$\wh\alpha_1$};
\draw[color=red!70!black] (.2,.9) node[right]{$\wh\alpha_2$};
\draw (0,0) node{$+$} node[above]{$\wh p$};
\draw (.6,0) node{$\bullet$} node[right]{$B$};
\draw (-.6,0) node{$\bullet$} node[left]{$A$};
\end{tikzpicture}
\end{center}
\end{minipage}
\hspace{30pt}
\begin{minipage}{.3\linewidth}
\begin{center}
\begin{tikzpicture}[scale=1]
\draw[color=blue!70!black, thick, ->] (-.5,.5) .. controls +(-1,1) and +(-1,-1) .. (-.5,-.5) -- (0,0) -- (-.5,.5);
\draw[color=red!70!black, thick, ->] (.5,-.5) .. controls +(1,-1) and +(1,1) .. (.5,.5) -- (0,0) -- (.5,-.5);
\draw[color=blue!70!black] (-.2,.9) node[left]{$\wh\alpha_1$};
\draw[color=red!70!black] (.2,.9) node[right]{$\wh\alpha_2$};
\draw (0,0) node{$+$} node[above]{$\wh p$};
\draw (.6,0) node{$\bullet$} node[right]{$B$};
\draw (-.6,0) node{$\bullet$} node[left]{$A$};
\end{tikzpicture}
\end{center}
\end{minipage}
\caption{\label{FigModifQ} Configurations of Proposition \ref{coveringmap2}. The right one is impossible.}
\end{center}
\end{figure}
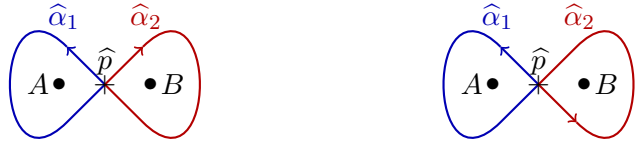 

\begin{proof}[Proof of Proposition \ref{coveringmap1}]
1. Denote by $\varphi$ the map $\pi_{1}(\hat{S},\hat{x_0})\rightarrow H_1(\hat{S},\R)$ which is the composition of the Hurewicz map $\pi_{1}(\hat{S},\hat{x_0}) \rightarrow H_{1}(\hat{S},\Z) \equiv \pi_{1}(\hat{S},\hat{x_0})/[\pi_{1}(\hat{S},\hat{x_0}),\pi_{1}(\hat{S},\hat{x_0})]$ with the map $H_{1}(\hat{S},\Z) \rightarrow H_{1}(\hat{S},\R) \equiv H_{1}(\hat{S},\Z) \otimes_{\Z} \R$. As the elements $[\gamma_{1}]_{x_0}$ and $[\gamma_{2}]_{x_0}$ generate the group $\pi_{1}(\hat{S},\hat{x_0})$, the vectors $\varphi([\gamma_{1}]_{x_0})$ and $\varphi([\gamma_{2}]_{x_0})$ generate the vector space $H_{1}(\hat{S},\R)$ and $\dim(H_{1}(\hat{S},\R))\leq 2$. Hence the surface $\hat{S}$ is either the sphere with $0$, $1$, $2$ or $3$ punctures or the $2$-torus with $0$ or $1$ puncture, as $\hat{S}$ is orientable.

However, a closed geodesic of $S$ has a non trivial free homotopy class: otherwise, there would exist closed geodesics in $\Hy^2$. Hence $\pi_{1}(\hat{S},\hat{x_0})$ is not trivial and $\hat{S}$ is not homeomorphic to the sphere nor to the sphere with one puncture.

Moreover, the classes $[\gamma_{1}]_{x_0}$ and $[\gamma_{2}]_{x_0}$ are distinct and $[\gamma_{1}]_{x_0} \neq [\gamma_{2}]_{x_0}^{-1}$ as the geodesics $\gamma_{1}$ and $\gamma_{2}$ are transverse and there is a unique geodesic representative in a free homotopy class of a loop. More generally, as $\gamma_{1}$ and $\gamma_{2}$ are primitive closed geodesics, it is not possible that there exist nontrivial integers $n_{1}$ and $n_{2}$ such that $[\gamma_{1}]_{x_0}^{n_{1}}=[\gamma_{2}]_{x_0}^{n_{2}}$: otherwise, a geodesic representative of this class is obtained by turning $\mathrm{lcm}(n_{1},n_{2})$ times around a closed geodesic and $\gamma_{1}$ as well as $\gamma_{2}$ would not be primitive. Hence $\pi_{1}(\hat{S},\hat{x_0})$ is not isomorphic to $\Z$ either and $\hat{S}$ is not homeomorphic to the sphere with two punctures.

Finally, as the surface $S$ is endowed with a hyperbolic metric, the group $\pi_{1}(S,x_0)$ does not contain $\Z^2$. The surface $\hat{S}$ cannot be homeomorphic to the $2$-torus. Hence the surface $\hat{S}$ is either the sphere with three punctures or the torus with one puncture.
\medskip

2. We will prove point 4., point 2. being a consequence of it. For $i=1,2$, denote by $\tilde{\gamma}_{i}$ the lift of $\hat \gamma_{i}$ to $\tilde S$ which contains the point $\tilde{x_0}$, \emph{i.e.} the connected component of $\pi^{-1}(\gamma_{i})$ which contains the point $\tilde{x_0}$. In what follows, we will use a fixed covering map $\tilde{S} \rightarrow \hat{S}$.

Suppose for a contradiction that the loop $\hat{\gamma}_{1}$ is simple and that the surface $\hat{S}$ is homeomorphic to the sphere with three punctures. Then the loop $\hat{\gamma}_{1}$ is freely homotopic to a small loop $\hat\gamma'_{1}$ around a puncture which is disjoint from $\hat{\gamma}_{2}$. However, it is impossible. Indeed, the geodesic lines $\tilde{\gamma}_{1}$ and $\tilde{\gamma}_{2}$ meet transversely at the point $p$ and hence the lift of $\hat\gamma'_{1}$ to $\tilde S$, which is as a finite distance to $\tilde{\gamma}_{1}$, has to meet the geodesic line $\tilde{\gamma}_{2}$.
\medskip

3. Suppose the surface $\hat{S}$ is homeomorphic to the torus with one puncture. If one of the loops $\hat{\gamma}_{1}$ or $\hat{\gamma}_{2}$ is not an essential loop of $\hat{S}$, then it represents the trivial class in $H_{1}(\hat{S},\R)$. But this is not possible as $\dim(H_{1}(\hat{S},\R))=2$ and the classes of $\hat{\gamma}_{1}$ and $\hat{\gamma}_{2}$ in $H_{1}(\hat{S},\R)$ generate the vector space $H_{1}(\hat{S},\R)$.
\end{proof}

We now come to the proof of Proposition \ref{coveringmap2}. As a first step, we get a weak version of it, similar to the conclusion of Proposition \ref{coveringmap1}.

\begin{lemma}\label{Lemcoveringmap2}
Under the hypotheses of Proposition \ref{coveringmap2}, the surface $\hat S$ is homeomorphic to either the sphere with three punctures, or the torus with one puncture.
\end{lemma}

\begin{proof}[Proof of Lemma \ref{Lemcoveringmap2}]
Quite similarly to the proof of Proposition \ref{coveringmap1}, one can get that the surface $\hat S$ is homeomorphic to either the sphere with three punctures, or the torus with one puncture. As in Proposition \ref{coveringmap1}, we can prove that $\dim(H_{1}(\hat{S},\R)) \leq 2$.

Denote by $\gamma$ the primitive closed geodesic which is the concatenation of the loops $\gamma_{1}$ and $\gamma_{2}$ based at $p$. As the loop $\gamma$ has minimal length in its free homotopy class, it is not possible that either $[\gamma_{1}]_{x_0}$ or $[\gamma_{2}]_{x_0}$ is trivial. Moreover, if there existed integers $n_{1}$ and $n_{2}$ such that $[\gamma_{1}]_{x_0}^{n_{1}}=[\gamma_{2}]_{x_0}^{n_{2}}$, then the free homotopy class of $[\gamma_{1}]_{x_0}^{n_{1}}$ is represented geodesically by turning $n=\mathrm{lcm}(n_{1},n_{2})$ times around a closed geodesic $\gamma'$. Then the free homotopy class of $\gamma$ would be represented by turning $n/n_{1}+n/n_{2}$ times around $\gamma'$. Hence $\gamma$ would not be primitive, a contradiction.

The rest of the proof goes through as in the proof of Proposition \ref{coveringmap1}.
\end{proof}

Hence, to prove the first point of Proposition \ref{coveringmap2}, it remains to prove that $\hat S$ cannot be homeomorphic to the torus with one puncture. We will need the following classical result of \cite{MR1512188}.

\begin{lemma}[Nielsen]\label{LemNielsen}
Cardinal two generating families of the free group $\langle a,b\rangle$ on two generators are obtained from the canonical one $(a,b)$ by so-called Nielsen transformations: permutations of elements of the basis, inversion of one of them, and multiplication of one of them by the other one (on the left or on the right).
\end{lemma}

\begin{proof}[Proof of Proposition \ref{coveringmap2}]
1. Suppose by contradiction that $\hat S$ is homeomorphic to the torus with one puncture. In this case, the classes of the loops $\hat\gamma_1$ and $\hat\gamma_2$ span the fundamental group $\pi_1(\hat S,\hat x_0)$.

Observe that any Nielsen transformation corresponds to a homeomorphism of the torus with one puncture. Hence, by Lemma~\ref{LemNielsen}, there exists a homeomorphism of $\hat S$ sending the pair of classes $([\hat\gamma_1]_{\hat x_0},\,[\hat\gamma_2]_{\hat x_0})$ to the canonical generators of $\pi_1(\hat S,\hat x_0)$. So from now we suppose that these classes are the canonical ones. An example of configuration of the loops $\hat\gamma_1$ and $\hat\gamma_2$ is depicted in Figure \ref{FigHomotopTor}. The curves $\hat \alpha_1$ and $\hat\alpha_2$ can be homotoped to curves having only one intersection. In particular, the curve $\hat\gamma$ can be homotoped to a one with no self intersection (see Figure~\ref{FigModifP})
has no transverse intersection in $\hat S$. This implies that $\gamma$ has no transverse intersection in $S$ for one of the deck transformations $T_1$ or $T_2$ associated respectively to the loops $\gamma_1$ and $\gamma_2$, contradiction.

\begin{figure}[ht]
\begin{minipage}{.45\textwidth}
\begin{center}
\begin{tikzpicture}[scale=3]
\clip (-.5,.5) rectangle (.5,1.5);
\foreach \i in {-1,...,0}{
\draw[color=blue!60!black, thick] (-.25+\i,-.25-\i) .. controls +(.2,-.2) and +(-.2,.2) .. (1.25+\i,.25-\i);
\draw[color=blue!60!black, thick] (.25+\i,1.25-\i) .. controls +(.2,-.2) and +(-.2,.2) .. (-.25+\i,-.25-\i);
\draw[color=blue!60!black, thick] (.25+\i,2.25-\i) .. controls +(.2,-.2) and +(-.2,.2) .. (-.25+\i,.75-\i);
}
\fill [white] (-.5,.5) -- (-.5,1.5) -- (0,1) -- cycle;

\draw (.32,1.32) node{$\times$};
\draw[color=blue!60!black] (.18,1.18) node{$\hat\alpha_1$};
\draw[color=green!60!black] (-.3,.7) node{$\hat\alpha_2$};
\draw[thick] (-.5,.5) rectangle (.5,1.5);

\foreach \i in {-1,...,0}{
\draw[color=green!60!black, thick] (-.25+\i,.75-\i) .. controls +(.2,-.2) and +(-.2,.2) .. (1.25+\i,1.25-\i);
}
\clip (-.5,.5) -- (.5,.5) -- (-.5,1.5) -- cycle;
\foreach \i in {-1,...,0}{
\draw[color=green!60!black, thick] (-.25+\i,-.25-\i) .. controls +(.2,-.2) and +(-.2,.2) .. (1.25+\i,.25-\i);
\draw[color=green!60!black, thick] (.25+\i,2.25-\i) .. controls +(.2,-.2) and +(-.2,.2) .. (-.25+\i,.75-\i) ;
}
\end{tikzpicture}
\caption{\label{FigHomotopTor}Example of trajectory $\hat\alpha$ in $\hat S\simeq \T^2\setminus\{0\}$, obtained as the concatenation of $\hat\alpha_1$ and $\hat\alpha_2$ (left).}
\end{center}
\end{minipage}
\hfill
\begin{minipage}{.5\textwidth}
\begin{minipage}{.48\linewidth}
\begin{center}
\begin{tikzpicture}[scale=1]
\draw[dotted, color=blue!70!black, thick, ->] (-.7,.7) .. controls +(-1,1) and +(-1,-1) .. (-.7,-.7) -- (-.5,-.5);
\draw[color=blue!70!black, thick, ->] (-.7,-.7) -- (0,0) -- (-.7,.7);
\draw[dotted, color=red!70!black, thick, <-] (.7,.7) .. controls +(1,1) and +(1,-1) .. (.7,-.7) -- (.5,-.5);
\draw[color=red!70!black, thick, <-] (.7,-.7) -- (0,0) -- (.7,.7);
\draw[color=blue!70!black] (-.2,0) node[left]{$\hat \gamma_1$};
\draw[color=red!70!black] (.2,0) node[right]{$\hat \gamma_2$};
\draw (0,0) node{$\bullet$} node[above]{$\hat x_0$};
\end{tikzpicture}
\end{center}
\end{minipage}
\hfill
\begin{minipage}{.48\linewidth}
\begin{center}
\begin{tikzpicture}[scale=1]
\draw[dotted, color=blue!70!black, thick, ->] (-.7,.7) .. controls +(-1,1) and +(-1,-1) .. (-.7,-.7) -- (-.5,-.5);
\draw[color=blue!70!black, thick] (-.7,-.7) -- (-.5,-.5) to[out=45,in=180] (0,-.3);
\draw[color=blue!70!black, thick] (-.7,.7) -- (-.5,.5) to[out=-45,in=180] (0,.3);
\draw[dotted, color=red!70!black, thick, <-] (.7,.7) .. controls +(1,1) and +(1,-1) .. (.7,-.7) -- (.5,-.5);
\draw[color=red!70!black, thick] (.7,-.7) -- (.5,-.5) to[out=135,in=0] (0,-.3);
\draw[color=red!70!black, thick] (.7,.7) -- (.5,.5) to[out=-135,in=0] (0,.3);
\draw (0,0) node{$\bullet$} node[left]{$\hat x_0$};
\end{tikzpicture}
\end{center}
\end{minipage}
\caption{\label{FigModifP}Modification of $\hat \gamma_1$ and $\hat \gamma_2$ close to $\hat x_0$ in proof of Proposition \ref{coveringmap2}.}
\end{minipage}
\end{figure}

\medskip

2. It is a classical fact that the only homotopy classes of simple loops in the three punctured sphere are the ones winding once around one puncture and not around the others (and the trivial one), in other words there are 7 of them: $0,a,a^{-1},b,b^{-1},c,c^{-1}$.  As $([\hat\gamma_1]_{\hat x_0},\,[\hat\gamma_2]_{\hat x_0})$ generates $\pi_1(\hat S,\hat x_0)$, there are two distinct classes of such couples of classes up to homeomorphism, represented by $(a,b)$ and $(a,b^{-1})$. But only one of them corresponds to a curve $\hat\gamma$ with a transverse self-intersection (see Figure~\ref{FigModifQ}), which concludes the proof.
\end{proof}

\subsection{Preliminaries on homological rotation sets}

As the main result of this section (Proposition~\ref{lifttorus}) is stated in terms of homological rotation sets, we state here some facts about these sets, which were defined in the introduction (Definition~\ref{DefHomolo}).

Suppose $S=\T^2=\R^2 / \Z^2$. Then the homology classes of the loops $t \mapsto (t,0)$ and $t \mapsto (0,t)$ form a basis of $H_1(\T^2)$ and induce an identification $H_1(\T^2) \equiv \R^2$. Via this identification, the set $\rho_{H_1}(f)$ is identified with the rotation set $\rho(f)$ as defined by Misiurewicz and Ziemian \cite{MR1053617}.

There is no obvious relationship between homotopical and homological rotation sets. However, we have the following proposition.

\begin{proposition} \label{topiclogic}
Suppose that $S$ is closed, and let $(\alpha,\beta)$ be a geodesic line which projects to a closed geodesic $\gamma$ of $S$. Suppose $(\alpha,\beta,v)$, with $v>0$, is an extremal point of $\rho(\tilde{f})$. Then the homological vector $\frac{v}{\ell(\gamma)}[\gamma]_{H_1}$ belongs to the homological rotation set $\rho_{H_1}(f)$. Moreover, there exists a point $x \in S$ such that the orbit of $x$ realises $\frac{v}{\ell(\gamma)}[\gamma]_{H_1}$ in $\rho_{H_1}(f)$ and $(\gamma,v)$ in $\rho(f)$.
\end{proposition}

\begin{proof} 
Fix a generating set $\mathcal{G}$ of the group $\pi_{1}(S)$, which we see as the group of deck transformations of the covering map $\tilde{S}=\Hy^2 \rightarrow S$. We denote by $l_{\mathcal{G}}$ the wordlength with respect to this generating set $\mathcal{G}$. The proof of Proposition \ref{topiclogic} is a consequence of Proposition \ref{sublindistance} and of the well-known Svar\v{c}-Milnor lemma.

\begin{lemma}[Svar\v{c}-Milnor] \label{svarcmilnor}
Fix $\tilde{x} \in \Hy^{2}=\tilde{S}$. Then there exists a constant $C>1$ such that, for any $T \in \pi_{1}(S)$,
$$ \frac{1}{C} d(\tilde{x},T\tilde{x}) \leq l_{\mathcal{G}}(T) \leq C d(\tilde{x},T\tilde{x}).$$
\end{lemma}

Take a point $\tilde{x}$ of $\Hy^2$ given by Proposition \ref{sublindistance} and let $x =\pi(\tilde{x})$. Recall we called $l_{n,x}$ the loop which is the concatenation of the path $(f_{t}(x))_{t \in [0,n]}$ with a geodesic path $g_{f^{n}(x),x}$ between the points $f^{n}(x)$ and $x$ of length lower than or equal to $D=\mathrm{diam}(S)$. Denote by $T_n \in \pi_{1}(S)$ the deck transformation corresponding to the loop $l_{n,x}$ with basepoint $\tilde{x}$. By definition of $l_{n,x}$, for any $n \geq 1$,
$$d\big(T_{n}(\tilde{x}),\tilde{f}^{n}(\tilde{x}))\big) \leq D.$$
For any $n \geq 1$, let $k_{n}=\lfloor \frac{nv}{\ell(\gamma)} \rfloor$. Denote by $T$ the deck transformation corresponding to $(\alpha,\beta)$. Then, for any $n \geq 1$,
$$ \begin{array}{rcl}
d\big(T^{k_{n}}(\tilde{x}),T_{n}(\tilde{x})\big) &  \leq &  d\big(T^{k_{n}}(\tilde{x}),T^{k_{n}}(\pi_{\alpha,\beta}(\tilde{x}))\big)
+d\big(T^{k_{n}}(\pi_{\alpha,\beta}(\tilde{x})),\pi_{\alpha,\beta}(\tilde{f}^{n}(\tilde{x}))\big) \\
& & + d\big(\pi_{\alpha,\beta}(\tilde{f}^{n}(\tilde{x})),\tilde{f}^{n}(\tilde{x})\big)
+d\big(\tilde{f}^{n}(\tilde{x}),T_{n}(\tilde{x})\big).
\end{array}$$

However, for any $n \geq 1$,
$$ \left\{ \begin{array}{rcl}
d\big(\tilde{f}^{n}(\tilde{x}),T_{n}(\tilde{x})\big) & \leq & D \\
d\big(T^{k_{n}}(\tilde{x}),T^{k_{n}}(\pi_{\alpha,\beta}(\tilde{x}))\big)& = & d\big(\tilde{x},\pi_{\alpha,\beta}(\tilde{x})\big)
\end{array}
\right.$$
and, by Proposition \ref{sublindistance},
\[\lim_{n \rightarrow +\infty} \frac{1}{n}d(\pi_{\alpha,\beta}(\tilde{f}^{n}(\tilde{x})),\tilde{f}^{n}(\tilde{x})) = 0 \]
and 
\begin{align*}
\lim_{n \rightarrow +\infty} \frac{1}{n} & d\big(T^{k_{n}}(\pi_{\alpha,\beta}(\tilde{x})),\pi_{\alpha,\beta}(\tilde{f}^{n}(\tilde{x}))\big) \\
 & = \lim_{n \rightarrow +\infty} \frac{1}{n}\Big|d\big(T^{k_{n}}(\pi_{\alpha,\beta}(\tilde{x})),\pi_{\alpha,\beta}(\tilde{x})\big)-d\big(\pi_{\alpha,\beta}(\tilde{x}),\pi_{\alpha,\beta}(\tilde{f}^{n}(\tilde{x}))\big)\Big| & \\
& = |v-v|=0.
\end{align*}
Hence $$ \lim_{n \rightarrow +\infty} \frac{1}{n} d (T^{k_{n}}(\tilde{x}),T_{n}(\tilde{x}))=0.$$
By the Svar\v{c}-Milnor lemma, this implies that $ \displaystyle \lim_{n \rightarrow +\infty} \frac{1}{n}l_{\mathcal{G}}(T^{-k_{n}}T_n)=0$ so that
$$ \lim_{n \rightarrow +\infty} \frac{1}{n} [l_{n,x}]_{H_1}=\lim_{n \rightarrow +\infty} \frac{1}{n} [\gamma^{k_{n}}]_{H_1}=\lim_{n \rightarrow +\infty} \frac{k_{n}}{n} [\gamma]_{H_1}=\frac{v}{\ell(\gamma)}[\gamma]_{H_1}$$
and the orbit of $x$ realises $\frac{v}{\ell(\gamma)}[\gamma]_{H_1}$. The orbit of $x$ also realises $(\gamma,v)$ in $\rho(f)$ by the remark below Proposition \ref{sublindistance}.
\end{proof}

\subsection{Homological consequences when two geodesics intersect}

Let $f \in \mathrm{Homeo}_0(S)$ and let $\gamma_0$ and $\gamma_1$ be two closed geodesics of $S$, one of which is simple. Suppose that these closed geodesics meet at a point $x_0$ of $S$. Let us denote by $\hat{\pi}:(\hat{S},\hat{x}_{0}) \rightarrow (S,x_0)$ a covering map associated to $(\gamma_0,\gamma_1,x_0)$. By Proposition \ref{coveringmap1}, the surface $\hat{S}$ is homeomorphic to $\T^2 \setminus \left\{ x_{\infty} \right\}$, where $x_\infty$ is a point of $\T^2$. We set $\hat{S}=\T^2 \setminus \left\{ x_{\infty} \right\}$. We fix a lift $\tilde{x}_0$ of the point $x_0$.

By the lifting theorem, there exists a homeomorphism $\hat{f}$ in $\mathrm{Homeo}_{0}(\T^2 \setminus \left\{ x_{\infty} \right\})$ such that $\hat{\pi} \circ \hat{f}=f \circ \hat{\pi}$. Observe that $\T^2$ is the Alexandroff compactification of $\T^2 \setminus \left\{ x_{\infty} \right\}$ so that the homeomorphism $\hat{f}$ extends to an element of $\mathrm{Homeo}_0(\T^2)$ which fixes the point $x_\infty$. By abuse of notation, we also denote by $\hat{f}$ this element of $\mathrm{Homeo}_0(\T^2)$ and we also call it a lift of $f$.

As there is only one homotopy class of isotopy $(f_t)_{ t \in [0,1]}$ between $\Id_S$ and $f$, there is only one homotopy class of isotopy between $\Id_{\T^2}$ and $\hat{f}$ which lifts $(f_t)_{t \in [0,1]}$. We fix such a lift $(\hat{f}_t)_{ t \in [0,1]}$. We denote $\rho(\hat{f})=\rho_{H_{1}(\T^2)}(\hat{f})$ the rotation set of $\hat{f}$ with respect to this isotopy. We denote by $(\tilde{f}_t)_{ t\in [0,1]}$ a lift of the isotopy $(f_t)_{ t \in [0,1]}$ to $\Hy^2$. Of course, all those isotopies can be extended in the usual way to any $t \in \R$.

\begin{proposition} \label{lifttorus}
Let $\gamma_0$ and $\gamma_1$ be two closed geodesics of $S$, one of which is simple. Suppose that there exist $v_0>0$ and $v_1>0$ such that $(\gamma_0,v_0)$ and $(\gamma_1,v_1)$ belong to $\rho(f)$.

Then the convex hull in $H_1(\T^2)$ of $0$, $\frac{v_0}{\ell(\gamma_0)}[\hat{\gamma}_0]_{H_{1}(\T^2)}$ and $\frac{v_1}{\ell(\gamma_1)}[\hat{\gamma}_1]_{H_{1}(\T^2)}$ is contained in $\rho(\hat{f})$.
\end{proposition}

With this proposition, we can use known theorems about the rotation set of a homeomorphism of the 2-torus in order to deduce the following corollary.

\begin{corollary} \label{entropy}
Let $\gamma_0$ and $\gamma_1$ be two closed geodesics of $S$, one of which is simple. Suppose these closed geodesics meet and that there exist $v_0>0$ and $v_1>0$ such that $(\gamma_0,v_0)$ and $(\gamma_1,v_1)$ belong to $\rho(f)$. Then
\begin{enumerate}
\item The topological entropy of $f$ is positive.
\item Any rational point of $H_1(S,\R)$ in the interior of the triangle $T$ formed by the points $0$, $\frac{v_0}{\ell(\gamma_0)}[\gamma_0]_{H_{1}(S)}$ and $\frac{v_1}{\ell(\gamma_1)}[\gamma_1]_{H_{1}(S)}$ is realised by a periodic orbit.
\end{enumerate}
\end{corollary}

Note that point 1. of this corollary was already known, even in the case where both geodesics are not simple (we will handle this case in next section). It was done in  \cite{PabloUnpublished, MR611385, MR2003742} by using a different technique based on Nielsen-Thurston classification as in \cite{MR1101087} for the torus case: one removes some fixed points of the homeomorphism and proves that the homeomorphism defined relative to these fixed points has a pseudo-Anosov component. However, it seems to us that this technique cannot give informations about the existence of new rotation vectors as in point 2. of this corollary, or as in the results in the next two sections.

\begin{proof}[Proof of Proposition \ref{lifttorus}]
We can suppose that $(\gamma_0,v_0)$ and $(\gamma_1,v_1)$ are extremal points of $\rho(f)$ as the set $\rho(\hat{f})$ is convex by a result by Misiurewicz and Ziemian (see \cite{MR1053617}). We denote by $(\alpha_0,\beta_0)$ the lift of $\gamma_0$ containing the point $\tilde{x}_0$ and by $T$ the deck transformation associated to $(\alpha_0,\beta_0)$. By Proposition \ref{sublindistance}, there exists a point $\tilde{x}$ in $\tilde{S}=\Hy^2$ such that
$$\left\{ \begin{array}{rcl}
\displaystyle \lim_{n \rightarrow +\infty} \frac{1}{n}d\big(\tilde{f}^n(\tilde{x}),\tilde{x}\big) & = & v_0\\
\displaystyle \lim_{n \rightarrow +\infty} \frac{1}{n}d\big(\pi_{\alpha_0,\beta_0}(\tilde{f}^n(\tilde{x})),\tilde{f}^n(\tilde{x})\big) & = & 0.
\end{array} \right.$$
Let us set, for any $n \geq 0$, $k_{n}=\lfloor \frac{nv_0}{\ell(\gamma_0)} \rfloor$. Then
\begin{align*}
\lim_{n \rightarrow +\infty} & \frac{1}{n} d\big(\pi_{\alpha_0,\beta_0}(\tilde{f}^n(\tilde{x})),T^{k_{n}}(\pi_{\alpha_0,\beta_0}(\tilde{x}))\big) \\
 & = \lim_{n \rightarrow +\infty} \frac{1}{n} \Big|d\big(\pi_{\alpha_0,\beta_0}(\tilde{f}^n(\tilde{x})),\pi_{\alpha_0,\beta_0}(\tilde{x}))- d(\pi_{\alpha_0,\beta_0}(\tilde{x}),T^{k_{n}}(\pi_{\alpha_0,\beta_0}(\tilde{x}))\big)\Big| \\
 &  =    |v_{0}-v_{0}|=0.
\end{align*}

Hence
$$ \lim_{n \rightarrow +\infty}  \frac{1}{n} d\big(\tilde{f}^n(\tilde{x}),T^{k_{n}}(\pi_{\alpha_0,\beta_0}(\tilde{x}))\big)=0.$$
Let $\hat{x}$ be the projection of $\tilde{x}$ on $\T^2 \setminus \left\{ x_{\infty} \right\}$. We endow $\T^2$ with the Euclidean metric $g_{Euc}$ which turn the simple closed curves $\hat{\gamma}_{0}$ and $\hat{\gamma}_1$ into length $1$ orthogonal geodesics. Let us call $g_{hyp}$ the hyperbolic metric on $\T^2 \setminus \left\{ x_{\infty} \right\}$ induced by the hyperbolic metric on $\tilde{S}= \Hy^2$.

Recall that the metric $g_{Hyp}$ is complete so that $g_{Hyp} \rightarrow \infty$ as we get closer to the point $x_{\infty}$. Hence there exists $C>0$ such that $g_{Euc} \leq C g_{Hyp}$. Therefore, if we call $d_{Euc}$ the Euclidean distance on $\tilde{S}$ induced by $g_{Euc}$, we have, for any $n \geq 1$,
$$ d_{Euc}\big(\tilde{f}^{n}(\tilde{x}), T^{k_{n}}(\pi_{\alpha_0,\beta_0}(\tilde{x})\big)) \leq C d\big(\tilde{f}^{n}(\tilde{x}),T^{k_{n}}(\pi_{\alpha_0,\beta_0}(\tilde{x}))\big)$$
so that
$$ \lim_{n \rightarrow +\infty} \frac{1}{n} d_{Euc}\big(\tilde{f}^{n}(\tilde{x}),T^{k_{n}}(\pi_{\alpha_0,\beta_0}(\tilde{x}))\big)=0$$
and
$$ \lim_{n \rightarrow +\infty} \frac{k_{n}}{n}[\hat{\gamma}_0]_{H_1}=\frac{v_0}{\ell(\gamma_0)}[\hat{\gamma}_0]_{H_1} \in \rho(\hat{f}).$$
In the same way, we prove that $\frac{v_0}{\ell(\gamma_1)}[\hat{\gamma}_1]_{H_1} \in \rho(\hat{f})$. Observe also that the point $x_{\infty}$ is fixed under $\hat{f}$ with homological rotation number $0$. It suffices to recall that the rotation set $\rho_{H^1(\T^2)}(\hat{f})$ is convex \cite{MR1053617} to conclude.
\end{proof}

To prove Corollary \ref{entropy}, we need to recall some facts about topological entropy. Fix a compact metric space $(X,d)$ and a homeomorphism $h$ of $X$. In our specific case, $X=S$ and $h=f$. For any integer $n \geq 1$, we define the Bowen distance
$$\begin{array}{rrcl}
d_{n}: & X \times X & \longrightarrow & \R_{+} \\
 & (x,y) & \longmapsto & \displaystyle \max_{0\leq k \leq n-1} d\big(h^{k}(x),h^{k}(y)\big)
\end{array}$$
which is topologically equivalent to $d$. For any $\epsilon >0$ and $n \geq 1$, a subset $A$ of $X$ is said to be \emph{$(n,\epsilon)$-separated} if, for any distinct points $x$ and $y$ of $A$, $d_{n}(x,y) \geq \epsilon$. By compactness of $X$, such a set has to be finite. Denote by $a_{n,\epsilon}$ the maximal cardinality of an $(n,\epsilon)$-separated subset of $X$. Then
 $$h_{top}(h)=\lim_{\epsilon \rightarrow 0} \liminf_{n \rightarrow +\infty} \frac{\log(a_{n,\epsilon})}{n}=\lim_{\epsilon \rightarrow 0} \limsup_{n \rightarrow +\infty} \frac{\log(a_{n,\epsilon})}{n}.$$

Let us recall another way to compute the topological entropy. For any integer $n>0$ and any $\epsilon >0$, we call $(n,\epsilon)$-ball any open ball of radius $\epsilon$ for the distance $d_{n}$. We will denote by $B_{n}(x,\epsilon)$ the $(n,\epsilon)$-ball of center $x \in X$. Denote by $b_{n,\epsilon}$ the minimal cardinality of a cover of $X$ by $(n,\epsilon)$-balls. Then   
$$h_{top}(h)=\lim_{\epsilon \rightarrow 0} \liminf_{n \rightarrow +\infty} \frac{\log(b_{n,\epsilon})}{n}=\lim_{\epsilon \rightarrow 0} \limsup_{n \rightarrow +\infty} \frac{\log(b_{n,\epsilon})}{n}.$$

We will use the two following classical properties of the topological entropy.
\begin{enumerate}
\item If $Y$ is an $h$-invariant closed subset of $X$, then $h_{top}(h) \geq h_{top}(h_{|Y})$.
\item If $\hat X$ is a compact metric space, if $\hat{\pi}:\hat{X} \rightarrow X$ is an onto continuous map, and $\hat{h}:\hat{X} \rightarrow \hat{X}$ is a homeomorphism such that $\hat{\pi}\hat{h}=h \hat{\pi}$, then $h_{top}(h) \leq h_{top}(\hat{h})$.
\end{enumerate}

To prove Corollary \ref{entropy}, we need the following general result about topological entropy.

\begin{proposition} \label{entropycovering}
Let $\hat{\pi}: \hat{X} \rightarrow X$ be an onto local isometry between the compact metric spaces $\hat{X}$ and $X$. Let $\hat{h}: \hat{X} \rightarrow \hat{X}$ and $h:X \rightarrow X$ be homeomorphisms such that $\hat{\pi}\hat{h}=h\hat{\pi}$. Then
$$h_{top}(\hat{h})=h_{top}(h).$$
\end{proposition}

Now, we use this proposition to prove Corollary \ref{entropy}. We will prove Proposition \ref{entropycovering} afterwards.

\begin{proof}[Proof of Corollary \ref{entropy}]
1. By Proposition \ref{lifttorus}, the rotation set of $\hat{f}$ has nonempty interior. By Theorem 1 of the article \cite{MR1101087} by Llibre and MacKay, there exists an $\hat{f}$-invariant compact subset $\hat{K}$ of $\T^2$ such that:
\begin{enumerate}
\item The set $\hat{K}$ does not contain any fixed point of $\hat{f}$. Hence the point $x_{\infty}$ does not belong to $\hat K$ and $\hat K \subset \hat{S}= \T^2 \setminus \left\{ x_{\infty} \right\}$.
\item The homeomorphism $\hat{f}_{|\hat{K}}$ is conjugated to a subshift with positive topological entropy.
\end{enumerate}
In particular, $h_{top}(\hat{f}_{|\hat{K}})>0$. Now, apply Proposition \ref{entropycovering} to $\hat{h}=\hat{f}_{|\hat{K}}$ and $h=f_{|K}$, where $K$ is the image of $\hat{K}$ under the covering map $\hat{S} \rightarrow S$. Then
$$ h_{top}(f_{|K})=h_{top}(\hat{f}_{|\hat{K}}).$$
However $h_{top}(f) \geq h_{top}(f_{|K})$ so that $h_{top}(f) >0$.
\medskip

2. Take a rational point $\eta$ of $H_1(S,\R)$ such that there exist real numbers $0< \lambda_0<1$ and $0< \lambda_1<1$ with $\lambda_0+ \lambda_1 <1$ such that
$$\eta= \lambda_0 \frac{v_{0}}{l(\gamma_{0})}[\gamma_{0}]_{ H_{1}(S)}+ \lambda_1 \frac{v_{1}}{l(\gamma_{1})}[\gamma_{1}]_{ H_{1}(S)}.$$
Let $$\hat{\eta}=\lambda_0 \frac{v_{0}}{l(\gamma_{0})}[\hat{\gamma}_{0}]_{ H_{1}(\T^2)}+ \lambda_1 \frac{v_{1}}{l(\gamma_{1})}[\hat{\ell}_{1}]_{ H_{1}(\T^2)} $$
and observe that the class $\hat{\eta}$ is a rational point of $H_1(\T^2)$.

Write $\hat{\eta}=\frac{p}{q}[\hat{\gamma}]_{H_1(\T^2)}$, where $p$ and $q$ are positive integers and $[\hat{\gamma}]_{H_1(\T^2)}$ is an integral undivisible class in $H_1(\T^2)$ and is hence represented by a simple loop $\hat{\gamma}$ of $\T^2$.

Then the class $\hat{\eta}$ is a rational point which lies in the interior of $\rho(\hat{f})$ by Proposition~\ref{lifttorus}. By Theorems by Franks \cite{MR958891} and Llibre-MacKay \cite{MR1101087}, the class $\hat{\eta}$ is realised by a primitive periodic orbit, that is:
\begin{enumerate}
\item there exists a point $\hat{x}$ of $\T^2 \setminus \left\{ x_{\infty} \right\}$ such that $\hat{f}^q(\hat{x})=\hat{x}$;
\item the loop $(\hat{f}_{t}(\hat{x}))_{t \in[0,q]}$ is homologous to the class $p[\hat{\gamma}]_{H_1(\T^2)}$.
\end{enumerate}
Let $x=\hat{\pi}(\hat{x})$ be the projection of the point $\hat{x}$ on the surface $S$. Then $f^{q}(x)=x$ and the loop $(f_t(x))_{ t \in [0,q]}$ is homologous to the class
$$ p [\hat{\pi} \circ \hat{\gamma}]_{H_1(S)}=q\eta$$
so that the vector $\eta \in H_1(S)$ is realised by a (primitive) periodic orbit.
\end{proof}

\begin{proof}[Proof of Proposition \ref{entropycovering}]
The relation $\hat{\pi} \hat{h}=h \hat{\pi}$ gives immediately that $h_{top}(\hat{h}) \geq h_{top}(h)$.  Hence it suffices to prove that $h_{top}(\hat{h}) \leq h_{top}(h)$.
Let $\displaystyle \delta = \sup_{x \in X} \# \hat{\pi}^{-1}(\left\{x \right\})$. Note that $\delta$ is finite by compactness of $\hat{X}$ and as $\hat{\pi}$ is a local isometry.

As $\hat{\pi}$ is a local homeomorphism  and $\hat{X}$ is compact, there exists $\alpha >0$ such that, for any distinct points $\hat{x}$ and $\hat{y}$ of $\hat{X}$ such that $\pi(\hat{x})=\pi(\hat{y})$, we have $\hat{d}(\hat{x},\hat{y}) \geq \alpha$.

Take $\epsilon >0$ small enough so that the following properties hold:
\begin{enumerate}
\item $\epsilon < \frac{\alpha}{2}$.
\item For any points $\hat{x}$ and $\hat{y}$ of $\hat{X}$ with $\hat{d}(\hat{x},\hat{y}) < \epsilon$, we have $\hat{d}(\hat{h}(\hat{x}),\hat{h}(\hat{y})) < \frac{\alpha}{2}$.
\item The restriction of the map $\hat{\pi}$ to any ball of radius $\epsilon$ is an isometry onto a  ball of $X$.
\end{enumerate} 

Fix $n>0$ and take a maximal $(n,\epsilon)$-separated subset $A_{n,\epsilon}$ of $X$ for $h$. The central point of the proof is the following claim.

\begin{claim} \label{ballcover}
$$\hat{X}= \bigcup_{\hat{x} \in \hat{\pi}^{-1}(A_{n,\epsilon})} B_{n}(\hat{x},\epsilon).$$
\end{claim}

Before proving the claim, let us see why it yields Proposition \ref{entropycovering}. The claim implies that
$$\begin{array}{rcl}
h_{top}(\hat{h}) & \leq & \displaystyle \lim_{\epsilon \rightarrow 0} \liminf_{n \rightarrow +\infty} \frac{\# \hat{\pi}^{-1}(A_{n,\epsilon})}{n} \\
 & \leq & \displaystyle \lim_{\epsilon \rightarrow 0} \liminf_{n \rightarrow +\infty} \frac{\log(\delta \# A_{n,\epsilon})}{n} \\
 & \leq & h_{top}(h).
\end{array}$$
\end{proof}

\begin{proof}[Proof of Claim \ref{ballcover}]
Let $\hat{y} \in \hat{X}$ and $y=\hat{\pi}(\hat{y}) \in X$. By maximality of the set $A_{n,\epsilon}$, there exists a point $x$ of $A_{n,\epsilon}$ such that, for any $0 \leq k \leq n-1$,
$$ d\big(h^{k}(y),h^{k}(x)\big) < \epsilon.$$

Let
$$ \hat{\pi}^{-1}(\left\{ x \right\})= \big\{ \hat{x}_{i} \ | \ 1 \leq i \leq l \big\},$$
where the points $\hat{x}_{i}$ are pairwise distinct. As $d(x,y) < \epsilon$, there exists an index $i$ such that $\hat{d}(\hat{y},\hat{x}_i) < \epsilon$. Likewise, as $d(h(y),h(x)) < \epsilon$, there exists an index $j$ such that $\hat{d}(\hat{h}(\hat{y}),\hat{h}(\hat{x}_j)) < \epsilon$.

Let us prove that $i=j$. Indeed, we have
$$\hat{d}(\hat{h}(\hat{x}_i),\hat{h}(\hat{x}_j)) \leq \hat{d}(\hat{h}(\hat{x}_i),\hat{h}(\hat{y}))+\hat{d}(\hat{h}(\hat{y}),\hat{h}(\hat{x}_j))< \frac{\alpha}{2}+\frac{\alpha}{2}=\alpha.$$
Moreover, $\hat{\pi}(\hat{h}(\hat{x}_i))=h(x)=\hat{\pi}(\hat{h}(\hat{x}_j))$. By the definition of $\alpha$, this implies that $\hat h(\hat x_i) = \hat h(\hat x_j)$ and thus that $\hat x_i=\hat x_j$. 

In the same way, an induction proves that, for any $0 \leq k < n$, 
$$\hat{d}(\hat{h}^k(\hat{y}), \hat{h}^k(\hat{x}_i)) < \epsilon.$$
Hence 
$$ \hat{y} \in \bigcup_{1 \leq i \leq l} B_{n}(\hat{x}_i,\epsilon) \subset \bigcup_{\hat{x} \in \hat{\pi}^{-1}(A_{n,\epsilon})} B_{n}(\hat{x},\epsilon).$$
\end{proof}

\section{Closed geodesics with auto-intersection}\label{SecClosed}

This section is the first one where we use le Calvez and Tal forcing theory. Its aim is to prove Theorem~\ref{ExistSuperChevalIntro} if the introduction, that we will state as Theorem~\ref{ExistSuperCheval}.

After introducing some tools of forcing theory in Subsection~\ref{SubSecForcing}, we will define rotational horseshoes and prove some properties of them (Subsection~\ref{SubSecmarkov}). We will then get the fact that geometric auto-intersections of closed trajectories give rise to $\F$-transverse intersections in Subsection~\ref{SubSecGeomtrans}. This will lead us to the main theorem of this section in Subsection~\ref{SubSecHorse}, which will be preceded by a first step of independent interest, about the forcing of new periodic orbits, performed in Subsection~\ref{SubSecCrea}.

\subsection{Some results of forcing theory for transverse trajectories}\label{SubSecForcing}

This paragraph is a short introduction to the techniques and the results of Le Calvez and Tal \cite{MR3787834,1803.04557} that will be used in the sequel.

In the sequel, we will call \emph{line} any properly embedded topological line of the plane. For any surface $S$, we call \emph{singular foliation of $S$} any foliation $\F$ of an open subset $\dom \F$ of $S$. The set $S \setminus \dom \F$ is called the \emph{set of singularities} of $\F$. We will call \emph{end of a leaf $\phi$} either its $\alpha$ or its $\omega$-limit in $S$ or in $\tilde S$ (depending on the context).

Let $\F$ be an oriented nonsingular foliation of the plane. By classical plane topology (see Haefliger-Reeb \cite{MR89412}), each leaf $\phi$ of the foliation is a line, hence its complement possesses two connected components: the left of $\phi$, denoted by $L(\phi)$, and the right of $\phi$, denoted by $R(\phi)$ (that are chosen according to a fixed orientation of the plane and the orientation of $\phi$).

\begin{definition}
Let $\F$ be an oriented nonsingular foliation of a surface\footnote{Not necessarily closed.} $S$ and $\alpha : [0,1]\to S$ be a path. For $x\in S$, we denote by $\phi_x$ the leaf of $\F$ passing by $x$. We say that $\alpha$ is \emph{positively transverse} to $\F$ (abbreviated by $\F$-transverse) if for any $t\in[0,1]$, in the universal cover\footnote{This universal cover is always homeomorphic to $\R^2$, as there is no nonsingular foliation on the sphere.} of $S$ one has
\[\wt\alpha([0,t)) \subset L(\wt \phi_{\wt \alpha(t)})
\qquad \text{and} \qquad
\alpha((t,1]) \subset R(\wt \phi_{\wt \alpha(t)}).\]
\end{definition}

Let $\F$ be a (singular) foliation of a surface.
The following result can be obtained as a combination of \cite{MR2217051} with \cite{bguin2016fixed}.

\begin{theorem}\label{ThExistIstop}
Let $S$ be a surface and $f\in\Homeo_0(S)$.
Then there exists an isotopy $I$ linking $\Id$ to $f$, a transverse topological oriented singular foliation $\F$ of $S$ with $(\dom \F)^\complement = \bigcap_{t\in[0,1]} \fix I^t \subset \fix f$, and for any $z\in \dom\F$, a $\F$-transverse path linking $z$ to $f(z)$ which is homotopic in $\dom \F$, relative to its endpoints, to the arc $(I^t(z))_{t\in[0,1]}$.
\end{theorem}

\begin{definition}\label{DefDessus}
Let $\phi$, $\phi_1$ and $\phi_2$ three oriented lines of the plane. We will say that \emph{$\phi_2$ is above $\phi_1$ relative to $\phi$} if 
\begin{itemize}
\item these three lines are pairwise disjoint;
\item none of these lines separates the two others;
\item if $\alpha_i$, $i=1,2$, are two disjoint paths linking a point of $\phi_i$ to a point $\phi(t_i)$, then\footnote{Each line is parametrized according to its orientation.} $t_2>t_1$.
\end{itemize}
\end{definition}

Let $\F$ be an oriented nonsingular foliation of the plane, $J_1$, $J_2$ be two intervals and $\alpha_i = J_i\to\R^2$, $i=1,2$, two $\F$-transverse paths.

\begin{definition}\label{DefInterTrans}
We say that $\alpha_1 : J_1\to\R^2$ and $\alpha_2 : J_2\to\R^2$ \emph{intersect $\F$-transversally and positively} if there exists $a_i<t_i< b_i\in J_i$ such that
\begin{itemize}
\item $\phi_{\alpha_1(t_1)} = \phi_{\alpha_2(t_2)} = \phi$ ;
\item $\phi_{\alpha_1(a_1)}$ is above $\phi_{\alpha_2(a_2)}$ relative to $\phi$ ;
\item $\phi_{\alpha_2(b_2)}$ is above $\phi_{\alpha_1(b_1)}$ relative to $\phi$.
\end{itemize} 
\end{definition}

The same notion can be defined in $\dom \F$, by asking that some lifts of the paths to $\wt\dom\F$ intersect $\F$-transversally.

In the sequel, when it is obvious from the context, we will omit the mention ``and positively'' when talking about $\F$-transverse intersection.

Fix a homeomorphism $f \in \Homeo_0(S)$ and let $\F$ be a singular foliation of $S$ given by Theorem \ref{ThExistIstop}. We denote by $\hat{f}$ the canonical lift of $f$ to the universal cover $\wt \dom \F$ of $\dom \F$.

\begin{definition}\label{DefAdmis}
We say that a $\F$-transverse path $\alpha : [a,b] \to \dom \F$ is \emph{admissible of order $n$} if there exists a lift $\hat{\alpha}$ of $\alpha$ to $\wt\dom \F$ such that $\hat{f}^n(\phi_{ \hat{\alpha}(a)}) \cap \phi_{\hat{\alpha}(b)} \neq \emptyset$.
\end{definition}

The following is the fundamental proposition of \cite{MR3787834} (Proposition 20).

\begin{proposition}\label{PropFondLCT}
Suppose that $\alpha_1 : [a_1,b_1] \to \dom\F$ and $\alpha_2 : [a_2,b_2] \to \dom\F$ are transverse paths that intersect $\F$-transversally at $\alpha_1(t_1) = \alpha_2(t_2)$. If $\alpha_1$ is admissible of order $n_1$, and $\alpha_2$ is admissible of order $n_2$, then $\alpha_1|_{[a_1,t_1]}\alpha_2|_{[t_2,b_2]}$ and $\alpha_2|_{[a_2,t_2]}\alpha_1|_{[t_2,b_2]}$ are both admissible of order $n_1+n_2$.
\end{proposition}

A consequence of this proposition is the following (\cite{MR3787834}, Proposition 23).

\begin{proposition}\label{PropFondLCT2}
Suppose that $\alpha : [a, b] \to\dom \F$ is a transverse path admissible of order $n$ and that $\alpha$ intersects itself $\F$-transversally at $\alpha(s)=\alpha(t)$, with $s < t$. Then $\alpha|_{[a,s]} \alpha|_{[t,b]}$ is admissible of order $n$ and $\alpha|_{[a,s]} \big(\alpha|_{[s,t]}\big)^q \alpha|_{[t,b]}$ is admissible of order $nq$ for every $q > 1$.
\end{proposition}

We finish this crash course by a result on admissibility of trajectories. It uses the following definition.

\begin{definition}
We say that a transverse path $\alpha : J\to \R^2$ has a \emph{leaf on its right} if there exists $a < b$ in $J$ and a leaf $\phi$ in $L(\phi_{\alpha(a)}) \cap R(\phi_{\alpha(b)})$ that lies in the right of $\alpha|_{[a,b]}$. Similarly, one can define the notion
of having a \emph{leaf on its left}.
\end{definition}

The following is Proposition~19 of \cite{MR3787834}.

\begin{proposition}\label{PropPasFondLCT}
Let $\alpha : [a, b] \to \dom \F$ be an $\F$-transverse path that is not admissible of order $n$ but is a subpath of an $\F$-transverse path of order $n$.
Then any lift $\hat\alpha$ of $\alpha$ to $\wt\dom\F$ has no leaf on its right and no leaf on its left.
\end{proposition}

\subsection{Markovian intersections}\label{SubSecmarkov}

We now study Markovian intersections. They will be used to get topological rotational horseshoes.

\begin{definition}\label{Defmarkov}
Let $S$ be a surface. We call \textit{rectangle} of $S$ a subset $R \subset S$ satisfying $R = h([0,1]^2)$ for some homeomorphism $h: [0,1]^2\to h([0,1]^2)\subset S$. We call \textit{sides} of $R$ the image by $h$ of the sides of $[0,1]^2$. We call \textit{horizontal} the sides $R^- = h([0,1] \times \left\{ 0 \right\})$ and $R^+ = h([0,1] \times \left\{ 1 \right\})$ and \textit{vertical} the two others. We say that a rectangle $R' \subset R$ is a \emph{strict horizontal} (resp. vertical) \emph{subrectangle} of $R$ if the horizontal (resp. vertical) sides of $R'$ are strictly disjoint from those of $R$ and the vertical (resp. horizontal) sides of $R'$ are included in those of $R$.
\end{definition}

\noindent
Given $x \in \R^2$, we will denote by $\pi_2(x)$ its second coordinate. Following \cite{MR2060531}, we define Markovian intersections in the following way:

\begin{definition}\label{def:markov}
Let $R_1$ and $R_2$ be two rectangles of a surface $S$. We say that the intersection $R_1 \cap R_2$ is \emph{Markovian} if there exists a homeomorphism $h$ from a neighbourhood of $R_1\cup R_2$ to an open subset of $\R^2$ such that:
\begin{itemize}
\item $h(R_2) = [0,1]^2$;
\item either $h(R_1^+) \subset \left\{x\mid \pi_2(x) > 1 \right\}$ and $h(R_1^-) \subset \left\{x\mid \pi_2(x) < 0 \right\}$, or $h(R_1^-) \subset \left\{x\mid \pi_2(x) > 1 \right\}$ and $h(R_1^+) \subset \left\{x\mid \pi_2(x) < 0 \right\}$;
\item $h(R_1) \subset \left\{x\mid \pi_2(x) < 0 \right\} \cup [0,1]^2 \cup \left\{x\mid \pi_2(x) > 1 \right\}$.
\end{itemize}
\end{definition}

\begin{figure}[h!]
\begin{center}
\includegraphics[scale=1]{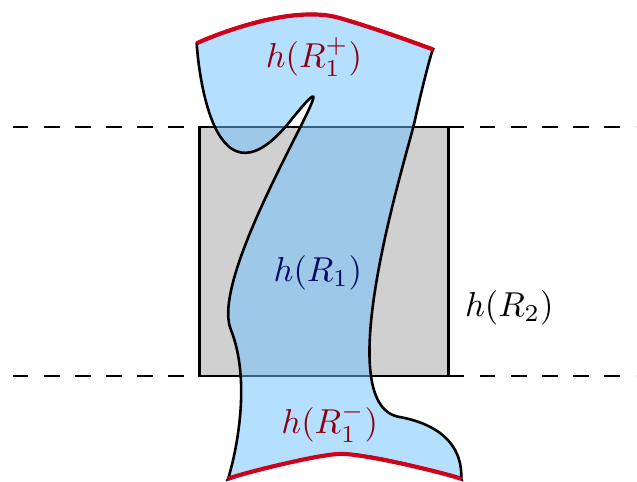}
\caption{A Markovian intersection\label{FigExMarkov}}
\end{center}
\end{figure}

The proofs of the following two results can be obtained as a combination of Theorem 16 and Corollary 12 of \cite{MR2060531}.

\begin{proposition}\label{LemChain}
Given a homeomorphism $f$ of a surface $S$ and three rectangles $R_1$, $R_2$ and $R_3$, if the intersections $f(R_1) \cap R_2$ and $f(R_2) \cap R_3$ are Markovian, then the intersection $f^2(R_1) \cap R_3$ is Markovian too (and in particular is nonempty).
\end{proposition}

\begin{proposition}\label{LemPointFixe}
Let $f$ be a homeomorphism of $S$ and $R$ a rectangle such that $f(R) \cap R$ is Markovian. Then there exists a fixed point for $f$ in $R$.
\end{proposition}

The following is a particular case of Homma's generalization \cite{MR58194} of Schoenflies theorem, it will be used to find rectangles and Markovian intersections.

\begin{theorem}[Homma]\label{PropHomma}
Any homeomorphism of 
\[\Big(\big((\R\times\{0\})\cup (\R\times \{1\}) \cup (\{0\}\times [0,1]) \cup (\{1\}\times [0,1]) \big)\cap B(0,10)\Big)\cup \partial B(0,10)\]
to its image in $\R^2$ can be extended to a self-homeomorphism of $\R^2$.
\end{theorem}

The following definition is a variation over the concept of rotational horseshoe defined in \cite{MR3784518} and used in \cite{1803.04557}.

\begin{definition}\label{DefRotHorse}
Let $S$ be a surface with negative Euler characteristic and $f$ a homeomorphism of $S$. We denote by $\tilde f$ the canonical lift of $f$ to $\wt S\simeq\Hy^2$.

We say that $f$ has a \emph{rotational horseshoe with deck transformations $U_1,\dots,U_k$} if there exists a rectangle $R$ of $\wt S$ such that, for any $1\le i \le k$, the intersections $U_i R \cap \tilde f(R)$ are Markovian.
\end{definition}

For any finite set $\{1,\dots,k\}^\Z$, we denote by $\sigma : \{1,\dots,k\}^\Z \rightarrow \{1,\dots,k\}^\Z$ the shift map, \emph{i.e} the map which, to a sequence $(a_{i})_{i \in \Z}$, associates the sequence $(a_{i+1})_{i \in \Z}$. 

From Propositions \ref{LemChain} and \ref{LemPointFixe}, it follows the following ``semi-conjugacy'' result (which allows to link our notion of horseshoe with the one of \cite{1803.04557}).

\begin{proposition}\label{PropConjug}
Suppose that $f$ has a rotational horseshoe with deck transformations $U_1,\dots,U_k$, and suppose that these transformations form a free group. Then there exists $q\ge 0$, a compact invariant subset $\tilde Y$ of $\tilde S$, a homeomorphism $g$ of $\tilde Y$ and a surjective continuous map $h_1 : \tilde Y \to \{1,\dots,k\}^\Z$ such that the following diagram commutes:
\begin{center}
\begin{tikzpicture}[scale=1]
\node (A) at (0,0) {$Y$};
\node (B) at (3,0) {$Y$};
\node (C) at (0,1.5) {$\tilde Y$};
\node (D) at (3,1.5) {$\tilde Y$};
\node (E) at (0,3) {$\{1,\dots,k\}^\Z$};
\node (F) at (3,3) {$\{1,\dots,k\}^\Z$};

\draw[->,>=latex,shorten >=3pt, shorten <=3pt] (A) to node[midway, above]{$f^q$} (B);
\draw[->,>=latex,shorten >=3pt, shorten <=3pt] (C) to node[midway, above]{$\tilde g$} (D);
\draw[->,>=latex,shorten >=3pt, shorten <=3pt] (E) to node[midway, above]{$\sigma$} (F);
\draw[->,>=latex,shorten >=3pt, shorten <=3pt] (C) to node[midway, left]{$h_1$} (E);
\draw[->,>=latex,shorten >=3pt, shorten <=3pt] (D) to node[midway, right]{$h_1$} (F);
\draw[->,>=latex,shorten >=3pt, shorten <=3pt] (C) to node[midway, left]{$\pi$} (A);
\draw[->,>=latex,shorten >=3pt, shorten <=3pt] (D) to node[midway, right]{$\pi$} (B);
\end{tikzpicture}
\end{center}
($\pi$ denotes the canonical projection)
and moreover
\begin{itemize}
\item the preimage by $h_1$ of every $p$-periodic sequence for $\sigma$ contains a point which projects to a $p$-periodic sequence for $f^q$;
\item for any $\tilde y\in \tilde Y$ and any $n>0$, one has
\begin{align*}
\tilde f^{qn}(\tilde y) & \in U^q_{h_1(\tilde f^{q(n-1)}(\tilde y))}U^q_{h_1(\tilde f^{q(n-2)}(\tilde y))}\dots U^q_{h_1(\tilde y)}(\tilde Y), \quad \text{and}\\
\tilde f^{-qn}(\tilde y) & \in U_{h_1(\tilde f^{-qn}(\tilde y))}^{-q}U_{h_1(\tilde f^{q(-n+1)}(\tilde y))}^{-q}\dots U_{h_1(f^{-q}(\tilde y))}^{-q}(\tilde Y).
\end{align*}
\end{itemize}
\end{proposition}

Similar properties hold for $f$ instead of $f^q$: we get some classical consequences of the semi-conjugacy to a shift without the semi-conjugacy property itself.

\begin{proposition}\label{PropConjug2}
Suppose that $f$ has a rotational horseshoe with deck transformations $U_1,\dots,U_k$. Then 
\begin{itemize}
\item For any word $T\in \langle U_1,\dots,U_k\rangle_+$ of length $q$, there exists $\tilde x\in \tilde S$ such that $\tilde f^q(\tilde x) = T\tilde x$ (in other words, $x$ is a $q$-periodic point associated to the deck transformation $T$);
\item There exists $D>0$ such that, for any word $(w_i)_i\in \{1,\dots,k\}^\Z$, there exists $\tilde x \in \tilde S$ such that, for any $i\ge 0$, 
\[d\big(\tilde f^i(\tilde x),U_{w_i}\dots U_{w_0}\tilde x\big) \le D, \qquad
d\big(\tilde f^{-i}(\tilde x),U_{w_{-i}}^{-1}\dots U_{w_{-1}}^{-1} \tilde x\big) \le D.\]
\item if $U_1,\dots,U_k$ form a free group, then the topological entropy of $f$ is bigger than $\log k$.
\end{itemize}
\end{proposition}

\begin{proof}[Proof of Proposition~\ref{PropConjug}]
We use notation from Definition~\ref{DefRotHorse}.

As $R$ is compact, as the group generated by $U_1,\dots,U_k$ is free, and as $\pi_1(S)$ acts properly discontinuously on $\tilde S$, there exists $q\in\N^*$ such that for any nontrivial word $T$ in $U_1^q,\dots,U_k^q$, one has $T R\cap R = \emptyset$.

For $(w_i)\in\{1,\dots,k\}^{\Z}$, let us define 
\[R_{(w_i)}^n = R \cap \bigcap_{0\le i < n}\left( \tilde{f}^{-qi}(U_{w_{i-1}}^{q}\dots U_{w_0}^{q} R) \, \cap \, \tilde{f}^{qi}(U_{w_{-i}}^{-q}\dots U_{w_{-1}}^{-q} R)\right)\]
and 
\[\tilde Y = \bigcap_{n \geq 0}\  \bigcup_{(w_i)\in\{1,\dots,k\}^\Z} R_{(w_i)}^n.\]
We will denote by $Y$ the projection of $\tilde Y$ on $S$. Note that $\tilde Y$ is a decreasing intersection of compact subsets of $R$, so it is compact.
 
Remark that, if $(w_i)\in\{1,\dots,k\}^{\Z}$ and $  \tilde{x}\in \bigcap_{n\ge 0} R_{(w_i)}^n$, then (because of the definition of $q$) for any $i\in\Z$ there is a unique deck transformation $T_i\in\langle U_1^q,\dots,U_k^q\rangle$ such that $\tilde{f}^{qi}(\tilde{x})\in T_i R$. By the very definition of $R_{(w_i)}^n$, one has $\tilde{f}^{q(i+1)}(x) \in U_{w_i}^qT_i R$. Then $T_{i+1}=U_{w_i}^qT_i$ and there exists a unique sequence $(w_i)\in\{1,\dots,k\}^{\Z}$ such that $ \tilde x\in \bigcap_{n\ge 0} R_{(w_i)}^n$.

Moreover, the previous remark also implies the following equality:
\[\wt Y \overset{\text{def.}}{=} \bigcap_{n\ge 0}\  \bigcup_{(w_i)\in\{1,\dots,k\}^\Z} R_{(w_i)}^n =  \bigcup_{(w_i)\in\{1,\dots,k\}^\Z}\ \bigcap_{n\ge 0} R_{(w_i)}^n.\]
Indeed, if a point $\tilde{x}$ belongs to the left-hand side set, then, for any $i \in \Z$, there exists a unique $T_i$ such that $\tilde{f}^{qi}(\tilde{x}) \in T_i R$ and a unique $w_i$ such that $T_{i+1}=U_{w_i}^q T_i$. Hence the point $\tilde{x}$ belongs to $\bigcap_{n\ge 0} R_{(w_i)}^n$. The other inclusion is trivial.

Then, for any $\tilde x\in\wt Y$, there exists a unique sequence $(w_i)\in\{1,\dots,k\}^{\Z}$ such that $ \tilde x\in \bigcap_{n\ge 0} R_{(w_i)}^n$. This allows to talk about the \emph{trajectory} of a point $\tilde x\in \tilde Y$, which we define as the unique sequence $(w_i)\in\{1,\dots,k\}^{\Z}$ such that $\tilde x\in \bigcap_{n\ge 0} R_{(w_i)}^n$. We define
\begin{align*}
h_1 : \tilde Y & \longrightarrow \{1,\dots,k\}^{\Z}\\
x & \longmapsto (w_i).
\end{align*}
as the map which, to any point of $\tilde Y$, associates its trajectory.

Repeated applications of Proposition \ref{LemChain} imply that for any $(w_i)\in\{1,\dots,k\}^{\Z}$, the sets $R_{(w_i)}^n$ are nonempty (and compact), hence $\bigcap_{n\ge 0} R_{(w_i)}^n$ is also a nonempty compact set. This shows that the map $h_1$ is surjective.

We define the map $\tilde g$ by $\tilde g|_{\tilde f^{-q}(U_{w_i}^q R)\cap R} = U_{w_i}^{-q}\circ \tilde f^q$. As the sets $\tilde f^{-q}(U_{w_i}^q R)\cap R$ are all at positive distance one to the others, the map $\tilde g$ defines a homeomorphism of $\tilde Y$.
By the very construction of $h_1$, the diagram of Proposition \ref{PropConjug} commutes.

Fix $(w_{i})_{i \in \Z} \in \{1,\dots,k\}^{\Z}$. Let $n \geq 2$. Observe that $R^{n}_{(w_i)}$  is a neighbourhood of $R^{n+1}_{(w_i)}$ so that $R^{n}_{(w_i)} \cap \tilde{Y}$ is an open subset of $\tilde{Y}$ and projects to an open subset of $Y$. Observe that the projection on the coordinates between $-n+2$ and $n-1$ of the map $h_1$ is constant on this open subset so that the map $h_1$ is continuous.

Finally, Proposition~\ref{LemPointFixe} implies that in the preimage by $h_1$ of any periodic word, there is a periodic point of $f$ of the same period. This finishes the proof.
\end{proof}

\begin{proof}[Proof of Proposition~\ref{PropConjug2}]
The first point of the proposition is a simple application of Propositions~\ref{LemChain} and \ref{LemPointFixe}.
\medskip

For the second point, we can use again the strategy of the proof of Proposition~\ref{PropConjug2}, by considering the compact set
\[R_{(w_i)}^n = R \cap \bigcap_{0\le i < n}\left( \tilde{f}^{-i}(U_{w_{i-1}}\dots U_{w_0} R) \, \cap \, \tilde{f}^{i}(U_{w_{-i}}^{-1}\dots U_{w_{-1}}^{-1} R)\right).\]
One gets easily as a consequence of Proposition~\ref{LemChain} that the set $\bigcap_{n \geq 0} R_{(w_i)}^n$ is nonempty, and any element of it can be used to get the desired conclusion (taking $D$ as the diameter of $R$ for instance).
\medskip

Concerning the entropy, let us change a bit the definition of $R_{(w_i)}^n$ to consider only positive times:
\[\overline R_{(w_i)}^n = R \cap \bigcap_{0\le i < n} \tilde{f}^{-i}(U_{w_{i-1}}\dots U_{w_0} R).\]
Note that, as $R$ is compact, as the group generated by $U_1,\dots,U_k$ is free, and as $\pi_1(S)$ acts properly discontinuously on $\tilde S$, there exists $N_0\in\N$ such that for any nontrivial word $T$ in $U_1,\dots,U_k$ of length bigger than $N_0$, one has $T R\cap R = \emptyset$.

This implies that, for any words $(w_i)_{0\le i\le n},(w'_i)_{0\le i\le n} \in \{1,\dots,k\}^{n+1}$, if there exists $N_0\le i_0\le n$ such that $w_{i_0}\neq w'_{i_0}$, then $U_{w_n}\dots U_{w_0}R \, \cap \, U_{w'_n}\dots U_{w'_0}R = \emptyset$, and so $\overline R_{(w_i)}^n \cap \overline R_{(w'_i)}^n = \emptyset$. 

By \cite{MR295352}, there exists a finite cover $\hat S$ of $S$ such that $R$ projects injectively on\footnote{A group is \emph{residually finite} if, for any finite subset $F\subset G$, there exists a finite quotient $G/H$ in which $F$ projects injectively.  By \cite{MR295352}, any surface group is residually finite, and we can apply this property to the finite set $F$ of deck transformations $T$ such that $TR\cap R\neq\emptyset$. The finite index resulting subgroup $H$ corresponds to a finite cover of the surface in which $R$ projects injectively.} $\hat S$. We denote by $\hat f$ the map induced by $\tilde f$ on $\hat S$. Fix $n >N_0$. Consider $\varepsilon$ the minimum distance between the projections of the compact sets $\overline R_{(w_i)}^n$ on $\hat S$, with $(w_i)\in \{1\}^{N_0}\times \{1,\dots,k\}^{n-N_0}$. By the previous paragraph, these sets are pairwise disjoint in $\tilde S$, and as they are subsets of $R$, they project injectively on $\hat S$: we have $\varep>0$.

Now, take $\ell\ge 0$ and consider the family of words $(w_i)\in (\{1\}^{N_0}\times \{1,\dots,k\}^{n-N_0})^\ell$. Taking one point in each of these sets, we get a subset of $\hat S$ of cardinality $k^{\ell(n-N_0)}$ which is $(\ell n,\varepsilon)$ separated. This implies that (taking $\ell\to+\infty$)
\[h_{top}(\hat f) \ge \frac{n-N_0}{n} \log	 k,\]
and so, taking $n\to+\infty$, that $h_{top}(\hat f) \ge \log k$.

We get the conclusion of the proposition by using Proposition~\ref{entropycovering} which tells us that $h_{top}(f)  = h_{top}(\hat f)$.
\end{proof}

\subsection{Geometric vs. $\F$-transverse intersections}\label{SubSecGeomtrans}

In this subsection, we prove that an $\F$-transverse loop on a surface that has a geometric auto-intersection (in the geometric meaning of Definition~\ref{DefTransverseInter}) must have an $\F$-transverse auto-intersection, associated to a deck transformation that projects to the deck transformation of the geometric auto-intersection.

In the sequel, we will denote paths with marked points to denote their lifts to the universal cover starting on the common marked point. For instance, $\alpha\cdot$ and $\beta\cdot$ denote some lifts of respectively $\alpha$ and $\beta$ whose right ends coincide.

We fix a surface $S$ (not necessarily closed) of negative Euler characteristic, a singular foliation $\F$ of $S$, and an $\F$-transverse loop $\alpha : \R\to \dom \F$ (which means that for every $t$, one has $\alpha(t) = \alpha(t+1)$).
We suppose that $\alpha$ auto-intersects geometrically (see Definition~\ref{DefTransverseInter}) at $\alpha(t_1) = \alpha(t_2)$, with $t_1<t_2<t_1+1$. We let $\alpha_1=\alpha|_{[t_1,t_2]}$ and $\alpha_2= \alpha|_{[t_2,t_1+1]}$. Let $\wt\alpha$ be a lift of $\alpha$ to $\wt S$ and $\wt \alpha_1$ be the lift of $\alpha_1$ which starts from the same point as $\wt \alpha$.  Also, let $T$ and $T_1$ be the deck transformations of the universal cover $\wt S\to S$ respectively associated to $\alpha|_{[t_1,t_2]}$ and $\alpha$ and which respectively preserve $\wt \alpha$ and $\wt \alpha_{1}$. Let $T_2$ be the deck transformation such that $T=T_2T_1$. Finally, let $\wt{\F}$ be the lift of $\F$ to $\wt S$.

\begin{proposition}\label{GeomImpliqFeuill}
If $\alpha$ auto-intersects geometrically, then there exists $u_1$ and $u_2$ such that the paths $T_1 \wt \alpha|_{[u_2,u_2+1]}$ and $\wt \alpha|_{[u_1,u_1+1]}$ intersect $\wt \F$-transversally at $T_1\wt\alpha(t_1) = \wt\alpha(t_2)$.

In particular, the paths $\alpha_1\alpha_2\cdot\alpha_1\alpha_2$ and $\alpha_2\alpha_1\cdot\alpha_2\alpha_1$ intersect $\F$-transversally at the marked point. More generally, for $i,j,k,\ell\ge 1$, the paths $\alpha_1\alpha_2^i\cdot\alpha_1^k\alpha_2$ and $\alpha_2\alpha_1^j\cdot\alpha_2^\ell\alpha_1$ intersect $\F$-transversally at the marked point. 
\end{proposition}

\begin{remark}
The proof of this proposition shows that if none of the deck transformations $T_1$ and $T_2$ is a prefix/suffix of the other, then the conclusion is stronger: the paths $\alpha_2\cdot\alpha_1$ and $\alpha_1\cdot\alpha_2$ intersect $\F$-transversally at the marked point.
\end{remark}

\begin{remark}
In the end of the proof one has to consider the case where $T_1$ is a suffix of $T_2$, and $T_2$ is a suffix of $T_2T_1$. This case can happen, as can be seen by considering the words $w_1 = 1221$ and $w_2 = 21\,1221$: $w_1$ is a suffix of $w_2$, and $w_2$ is a suffix of $w_2w_1$. This suggests that in general, the conclusion of the proposition cannot be improved.
\end{remark}

\begin{proof}
We denote $\alpha_1 = \alpha|_{[t_1,t_2]}$ and $\alpha_2 = \alpha|_{[t_2,t_1+1]}$. Let $\check{\dom}(\F)$ be the covering of $\dom(\F)$ associated to $(\alpha_1,\alpha_2,x_0)$. By Proposition~\ref{coveringmap2}, the surface $\check{\dom}(\F)$ is homeomorphic to the three punctured sphere $\Sp^2\setminus\{A,B,C\}$. The lifts of $\alpha$, $\alpha_1$ and $\alpha_2$ to $\check{\dom}(\F)$ are respectively denoted by $\check{\alpha}$, $\check{\alpha_1}$ and $\check{\alpha_2}$. We denote $a$ resp. $b$ two simple loops generating $\pi_1(\Sp^2\setminus\{A,B,C\})$, winding once around $A$ and not around $B$ or $C$ (resp. once around $B$ and not around $A$ or $C$).

During the proof we will use the following fact: if $\F$ is a singular foliation on $\Sp^2$, and $\alpha$ is a $\F$-transverse Jordan curve in $\Sp^2$, then each connected component of $\alpha^\complement$ has to contain at least one singularity of the foliation $\F$. 

Let $\beta_1 = \check{\alpha}|_{[s_1,s'_1]}$ be a subpath of $\check{\alpha}$ which is a simple loop. As $\check{\alpha}$ is transverse, $\beta_1$ is essential: indeed, this Jordan curve separates $\Sp^2$ in two connected components, and each of them has to contain a singularity of the lift $\check{\F}$ of $\F$ to $\check{\dom}(\F)\subset \Sp^2$, because $\beta_1$ is transverse; it then suffices to remember that the only singularities of $\check\F$ in $\Sp^2$ are the punctures. Hence, as the only essential simple loops in the three punctured sphere are the ones winding once around one puncture and not around the others, we can suppose (up to a permutation of $A$, $B$ and $C$) that the loop $\beta_1$ is homotopic to $a$.

Consider the loop $\overline{\beta_1} \doteq \check{\alpha}|_{[s'_1,s_1+1]}$. It is not contractible : otherwise, the loop $\check{\alpha}$ would be homotopic to $\beta_1$, which is not possible as $\check{\alpha}$ has a geometric self intersection. Again, let $\beta_2 = \overline{\beta_1}|_{[s_2,s'_2]}$ be a subpath of $\overline{\beta_1}$ which is an essential simple loop. If $\beta_2$ is homotopic to $a$ or $a^{-1}$, we can iterate the process by considering the loop $\overline{\beta_1}|_{[s'_2,s_2+1]}$ or the loop $\overline{\beta_1}|_{[s'_1,s_2]}$: one of them is homotopically non trivial as $\check{\alpha}$ cannot be homotopic to a power of $\beta_1$, by definition of $\check{\dom}(\F)$)\dots{} Eventually, we find an essential simple loop $\beta_n$, which is a concatenation of pieces of the path $\check{\alpha}$, which is neither homotopic to $a$ nor to $a^{-1}$. As before, we can suppose that this loop $\beta_n$ is homotopic to $b$ (changing $b$ to $b^{-1}$ if necessary).

From now on we will denote $\beta_A = \beta_1$ and $\beta_B = \beta_n$. Let $\Phi_A$ be the union of the leaves met by $\beta_A$, and $\Phi_B$ the union of the leaves met by $\beta_B$. These are open annuli, $\Phi_A$ separating $A$ from $B$ and $C$, and $\Phi_B$ separating $B$ from $A$ and $C$. Remark that the complement of $\Phi_A$ (resp. $\Phi_B$) in $\Sp^2$ is made of two connected components, that are closed.

\begin{claim}\label{ClaimDisjoint}
The loops $\beta_A$ and $\beta_B$ are $\check{\F}$-equivalent to disjoint loops.
\end{claim}

\begin{proof}
Replacing $\beta_A$ and $\beta_B$ by $\check{\F}$-equivalent loops if necessary, we can suppose that the number of intersections between them is finite. 

Suppose that $\beta_A$ and $\beta_B$ are disjoint, otherwise there is nothing to do. the only nontrivial case is when $\beta_A$ meets the connected component $O$ of $\beta_B^\complement$ that contains $B$.

Let $t_1<t_2$ be such that $\beta_A|_{[t_1,t_2]}$ meets $\beta_B$ at its endpoints, and that $\beta_A|_{]t_1,t_2[}$ is included in $O$. Then $\beta_A|_{]t_1,t_2[}$ separates $O$ into two connected components, one of them containing $B$ and the other one, denoted by $O'$, not containing it. Suppose that $O'$ is locally on the right of $\beta_A|_{]t_1,t_2[}$ (the other case being identical). Each leaf meeting $O'$ has to get out of it as $O'$ does not contain any singularity of $\check\F$. In particular, each leaf entering in $O'$ through $\beta_A|_{]t_1,t_2[}$ has to get out of $O'$ by $\beta_B$. This implies that $\beta_A|_{]t_1,t_2[}$ stays in $\Phi_B$, and hence that $\beta_A$ does not meet the connected component of $\Phi_B^\complement$ containing $B$.

By local compactness, the distance between $\beta_A$ and the connected component of $\Phi_B^\complement$ containing $B$ is positive. Remark that the leaves of resp. $\Phi_A$ and $\Phi_B$ are naturally indexed by the transverse loops $\beta_A$ and $\beta_B$, and in particular are endowed with a natural cyclic order. By considering a continuous parametrization of the leaves of $\Phi_B$ by $\Sp^1\times \R$ ($\Sp^1$ corresponding to the point of $\beta_B$ met by the leaf and $\R$ to the parametrization of the leaf itself), by flowing $\beta_B$ along the leaves of $\Phi_B$ in the direction of $B$, one can easily find a loop $\F$-equivalent to $\beta_B$ and which is disjoint from $\beta_A$.
\end{proof}

\begin{claim}\label{ClaimStay}
The loop $\check{\alpha}$ stays in $\Phi_A \cup\Phi_B$. In particular, $\Phi_A \cap\Phi_B\neq\emptyset$.
\end{claim}

\begin{proof}
The second part of the claim follows from the first one as the union of the leaves met by $\check{\alpha}$ is connected.

Suppose for a contradiction that $\check{\alpha}$ does not stay in $\Phi_A \cup\Phi_B$. Let $x$ be a point of $\check{\alpha}$ outside of $\Phi_A \cup\Phi_B$, and $\phi\doteq \phi_x$. Then, by Poincaré-Bendixson theorem, the ends of $\phi$ are either topological circles, or contain singularities of $\check{\F}$. In the circle case, either it contains a singularity, or both connected components of its complement contain a singularity. 

Remark that $\phi$ and its ends are disjoint from $\Phi_A\cup \Phi_B$, and that $\Phi_A$ and $\Phi_B$ separate all singularities of $\check{\F}$.

In the case where $\phi$ is a circle, then it separates $\Sp^2$ into two disjoint connected components and so it prevents transverse trajectories passing through it (e.g. $\check{\alpha}$) to be recurrent, which is a contradiction.

The same argument can be applied when both ends of $\phi$ contain a singularity: in this case, as $\phi$ is contained in a single connected component of $(\Phi_A\cup\Phi_B)^\complement$, this singularity $D\in\{A,B,C\}$ is the same for both ends of $\phi$. Then, the ends of $\phi$ are made of the union of $D$ with possibly leaves of $\check{ \F}$ that are homoclinic to $D$. In the case such leaves exist, they all separate $\Sp^2$ in two connected components, one of which containing the whole loop $\check{\alpha}$. So it does not change anything dynamically to quotient by these connected components, and this crushing allows to reduce to the case where both ends are reduced to $\{D\}$. Replacing the circle by $\phi\cup \{D\}$ in the previous paragraph leads to a contradiction.

Suppose now that $\phi$ is not a circle and that at least one end of it is a closed leaf. Then this leaf $L$ separates $\Sp^2$ into two disjoint connected components. As $\check{ \alpha}$ is a transverse loop, it cannot meet $L$.
As the loop $\check{\alpha}$ meets $\Phi_A$, $\Phi_B$ and $\phi$, one of the connected components of $L^\complement$ contains $\Phi_A$, $\Phi_B$ and $\phi$, and the other one contains a singularity. Observe that the other end of $\phi$ cannot contain a singularity: it is a closed leaf $L'$. Moreover, $\phi$ belongs to the connected component of the complement of $L'$ that does not contain $\Phi_A$ and $\Phi_B$, a contradiction as $\check{ \alpha}$ cannot meet $L'$. 
\end{proof}

Recall that the leaves of resp. $\Phi_A$ and $\Phi_B$ are naturally indexed by the transverse loops $\beta_A$ and $\beta_B$, and in particular are endowed with a natural cyclic order.

\begin{claim}\label{ClaimInterval}
The set of leaves of $\Phi_A\cap\Phi_B$ is an interval of leaves of $\Phi_A$ (resp. $\Phi_B$).
\end{claim}

\begin{proof}
Les us reason in the plane $\Sp^2\setminus\{C\}$.

By Claim~\ref{ClaimDisjoint}, there exist transverse loops $\beta'_A$ and $\beta'_B$, which are respectively equivalent to $\beta_A$ and $\beta_B$ and which are disjoint. Note that in this plane, the two Jordan curves $\beta'_A$ and $\beta'_B$ bound bounded domains that are disjoint: if one bounded domain was included in the other one, it would have to contain both $A$ and $B$ which is impossible. Also, by considering a leaf of $\Phi_A\cap\Phi_B$, which meets both $\beta'_A$ and $\beta'_B$, we see that $\beta'_A$ and $\beta'_B$ turn in opposite directions (relative to a fixed orientation of $\Sp^2\setminus \{C\}$).

Now, let $\phi_1$ and $\phi_2$ be two distinct leaves of $\Phi_A\cap\Phi_B$; we want to show that a one of the intervals of leaves of $\Phi_A$, $[\phi_1,\phi_2]_{\Phi_A}$ or $[\phi_2,\phi_1]_{\Phi_A}$, is contained in $\Phi_A\cap\Phi_B$ (the proof is identical for $\Phi_B$). Denote by $\phi'_1$ and $\phi'_2$ the leaf segments of resp. $\phi_1$ and $\phi_2$ that are bounded by $\beta'_A$ and $\beta'_B$. In this case, the open set
\[\Big(\beta'_A \cup \beta'_B \cup \phi'_1 \cup \phi'_2\Big)^\complement\]
is made of four connected components (the considered paths bound four different Jordan curves), one of which, denoted by $D$, is containing no singularity. Part of its boundary is made of a segment of $\beta'_A$ between $\phi_1$ and $\phi_2$. All the leaves crossing this segment have to get out of $D$ (because $D$ is a disk containing no singularity), and it can make it only by crossing $\beta'_B$. Hence, a whole interval of leaves between $\phi_1$ and $\phi_2$ is included in $\Phi_A\cap\Phi_B$.
\end{proof}

\begin{figure}[ht]
\begin{minipage}{.65\linewidth}
\includegraphics[width=\linewidth]{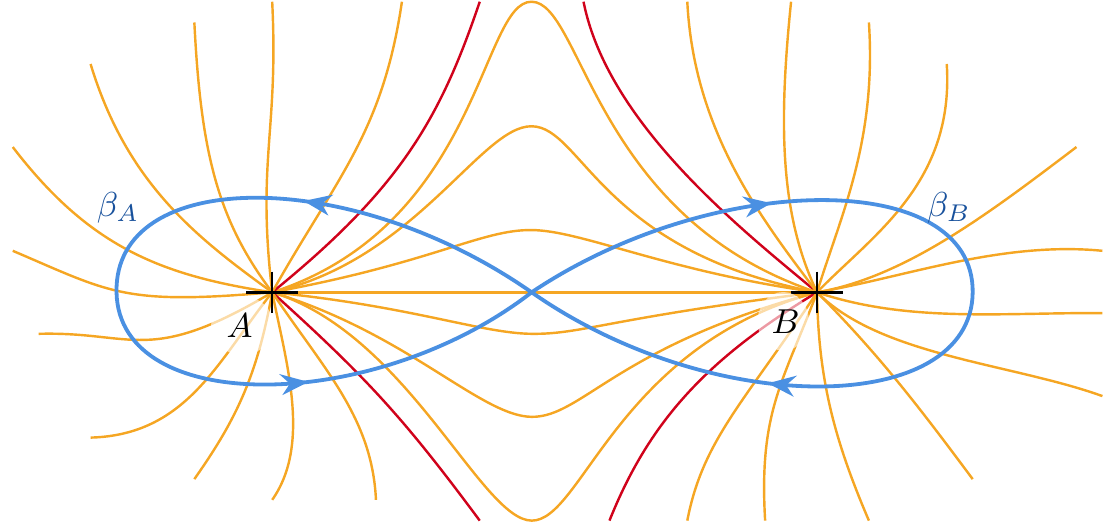}
\caption{One possible shape for the set of leaves that meet $\beta_A$ and $\beta_B$. The four red leaves are the boundaries of the sets $\Phi_A\cap\Phi_B$, $\Phi_A\setminus\Phi_B$ and $\Phi_B\setminus\Phi_A$.}\label{ShapeLeaves3Sphere}
\end{minipage}
\hfill
\begin{minipage}{.33\linewidth}
\includegraphics[width=\linewidth]{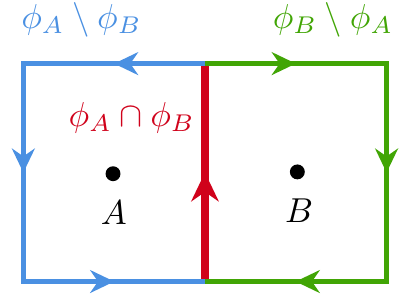}
\caption{The set of leaves of $\Phi_A\cap\Phi_B$ has topologically this shape (be careful, this space is not Hausdorff: the boundaries of $\Phi_A\setminus\Phi_B$ and $\Phi_B\setminus\Phi_A$ do not coincide).}\label{ShapeLeaves3Sphere2}
\end{minipage}
\end{figure}

From now on, replacing the transverse loops $\beta_A$ and $\beta_B$ by $\check{\F}$-equivalent ones if necessary, we suppose that they are indexed by $\R/\Z$, and that they meet at a single point $\beta_A(0)=\beta_B(0)$. By changing $\alpha$ to an $\F$-equivalent loop if necessary, we can suppose that $\check\alpha(t_1) = \beta_A(0) = \beta_B(0) = \check\alpha(t_2)$.

Let $\phi_a\in\Phi_A\setminus\Phi_B$ and $\phi_b\in\Phi_B\setminus\Phi_A$ (these sets are nonempty, otherwise it would contradict the fact that $\beta_A$ and $\beta_B$ are simple and bound different singularities of $\check{ \F}$). Changing the speeds of $\beta_A$ and $\beta_B$ if necessary, we suppose that $\beta_A(1/2)\in \phi_a$ and $\beta_B(1/2)\in \phi_b$. By Claims \ref{ClaimDisjoint} and \ref{ClaimInterval}, the set $\Phi_A\cap\Phi_B$ is a nonempty open topological disk.

\begin{claim} \label{ClaimEquivalencetoword}
For any transverse loop $\check \gamma$ contained in $\Phi_A\cup\Phi_B$, there exists a unique word $a_1 \ldots a_n$ on the letters $A$ and $B$ such that $\check \gamma$ is $\check{\F}$-equivalent to the loop $\beta_{a_1}\dots\beta_{a_n}$.
In particular, there exists a unique word $w=w_1,\dots,w_k \in\{A,B\}^k$ such that $\check{\alpha}|_{[t_1,t_1+1]}$ is $\check{\F}$-equivalent to the loop $\beta_{w_1}\dots\beta_{w_k}$. 
\end{claim}

\begin{proof}
Such a loop $\check\gamma$ cannot be contained in $\Phi_A\cap\Phi_B$, as it is recurrent. Similarly, it cannot be contained in $\Phi_A\setminus\Phi_B$, nor in $\Phi_B\setminus\Phi_A$. Hence, the projection of this loop $\check\gamma$ on the set of leaves of $\Phi_A\cap\Phi_B$ has to follow the oriented paths of Figure~\ref{ShapeLeaves3Sphere2}.

Hence, the homotopy class of the transverse loop $\check{\gamma}$ is determined by the sequence of leaves $\phi_a,\phi_b$ met by $\check{\gamma}$: for instance, if $\check{ \gamma}$ meets successively $\phi_a,\phi_a,\phi_b$ and $\phi_a$, then the homotopy type of $\check{\gamma}$ is the one of $\beta_A^2\beta_B\beta_A$. So the homotopy type of $\check{\alpha}$ is a word in $\beta_A$ and $\beta_B$ (it does not contain neither $\beta_A^{-1}$ nor $\beta_B^{-1}$). This implies the claim.
\end{proof}

\begin{claim}
The transverse paths
\[\beta_A|_{[1/2,1]}\beta_B|_{[0,1/2]} \quad \text{and} \quad \beta_B|_{[1/2,1]}\beta_A|_{[0,1/2]}\]
have an $\F$-transverse intersection at $\beta_A(0) = \beta_B(0)$.
\end{claim}

\begin{proof}
Let $\wh{\beta_A}$ and $\wh{\beta_B}$ be two lifts of resp. $\beta_A|[0,1]$ and $\beta_B|_{[0,1]}$ to $\wt\dom(\F)$ that meet at $\wh{\beta_A}(1) = \wh{\beta_B}(0)$. We denote $T_A$ (resp. $T_B$) the deck transformation of $\wt\dom(\F)\to\check{\dom}(\F)$ corresponding to the essential loop $\beta_A$ (resp. $\beta_B$) which preserves $\wh{\beta_A}$ (resp. $\wh{\beta_B}$). Then (see Figure~\ref{FigGeomImpliqFeuill}) we have $\wh{\beta_A}(1) = \wh{\beta_B}(0) = T_A\wh{\beta_A}(0) = T_B^{-1}\wh{\beta_B}(1)$.
As $\phi_a\in\Phi_A\setminus\Phi_B$ and $\phi_b\in\Phi_B\setminus\Phi_A$, we deduce that 
\[R(\phi_{\wh{\beta_B}(1/2)}) \cap R(\phi_{T_A\wh{\beta_A}(1/2)}) = L(\phi_{T_B^{-1}\wh{\beta_B}(1/2)}) \cap L(\phi_{\wh{\beta_A}(1/2)}) =  \emptyset.\]
Moreover, the fact that all pairs of curves that are homotopic to $\beta_A|_{[1/2,1]}\beta_B|_{[0,1/2]}$, resp. $\beta_B|_{[1/2,1]}\beta_A|_{[0,1/2]}$ meet (this comes from the fact that the curves $\beta_A$ and $\beta_B$ turn in different directions in $\Sp^2\setminus \{C\}$) implies that 
\begin{itemize}
\item Either $\phi_{\wh{\beta_B}(1/2)}$ is above $\phi_{T_B^{-1}\wh{\beta_B}(1/2)}$ relative to $\phi_{\wh{\beta_B}(0)}$ and $\phi_{T_A\wh{\beta_B}(1/2)}$ is above $\phi_{\wh{\beta_B}(1/2)}$ relative to $\phi_{\wh{\beta_B}(0)}$.
\item Or $\phi_{\wh{\beta_B}(1/2)}$ is below $\phi_{T_B^{-1}\wh{\beta_B}(1/2)}$ relative to $\phi_{\wh{\beta_B}(0)}$ and $\phi_{T_A\wh{\beta_B}(1/2)}$ is below $\phi_{\wh{\beta_B}(1/2)}$ relative to $\phi_{\wh{\beta_B}(0)}$.
\end{itemize}
In both cases we have an $\F$-transverse intersection between $\wh{\beta_A}|_{[1/2,1]}\wh{\beta_B}|_{[0,1/2]}$ and $T_B^{-1}\wh{\beta_B}|_{[1/2,1]}T_A\wh{\beta_A}|_{[0,1/2]}$.
\end{proof}

\begin{figure}[ht]
\begin{center}
\includegraphics[scale=1]{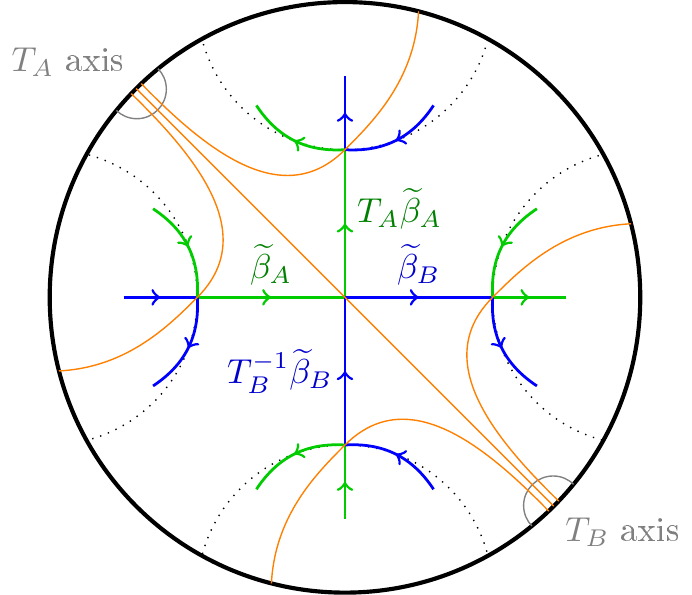}
\caption{The configuration of Proposition~\ref{GeomImpliqFeuill} in $\wt\dom\F\simeq \Hy^2$: the trajectories that are lifts of $\beta_A$ are in green, and the ones that are lifts of $\beta_B$ in blue; the leaves of the foliation are in orange.}\label{FigGeomImpliqFeuill}
\end{center}
\end{figure}

We are ready to prove that the paths $T_1 \wt \alpha|_{[u_2,u_2+1]}$ and $\wt \alpha|_{[u_1,u_1+1]}$ intersect $\wt \F$-transversally at $T_1\wt\alpha(t_1) = \wt\alpha(t_2)$.

By Claim \ref{ClaimEquivalencetoword}, the transverse loop $\check{\alpha}|_{[t_1,t_2]}$ is $\check{\F}$-equivalent to a subword $\beta_{w_1}\dots\beta_{w_\ell}$ of $w=\beta_{w_1}\dots\beta_{w_k}$.

Let us periodize the word $w$, and consider the word $w'$ such that $w'_i = w_{i-\ell}$ for any $i\in\Z$. Observe that the loop $\check{\alpha|}_{[t_1,t_1+1]}$ is $\check{\F}$-equivalent to $\beta_{w_1}\beta_{w_2} \ldots \beta_{w_k}$ and that the loop $\check{\alpha}|_{[t_2,t_2+1]}$ is $\check{\F}$-equivalent to $\beta_{w'_1}\beta_{w'_2} \ldots \beta_{w'_k}$.  As $\check{\alpha}$ and $T_1\check{\alpha}$ have a geometric transverse intersection, we cannot have $w=w'$. As both words $w$ and $w'$ are periodic of period $k$, this implies that there exists $i_0\le 0$, $j_0>0$, with $j_0-i_0\le k +1$, such that 
\[w|_{\{i_0+1,\dots,j_0-1\}} = w'|_{\{i_0+1,\dots,j_0-1\}}\,, \quad w_{i_0}\neq w'_{i_0} \quad \text{and} \quad w_{j_0}\neq w'_{j_0}.\]

As the curves $\beta_A$ and $\beta_B$ intersect at a single point, and as the homotopy group they generate is free, the union of lifts of these two loops to $\wt\dom(\F)$ is a complete binary tree (as in Figure~\ref{FigGeomImpliqFeuill}).

Because of this, and because the intersection between $\hat{\alpha}$ and $T_1\hat{\alpha}$ is geometrically transverse, either $w_{i_0}=w'_{j_0}=A$ and $w_{j_0}=w'_{i_0}=B$, or $w_{i_0}=w'_{j_0}=B$ and $w_{j_0}=w'_{i_0}=A$. In particular, denoting $u_{-2}$, $u_{-1}$, $u_1$ and $u_2$ the times in $\wh\alpha$ corresponding to resp. $\beta_{w_{i_0}}(1/2)$, $\beta_{w_{i_0}}(1)$, $\beta_{w_{j_0}}(0)$ and $\beta_{w_{j_0}}(1/2)$, and $u'_{-2}$, $u'_{-1}$, $u'_1$ and $u'_2$ the times in $T_1\wh\alpha$ corresponding to resp. $\beta_{w'_{i_0}}(1/2)$, $\beta_{w'_{i_0}}(1)$, $\beta_{w'_{j_0}}(0)$ and $\beta_{w'_{j_0}}(1/2)$
\begin{itemize}
\item Either $\phi_{\wh{\alpha}(u_{-2})}$ is above $\phi_{T_1\wh{\alpha}(u'_{-2})}$ relative to $\phi_{\wh{\alpha}(u_{-1})} = \phi_{T_1\wh{\alpha}(u'_{-1})}$, and $\phi_{\wh{\alpha}(u_{2})}$ is below $\phi_{T_1\wh{\alpha}(u'_{2})}$ relative to $\phi_{\wh{\alpha}(u_{1})} = \phi_{T_1\wh{\alpha}(u'_{1})}$.
\item Or $\phi_{\wh{\alpha}(u_{-2})}$ is below $\phi_{T_1\wh{\alpha}(u'_{-2})}$ relative to $\phi_{\wh{\alpha}(u_{-1})} = \phi_{T_1\wh{\alpha}(u'_{-1})}$, and $\phi_{\wh{\alpha}(u_{2})}$ is above $\phi_{T_1\wh{\alpha}(u'_{2})}$ relative to $\phi_{\wh{\alpha}(u_{1})} = \phi_{T_1\wh{\alpha}(u'_{1})}$.
\end{itemize}
In both cases the two transverse paths $\wh{\alpha}|_{[u_{-2},u_2]}$ and $T_1\wh{\alpha}|_{[u'_{-2},u'_2]}$ intersect $\F$-transversally. Because $j_0-i_0\le k{+1}$, we have that $u_2-u_{-2} \le 1$ and $u'_2-u'_{-2} \le 1$. In particular, this implies that the two paths $\alpha_1\alpha_2\cdot\alpha_1\alpha_2$ and $\alpha_2\alpha_1\cdot\alpha_2\alpha_1$ intersect $\F$-transversally at the marked point.
\bigskip

We now prove that for $i,j,k,\ell\ge 1$, the paths $\alpha_1\alpha_2^i\cdot\alpha_1^k\alpha_2$ and $\alpha_2\alpha_1^j\cdot\alpha_2^\ell\alpha_1$ intersect $\F$-transversally at the marked point. To fix notations, we suppose that the leaf passing through the left end of $\alpha_1\alpha_2\cdot$ is over the leaf passing through the left end of $\alpha_2\alpha_1\cdot$ relative to the leaf passing through the right end of both paths $\alpha_1\alpha_2\cdot$ and $\alpha_2\alpha_1\cdot$. We want to prove that the leaf passing through the left end of $\alpha_1\alpha_2^i\cdot$ is over the leaf passing through the left end of $\alpha_2\alpha_1^j\cdot$ relative to the leaf passing through the right end of both paths $\alpha_1\alpha_2^i\cdot$ and $\alpha_2\alpha_1^j\cdot$. This will prove the proposition, as the reasoning for the right parts of the paths $\cdot \alpha_1\alpha_2$ and $\cdot\alpha_2\alpha_1$ is identical.

Suppose first that the leaves passing through the left ends of respectively $\alpha_1\cdot$ and $\alpha_2\cdot$ are not comparable (meaning that none of them is in the left of the other). Then, the leaf passing through the left end of $\alpha_2\cdot$ has to be above  the leaf passing through the left end of $\alpha_1\cdot$ relative to the leaf passing through their common right end, as by hypothesis the leaf passing through the left end of $\alpha_1\alpha_2\cdot$ is over the leaf passing through the left end of $\alpha_2\alpha_1\cdot$ relative to the leaf passing through their common right end. This suffices to get the desired property.

Suppose now that the leaves passing through the left ends of respectively $\alpha_1\cdot$ and $\alpha_2\cdot$ are comparable (meaning that one of them is contained in the left of the other). This means that one of the two paths $\alpha_1\cdot$ and $\alpha_2\cdot$ is homotopic (relative to endpoints) to a subpath of the other; more precisely, exchanging $\alpha_1$ and $\alpha_2$ if necessary, there exists a path $\beta_2$ and $p>0$ such that (up to homotopy) $\alpha_2\cdot = \beta_2\alpha_1^p\cdot$, that $\alpha_1\cdot$ is not equivalent to a suffix of $\beta_2\cdot$, and that $\beta_2$ is not homotopically trivial (otherwise the homotopy type of $\alpha_2$ would be a power of the one of $\alpha_1$, which is impossible). 

When $\alpha_1$ and $\beta_2$ are seen as words in $\beta_A$ and $\beta_B$ (the loops generating the fundamental group of $\check\dom\F$), the length of $\alpha_1\beta_2\cdot$ is bigger than the length of $\alpha_1\cdot$. Moreover, $\alpha_1\cdot$ is not a suffix of $\beta_2\cdot$. This implies that the leaves at the left end of $\alpha_1\cdot$ and $\alpha_1\beta_2\cdot$ are not comparable. Recall that by hypothesis, the leaf passing through the left end of $\alpha_1\beta_2\alpha_1^p\cdot$ is over the leaf passing through the left end of $\beta_2\alpha_1^{p+1}\cdot$ relative to the leaf passing through their common right end, so  the leaf passing through the left end of $\alpha_1\beta_2\cdot$ is above the leaf passing through the left end of $\alpha_1\cdot$ relative to the leaf passing through their common right end.

Now, let us compare the leaves on the left ends of $\alpha_1\alpha_2^i\cdot$ and $\alpha_2\alpha_1^j\cdot$. As $\alpha_1\alpha_2^i\cdot = \alpha_1(\beta_2\alpha_1^p)^i\cdot$, if we compare successfully the leaves on the left ends of suffixes of them, namely $\alpha_1\beta_2\alpha_1^p\cdot$ and $\alpha_1^{p+j}\cdot$, we are done. But we already know that the leaf on the left end of $\alpha_1\beta_2\cdot$ is above the leaf on the left end of $\alpha_1\cdot$ relative to the leaf passing through their common right end, so the leaf on the left end of $\alpha_1\alpha_2^i\cdot$ is above the leaf on the left end of $\alpha_2\alpha_1^j\cdot$ relative to the leaf passing through their common right end. This proves the proposition.
\end{proof}

\subsection{Setting}\label{Secsetting}

We set here some notations for the two next paragraphs.

Let $f\in \Homeo_0(S)$, and $\gamma$ a closed geodesic with a geometric auto-intersection associated to the deck transformation $T_1$ (in the sense of Definition \ref{DefTransverseInter}). Denote $T_2$ the deck transformation so that $T=T_2T_1$ is a deck transformation associated to the closed geodesic $\gamma$. We suppose that $(\gamma,\ell(\gamma))\in \rho(f)$.

By (iii) of Proposition \ref{PropRealPtExtRat}, there exists a fixed point $x$ of $f$ having rotation vector $(\gamma,\ell(\gamma))$. Consider the foliation $\F$ and the isotopy $I$ given by Theorem~\ref{ThExistIstop}. We denote by $\hat{f}$ the canonical lift of $f$ to $\wt\dom \F$. As lifts of $x$ to $\tilde{S}$ are not fixed by the lift $\tilde f$ of $f$ to $\tilde{S}$, the point $x$  belongs to $\dom\F$; this allows to consider a closed transverse loop $\alpha$ of $S$ associated to the trajectory of $x$ which is homotopic to the closed geodesic $\gamma$. This loop is admissible of order $1$. We denote by $\tilde{\alpha}$ a lift of $\alpha$ to $\tilde{S}$ which corresponds to the deck transformation $T$ and by $\hat{\alpha}$ a lift of $\tilde{\alpha}$ to $\wt \dom(\F)$.

By Proposition~\ref{GeomImpliqFeuill}, the loops $\tilde \alpha$ and $T_1 \tilde\alpha$ intersect $\F$-transversally at $\tilde \alpha(t_1) = T_1^{-1} \tilde \alpha(t_2)$, for $t_1<t_2<t_1+1$. We denote $\alpha_1 = \alpha|_{[t_1,t_2]}$ and $\alpha_2 = \alpha_{[t_2,t_1]}$. 

Note that, for any $n \geq 1$, the transverse paths $(\alpha_1 \alpha_2)^n$ and $(\alpha_2 \alpha_1)^n$ are admissible of order $n+1$.

\subsection{Creation of new periodic points}\label{SubSecCrea}

We use notation from Subsection~\ref{Secsetting}.

As a preliminary to the existence of a rotational horseshoe (Theorem~\ref{ExistSuperCheval}), we prove the existence, for any finite word $(w_i)\in \{1,2\}^k$, of periodic orbits rotating in the direction $T_{w_1}\dots T_{w_k}$ (Proposition~\ref{PropExistPeriodic}). Note that the periods we get for these periodic orbits are better than the ones that can be obtained from Theorem~\ref{ExistSuperCheval}.

Let $(w_i)\in \{1,2\}^{\Z/k\Z}$ be a periodic word of length $k$. We suppose that its period $k$ is minimal.

Let us consider the smallest periodic word $(\overline w_j)\in \{1,2\}^{\Z/m\Z}$ of the form $(1\, 2)^\ell$ or $(2\, 1)^\ell$ obtained from $(w_i)$ by adding some letters (hence $m=2\ell$).

It can be seen that if we break the word $(w_i)$ into blocks $(b_j)_{j=1}^{j_0}$ of consecutive identical letters (counting the first and the last blocks together as only one block if they contain the same letters), then (notice that $j_0$ is even)
\[\ell = \frac12\left(k+\sum_{j=1}^{j_0} (\operatorname{length}(b_j)-1) \right) = k-\frac{j_0}{2} < k.\]
For example, if $(w_i) = (1\, 2\, 2\, 1\, 2\, 1)$, then the smallest word of the form $(1\, 2)^\ell$ or $(2\, 1)^\ell$ obtained from $(w_i)$ by adding some letters is $(1\, 2\,{\color{gray}\underline 1\,} 2\, 1\, 2\, 1\, {\color{gray}\underline 2})$ and so $\ell  = 4$.

\begin{proposition}\label{PropExistPeriodic}
Let $w=(w_i)\in \{1,2\}^{\Z/k\Z}$ be a periodic word of length $k$, such that $w_1=w_k$. Let $\ell=\ell(w_i)$ be defined as above.
For any $r\ge 1$ and any $q\ge r\ell + 3$, there exists a periodic point $x$ for $f$, with a lift $\tilde{x}$ to $\tilde{S}$ such that
\[\tilde f^{q}(\tilde x) = \big(T_{w_1}\dots T_{w_k}\big)^r(\tilde x).\]
\end{proposition}

\begin{proof}
If $k\ge 2$, as $(w_i)$ is periodic of period $k\ge 2$, either it is equal to $(12)^\infty$, and there is nothing to prove (as this word corresponds to the initial rotation vector), or it contains at least one block $b_j$ of length $\ge 2$. By cyclically permuting the letters of the word $w$, we can suppose this block is $b_1$. Indeed, if we find a point $\tilde{y}$ such that
$$\tilde{f}^{q}(\tilde{y})=T_{w_{\xi}}T_{w_{\xi+1}} \ldots T_{w_{k}} T_{w_1} \ldots T_{w_{\xi-1}}\tilde{y}$$ for some integer $\xi \in [1,k]$, then the point $\tilde{x}=T_{w_1} \ldots T_{w_{\xi-1}}\tilde{y}$ will satisfy the proposition. Changing the roles of $1$ and $2$ if necessary, we can suppose that this block $b_1$ is made of 1s (hence the word $(w_i)$ starts as $1^{\operatorname{length}(b_1)}$).

Consider the finite word 
\[a=a_1 a_2 \ldots a_n= 1\,2\big(\overline w_3 \dots \overline w_m \overline w_1 \overline w_2 \big)^r 1\,2\]
(note that $\overline w_3 = w_2 = 1$ and $\overline w_2 = 2$).
As $a=(1\,2)^{rl+2}$, the associated transverse path $\alpha_{a_1}\dots \alpha_{a_n}=(\alpha_1 \alpha_2)^{rl+2}$ is admissible of order $r\ell+3$.

The last statement of Proposition~\ref{GeomImpliqFeuill} implies that for any $k,k'\ge 1$, the marked path $\alpha_2\alpha_1^{k}\cdot \alpha_2\cdot\alpha_1^{k'}\alpha_2$ has self $\F$-transverse intersection at its marked points (and the same holds for $\alpha_1\alpha_2^{k}\cdot \alpha_1\cdot\alpha_2^{k'}\alpha_1$). Using Proposition \ref{PropFondLCT2} and the first self transverse intersection, we can remove a letter $2$ between the positions $2$ and $n-3$ from the word $a_1 \ldots a_n$ to obtain a new word which is admissible of order $r\ell +3$. In the same way, by using the second self transverse intersection, we can remove a letter $1$ between the positions $2$ and $n-3$ from the word $a_1 \ldots a_n$ to obtain a new word which is admissible of order $r\ell +3$.
So, by successive applications of Proposition~\ref{PropFondLCT2}, we get that the path 
\[\beta = \alpha_{1}\alpha_2\big(\alpha_{w_2}\dots \alpha_{w_k} \alpha_{w_1} \big)^r \alpha_{\overline{w}_2}\alpha_1\alpha_2
= \alpha_{1}\alpha_2\big(\alpha_{1}\dots \alpha_{w_k} \alpha_{1} \big)^r \alpha_{2}\alpha_1\alpha_2\]
is also admissible of order $r\ell+3$. Let $\hat{\beta}$ be the lift of $\beta$ corresponding to the deck transformation $T_1T_2\big(T_{w_2}\dots T_{w_k} T_{w_1} \big)^r T_2 T_1 T_2$. Remark that the paths $\hat{\beta}$ and $T_1T_2  \big(T_{w_2}\dots T_{w_k} T_{w_1} \big)^r T_2^{-1} T_1^{-1} \hat{\beta}$ intersect $\widehat{\F}$-transversally, simply because $\beta$ starts with $\alpha_1\alpha_2 \cdot \alpha_1^i\alpha_2\dots$ and ends with $\dots\alpha_2\alpha_1^j \cdot\alpha_2\alpha_1\alpha_2$ (Proposition~\ref{GeomImpliqFeuill}).
Thus, by Theorem~\ref{ExistPasSuperCheval} (which is Theorem M of \cite{1803.04557}), for any $q \ge r\ell +3$, there exists a point $\tilde y\in\wt S$ such that
\[\tilde f^{q}(\tilde y) = T_1T_2\big(T_{w_2}\dots T_{w_k} T_{w_1}\big)^r T_2^{-1} T_1^{-1}(\tilde y).\]
We can then take $\tilde{x}= T_1 T_2^{-1}T_1^{-1}(\tilde{y})$. to obtain a point $\tilde{x}$ which satisfies the conclusion of the proposition.

The arguments work identically when the initial word $(w_i)$ is constant (\emph{i.e.} is equal to either $1^\infty$ or $2^\infty$).
\end{proof}

\subsection{Horseshoe}\label{SubSecHorse}

We now come to the main theorem of this section, that concerns the existence of a rotational horseshoe (Theorem~\ref{ExistSuperChevalIntro} of the introduction). Again, we use notation from Subsection~\ref{Secsetting}.

\begin{theorem}\label{ExistSuperCheval}
Let $f\in \Homeo_0(S)$, and $\gamma$ a closed geodesic with a geometric auto-intersection associated to the deck transformation $T_1$ (in the sense of Definition \ref{DefTransverseInter}). Denote $T_2$ the deck transformation such that $T=T_1T_2$ is the deck transformation associated to the closed geodesic $\gamma$. 

Suppose that $(\gamma,\ell(\gamma))\in \rho(f)$.
Then, $f^7$ has a topological horseshoe associated to the deck transformations $T_1$, $T_1^2$, $T_2$, $T_1T_2$, $T_2T_1$ and $T_1T_2T_1$.
\end{theorem}

In particular, this implies that $h_{top}(f)\ge \log 7/5$.

\begin{proof}
The configuration of the beginning of the proof is depicted in Figure~\ref{FigExistSuper1}. In particular, $\hat{\alpha}_1$ and $\hat{\alpha}_2$ are two lifts of $\alpha_1$ and $\alpha_2$ to $\wt \dom(\F)$ that have $\hat{\alpha}(t_2)$ as final point. By abuse of notation, we denote by $T_1$ and $T_2$ the lifts to $\wt\dom(\F)$ of the corresponding deck transformations of $\tilde{S}$ which are respectively associated to $\hat{\alpha}_1$ and $\hat{\alpha}_2$.

In $\wt\dom \F$, denote $\phi_a = T_2^{-2}T_1^{-1}(\phi_{\hat \alpha(t_2)})$ and $\phi_b = T_2^2T_1T_2(\phi_{\hat\alpha(t_2)})$.

\begin{figure}
\begin{center}
\includegraphics[width=\linewidth]{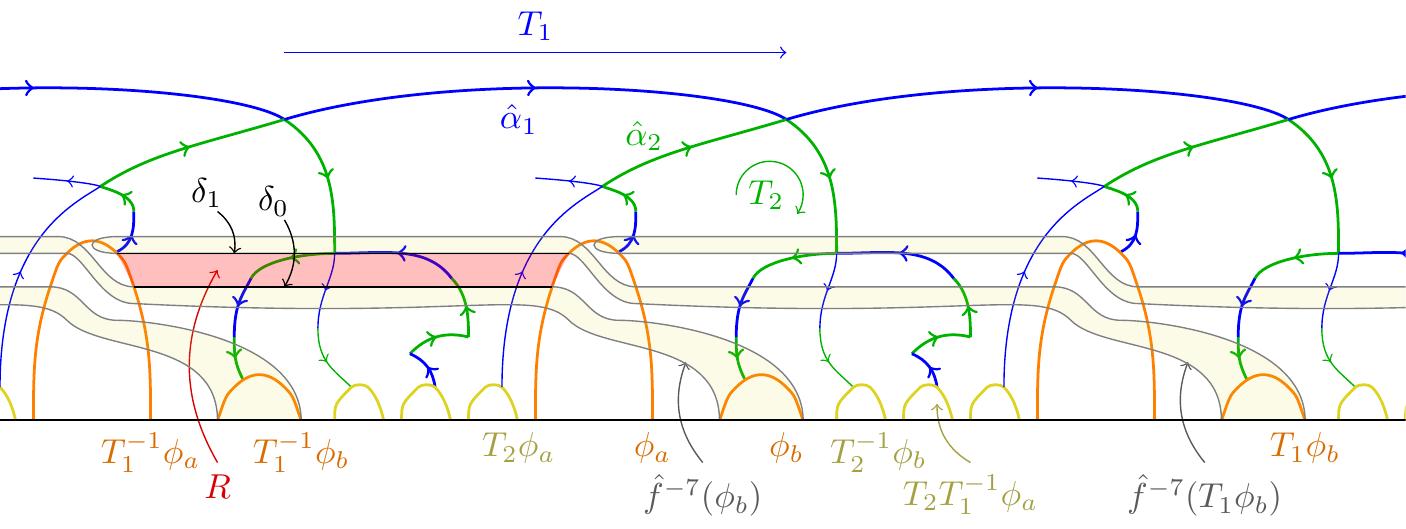}
\caption{Configuration of Theorem \ref{ExistSuperCheval}. Note that we do not know \emph{a priori} whether the leaf $\phi_a$ is located in the left of $T_2 \phi_a$ or is below $T_2\phi_a$ relative to $T_1T_2^2\phi_a$.\label{FigExistSuper1}}
\end{center}
\end{figure}

Remark that any lift of $(\alpha_1\alpha_2)^k$ or $(\alpha_2\alpha_1)^k$ is admissible or order $k+1$. By successive applications of Proposition~\ref{PropFondLCT2}, allowed by Proposition~\ref{GeomImpliqFeuill}, we get that
\begin{itemize}
\item $\hat f^6(\phi_a)\cap \phi_b \neq\emptyset$ because $\alpha_1 \alpha_2^4 \alpha_1 \alpha_2$ is admissible of order $6$;
\item $\hat f^6(\phi_a)\cap T_1\phi_b \neq\emptyset$ because $\alpha_1 \alpha_2^2 \alpha_1 \alpha_2^{2} \alpha_1 \alpha_2$ is admissible of order $6$;
\item $\hat f^7(\phi_a)\cap T_1^2\phi_b \neq\emptyset$ because $\alpha_1 \alpha_2^2 \alpha_1^2 \alpha_2^{2} \alpha_1 \alpha_2$ is admissible of order $7$.
\end{itemize}
For example, for the first one, the path $(\alpha_1\alpha_2)^5$ is admissible of order 6, hence so does $\alpha_1 \alpha_2^4 \alpha_1 \alpha_2$. 

Note that, by Proposition~\ref{PropPasFondLCT}, we also have $\hat f^7(\phi_a)\cap \phi_b \neq\emptyset$ and $\hat f^7(\phi_a)\cap T_1\phi_b \neq\emptyset$.

\paragraph{Construction of the rectangle $R$ --}

As in Section 3.1 of \cite{1803.04557}, one can define
\[R_a = \bigcap_{k\in \Z} R(T_1^k\phi_a),\]
and the set $\X_{p}$ of paths joining $T_1^{-1}\phi_a$ to $\phi_a$ whose interior is a connected component of $T_1^p \hat f^{-7}(\phi_b) \cap R_a$. The following lemma proves that $\X_0, \X_1\neq\emptyset$.

\begin{lemma}\label{Lemma10}
Every simple path $\delta : [c,d]\to\wt\dom\F$ that joins $T_1^{-p_0}\phi_a$ to $T_1^{p_1} \phi_a$, with $p_0,p_1>0$, and which is $T_1$-free, meets $L(\phi_a)$.

Similarly, for any $t\in\R$, if $\hat f^{-7}(\phi_b((-\infty, t]))$ meets $T_1^{-p}\phi_a$ for some $p>0$, then it also meets $L(\phi_a)$.
\end{lemma}

For the first part of the lemma, the idea of proof is that, if the path $\delta$ meets neither $L(\phi_a)$ nor $T_1^{p_1+p_0} \phi_a$, then this path, together with the leaves $T_1^{-p_0}\phi_a$ and $T_1^{p_1} \phi_a$, separates the leaves $\phi_a$ and $T_{1}^{p_0+p_1}\phi_a$, which implies (by an application of Jordan theorem) that $\delta$ is not $T_1$-free. The case where the path $\delta$ meets $T_1^{p_1+p_0} \phi_a$ but not $T_1^{p_1+2p_0} \phi_a$ leads to a similar contradiction, and so on. For a more detailed proof, see Lemma 10 of \cite{1803.04557}. The proof of the second part of the lemma is identical.
\bigskip

By what we have just said, using Lemma \ref{Lemma10} (and similarly to Lemma~11 of \cite{1803.04557}), the sets $\X_0$ and $\X_1$ are nonempty. Moreover, because the sets $T_1^k \hat f^{-7}(\phi_b)$ are pairwise disjoint, two elements of respectively $\X_0$ and $\X_1$ are disjoint.

\begin{lemma}\label{LemOrientBordR}
There is a path $\delta_1\in \X_1$, and a path $\delta_0 \in \X_0$ lying in the connected component of the complement of $R_a^\complement \cup \delta_1$ containing $T_1^{-1}\phi_b$.
\end{lemma}

Before proving the lemma, let us point out that because of $\F$-transverse intersections (last conclusion of Proposition~\ref{GeomImpliqFeuill}), we have, for any $k\in\Z^*$ and any $n\in\Z$,
\begin{equation}\label{EqDisjImgPhiq}
\hat f^n(T_1^k\phi_b) \cap \phi_b  = \emptyset.
\end{equation}
This implies that $\delta_1$ is disjoint from $T_1^{-1}\phi_b$, and hence that $T_1^{-1}\phi_b$ lies in the complement of $R_a^\complement \cup \delta_1$.

\begin{proof}
Note that the union of elements of $\X_1$ forms a compact subset of $T_1\hat f^{-7}(\phi_b)$, so there are finitely many elements of $\X_1$. Consider the first one, $\delta_1$, for the order on $T_1\hat f^{-7}(\phi_b)$ induced by some parametrization of $\phi_b$, and denote $\hat f^7(\delta_1) = T_1\phi_b|_{[t_1,t_2]}$. By Lemma~\ref{Lemma10}, second part, $T_1 \hat f^{-7}(\phi_b(t_2))$ is the first intersection point of $T_1 \hat f^{-7}(\phi_b)$ (again, for the order induced by some oriented parametrization of $\phi_b$) with $T_1^{-1}\phi_a$ ; in particular the path $T_1\hat f^{-7}(\phi_b|_{(-\infty,t_2]})$ meets $T_1^{-1}\phi_a$ at a single point. The complement of $L(T_1^{-1} \phi_a)\cup T_1\hat f^{-7}(\phi_b|_{(-\infty,t_2]})$ has two connected components. We denote by $A$ the one containing $\phi_b$.

As the set $\hat{f}^{-7}(\phi_b)$ meets $T_1^{-1} \phi_a$, this set is not contained in $A$. Consider the first intersection point $\hat{f}^{-7}(\phi_b)(t'_1)$ between $\partial A$ and $\hat{f}^{-7}(\phi_b)$. This point must belong to $T_1^{-1}(\phi_a)$ as $\hat{f}^{-7}(\phi_b)\cap T_1\phi_b= \emptyset$ by \eqref{EqDisjImgPhiq}. Lemma~\ref{Lemma10}, second part implies that $\hat{f}^{-7}(\phi_{b|(-\infty,t'_1]}) \cap R(\phi_a) \neq \emptyset$, which gives a path $\delta_0$ and proves Lemma~\ref{LemOrientBordR}. 
\end{proof}

Consider two paths $\delta_0\in \X_0$ and $\delta_1$ in $\X_1$ given by Lemma~ \ref{LemOrientBordR}. Similarly to what is done in the proof of Proposition 12 of \cite{1803.04557}, take $\beta$ the path made of the bounded connected component of $T_1^{-1}\phi_a\setminus(\delta_0\cup\delta_1)$, and $\beta'$ the path made of the bounded connected component of $\phi_a \setminus (\delta_0\cup\delta_1)$.

It allows to define the topological rectangle $R$ bounded by the four curves $\beta$, $\beta'$, $\delta_0$ and $\delta_1$. Lemma \ref{LemOrientBordR} implies that in the direct orientation, the paths are ordered as: $\delta_1\beta \delta_0 \beta'$.
Note also that the set $\hat f^7(R)$ is a topological rectangle, with two edges which are subsets of resp. $\phi_b$ and $T_1 \phi_b$, the two others being images of pieces of resp. $T^{-1} \phi_a$ and $\phi_a$.

\paragraph{Proof of the existence of Markovian intersections --}

Because of $\F$-transverse intersections (last conclusion of Proposition~\ref{GeomImpliqFeuill}), we have, for any $k\ge 0$ and any $n\in\Z$, 
\begin{equation}\label{EqDisjImgPhiz}
\hat f^n(\phi_a) \cap T_1^kT_2T_1^{-1}\phi_a =\emptyset;
\end{equation}
for any $k\in\Z^*$ and any $n\in\Z$,
\begin{equation}\label{EqDisjImgPhia}
\hat f^n(\phi_a) \cap T_1^kT_2\phi_a
\quad = \quad
\hat f^n(T_1^k\phi_b) \cap T_2^{-1}\phi_b
\quad = \quad 
\hat f^n(T_2\phi_b) \cap T_1^k\phi_b
\quad  = \quad \emptyset,
\end{equation}
and similarly,
for any $k, n\in\Z$ with $(k,n)\neq (0,0)$,
\begin{equation}\label{EqDisjImgPhic}
\hat f^n(\phi_a) \cap T_1^k\phi_a = \emptyset.
\end{equation}

\begin{figure}
\begin{center}
\includegraphics[scale=1]{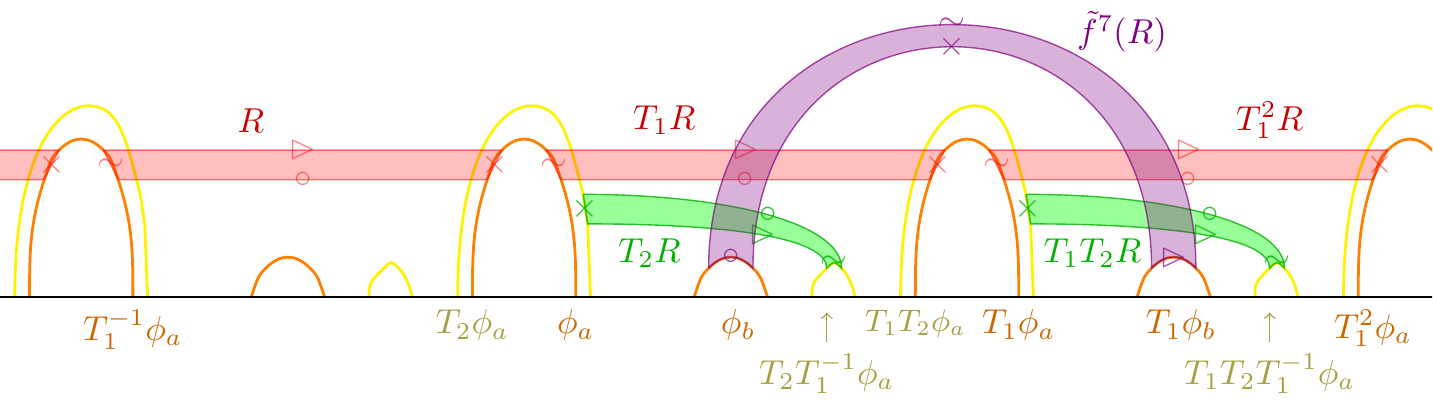}
\caption{Configuration of Theorem \ref{ExistSuperCheval}, the image of the rectangle $R$ by $\tilde f^7$ has Markovian intersections with $T_1R$, $T_1^2R$, $T_2R$ and $T_1T_2R$. It also has Markovian intersections with $T_2T_1R$ and $T_1T_2T_1R$ (not represented in the figure). Note that the relative position of $\phi_a$ and $T_2\phi_a$ is different from Figure~\ref{FigExistSuper1}, but \emph{a priori} possible.}\label{FigExistSuper2}
\end{center}
\end{figure}

The rectangle $\hat f^7(R)$ is disjoint from :
\begin{itemize}
\item $\phi_a$, $T_1\phi_a$ and $T_1^2\phi_a$, by \eqref{EqDisjImgPhic};
\item $T_1T_2\phi_a$ and $T_1^2T_2\phi_a$, by the first intersection of \eqref{EqDisjImgPhia}; 
\item $T_2\phi_a$. Indeed, the closure of the set $L(\phi_b)\cup \hat f^7(R) \cup L(T_1\phi_b)$ has two unbounded connected components in its complement, we denote by $C$ the one containing $\phi_a$. Because of the the orientation of $\partial R$ (which is a consequence of Lemma \ref{LemOrientBordR}), the closure of $C$ contains $\hat f^7(\beta)$, but is disjoint from $\hat f^7(\beta')$ (recall that $\beta$ is a piece of $T_1^{-1}\phi_a$). But by the first intersection of \eqref{EqDisjImgPhia}, $T_2\phi_a$ is disjoint from $\hat f^7(\beta)$, so $\hat f^7(R)$ is disjoint from $T_2\phi_a$, as $T_2 \phi_a \cap C \neq \emptyset$.
\item $T_2T_1^{-1}\phi_a$ and $T_1T_2T_1^{-1}\phi_a$. Indeed, by the same argument about orientation as before, it suffices to prove that the intersections $\hat f^7(\phi_a) \cap T_2T_1^{-1}\phi_a$ and $\hat f^7(T_1^{-1}\phi_a) \cap T_1T_2T_1^{-1}\phi_a$ are empty, which is true by \eqref{EqDisjImgPhiz}.
\end{itemize}

The rectangles $T_1^k R$ (for $k\in\Z$) are disjoint from the sets $T_1^\ell \phi_b$, by \eqref{EqDisjImgPhiq}.

The rectangle $T_2 R$ is disjoint from the sets $\phi_b$ and $T_1\phi_b$. Indeed, by the same reasoning about orientation as before, we just have to prove that the intersections $T_2\hat f^{-7}(T_1\phi_b)\cap \phi_b$ and $T_2\hat f^{-7}(\phi_b) \cap T_1\phi_b$ are empty, which is true by the two last intersections of \eqref{EqDisjImgPhia} .
\medskip

All these facts, combined with Homma's theorem (Theorem~\ref{PropHomma}), imply that the intersections of $\tilde f^7(R)$ with the following sets are Markovian (see Figures~\ref{FigExistSuper2} and \ref{FigExMarkov}): $T_1R$, $T_1^2R$, $T_2 R$, $T_1T_2R$, $T_2T_1R$ and $T_1T_2T_1R$. For example, Homma's theorem asserts that there exists a homeomorphism $h : \wt S\to \R^2$ such that
\[h(\phi_a) = \{0\}\times\R,\ h(T_1\phi_a) = \{1\}\times\R,\ h(\delta_0) = [0,1]\times\{0\},\ h(\delta_1) = [0,1]\times\{1\}.\]
Hence, because $R(\phi_a)\cup T_1R\cup L(T_1\phi_a)$ separates $\phi_b$ and $T_1\phi_b$ (this is a consequence of the previous listed facts),
\[h(\hat f^7(\delta_0))\subset h(\phi_b) \subset ]0,1[ \times (-\infty,0[, \quad 
h(\hat f^7(\delta_1))\subset h(T_0\phi_b) \subset ]0,1[ \times ]1,+\infty[,\]
and similarly, because $R(\phi_b)\cup \hat f^7(R) \cup L(T_1\phi_b)$ separates $\phi_a$ and $T_1\phi_a$, $h(\hat f^7(R)) \subset (0,1)\times\R$. The fact that the other intersections are Markovian can be proved similarly, using the previous listed facts. 
\end{proof}

By a proof which is very similar, we can get the following statement, which is a reformulation of \cite[Section 3]{1803.04557} to fit with the definition of rotational horseshoe we use here.

\begin{theorem}\label{ExistPasSuperCheval}
Let $S$ be an orientable surface, $f\in \Homeo_0(S)$, $\F$ a transverse foliation in the sense of Theorem~\ref{ThExistIstop} and $\alpha : [0,1]\to\dom(\F)$ an $\F$-transverse curve. Denote $\hat\alpha$ a lift of $\alpha$ to the universal cover $\wt\dom(\F)$ of $\dom(\F)$.

Suppose that $\alpha$ is admissible of order 1, and that there exists a deck transformation $T$ of $\wt\dom(\F)$ and $0<t_1<t_2<1$ such that $\hat\alpha$ and $T\hat\alpha$ have an $\F$-transverse intersection at $\hat \alpha(t_2) = T\hat\alpha(t_1)$.

Then, for any $r\ge 2$, $f^{r}$ has a topological horseshoe associated to the deck transformations $T, T^2, \dots, T^r$.
\end{theorem}

The proof of a similar statement can be found in \cite{1803.04557}, however, as it is very similar to the one of Theorem~\ref{ExistSuperCheval}, we include a short sketch of proof. 

\begin{proof}[Sketch of proof]
We give a sketch of proof for $r=2$.

\begin{figure}
\begin{center}
\includegraphics[scale=1]{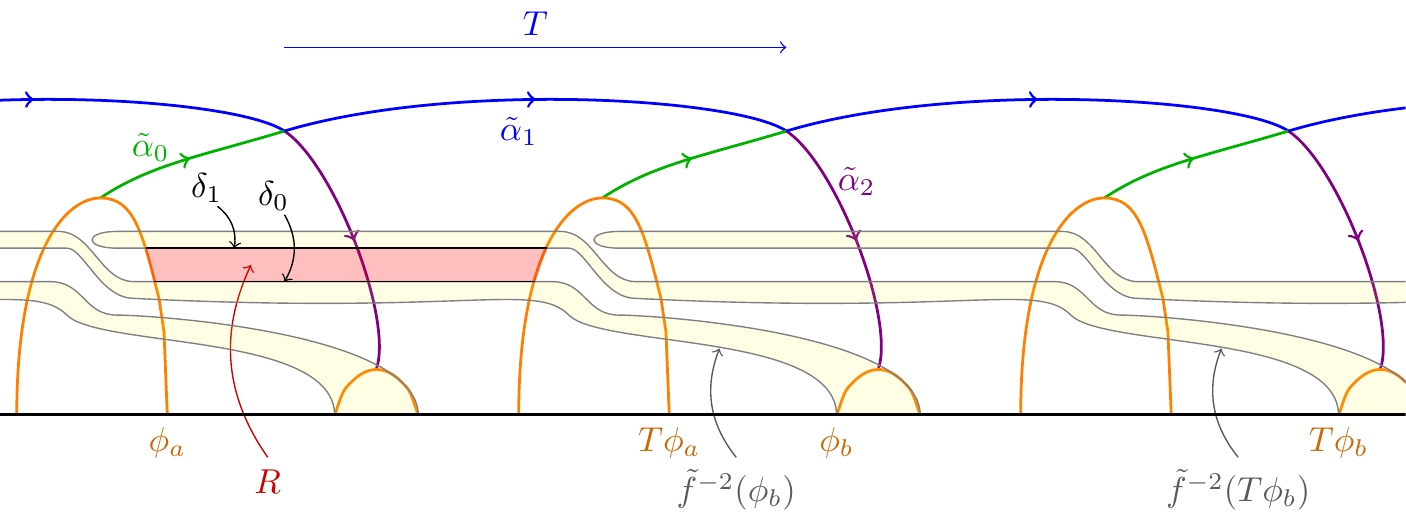}
\caption{Configuration of Theorem \ref{ExistPasSuperCheval}.}\label{FigExistSuper3}
\end{center}
\end{figure}

Denote $\hat\alpha_0 = \hat\alpha|_{[0,t_1]}$, $\hat\alpha_1 = \hat\alpha|_{[t_1,t_2]}$ and $\hat\alpha_2 = \hat\alpha|_{[t_2, 1]}$ (see Figure~\ref{FigExistSuper3}). We also set $\phi_a = \phi_{\hat\alpha(0)}$ and $\phi_b = \phi_{\hat\alpha(1)}$

Then, applying Proposition~\ref{PropFondLCT2}, the paths $\hat\alpha_0\tilde\alpha_1\tilde\alpha_2$ and $\hat\alpha_0 T^{-1}\tilde\alpha_2$ are admissible of order 1, and the path $\hat\alpha_0\hat\alpha_1(T\hat\alpha_1)(T\hat\alpha_2)$ is admissible of order 2. As in the proof of Theorem~\ref{ExistSuperCheval}, it is possible to find two pieces of resp. $\hat f^{-2}(\phi_b)$ and $\hat f^{-2}(T\phi_b)$, each of one meeting $\phi_a$ and $T\phi_a$ only at its endpoints. The set $R$ bounded by these paths and the bounded pieces of $\phi_a$ and $T\phi_a$ linking their ends is a rectangle, such that the intersections $\hat f^2(R) \cap TR$ and $\hat f^2(R) \cap T^2R$ are Markovian.
\end{proof}

\section{Two transverse closed geodesics}\label{LastSection}

Recall that $S$ is an orientable surface of finite type and negative Euler characteristic.
Let $\tilde{\gamma}_1$ and $\tilde{\gamma}_2$ be two geodesic lines of $\tilde{S} = \Hy^2$ which project to closed geodesics of $S$. We denote by $T_i$ the deck transformation associated to $\tilde{\gamma}_i$, for $i=1,2$. For any element $w$ of the semigroup $\langle T_1,T_2\rangle _+$ generated by $T_1$ and $T_2$, we denote by $\tilde{\gamma}(w)$ the geodesic axis of the deck transformation $w$.

In this last section, we prove Theorem~\ref{Th2transverseIntro} of the introduction. As the previous one, it is based on forcing theory of le Calvez-Tal. It deals with the case where in the rotation set, there are two closed geodesics with geometric intersection, each one associated with nonzero rotation speed. Contrary to the last section where we got the existence of a rotational horseshoe, here the proof does not give such an object associated to the deck transformations $T_1$ and $T_2$ (in fact, we even do not know if the two initial rotation vectors are realised by periodic orbits or not), but we get similar consequences.

The following is Theorem~\ref{Th2transverseIntro} of the introduction. 

\begin{theorem} \label{Th2transverse}
Suppose that there exist nonzero rotation vectors of directions $\tilde{\gamma}_1$ and $\tilde{\gamma}_2$ in $\rho(f)$ and that the geodesics $\tilde{\gamma}_1$ and $\tilde{\gamma}_2$ cross. Then, for any element $w$ in $\langle T_1,T_2\rangle _+$, there are nonzero vectors of direction $\tilde{\gamma}(w)$ in $\rho(f)$. 
\end{theorem}

In the course of the proof of the theorem, we can also recover the fact that, in this situation, the topological entropy of $f$ is positive. It was already known as a consequence of Theorem \ref{ExistSuperCheval} in the case where one of the closed geodesics $\gamma_1$ or $\gamma_2$ has an autointersection and a consequence of Corollary \ref{entropy} otherwise.

Observe that, if the element $w$ does not belong to the cyclic groups $\langle T_1\rangle $ nor $\langle T_2\rangle $, such a geodesic $\tilde{\gamma}(w)$ has a positive transverse intersection with one geodesic among $\tilde{\gamma}_1$ and $\tilde{\gamma}_2$ and a negative transverse intersection with the other one.  By Proposition \ref{PropRealPtExtRat}, this implies the following corollary.

\begin{corollary} \label{Cor2transverse}
Suppose that there exist nonzero rotation vectors of directions $\tilde{\gamma}_1$ and $\tilde{\gamma}_2$ in $\rho(f)$. Then, for any element $w$ in $\langle T_1,T_2\rangle _+$ which does belong neither to the cyclic group $\langle T_1\rangle $ nor to the cyclic group $\langle T_2\rangle $, there are infinitely many periodic orbits whose rotation vector is in the direction $\tilde{\gamma}(w)$.
\end{corollary}

Observe that the above corollary is equivalent to Theorem \ref{Th2transverse}. Actually, to prove Theorem \ref{Th2transverse}, we will prove Corollary \ref{Cor2transverse}. The rest of this section is devoted to the proof of Theorem \ref{Th2transverse}. In the first subsection, we will distinguish two cases in which the proof must be carried out. The second subsection is devoted to a useful notion that we use. The two following subsections are devoted to the proof of Theorem \ref{Th2transverse} in each of those cases.

\subsection{Two cases}

Fix $i=1,2$. Take a rational number $\frac{p_i}{q_i}>0$ such that $(\tilde{\gamma}_i,\frac{p_i}{q_i} \ell(\gamma_i))$ belongs to $\rho(\tilde{f})$ but is not an extremal point of $\rho(\tilde{f})$.

Recall that $\tilde{f}$ extends continuously to $\tilde S \simeq \overline{ \Hy^2}$ by fixing all the points of $\partial \Hy^2$. Denote by $\gamma_{i,-}$ and $\gamma_{i,+}$ the endpoints of $\tilde{\gamma}_i$ on $\partial \Hy^2$, where $\tilde{\gamma}_i$ is oriented from $\gamma_{i,-}$ to $\gamma_{i,+}$. Denote by $A_{i}$ the closed annulus $\big(\overline{ \Hy^2} \setminus \left\{ \gamma_{i,-},\gamma_{i,+} \right\} \big)/ \langle T_i\rangle $, by $\pi_{i}$ the projection from the interior of $A_i$ to $S$, and by $f_i$ the homeomorphism of $A_i$ induced by $\tilde{f}$ on $A_i$.

\begin{lemma} \label{LemRecurrentpoint}
There exists a $f_{i}$-birecurrent point $x_i$ of $A_i$ with lift $\tilde{x}_{i}$ to $\Hy^2$ such that  the two following properties are satisfied.
\begin{enumerate}
\item $$ \lim_{n \rightarrow \pm \infty} \tilde{f}^n (\tilde{x}_i)=\gamma_{i,\pm}$$
\item $$ \lim_{n \rightarrow \pm \infty} \frac{1}{n} d\big(\pi_{\tilde{\gamma}_i}(\tilde{x}_i),\pi_{\tilde{\gamma}_i}(\tilde{f}^n(\tilde{x}_i))\big) = v_i > \frac{p_i}{q_i} \ell(\gamma_i).$$ 
\item The orbit of $\tilde{x}_i$ under $\tilde{f}$ stays within a bounded distance from the geodesic $\tilde{\gamma}_i$.
\item The closure of the orbit of ${x}_i$ does not contain fixed points of\footnote{Meaning that the closure of the orbit in $S$ does not contain points that lift to fixed points of $\tilde f$.} $\tilde{f}$.
\end{enumerate}
Moreover, if the closed geodesic $\gamma_i$ has an autointersection, we also require the point $x_i$ to have a periodic orbit.
\end{lemma}

\begin{proof}
Suppose first that $\gamma_i$ has no autointersection. Take $w>0$ such that $(\tilde{\gamma}_i,w)$ is an extremal point of $\rho(f)$.
Use Proposition \ref{sublindistance} so that one of the following is true. 
\begin{enumerate}
\item Either there exists a point ${x}_i$, which lifts to a recurrent point of $A_i$, with one lift ${\tilde x}_i\in\Hy^2$ realising the rotation vector $(\tilde{\gamma}_i,w)$. Moreover, the orbit of ${\tilde x}_i$ stays at a bounded distance from the geodesic $\tilde{\gamma}_i$ and the closure of the orbit of ${x}_i$ does not contain fixed point of $\tilde{f}$. In this case, take $v_i=w$.
\item Or, for any $r$ rational strictly smaller than $\frac{w}{\ell(\gamma)}$, there exist a periodic orbit whose rotation number is $(\tilde{\gamma}_i,r \ell(\gamma))$. In this case, take $v_i =r_{0} \ell(\gamma)> \frac{p_i}{q_{i}} \ell(\gamma)$, for some $r_{0}> \frac{p_i}{q_i}$ to find a point which satisfies the lemma. 
\end{enumerate}

If the closed geodesic $\gamma_i$ has a transverse autointersection, fix a number $v=r_i \ell(\gamma)>\frac{p_i}{q_i} \ell(\gamma)$ where $r_i$ is rational and $(\tilde{\gamma}, r_i \ell(\gamma))$ belongs to $\rho(f)$. Then, by Proposition \ref{PropRealPtExtRat}, there exists a point $x_i$ whose orbit is periodic and which realises this rotation vector. This point $x_i$ satisfies the requirements of the lemma.
\end{proof}

We denote by $\F$ a foliation of $S$ so that Theorem \ref{ThExistIstop} is satisfied, by $\mathcal{F}_i$ the lift of the foliation $\mathcal{F}$ to $A_i$, by $\tilde{\F}$ its lift to $\tilde{S}$ and by $\hat{\F}$ its lift to $\widetilde{\dom}\F$. Denote by $\mathcal{I}^{\mathbb{Z}}_{\mathcal{F}_i}(x_i)$ a $\F_i$-transverse trajectory associated to the orbit of $x_i$ under $f_i$. We use similar notation for $\F$, $\tilde{\F}$ or $\hat{\F}$-transverse trajectories. Choose respective lifts $\tilde{x}_1$ and $\tilde{x}_2$ of $x_1$ and $x_2$ to $\tilde{S}$ such that the trajectories $\mathcal{I}^{\mathbb{Z}}_{\tilde{\mathcal{F}}}(\tilde{x}_1)$ and $\mathcal{I}^{\mathbb{Z}}_{\tilde{\mathcal{F}}}(\tilde{x}_2)$ have the same endpoints on $\overline{\Hy}^2$ as the geodesics $\tilde{\gamma}_1$ and $\tilde{\gamma}_2$. In particular, the trajectories $\mathcal{I}^{\mathbb{Z}}_{\tilde{\mathcal{F}}}(\tilde{x}_1)$ and $\mathcal{I}^{\mathbb{Z}}_{\tilde{\mathcal{F}}}(\tilde{x}_2)$ meet at some point $\tilde{y}_0$. Finally, choose respective lifts $\hat{x}_1$ and $\hat{x}_2$ of $\tilde{x}_1$ and $\tilde{x}_2$ to $\widetilde{\dom} \F$ so that the trajectories $\mathcal{I}^{\mathbb{Z}}_{\hat{\mathcal{F}}}(\hat{x}_1)$ and $\mathcal{I}^{\mathbb{Z}}_{\hat{\mathcal{F}}}(\hat{x}_2)$ meet at some lift $\hat{y}_0$ of $\tilde{y}_0$.
We will use the following properties of the transverse trajectory.

\begin{lemma} \label{LemBoundeddistance}
For $i=1,2$, the transverse trajectory $\mathcal{I}^{\mathbb{Z}}_{\tilde{\F}}(\tilde{x}_i)$ can be chosen in such a way that the following properties are satisfied.
\begin{enumerate}
\item There exists a closed neighbourhood $K_i$ of $\mathcal{I}^{\mathbb{Z}}_{{\F}}({x}_i)$ such that the supremum $M_i$ of the diameters of $\I_{\tilde{\F}}(\tilde{z})$, with ${z} \in K_i$, is finite.
\item There exists $R_i>0$ such that the transverse trajectory $\mathcal{I}^{\mathbb{Z}}_{\tilde{\F}}(\tilde{x}_i)$ stays in the $R_i$-neighbourhood of the geodesic $\tilde{\gamma}_i$. 
\end{enumerate}
\end{lemma}

\begin{proof}
Denote by $U \subset S$ the complement of the projection on $S$ of fixed points of $\tilde{f}$. For any point $x\in U$, it is possible to find a neighbourhood $V$ of $x$ so that, changing the trajectories to equivalent ones if necessary, the diameters of the lifts to $\tilde{S}$ of $\I_{\F}(y)$ are uniformly bounded for $y \in V$. As the closure of the orbit of the point $x_i$ is compact and projects to $U$ (Lemma~\ref{LemRecurrentpoint}), we can choose the transverse trajectories of $x_1$ and $x_2$ so that they project on compact sets contained in $U$. Compactness of those trajectories and the fact that the transverse trajectories can be chosen locally bounded yields the first point. 

The second point is a consequence of the first point and the fact that the orbit of $\tilde{x}_i$ stays at a bounded distance from the geodesic $\tilde{\gamma}_i$.
\end{proof}

As $\mathcal{I}_{\hat{\F}}^{\Z}(\tilde{x}_i)$ is an oriented immersed line, there is a natural order relation $<_i$ on this set. By abuse of notation, we also denote by $<_i$ the order relation on $\mathcal{I}_{\tilde{\F}}^{\Z}(\hat{x}_i)$ induced by the orientation on those immersed lines. For any two points $x$ and $y$ on $\mathcal{I}_{\hat{\F}}^{\Z}(\tilde{x}_i)$ (respectively on $\mathcal{I}_{\tilde{\F}}^{\Z}(\hat{x}_i)$), we let
$$ [x,y]_{i}=\left\{ z  \in \mathcal{I}_{\hat{\F}}^{\Z}(\hat{x}_i) \ | \ x \leq_i z \leq_i y \right\}$$
(respectively
 $$ [x,y]_{i}=\left\{ z  \in \mathcal{I}_{\tilde{\F}}^{\Z}(\tilde{x}_i) \ | \ x \leq_i z \leq_i y \right\}).$$
 
For any two leaves $\phi_1$ and $\phi_2$ of $\hat{\F}$ which meet $\mathcal{I}_{\hat{\F}}^{\Z}(\hat{x}_i)$, we set $\phi_1 <_i \phi_2$ if $\phi_1\cap \mathcal{I}_{\hat{\F}}^{\Z}(\hat{x}_i) <_i \phi_2 \cap \mathcal{I}_{\hat{\F}}^{\Z}(\hat{x}_i)$. In this way, we can also define $[\phi_1,\phi_2]_i=[\phi_1\cap \mathcal{I}_{\hat{\F}}^{\Z}(\hat{x}_i) ,\phi_2 \cap \mathcal{I}_{\hat{\F}}^{\Z}(\hat{x}_i)]_i$. We use a similar notation for leaves of $\tilde{\F}$. When we use it in the case of leaves of $\tilde{\F}$, we tacitly choose points in the intersection between the leaf and the trajectory. Each time we use this notation, the choice of those points will be irrelevant.

Finally, for any segment $J$ contained in $\mathcal{I}_{\tilde{\F}}^{\Z}(\tilde{x}_i)$, and for any leaves $\phi_1, \phi_2, \ldots, \phi_n$ of $\tilde{\F}$, we say that $J$ meets the leaves $\phi_1, \phi_2, \ldots, \phi_n$ in this order if there exist points $\tilde{y}_1,\tilde{y}_2,\ldots, \tilde{y}_n$ which belong respectively to $\phi_1 \cap \mathcal{I}_{\tilde{\F}}^{\Z}(\tilde{x}_i), \phi_2 \cap \mathcal{I}_{\tilde{\F}}^{\Z}(\tilde{x}_i), \ldots, \phi_n \cap \mathcal{I}_{\tilde{\F}}^{\Z}(\tilde{x}_i)$ such that the segment $J$ meets the points $\tilde{y}_1,\tilde{y}_2,\ldots ,\tilde{y}_n$ in this order.

The following lemma, which is roughly a simple consequence of Lemma \ref{LemRecurrentpoint} in terms of transverse trajectories, will be useful.

\begin{lemma} \label{LemRecurrenceleaves}
Let $k >0$, $\epsilon =\pm 1$, $i=1,2$ and $0 < v'_i<v_i$. Let $(\phi_j)_{1 \leq j \leq \ell}$ be a sequence of leaves such that the segment $[\tilde{x}_i,\tilde{f}^{k}(\tilde{x}_i)]_i$ meets the leaves $\phi_1,\phi_2,\ldots,\phi_{\ell}$ outside $\tilde{x}_i$ and $\tilde{f}^k(\tilde{x}_i)$ and in this order. Then, for infinitely many $n_i > 0$, there exists $r \geq \frac{v'_i n_i}{\ell(\gamma_i)}$ such that the segment $[\tilde{f}^{\epsilon n_i}(\tilde{x}_i),\tilde{f}^{\epsilon n_i+k}(\tilde{x}_i)]_i$ meets the leaves $T_i^{\epsilon r}\phi_1,T_i^{\epsilon r}\phi_2,\ldots,T_i^{\epsilon r}\phi_{\ell}$ in this order.
\end{lemma}

\begin{proof}
By Lemma 17 in \cite{MR3787834}, there exists a small disk $\tilde{D}$ around $\tilde{x}_i$ such that, for any point $\tilde{y}$ in $\tilde{D}$, the transverse trajectory $\mathcal{I}^{k}_{\tilde{\F}}(\tilde{y})$ associated to $(\tilde{f}_t(y))_{0 \leq t \leq k}$ meets the leaves $\phi_1,\phi_2, \ldots \phi_{\ell}$ in this order. By Lemma \ref{LemRecurrentpoint}, for infinitely many $n_i >0$, $\tilde{f}^{\epsilon n_i}(\tilde{x}_i) \in T_{i}^{\epsilon r}(\tilde{D})$, with $r \geq \frac{v'_i n_i}{\ell(\gamma_i)}$. Hence the trajectory $\mathcal{I}^{k}_{\tilde{\F}}(\tilde{f}^{\epsilon n_i}(T_{i}^{-\epsilon r}\tilde{x}_i))$ meets the leaves $\phi_1,\phi_2, \ldots \phi_{\ell}$ in this order. Taking the image under $T_{i}^{\epsilon r}$ proves the lemma.
\end{proof}

For $i=1,2$, observe that the trajectory $\mathcal{I}^{\Z}_{\tilde{\F}}(\tilde{x}_i)$ satisfies one of the following conditions:
\begin{enumerate}[label=($C_\arabic*$), ref=($C_\arabic*$)]
\item\label{C1} either there exists a deck transformation $\tau_i \in \pi_1(S) \setminus \langle T_i\rangle $ and a leaf $\phi$ of $\tilde{\F}$ such that $\phi$ meets the trajectories $\mathcal{I}^{\Z}_{\tilde{\F}}(\tilde{x}_i)$ and $\tau_i \mathcal{I}^{\Z}_{\tilde{\F}}(\tilde{x}_i)$;
\item\label{C2} or, for any deck transformation $\tau \in \pi_1(S) \setminus \langle T_i\rangle $, any leaf which meets $\mathcal{I}^{\Z}_{\tilde{\F}}(\tilde{x}_i)$ does not meet $\tau \mathcal{I}^{\Z}_{\tilde{\F}}(\tilde{x}_i)$.
\end{enumerate}

To give some geometric intuition around this notion, observe that condition \ref{C2} amounts to saying that, on the surface $S$, the union of leaves which meet $\mathcal{I}^{\Z}_{\F}(\pi_i(x_i))$ is contained in an annulus which is embedded in $S$.

To carry out the proof of the theorem, we will distinguish the two following cases :
\begin{enumerate}[label=Case \arabic*]
\item\label{Case1}: both trajectories $\mathcal{I}_{\tilde{\mathcal{F}}}^{\mathbb{Z}}(\tilde{x}_1)$ and $\mathcal{I}_{\tilde{\mathcal{F}}}^{\mathbb{Z}}(\tilde{x}_2)$ satisfy the condition \ref{C1}.
\item\label{Case2}: one of the trajectories satisfies \ref{C2}. 
\end{enumerate}

The two following sections are devoted to the proof of Theorem \ref{Th2transverse} in each of those two cases. In the first case, we prove that the trajectories $\mathcal{I}_{\tilde{\F}}^{\Z}(\tilde{x}_1)$ and $\mathcal{I}_{\tilde{\F}}^{\Z}(\tilde{x}_2)$ have an $\tilde{\F}$-transverse intersection, which allows us to prove Theorem~\ref{Th2transverse}. In the second case, however, it is possible that such trajectories never have $\tilde{\F}$-transverse intersection, but it is possible to change one of the trajectories to obtain two trajectories with $\tilde{\F}$-transverse intersection.

The next paragraph is devoted to a notion which will be useful for our proof.

\subsection{Essential intersection points}

Take any two points $\tilde{x}$ and $\tilde{y}$ of $\tilde{S}$ which are not singularities of the foliation $\tilde{\F}$. For any point $\tilde{z}=\I_{\tilde{\F}}^{\Z}(\tilde{x})(t_1)=\I_{\tilde{\F}}^{\Z}(\tilde{y})(t_2)$, with $t_1,t_2 \in \mathbb{R}$, of $\I_{\tilde{\F}}^{\Z}(\tilde{x}) \cap \I_{\tilde{\F}}^{\Z}(\tilde{y})$, we call \emph{lifts of $\I_{\tilde{\F}}^{\Z}(\tilde{x})$ and $\I_{\tilde{\F}}^{\Z}(\tilde{y})$ associated to $\tilde{z}$} any respective lifts $\hat{\I}_x$ and $\hat{\I}_y$ of $\I_{\tilde{\F}}^{\Z}(\tilde{x})$ and $\I_{\tilde{\F}}^{\Z}(\tilde{y})$ to $\widetilde{\dom}\F$ which meet at a lift $\hat{z}=\hat{\I}_x(t_1)=\hat{\I}_y(t_2)$ of $\tilde{z}$. In case of multiple intersection points, \emph{e.g.} when $\tilde{z}$ is an autointersection point of $\I_{\tilde{\F}}^{\Z}(\tilde{x})$, notice that the values of the parameters $t_1$ and $t_2$ are important in this definition. However, to simplify notation, we will frequently drop the mention of those parameters when we use this notion.

\begin{definition}
A point $\tilde{z}$ on $\I_{\tilde{\F}}^{\Z}(\tilde{x}) \cap \I_{\tilde{\F}}^{\Z}(\tilde{y})$ is an \emph{essential intersection point} between $\I_{\tilde{\F}}^{\Z}(\tilde{x})$ and $\I_{\tilde{\F}}^{\Z}(\tilde{y})$ if there exist lifts $\hat{\I}_{x}$ and $\hat{\I}_y$ of $\I_{\tilde{\F}}^{\Z}(\tilde{x})$ and $\I_{\tilde{\F}}^{\Z}(\tilde{y})$ to $\widetilde{\dom}\F$ associated to $\tilde{z}$ such that $\hat{\I}_{x} \setminus \big(\hat{\I}_{x}\cap \hat{\I}_{y}\big)$ has two unbounded components which lie in different connected components of $\widetilde{\dom}\F \setminus \hat{\I}_{y}$.
\end{definition}

Note that this definition is supported by the fact that all transverse trajectories in $\wt{\dom}(\F)$ are proper.
Observe that, if this definition holds, then any two lifts $\hat{\I}_{x}$ and $\hat{\I}_y$ of $\I_{\tilde{\F}}^{\Z}(\tilde{x})$ and $\I_{\tilde{\F}}^{\Z}(\tilde{y})$ associated to $\tilde{z}$ will satisfy the above property.

Two trajectories $\I_{\tilde{\F}}^{\Z}(\tilde{x})$ and $\I_{\tilde{\F}}^{\Z}(\tilde{y})$, with $\tilde{x} \in \tilde{S}$ and $\tilde{y} \in \tilde{S}$, are said to be \emph{geometrically transverse} if there exist $\alpha_x,\alpha_y,\beta_x,\beta_y\in\partial\Hy^2$ such that the three following conditions are satisfied: 
\begin{enumerate}
\item The sequence $\tilde{f}^n(\tilde{x})$ converges to $\alpha_x \in \partial \Hy^2$ when $n \rightarrow -\infty$ and to $\omega_x \in \partial \Hy^2$ when $n \rightarrow +\infty$.
\item The sequence $\tilde{f}^n(\tilde{y})$ converges to $\alpha_y \in \partial \Hy^2$ when $n \rightarrow -\infty$ and to $\omega_y \in \partial \Hy^2$ when $n \rightarrow +\infty$.
\item The geodesic lines $(\alpha_x,\omega_x)$ and $(\alpha_y,\omega_y)$ meet in $\Hy^2$.
\end{enumerate}

\begin{lemma}[Properties of essential intersection points] \label{LemEssentialpoints}
$ $
\begin{enumerate}
\item (symmetry) Let $\tilde{z}$ be an essential intersection point between $\I_{\tilde{\F}}^{\Z}(\tilde{x})$ and $\I_{\tilde{\F}}^{\Z}(\tilde{y})$. Then, for any two lifts $\hat{\I}_{x}$ and $\hat{\I}_y$ of $\I_{\tilde{\F}}^{\Z}(\tilde{x})$ and $\I_{\tilde{\F}}^{\Z}(\tilde{y})$ to $\widetilde{\dom}\F$ associated to $\tilde{z}$, the two unbounded components of $\hat{\I}_{y} \setminus \big(\hat{\I}_{x}\cap \hat{\I}_{y}\big)$ lie in different connected components of $\widetilde{\dom}\F \setminus \hat{\I}_{x}$.
\item (geometrically transverse implies essential) If the trajectories $\I_{\tilde{\F}}^{\Z}(\tilde{x})$ and $\I_{\tilde{\F}}^{\Z}(\tilde{y})$ are geometrically transverse, then there exists an essential intersection point between $\I_{\tilde{\F}}^{\Z}(\tilde{x})$ and $\I_{\tilde{\F}}^{\Z}(\tilde{y})$.
\item ($\F$-transverse intersections) Let $\tilde{z}$ be an essential intersection point between $\I_{\tilde{\F}}^{\Z}(\tilde{x})$ and $\I_{\tilde{\F}}^{\Z}(\tilde{y})$ and let $\hat{\I}_{x}$ and $\hat{\I}_y$ be two lifts of $\I_{\tilde{\F}}^{\Z}(\tilde{x})$ and $\I_{\tilde{\F}}^{\Z}(\tilde{y})$ to $\widetilde{\dom}\F$ associated to $\tilde{z}$. Suppose that the unbounded components $C_2$ and $C'_2$ of $\hat{\I}_{y} \setminus \big(\hat{\I}_{x}\cap \hat{\I}_{y}\big)$ meet respectively leaves $\phi_2$ and $\phi'_2$ of $\hat{\F}$  which do not meet $\hat{\I}_{x}$ and that the unbounded components $C_1$ and $C'_1$ of $\hat{\I}_{x} \setminus \big(\hat{\I}_{x}\cap \hat{\I}_{y}\big)$ meet respectively leaves $\phi_1$ and $\phi'_1$ of $\hat{\F}$  which do not meet $\hat{\I}_{x}$. Then the trajectories $\I_{\tilde{\F}}^{\Z}(\tilde{x})$ and $\I_{\tilde{\F}}^{\Z}(\tilde{y})$ intersect $\tilde{\F}$-transversally at $\tilde{z}$. More precisely, any segment on $\I_{\tilde{\F}}^{\Z}(\tilde{x})$ joining $\phi_1$ and $\phi'_1$ is $\tilde{\F}$-transverse to any segment on $\I_{\tilde{\F}}^{\Z}(\tilde{y})$ joining $\phi_2$ and $\phi'_2$.
\end{enumerate}
\end{lemma}

\begin{proof}
\begin{enumerate}
\item Fix two lifts $\hat{\I}_{x}$ and $\hat{\I}_y$ associated to $\tilde{z}$. Note that the fact that $\hat{\I}_{x}$ and $\hat{\I}_y$ are proper implies that if $\hat{\I}_{x} \setminus \big(\hat{\I}_{x}\cap \hat{\I}_{y}\big)$ has two unbounded components, then $\hat{\I}_{y} \setminus \big(\hat{\I}_{x}\cap \hat{\I}_{y}\big)$ also has two unbounded components.

Suppose that the property we want to prove does not hold: the two unbounded components $C_2$ and $C'_2$ of $\hat{\I}_{y} \setminus \big(\hat{\I}_{x}\cap \hat{\I}_{y}\big)$ are contained in the same connected component $L$ of $ \widetilde{\dom\F} \setminus \hat{\I}_{x}$.  

\begin{claim}
There exists an arc $\alpha$ joining the two components $C_2$ and $C'_2$, which meets $\hat{\I}_{y}$ only at its endpoints and which is disjoint from $\hat{\I}_{x}$. 
\end{claim}

\begin{proof}
See Figure~\ref{FigLemEssentialpoints1}.
Take a closed disk $D$ whose interior contains the closure of the union of the bounded components of $\hat{\I}_{y} \setminus \hat{\I}_{x}\cap \hat{\I}_{y}$ and of $\hat{\I}_{x} \setminus \hat{\I}_{x}\cap \hat{\I}_{y}$. Then we claim that the closure of some connected component of $\partial D \setminus (\hat{\I}_{x}\cup \hat{\I}_{y})$ gives the desired path. Indeed, the closure of some connected component of $\partial D \setminus \hat{\I}_{x}$ is contained in $L$ and joins the two unbounded components of $\hat{\I}_{x} \setminus \hat{\I}_{x}\cap \hat{\I}_{y}$: otherwise, the endpoints of those unbounded components would be contain in the exterior of the disk $D$, a contradiction. Among those connected components of $\partial D \setminus \hat{\I}_{x}$ which are contained in $L$ and which join the two unbounded components of $\hat{\I}_{x} \setminus \hat{\I}_{x}\cap \hat{\I}_{y}$, one of those has to meet both $C_2$ and $C'_2$, otherwise the ends of both $C_2$ and $C'_2$ would be contained in the exterior of $D$. Now, the closure of some connected components of $S \setminus (C_2 \cup C'_2)$ has to join $C_2$ and $C'_2$, giving the path $\alpha$ we want.
\end{proof}

\begin{figure}
\begin{center}
\includegraphics[scale=1]{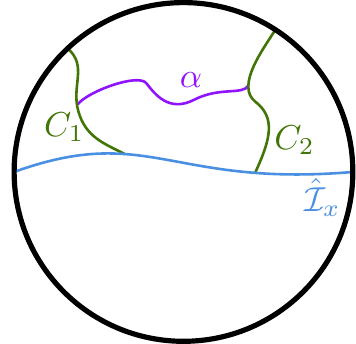}
\caption{Configuration of the proof of 1. of Lemma~\ref{LemEssentialpoints}.}\label{FigLemEssentialpoints1}
\end{center}
\end{figure}

Then some unbounded component $K_1$ of $\widetilde{\dom}\F \setminus \hat{\I}_{y}$ contains $\alpha$. Hence $K_1 \setminus \alpha$ has one bounded component $K_{1,1}$ and an unbounded component $K_{1,2}$  which is surrounded by $C_2 \cup \alpha\cup C'_2$. As the latter set does not meet $\hat{\I}_{x}$, the trajectory $\hat{\I}_{x}$ does not meet $K_{1,2}$. Hence the unbounded components of $\hat{\I}_{x} \setminus \hat{\I}_{x}\cap \hat{\I}_{y}$ have both to lie in the other unbounded component $K_2$ of $\widetilde{\dom}\F \setminus \hat{\I}_{y}$, in contradiction with the definition of essential intersection point.

\item Observe that the algebraic intersection number between the trajectories $\I_{\tilde{\F}}^{\Z}(\tilde{x})$ and $\I_{\tilde{\F}}^{\Z}(\tilde{y})$ is equal to $1$: it is well-defined as those trajectories meet in a compact subset of $\tilde{S}$. If those trajectories had only inessential intersection points, then the algebraic intersection number between those two trajectories would be equal to $0$.

Let us give some details. We define an equivalence relation on intersection points between $\I_{\tilde{\F}}^{\Z}(\tilde{x})$ and $\I_{\tilde{\F}}^{\Z}(\tilde{y})$ (or more precisely on couple of parameters which correspond to an intersection point). Two such intersection points $\tilde{z}_1$ and $\tilde{z}_2$ are equivalent if there exist lifts of $\I_{\tilde{\F}}^{\Z}(\tilde{x})$ and $\I_{\tilde{\F}}^{\Z}(\tilde{y})$ associated to $\tilde{z}_1$ which are also associated to $\tilde{z}_2$. Observe that, if this property is true, then it holds for any lifts of $\I_{\tilde{\F}}^{\Z}(\tilde{x})$ and $\I_{\tilde{\F}}^{\Z}(\tilde{y})$ associated to $\tilde{z}_1$. For each equivalence class $C$ of this equivalence relation, fix lifts $\hat{\I}_{1,C}$ and $\hat{\I}_{2,C}$ of $\I_{\tilde{\F}}^{\Z}(\tilde{x})$ and $\I_{\tilde{\F}}^{\Z}(\tilde{y})$ associated to this class $C$. The algebraic intersection number between $\I_{\tilde{\F}}^{\Z}(\tilde{x})$ and $\I_{\tilde{\F}}^{\Z}(\tilde{y})$ is then the sum over such classes $C$ of the algebraic intersection numbers $n_C$ between $\hat{\I}_{1,C}$ and $\hat{\I}_{2,C}$. For any class corresponding to an inessential intersection point, $n_C=0$.

\item Denote by $\phi_1$ and $\phi'_1$ (respectively $\phi_2$ and $\phi'_2$)  the two leaves met by $\hat{\I}_x$ (respectively $\hat{\I}_y$) mentioned in the statement of the lemma. We choose them in such a way that $\hat{\I}_x$ meets $\phi_1$ first and $\hat{\I}_y$ meets $\phi_2$ first. Denote by $\phi$ the leaf going though the lift $\hat{z}$ of $\tilde{z}$ which belongs to $\hat{\I}_x \cap \hat{\I}_{y}$. The definition of essential intersection point guarantees that, if the leaf $\phi_1$ is above $\phi_2$ with respect to $\phi$, then $\phi'_2$ is above $\phi'_1$ with respect to $\phi$: otherwise $\phi_1$ and $\phi'_1$ would be in the same connected component of the complement of $\hat{\I}_{y}$, a contradiction with the definition of an essential intersection point. In the same way, if the leaf $\phi_1$ is below $\phi_2$ with respect to $\phi$, then $\phi'_2$ is below $\phi'_1$ with respect to $\phi$. This implies that we have an $\tilde{\F}$-transverse intersection.
\end{enumerate}
\end{proof}

\subsection{Case 1}

In this subsection, we suppose that both trajectories $\mathcal{I}_{\tilde{\mathcal{F}}}^{\mathbb{Z}}(\tilde{x}_1)$ and $\mathcal{I}_{\tilde{\mathcal{F}}}^{\mathbb{Z}}(\tilde{x}_2)$ satisfy condition
\begin{enumerate}
\item[\ref{C1}] there exists a deck transformation $\tau_i \in \pi_1(S) \setminus \langle T_i\rangle $ and a leaf $\phi$ of $\tilde{\F}$ such that $\phi$ meets the trajectories $\mathcal{I}^{\Z}_{\tilde{\F}}(\tilde{x}_i)$ and $\tau_i \mathcal{I}^{\Z}_{\tilde{\F}}(\tilde{x}_i)$.
\end{enumerate}
For $i=1, 2$, recall that, by Lemma \ref{LemBoundeddistance}, the trajectory $\mathcal{I}_{\tilde{\mathcal{F}}}^{\mathbb{Z}}(\tilde{x}_i)$ stays at distance strictly less than $R_{i}$ from the geodesic $\tilde{\gamma}_i$. For any subset $A$ of $\Hy^2$ and any real number $R>0$, we let
\begin{equation}\label{EqDesNbh}
A_{R}=\left\{ \tilde{x} \in \Hy^2 \ | \ d(\tilde{x},A) < R \right\}.
\end{equation}

For notational convenience, we will identify the indices $i=1,2$ with an element of $\Z/2$.

The heart of the proof is to find suitable leaves of the transverse foliation in $\tilde S$ so that some orbits realising the rotation vectors in resp. directions $T_1$ and $T_2$ have an $\F$-transverse intersection (Paragraph~\ref{SubSecLeaves1}). Once finished this preparatory step, the two following paragraphs --- still quite technical -- are rather straightforward.

\subsubsection{Leaves and trajectories}\label{SubSecLeaves1}

In this section, we state some preliminary results on the possible behaviours of the leaves. 

Take a point $\tilde{x} \in \Hy^2$ and suppose that the $\tilde{\F}$-transverse trajectory $\I^{\Z}_{\tilde{\F}}(\tilde{x})$ has an $\omega$-limit set in $\overline{\Hy^2}$ which is reduced to a point $\gamma_+$ of $\partial \Hy^2$ and has an $\alpha$-limit set in $\overline{\Hy^2}$ which is reduced to a point $\gamma_-$ of $\partial \Hy^2$ which is different from $\gamma_{+}$. 

Denote by $S_L$ (respectively $S_{R}$) the segment of $\partial \Hy^2$ which is on the left (resp. on the right) of the geodesic joining $\gamma_{-}$ to $\gamma_{+}$. Denote by $L(\I^{\Z}_{\tilde{\F}}(\tilde{x}))$ (resp. $R(\I^{\Z}_{\tilde{\F}}(\tilde{x}))$) the unbounded connected component of $\Hy^{2} \setminus \I_{\tilde{\F}}^{\Z}(\tilde{x})$ whose trace on $\partial{\Hy^2}$ coincides with $S_L$ (resp. $S_R$).

\begin{lemma} \label{LemBoundedend}
Let $\phi$ be a leaf of $\tilde{\F}$. Suppose that the leaf $\phi$ meets the trajectory $\I_{\tilde{\F}}^{\Z}(\tilde{x})$ in two points $\tilde{z}_1$ and $\tilde{z}_2$. Then one unbounded component of $\phi \setminus \left\{ \tilde{z}_1,\tilde{z}_2 \right\}$ is contained in a disk bounded by a closed piece of $\I^{\Z}_{\tilde{\F}}(\tilde{x})$.
\end{lemma}

\begin{proof}
Let $S_I$ be the segment of trajectory between $\tilde{z}_1$ and $\tilde{z}_2$ and $S_\phi$ be the segment between $\tilde{z}_1$ and $\tilde{z}_2$ on the leaf $\phi$ (see Figure~\ref{FigLemBoundedend}). Denote by $Ext$ the unbounded component of the complement in $\Hy^2$ of $S_I \cup S_\phi$ and $Int$ be the complement in $\Hy^2$ of $Ext$. Observe that one of the two connected components of $\phi \setminus S_\phi$, which we call $\psi$, is contained in $Int$. Indeed, otherwise both connected components of $\phi \setminus S_\phi$ would be contained in $Ext$, and we would find two simple paths $\alpha$ and $\alpha'$ contained in $Ext$ and such that $\alpha\cup\alpha'\cup \phi$ separates $\Hy^2$ in two connected components, each one intersecting $\I^{\Z}_{\tilde{\F}}(\tilde{x})$. This would lead to a contradiction, as $S_I$ cannot cross $\alpha$ nor $\alpha'$, and has to cross $\phi$ positively twice.

Moreover, one of the connected components of $\I_{\tilde{\F}}^{\Z}(\tilde{x}) \setminus S_I$ meets $Int$ in a neighbourhood of its end $\tilde{z}_i$, with $i=1$ or $i=2$. Take the closest point $\tilde{z}'_i$, for the order on the trajectory, to $\tilde{z}_i$ on this connected component which meets  $S_{I}$ (such a point exists as the $\alpha$-limit and $\omega$-limit sets of the trajectory lie on $\partial \Hy^2$ so that this connected component cannot remain in $Int$).  Let $S'_{I}$ be the segment joining $\tilde{z}_i$ to $\tilde{z}'_i$ on the trajectory. Then the half-leaf $\psi$ does not meet the unbounded component of $S_{I} \cup S'_I$.
\end{proof}

\begin{figure}[ht]
\begin{minipage}{.47\linewidth}
\begin{center}
\includegraphics[scale=1]{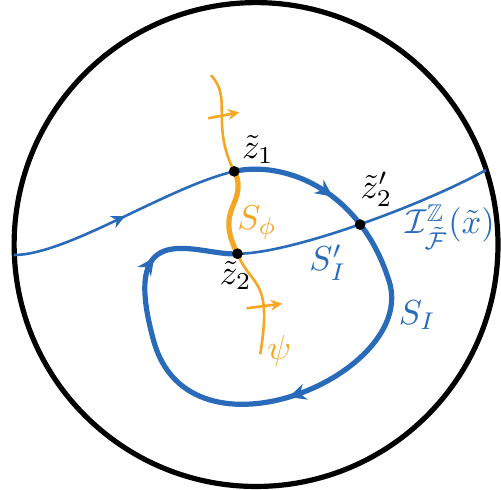}
\caption{\label{FigLemBoundedend}Configuration of the proof of Lemma~\ref{LemBoundedend}.}
\end{center}
\end{minipage}\hfill
\begin{minipage}{.47\linewidth}
\begin{center}
\includegraphics[scale=1]{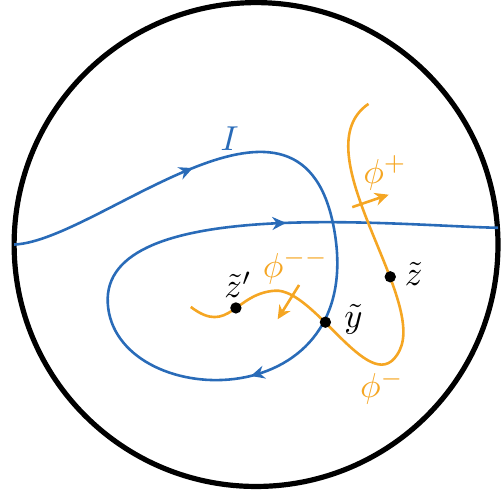}
\caption{\label{FigEndleaves}A possible configuration of the proof of Lemma~\ref{LemEndleaves}.}
\end{center}
\end{minipage}
\end{figure}

For any leaf $\phi$ of $\tilde{\F}$, we call \emph{neighbourhood of $+\infty$} (respectively $-\infty$) in $\phi$ any half-leaf contained in $\phi$ which contains all the points after (resp. before) some point of $\phi$ for the order relation induced by the orientation of $\phi$.

\begin{lemma} \label{LemEndleaves}
Let $\phi$ be a leaf of $\tilde{\F}$ which contains some point $\tilde{z}$ of $R(\I^{\Z}_{\tilde{\F}}(\tilde{x}))$. Let $\phi_{+}$ be the connected component of $\phi \setminus \left\{ \tilde{z} \right\}$ which contains the points after $\tilde{z}$ on $\phi$ and let $\phi_-$ be the other connected component of $\phi \setminus \left\{ \tilde{z} \right\}$.
\begin{enumerate}
\item If the half-leaf $\phi_+$ meets $\I_{\tilde{\F}}^{\Z}(\tilde{x})$, then  either $\phi_{+}$ is bounded in $\Hy^2$  or a neighbourhood of $+\infty$ in $\phi$ is contained in $L(\I^{\Z}_{\tilde{\F}}(\tilde{x}))$. Moreover, the intersection $\overline{\phi_+}\cap \overline{R(\I^{\Z}_{\tilde{\F}}(\tilde{x}))}$ is a segment of $\phi_+$.
\item If the half-leaf $\phi_-$ meets $\I_{\tilde{\F}}^{\Z}(\tilde{x})$, then the $\alpha$-limit set  of $\phi$ in $\overline{\Hy^2}$ does not meet $\partial \Hy^2$. Moreover, $\phi_-$ does not meet $L(\I^{\Z}_{\tilde{\F}}(\tilde{x}))$, and the intersection $\overline{\phi_-}\cap \overline{R(\I^{\Z}_{\tilde{\F}}(\tilde{x}))}$ is a segment of $\phi_-$.
\end{enumerate}
\end{lemma}

 Of course, we have a symmetric statement for a leaf which contains a point of $L(\I^{\Z}_{\tilde{\F}}(\tilde{x}))$ by exchanging $R$ with $L$, $\alpha$ with $\omega$ and $+$ with $-$.

\begin{proof}
A possible configuration for this proof is depicted in Figure~\ref{FigEndleaves}.

If the half-leaf $\phi_+$ has at two intersection points with the trajectory $\I_{\tilde{\F}}^{\Z}(\tilde{x})$, then, by Lemma \ref{LemBoundedend} and as $\tilde{z} \in R(\I_{\tilde{\F}}^{\Z}(\tilde{x}))$, the $\omega$-limit set of $\phi$ is bounded and the first point holds. If the leaf $\phi_{+}$ has exactly one intersection point $\tilde{y}$ with the trajectory $\I_{\tilde{\F}}^{\Z}(\tilde{x})$ then the unbounded component of $\phi_{+} \setminus \left\{ \tilde{y} \right\}$ has to be contained in a connected component of $\Hy^2 \setminus \I_{\tilde{\F}}^{\Z}(\tilde{x})$ which is different from $R(\I^{\Z}_{\tilde{\F}}(\tilde{x}))$: either it is contained in $L(\I^{\Z}_{\tilde{\F}}(\tilde{x}))$, or it stays in a bounded connected component of $\Hy^2\setminus \I^{\Z}_{\tilde{\F}}(\tilde{x})$. This proves the first point.

If the half-leaf $\phi_{-}$ has two intersection points with the trajectory $\I_{\tilde{\F}}^{\Z}(\tilde{x})$, then Lemma~\ref{LemBoundedend} implies the second point. Suppose that there is exactly one intersection point $\tilde{y}$ between the half-leaf $\phi_{-}$ and the trajectory $\I_{\tilde{\F}}^{\Z}(\tilde{x})$. Let $\phi_{--}$ be the unbounded component of $\phi_- \setminus \left\{ \tilde{y} \right\}$ and take a point $\tilde{z}'$ in $\phi_{--}$. The half-leaf $\phi_{--}$ cannot be contained in $R(\I_{\tilde{\F}}^{\Z}(\tilde{x}))$. Let us call $\phi_{\tilde{z}'\tilde{z}}$ the segment of $\phi$ between the points $\tilde{z}'$ and $\tilde{z}$. If $\phi_{--}$ was contained in $L(\I_{\tilde{\F}}^{\Z}(\tilde{x}))$, then the algebraic intersection number, relative to endpoints, of the trajectory $\I_{\tilde{\F}}^{\Z}(\tilde{x})$ with the segment $\phi_{\tilde{z}'\tilde{z}}$ would be equal to $-1$, which is not possible as  the trajectory $\I_{\tilde{\F}}^{\Z}(\tilde{x})$ is positively $\tilde{\F}$-transverse.
\end{proof}

The two previous lemmas did not use condition \ref{C1}. The next one is the first one which is specific to case 1.
	
\begin{lemma} \label{LemDisjointleavescase1}
Fix $i \in \Z /2$. For any neighbourhoods $U_{i,-}, U_{i,+}$ of resp. $\gamma_{i,-}$ and $\gamma_{i,+}$ in $\overline{\Hy^2}$, there exist leaves $\phi_i$ and $\phi'_i$ of $\tilde{\F}$ that meet the trajectory $\mathcal{I}^{\Z}_{\tilde{\F}}(\tilde{x}_i)$ with the following properties.
\begin{enumerate}
\item For any $n \geq 0$, we have $T_{i}^{-n}\phi_i\in U_{i,-}$ and $T_{i}^{n}\phi'_i\in U_{i,+}$.
\item The half-trajectory $(-\infty,\phi_i]_i$ belongs to $U_{i,-}$  and the half-trajectory $[\phi'_i,+\infty)_i$ belongs to $U_{i,+}$.
\end{enumerate}
\end{lemma}

We fix $i \in \Z/2$ and orient $\partial \overline{\Hy^2}$ in such a way that $\gamma_{i+1,-}$ lies in the positively oriented segment of $\partial \Hy^2$ which joins $\gamma_{i,-}$ to $\gamma_{i,+}$. For any points $a$, $b$ on $\partial \Hy^2$, we denote by $[a,b]_{\partial \Hy^2}$ the positively oriented segment of $\partial \Hy^2$ from $a$ to $b$.

\begin{proof}
We will distinguish two cases depending on whether the closed geodesic $\gamma_i$ is simple or not.

As a sidenote, observe that condition \ref{C1} is automatically satisfied by $\mathcal{I}^{\Z}_{\tilde{\F}}(\tilde{x}_i)$ if the geodesic $\gamma_i$ has an autointersection.

\paragraph{First case:} Suppose first that the closed geodesic $\gamma_i$ is not simple. Then there exists a deck transformation $\tau_{i} \in \pi_1(S) \setminus \langle T_i\rangle $ such that the geodesic lines $\tau_i\tilde{\gamma}_i$ and $\tilde{\gamma}_i$ meet. Observe that, for any $n \in \Z$, the geodesic lines $T_i^{n} \tau_i \tilde{\gamma}_i$ and $T_{i}^n \tilde{\gamma}_i=\tilde{\gamma}_i$ also meet. Fix $n \in \Z$. Recall that, by Lemma \ref{LemRecurrentpoint}, the orbit of $x_i$ is periodic so that $\I^{\Z}_{\tilde{\F}}(\tilde{x}_i)$ projects to a closed curve on $S$. Denote by $p$  the minimal positive number such that $T_i^p \mathcal{I}^{\Z}_{\tilde{\F}}(\tilde{x}_i) = \mathcal{I}^{\Z}_{\tilde{\F}}(\tilde{x}_i)$.  By Proposition~\ref{GeomImpliqFeuill}, the trajectories $\mathcal{I}^{\Z}_{\tilde{\F}}(\tilde{x}_i)$ and $T_{i}^{pn} \tau_i \mathcal{I}^{\Z}_{\tilde{\F}}(\tilde{x}_i)$ have an $\tilde{\F}$-transverse intersection.

In what follows, we construct a leaf $\phi_i$ which satisfies the conclusion of the lemma. The construction of $\phi'_i$ is similar and left to the reader.

\underline{Subcase 1:} We suppose that there exists $\tilde y$ in $\mathcal{I}^{\Z}_{\tilde{\F}}(\tilde{x}_i)$ such that $\phi_{\tilde y}$ has not $\gamma_{i,+}$ in its $\omega$-limit.  Suppose the point $\tau_{i}\gamma_{i,-}$ is on the right of the geodesic $\gamma_i$. If the $\alpha$-limit of the leaf $\phi_{\tilde{y}}$ does not contain  $\gamma_{i,+}$ either, for $k$ large enough, the leaf $T_i^{-pk}\phi_{\tilde y}$ can be used as the leaf $\phi_i$ of the lemma's conclusion. Suppose the $\alpha$-limit of the leaf $\phi_{\tilde{y}}$ contains the point $\gamma_{i,+}$.  Then there exists $J>0$ such that, for any $j \geq J$, $\tilde y \in L(T_i^{pj}\tau_i\mathcal{I}^{\Z}_{\tilde{\F}}(\tilde{x}_i))$ and the negative half-leaf in $\phi_{\tilde{y}}$ meets the trajectory $T_i^{pj}\tau_i\mathcal{I}^{\Z}_{\tilde{\F}}(\tilde{x}_i)$. By Lemma~\ref{LemEndleaves}.1 (symmetric version), the $\alpha$-limit of $\phi_{\tilde y}$ is contained in
$$\bigcap_{k \geq 0} \overline{R(T_i^{pk}\tau_i\mathcal{I}^{\Z}_{\tilde{\F}}(\tilde{x}_i))}=\left\{ \gamma_{i,+} \right\}.$$ 
Moreover, by Lemma \ref{LemEndleaves}.1, the $\omega$-limit of $\phi_{\tilde{y}}$ is contained in $\overline{L(T_i^{pJ}\tau_i\mathcal{I}^{\Z}_{\tilde{\F}}(\tilde{x}_i))}$. Hence no end of the leaf $\phi_{\tilde{y}}$ meets $ T_i^{pJ+p}\tau_i \gamma_{i,+}$ nor $ T_i^{pJ+p}\tau_i \gamma_{i,-}$. Then the leaf $\tau_{i}^{-1} T_{i}^{-pJ-p} \phi_{\tilde{y}}$ meets the trajectory $\mathcal{I}^{\Z}_{\tilde{\F}}(\tilde{x}_i)$ and does not meet $\gamma_{i,+}$ nor  $\gamma_{i,-}$. Hence, for $k$ large enough, the leaf $\phi_i=T_{i}^{-pk}\tau_{i}^{-1} T_{i}^{-pJ-p} \phi_{\tilde{y}}$ will satisfy the lemma.

Suppose now that the point $\tau_{i}\gamma_{i,-}$ is on the left of the geodesic $\gamma_{i}$. Then there exists $j>0$ such that $\tilde y \in R(T_i^{pj}\tau_i\mathcal{I}^{\Z}_{\tilde{\F}}(\tilde{x}_i))$. By Lemma~\ref{LemEndleaves}.2, the $\alpha$-limit of $\phi_{\tilde y}$ is contained in $\Hy^2\cup \overline{R(T_i^{pj}\tau_i\mathcal{I}^{\Z}_{\tilde{\F}}(\tilde{x}_i))}$ and for $k$ large enough, the leaf $T_i^{-pk}\phi_{\tilde y}$ can be used as the leaf $\phi_i$ of the lemma's conclusion.
\bigskip

\underline{Subcase 2:} Suppose that for any $\tilde y$ in $\mathcal{I}^{\Z}_{\tilde{\F}}(\tilde{x}_i)$, the leaf $\phi_{\tilde y}$ has $\gamma_{i,+}$ in its $\omega$-limit. Fix $\tilde y$ in $\mathcal{I}^{\Z}_{\tilde{\F}}(\tilde{x}_i)$.
Suppose that the point $\tau_{i} \gamma_{i,-}$ is on the right of the geodesic $\gamma_i$. Then there exists $J>0$ such that for any $j\ge J$, one has $\tilde y \in L(T_i^{pj}\tau_i\mathcal{I}^{\Z}_{\tilde{\F}}(\tilde{x}_i))$ and, by the hypothesis on the $\omega$-limit of $\phi_{\tilde{y}}$, the positive half-leaf in $\phi_{\tilde{y}}$ starting at $\tilde{y}$ has to cross each of those trajectories $T_i^{pj}\tau_i\mathcal{I}^{\Z}_{\tilde{\F}}(\tilde{x}_i)$. But by Lemma ~\ref{LemEndleaves}.2 (symmetric version), this implies that the $\omega$-limit of the leaf $\phi_{\tilde{y}}$ is contained in $\Hy^2$, a contradiction. This case cannot happen.

Suppose now that the point $\tau_{i} \gamma_{i,-}$ is on the left of the geodesic $\gamma_i$. Then there exists $J>0$ such that for any $j\ge J$, one has $\tilde y \in R(T_i^{pj}\tau_i\mathcal{I}^{\Z}_{\tilde{\F}}(\tilde{x}_i))$; this implies that the positive half-leaf starting at $\tilde y$ has to cross $T_i^{pj}\tau_i\mathcal{I}^{\Z}_{\tilde{\F}}(\tilde{x}_i)$. By Lemma~\ref{LemEndleaves}.2, the $\alpha$-limit of $\phi_{\tilde y}$ is contained in $\Hy^2\cup \overline{R(T_i^{pj}\tau_i\mathcal{I}^{\Z}_{\tilde{\F}}(\tilde{x}_i))}$.
Hence, $\phi_{\tilde y}$ crosses $T_i^{pJ+p}\tau_i\mathcal{I}^{\Z}_{\tilde{\F}}(\tilde{x}_i)$ and its ends do not contain $T_i^{pJ+p}\tau_i\gamma_{i,+}$ as, by Lemma \ref{LemEndleaves}.1, the $\omega$-limit of this leaf is equal to
\[\bigcap_{k \geq 0}\overline{L(T_i^{pk}\tau_i\mathcal{I}^{\Z}_{\tilde{\F}}(\tilde{x}_i))}=\left\{ \gamma_{i,+} \right\}.\]
In this case, for $k$ large enough, the leaf $T_i^{-pk}\tau_i^{-1}T_i^{-pJ-p}\phi_{\tilde y}$ can be used as the leaf $\phi_i$ of the lemma's conclusion.

\paragraph{Second case:} Suppose now that the geodesic $\gamma_i$ is simple. Recall that the geodesic $\tilde{\gamma}_i$ shares no endpoint on $\partial \Hy^2$ with one of its translates under an element of $\pi_1(S)$ by Lemma~\ref{commonendpoint}. By condition \ref{C1}, there exists a leaf $\phi$ of $\tilde{\F}$ which meets $\mathcal{I}^{\Z}_{\tilde{\F}}(\tilde{x}_i)$ and $\tau_{i} \mathcal{I}^{\Z}_{\tilde{\F}}(\tilde{x}_i)$. 

For any subset $A$ of $\Hy^2$, we denote by $\overline{A}$ its closure in $\overline{\Hy^2}$. The heart of the proof in this second case is the following lemma.

\begin{lemma} \label{LemGoodleaves}
There exist two leaves $\phi_a$ and $\phi_b$, each one meeting $\mathcal{I}^{\Z}_{\tilde{\F}}(\tilde{x}_i)$, and such that $\overline{\phi_a}$ is disjoint from $\gamma_{i,-}$  and $\overline{\phi_b}$ is disjoint from $\gamma_{i,+}$.
\end{lemma}

\begin{proof}
We need to distinguish two cases, depending on whether $\tilde{\gamma}_i$ separates $\tau_{i}\tilde{\gamma}_i$ and $\tau_{i}^{-1}\tilde{\gamma}_i$ or not.
\bigskip

\underline{First case:} Suppose $\tilde{\gamma}_i$ separates $\tau_{i}\tilde{\gamma}_i$ and $\tau_{i}^{-1}\tilde{\gamma}_i$. Note that this amounts to saying that the axis of $\tau_i$ crosses $\tilde \gamma_i$ (using the fact that $\gamma_i$ is simple).

For notational convenience we suppose in what follows that the geodesic $\tau_{i} \tilde{\gamma}_{i}$ is on the left of $\tilde{\gamma}_i$. 
\bigskip

If $\omega(\phi) \cap \partial \Hy^{2}= \emptyset$ and $\alpha(\phi) \cap \partial \Hy^2=\emptyset$, the lemma holds (for $\phi_a=\phi_b=\phi$).
\bigskip

\begin{figure}
\begin{center}
\includegraphics[scale=1]{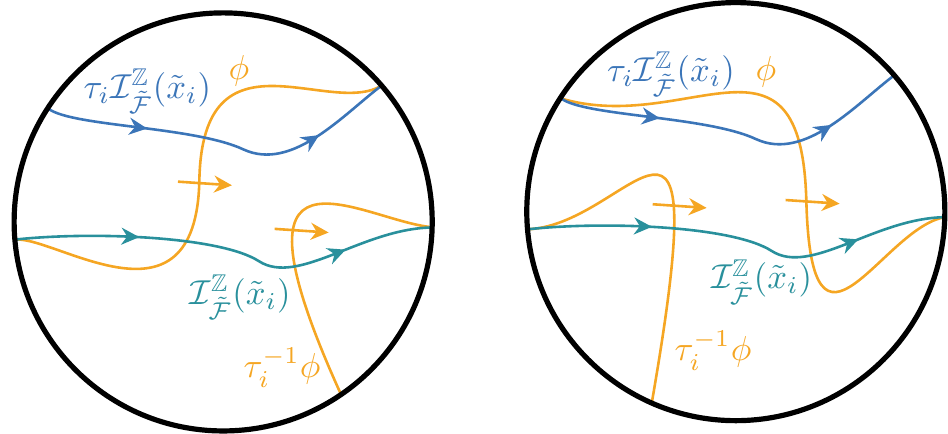}
\caption{Configuration of the first case of the proof of Lemma~\ref{LemGoodleaves}: the two different cases depending whether the trajectory crosses first $\phi$ or $\tau_i\phi$.}\label{FigGoodleaves1}
\end{center}
\end{figure}

Suppose that $\emptyset \neq \omega(\phi) \cap \partial \Hy^2$ is contained in the segment $[\tau_{i} \gamma_{i,+},\tau_{i} \gamma_{i,-}]_{\partial \Hy^2}$ of $\partial \Hy^2$. Then a neighbourhood of $+\infty$ in $\phi$ is contained in $L(\I^{\Z}_{\tilde{\F}}(\tilde{x}_i))$.
If $\alpha(\phi) \subset \Hy^2$, the lemma holds (for $\phi_a=\phi_b=\phi$).
Otherwise, by Lemma~\ref{LemEndleaves}.1. (symmetric version), a neighbourhood of $-\infty$ in $\phi$ is contained in $R(\I^{\Z}_{\tilde{\F}}(\tilde{x}_i))$. Hence the leaf $\tau_{i}^{-1} \phi$ joins the segment $[\tau_{i}^{-1} \gamma_{i,-}, \tau_{i}^{-1} \gamma_{i,+}]_{\partial \Hy^2}$ to the trajectory $\I^{\Z}_{\tilde{\F}}(\tilde{x}_i)$. Observe also that a neighbourhood of $+\infty$ in $\tau_{i}^{-1} \phi$ is disjoint from $R(\I^{\Z}_{\tilde{\F}}(\tilde{x}_i))$ by Lemma \ref{LemBoundedend}. However, either $\phi=\tau_{i}^{-1} \phi$, in which case the claim holds because both ends of $\phi$ are disjoint from the ends of $\gamma_i$ ($\phi_a=\phi_b=\phi$), or the leaves $\phi$ and $\tau_{i}^{-1} \phi$ are disjoint. Suppose the latter holds (see Figure~\ref{FigGoodleaves1}).

If the trajectory $\I^{\Z}_{\tilde{\F}}(\tilde{x}_i)$ meets $\phi$ before it meets $\tau_{i}^{-1} \phi$ (left of Figure~\ref{FigGoodleaves1}), then the set $\tau_{i}^{-1} \overline{\phi} \cap  \overline{R(\I^{\Z}_{\tilde{\F}}(\tilde{x}_i))}$ separates a neighbourhood of $-\infty$ in $\phi$ from $\gamma_{i,+}$ in $\overline{R(\I^{\Z}_{\tilde{\F}}(\tilde{x}_i))}$ and the set $ \overline{\phi} \cap \overline{L(\I^{\Z}_{\tilde{\F}}(\tilde{x}_i))} $ separates a neighbourhood of $+\infty$ in $\tau_{i}^{-1}\phi$ from $\gamma_{i,-}$ in $\overline{L(\I^{\Z}_{\tilde{\F}}(\tilde{x}_i))}$, hence the lemma holds for $\phi_a=\tau_i^{-1}\phi$ and $\phi_b=\phi$. If the trajectory $\I^{\Z}_{\tilde{\F}}(\tilde{x}_i)$ meets $\tau_{i}^{-1}\phi$ before it meets $\phi$, then the set $\tau_{i}^{-1} \overline{\phi} \cap \overline{R(\I^{\Z}_{\tilde{\F}}(\tilde{x}_i))}$ separates a neighbourhood of $-\infty$ in $\phi$ from $\gamma_{i,-}$ in $\overline{R(\I^{\Z}_{\tilde{\F}}(\tilde{x}_i))}$ and the set $ \overline{\phi} \cap \overline{L(\I^{\Z}_{\tilde{\F}}(\tilde{x}_i))}$ separates a neighbourhood of $+\infty$ in $\tau_{i}^{-1}\phi$ from $\gamma_{i,+}$ in $\overline{L(\I^{\Z}_{\tilde{\F}}(\tilde{x}_i))}$, hence the lemma holds for $\phi_a=\phi$ and $\phi_b=\tau_i^{-1}\phi$.
\bigskip

Suppose now that $\emptyset \neq \omega(\phi) \cap \partial \Hy^2$ is not contained in $[\tau_{i} \gamma_{i,+}, \tau_{i} \gamma_{i,-}]_{\partial \Hy^2}$ (such a configuration is depicted in Figure~\ref{FigGoodleaves2}). Then any neighbourhood of $+\infty$ in $\phi$ meets $R(\tau_{i}\I^{\Z}_{\tilde{\F}}(\tilde{x}_i))$. By Lemma \ref{LemEndleaves}.2., $\alpha(\phi)\subset \Hy^2$ and a neighbourhood of $-\infty$ in $\phi$ is disjoint from $L(\tau_{i} \I^{\Z}_{\tilde{\F}}(\tilde{x}_i))$. 

If any neighbourhood of $+\infty$ in $\phi$ also met $L(\tau_i \I^{\Z}_{\tilde{\F}}(\tilde{x}_i))$, then it would have to cross $\tau_i \I^{\Z}_{\tilde{\F}}(\tilde{x}_i)$ and, by Lemma \ref{LemEndleaves}.1., we would have $\omega(\phi) \subset [\tau_{i} \gamma_{i,+}, \tau_{i} \gamma_{i,-}]_{\partial \Hy^2}$, a contradiction. Hence a neighbourhood of $+\infty$ in $\phi$ is disjoint from $L(\tau_{i} \I^{\Z}_{\tilde{\F}}(\tilde{x}_i))$ and $\omega(\phi) \cap (\tau_{i} \gamma_{i,+}, \tau_{i} \gamma_{i,-})_{\partial \Hy^2}=\emptyset$.

If any neighbourhood of $+\infty$ in $\phi$ meets ${R(\I^{\Z}_{\tilde{\F}}(\tilde{x}_i))}$, then by Lemma \ref{LemEndleaves}.1. there exists a neighbourhood of $+\infty$ in $\phi$ included in ${R(\I^{\Z}_{\tilde{\F}}(\tilde{x}_i))}$ and so the lemma is satisfied for $\phi_a=\phi_b=\tau_i^{-1}\phi$, as the set $\tau_{i}^{-1} \overline{\phi}$ meets no ends of the geodesic $\tilde{\gamma}_i$.

Otherwise, a neighbourhood of $+\infty$ in $\phi$ is disjoint from $L(\tau_i \I^{\Z}_{\tilde{\F}}(\tilde{x}_i)) \cup {R(\I^{\Z}_{\tilde{\F}}(\tilde{x}_i))}$. Let us prove that, in this case, $\tau_i \I^{\Z}_{\tilde{\F}}(\tilde{x}_i) \cap \I^{\Z}_{\tilde{\F}}(\tilde{x}_i)\neq \emptyset$. Suppose the contrary. This implies that $\tau_i \I^{\Z}_{\tilde{\F}}(\tilde{x}_i) \subset L\big ( \I^{\Z}_{\tilde{\F}}(\tilde{x}_i)\big)$ and $\I^{\Z}_{\tilde{\F}}(\tilde{x}_i) \subset R\big ( \tau_i \I^{\Z}_{\tilde{\F}}(\tilde{x}_i)\big)$. As $\phi$ meets both trajectories, there exists 
\[\tilde y\in \phi \cap R\big ( \tau_i \I^{\Z}_{\tilde{\F}}(\tilde{x}_i)\big) \cap L\big ( \I^{\Z}_{\tilde{\F}}(\tilde{x}_i)\big).\]
By Lemma~\ref{LemEndleaves}.1, the positive half-leaf $\phi_+$ starting at $\tilde y$ cannot meet $\tau_i \I^{\Z}_{\tilde{\F}}(\tilde{x}_i)$ (otherwise $\omega(\phi)$ would be either bounded, or included in $\overline{L\big ( \tau_i \I^{\Z}_{\tilde{\F}}(\tilde{x}_i)\big)} $). Similarly, by Lemma~\ref{LemEndleaves}.2 (symmetric version), $\phi_+$ cannot meet $\I^{\Z}_{\tilde{\F}}(\tilde{x}_i)$. Hence, the negative half-leaf $\phi_-$ starting at $\tilde y$ meets both $\tau_i \I^{\Z}_{\tilde{\F}}(\tilde{x}_i)$ and $\I^{\Z}_{\tilde{\F}}(\tilde{x}_i)$. 
Using once again Lemma~\ref{LemEndleaves}, we deduce that $\alpha(\phi) \in R\big ( \tau_i \I^{\Z}_{\tilde{\F}}(\tilde{x}_i)\big)^\complement \cap L\big ( \I^{\Z}_{\tilde{\F}}(\tilde{x}_i)\big)^\complement$. The latter set is hence nonempty. This proves that $\tau_i \I^{\Z}_{\tilde{\F}}(\tilde{x}_i) \cap \I^{\Z}_{\tilde{\F}}(\tilde{x}_i)\neq \emptyset$. 

In this case, the intersection $\overline{R(\tau_i\I^{\Z}_{\tilde{\F}}(\tilde{x_i}))}\cap \overline{L(\I^{\Z}_{\tilde{\F}}(\tilde{x_i}))}$ has two unbounded connected components, with respective boundaries in $\partial\Hy^2$ $(\tau_i\gamma_{i,-},\gamma_{i,-})_{\partial\Hy^2}$ and $(\gamma_{i,+}, \tau_i\gamma_{i,+})_{\partial\Hy^2}$ (see Figure~\ref{FigGoodleaves2}). 

Consider the second of these connected components; its boundary in $\Hy^2$ contains pieces of both $\tau_i\I^{\Z}_{\tilde{\F}}(\tilde{x_i})$ and $\I^{\Z}_{\tilde{\F}}(\tilde{x_i})$. Let $\tilde z \in \tau_i\I^{\Z}_{\tilde{\F}}(\tilde{x_i})\cap \I^{\Z}_{\tilde{\F}}(\tilde{x_i})$ on this boundary, and $\phi'$ be the leaf passing by $\tilde z$. Note that $\phi'$ meets both $\overline{R(\tau_i\I^{\Z}_{\tilde{\F}}(\tilde{x_i}))}$ and $\overline{L(\I^{\Z}_{\tilde{\F}}(\tilde{x_i}))}$ in any neighbourhood of $\tilde z$. Denote by $\phi'_+$ the positive half-leaf of $\phi'$ starting at $\tilde z$. 
\begin{itemize}
\item If $\phi'_+$ is not included in $R(\tau_i\I^{\Z}_{\tilde{\F}}(\tilde{x_i}))$, then by Lemma~\ref{LemEndleaves}.1 either $\omega(\phi')\subset\Hy^2$, or $\omega(\phi')\subset L(\tau_i\I^{\Z}_{\tilde{\F}}(\tilde{x_i}))$ (right of Figure~\ref{FigGoodleaves2}). In the latter case, we already proved that the conclusion of the lemma holds.
\item If $\phi'_+$ is included in $R(\tau_i\I^{\Z}_{\tilde{\F}}(\tilde{x_i}))$, but $\phi'_+$ is not included in $L(\I^{\Z}_{\tilde{\F}}(\tilde{x_i}))$, then by Lemma~\ref{LemEndleaves}.2 (symmetric version) $\omega(\phi')\subset\Hy^2$.
\item Otherwise, $\omega(\phi')\subset\Hy^2\cup [\gamma_{i,+}, \tau_i\gamma_{i,+}]_{\partial\Hy^2}$ (left of Figure~\ref{FigGoodleaves2}).
\end{itemize}
Remark also that in all cases, by Lemma~\ref{LemEndleaves}, we have $\alpha(\phi')\cap \partial\Hy^2 \subset \overline{R(\I^{\Z}_{\tilde{\F}}(\tilde{x_i}))}$.

\begin{figure}
\begin{center}
\includegraphics[scale=1]{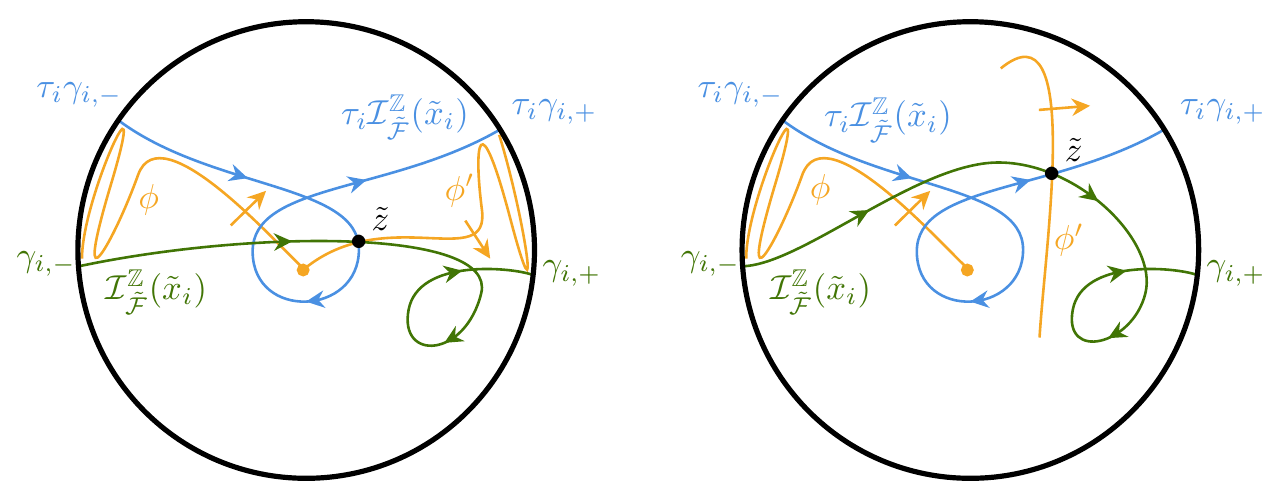}
\caption{Configuration of the first case of the proof of Lemma~\ref{LemGoodleaves}: finding another leaf $\phi'$.}\label{FigGoodleaves2}
\end{center}
\end{figure}

If $\omega(\phi') \subset \Hy^2$, then $\alpha(\phi') \cap \partial \Hy^2 \subset \overline{R(\I^{\Z}_{\tilde{\F}}(\tilde{x}_i))}$, and hence $\alpha(\tau_{i}^{-1} \phi')\cap\partial\Hy^2 \subset [\tau_{i}^{-1} \gamma_{i,-},\tau_{i}^{-1} \gamma_{i,+}]_{\partial \Hy^2}$ and $\omega(\tau_{i}^{-1} \phi') \subset \Hy^2$ so that the lemma holds for $\phi_a=\phi_b = \tau_i^{-1}\phi$.

In the other case, we have $\omega(\phi')\subset\Hy^2\cup [\gamma_{i,+}, \tau_i\gamma_{i,+}]_{\partial\Hy^2}$, and by Lemma \ref{LemEndleaves}.2., $\alpha(\phi')\subset \Hy^2$. We can perform the same construction for the other connected component of $\overline{R(\tau_i\I^{\Z}_{\tilde{\F}}(\tilde{x_i}))}\cap \overline{L(\I^{\Z}_{\tilde{\F}}(\tilde{x_i}))}$ to find another leaf $\phi''$. Again, the only case in which we still have not proved the lemma is when $\omega(\phi'')\subset\Hy^2\cup [\tau_i\gamma_{i,-}, \gamma_{i,-}]_{\partial\Hy^2}$. But in this case $\phi_a=\phi'$ and $\phi_b = \phi''$ make the lemma work.
\bigskip

The case $\omega(\phi) \subset \Hy^2$ and $\alpha(\phi) \cap \partial \Hy^2 \neq \emptyset$ is identical to the previous one, the details are left to the reader.

\bigskip

\underline{Second case:} Suppose that the geodesic line $\tilde{\gamma}_i$ does not separate $\tau_{i} \tilde{\gamma}_i$ and $\tau_{i}^{-1} \tilde{\gamma}_i$. Note that this amounts to saying that the axis of $\tau_i$ is disjoint from $\tilde \gamma_i$ (using the fact that $\gamma_i$ is simple). In this case, we need the following claim.

\begin{claim} \label{ClaimBoundedend}
One end of $\phi$ is contained in singularities of $\tilde{\F}$, \emph{i.e.} it is bounded in $\tilde{S}=\Hy^2$.
\end{claim}

\begin{proof}
Suppose that both ends of $\phi$ meet $\partial \Hy^2$.

Then the leaf $\phi$ meets the trajectories $\mathcal{I}^{\Z}_{\tilde{\F}}(\tilde{x}_i)$ and $\tau_{i} \mathcal{I}^{\Z}_{\tilde{\F}}(\tilde{x}_i)$ at only one point by Lemma \ref{LemBoundedend}. Observe also that, at the point of intersection $\phi \cap \mathcal{I}^{\Z}_{\tilde{\F}}(\tilde{x}_i)$, the leaf $\phi$ must go from one unbounded component of $\Hy^2 \setminus \mathcal{I}^{\Z}_{\tilde{\F}}(\tilde{x}_i) $ to the other one and the same holds at the point $\phi \cap  \tau_i \mathcal{I}^{\Z}_{\tilde{\F}}(\tilde{x}_i)$ for the trajectory $\tau_i \mathcal{I}^{\Z}_{\tilde{\F}}(\tilde{x}_i)$.

As the trajectories $\mathcal{I}^{\Z}_{\tilde{\F}}(\tilde{x}_i)$ and $\tau_{i} \mathcal{I}^{\Z}_{\tilde{\F}}(\tilde{x}_i)$ are both $\tilde{\F}$-transverse, the algebraic intersection number between  each of these trajectories and $\phi$ must be equal to $1$. However, as the geodesic line $\tilde{\gamma}_i$ does not separate $\tau_{i} \tilde{\gamma}_i$ and $\tau_{i}^{-1} \tilde{\gamma}_i$, the algebraic intersection number between $\tilde{\gamma}_{i}$ and $\phi$ and the algebraic intersection number between $\tau_{i} \tilde{\gamma}_i$ and $\phi$ must be opposite to each other: the axis of $\tau_{i}$ does not cross $\tilde{\gamma}_i$. Hence the algebraic intersection number between $\mathcal{I}^{\Z}_{\tilde{\F}}(\tilde{x}_i)$ and $\phi$ and the algebraic intersection number between $\tau_i \mathcal{I}^{\Z}_{\tilde{\F}}(\tilde{x}_i)$ and $\phi$ must be opposite to each other, a contradiction. 
\end{proof}

\begin{figure}
\begin{center}
\includegraphics[width=\linewidth]{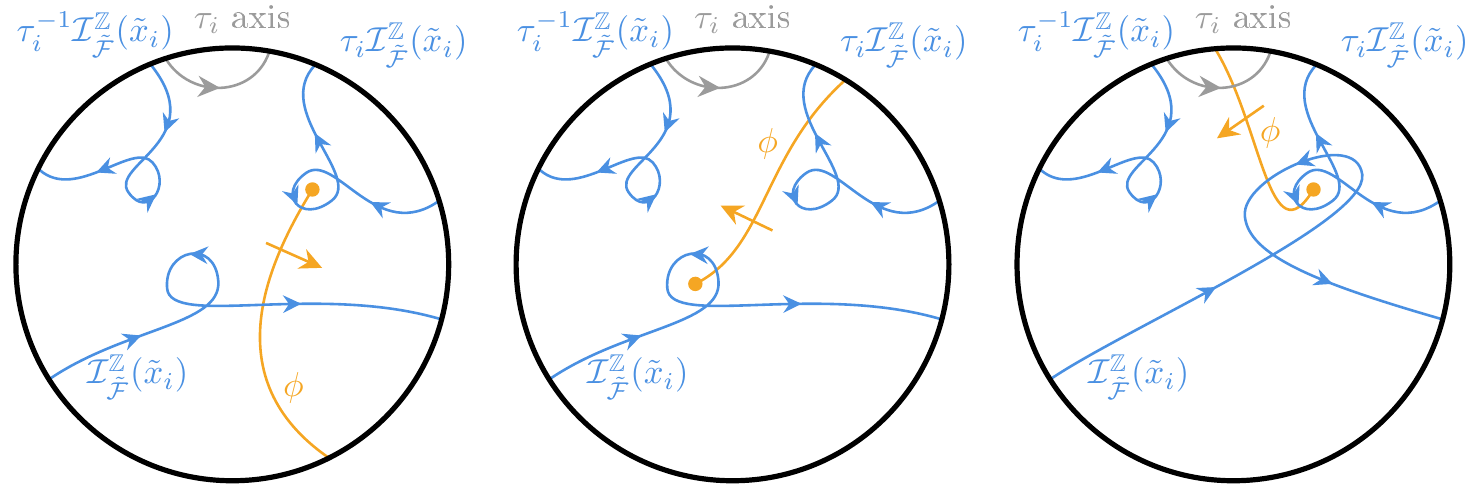}
\caption{Configuration of the second case of the proof of Lemma~\ref{LemGoodleaves} (where the geodesics $\tau_i\tilde\gamma_i$ and $\tau_i^{-1}\tilde\gamma_i$ are on the same side of the geodesic $\tilde\gamma_i$): the three different cases having to be considered.}\label{FigDisjointleavescase1}
\end{center}
\end{figure}

To simplify notation, we suppose that the geodesic $\tau_{i} \tilde{\gamma}_i$ is on the left of $\tilde{\gamma}_i$ and that the geodesic $\tau_{i} \tilde{\gamma}_i$ is above $\tau_{i}^{-1} \tilde{\gamma}_i$ with respect to $\tilde{\gamma}_i$. If $\alpha(\phi) \subset \Hy^{2}$ and $\omega(\phi) \subset \Hy^2$, the lemma holds. Otherwise, by Claim \ref{ClaimBoundedend}, either $\alpha(\phi) \cap \partial \Hy^2 \neq \emptyset$ and $\omega(\phi) \subset \Hy^2$ or $\omega(\phi) \cap \partial \Hy^2 \neq \emptyset$ and $\alpha(\phi) \subset \Hy^2$.

Suppose first that $\alpha(\phi) \cap \partial \Hy^2 \neq \emptyset$ and $\omega(\phi) \subset \Hy^2$. By Lemma \ref{LemEndleaves}.1. (symmetric version), either a neighbourhood at $-\infty$ of $\phi$ is contained in $L(\mathcal{I}^{\Z}_{\tilde{\F}}(\tilde{x}_i))$ or $\alpha(\phi) \cap \partial \Hy^{2} \subset \overline{R(\mathcal{I}^{\Z}_{\tilde{\F}}(\tilde{x}_i))}$. If $\alpha(\phi) \cap \partial \Hy^{2} \subset \overline{R(\mathcal{I}^{\Z}_{\tilde{\F}}(\tilde{x}_i))}$, then the set $\overline{\tau_{i}^{-1} \phi}$ meets none of the ends of the geodesic $\tilde{\gamma}_i$ so that the lemma holds for $\phi_a=\phi_b=\tau_i^{-1}\phi$ (see Figure~\ref{FigDisjointleavescase1}, left).  

In the other case, a neighbourhood at $-\infty$ of $\phi$ is contained in $L(\mathcal{I}^{\Z}_{\tilde{\F}}(\tilde{x}_i))$.
As before, by Lemma \ref{LemEndleaves}.1. (symmetric version), either a neighbourhood at $-\infty$ of $\phi$ is contained in $L(\tau_{i}\mathcal{I}^{\Z}_{\tilde{\F}}(\tilde{x}_i))$, or $\alpha(\phi) \cap \partial \Hy^{2} \subset \overline{R(\tau_{i}\mathcal{I}^{\Z}_{\tilde{\F}}(\tilde{x}_i))}$. In the latter case the lemma holds for $\phi_a=\phi_b=\phi$ (see Figure~\ref{FigDisjointleavescase1}, middle).

Suppose now that a neighbourhood at $-\infty$ of $\phi$ is contained in $L(\mathcal{I}^{\Z}_{\tilde{\F}}(\tilde{x}_i)) \cap  L(\tau_{i}\mathcal{I}^{\Z}_{\tilde{\F}}(\tilde{x}_i))$. Then both trajectories $\mathcal{I}^{\Z}_{\tilde{\F}}(\tilde{x}_i)$ and $\tau_i \mathcal{I}^{\Z}_{\tilde{\F}}(\tilde{x}_i)$ meet.
Indeed, by Lemma~\ref{LemEndleaves}.2. (symmetric version), $\omega(\phi)$ has to be contained in some bounded component of the complement of $\mathcal{I}^{\Z}_{\tilde{\F}}(\tilde{x}_i)$, and in some bounded component of the complement of $\tau_i\mathcal{I}^{\Z}_{\tilde{\F}}(\tilde{x}_i)$. Hence, one of these trajectories meets one bounded component of the complement of the other one, and these two trajectories meet. In particular, the set $\overline{L\big(\mathcal{I}^{\Z}_{\tilde{\F}}(\tilde{x}_i)} \big) \cap \overline{L\big(\tau_i\mathcal{I}^{\Z}_{\tilde{\F}}(\tilde{x}_i)\big)}$ has two unbounded connected components, one intersecting $\partial\Hy^2$ on $ [\tau_{i} \gamma_{i,+}, \gamma_{i,-}]_{\partial \Hy^2}$, the other one on $[ \gamma_{i,+}, \tau_i \gamma_{i,-}]_{\partial \Hy^2}$. 
Thus, either $\alpha(\phi)\cap\partial\Hy^2 \subset [\tau_{i} \gamma_{i,+}, \gamma_{i,-}]_{\partial \Hy^2}$ or $\alpha(\phi) \subset [ \gamma_{i,+}, \tau_i \gamma_{i,-}]_{\partial \Hy^2}$. In the first case, the set $\overline{\phi}$ does not meet $\gamma_{i,+}$ and the set $\tau_{i}^{-1} \overline{\phi}$ does not meet $\gamma_{i,-}$ and the lemma holds for $\phi_a=\tau_i^{-1}\phi$ and $\phi_b=\phi$ (see Figure~\ref{FigDisjointleavescase1}, right). In the second case, the set $\overline{\phi}$ does not meet $\gamma_{i,-}$ and the set $\tau_{i}^{-1} \overline{\phi}$ does not meet $\gamma_{i,+}$ and the lemma holds for $\phi_a=\phi$ and $\phi_b=\tau_i^{-1}\phi$ .

Finally, suppose that $\omega(\phi) \cap \partial \Hy^2 \neq \emptyset$ and $\alpha(\phi) \subset \Hy^{2}$. This case is similar to the previous one so we will give less details. By Lemma \ref{LemEndleaves}.1., either a neighbourhood at $+\infty$ of $\phi$ is contained in $R(\tau_{i}\mathcal{I}^{\Z}_{\tilde{\F}}(\tilde{x}_i))$ or $\omega(\phi) \cap \partial \Hy^{2} \subset \overline{L(\tau_{i}\mathcal{I}^{\Z}_{\tilde{\F}}(\tilde{x}_i))}$ and either a neighbourhood at $+\infty$ of $\phi$ is contained in $R(\mathcal{I}^{\Z}_{\tilde{\F}}(\tilde{x}_i))$ or $\omega(\phi) \cap \partial \Hy^{2} \subset \overline{L(\mathcal{I}^{\Z}_{\tilde{\F}}(\tilde{x}_i))}$. If a neighbourhood at $+\infty$ of $\phi$ is contained in either $R(\tau_{i}\mathcal{I}^{\Z}_{\tilde{\F}}(\tilde{x}_i))$ or $R(\mathcal{I}^{\Z}_{\tilde{\F}}(\tilde{x}_i))$, the lemma holds: take $\phi_a=\phi_b=\phi$ in the first case and $\phi_a=\phi_b=\tau_i^{-1} \phi$ in the second one. Otherwise, $\omega(\phi) \subset \overline{L(\mathcal{I}^{\Z}_{\tilde{\F}}(\tilde{x}_i)) \cap L(\tau_i \mathcal{I}^{\Z}_{\tilde{\F}}(\tilde{x}_i))}$. In this case, by the last part of Lemma \ref{LemEndleaves}.1. (symmetric version), $\alpha(\phi)$ is contained in $ L(\mathcal{I}^{\Z}_{\tilde{\F}}(\tilde{x}_i))^{c} \cap L(\tau_i \mathcal{I}^{\Z}_{\tilde{\F}}(\tilde{x}_i))^{c}$ so that the trajectories $\mathcal{I}^{\Z}_{\tilde{\F}}(\tilde{x}_i)$ and $\tau_i \mathcal{I}^{\Z}_{\tilde{\F}}(\tilde{x}_i)$ meet. Hence, as in the previous case, either $\omega(\phi) \cap \partial \Hy^2 \subset [\tau_{i}\gamma_{i,+}, \gamma_{i,-}]_{\partial \Hy^2}$ or $\omega(\phi) \cap \partial \Hy^2 \subset [\gamma_{i,+}, \tau_{i}\gamma_{i,-}]_{\partial \Hy^2}$ and the lemma holds : take $\phi_a=\tau_i^{-1} \phi$ and $\phi_b=\phi$ in the first case and  $\phi_a= \phi$ and $\phi_b= \tau_i^{-1}\phi$ in the second one.
\end{proof}

By Lemma \ref{LemGoodleaves}, there exist two leaves $\psi_i$ and $\psi'_i$ which meet $\I^{\Z}_{\tilde{\F}}(\tilde{x}_i)$ such that $\overline{\psi}_i$ is disjoint from $\gamma_{i,+}$ and $\overline{\psi}'_i$ is disjoint from $\gamma_{i,-}$. Note that it is possible that $\psi_i=\psi'_i$. By Lemma \ref{LemRecurrenceleaves}, for arbitrarily large $n >0$, the leaf $T_{i}^{- n}  \psi_i$ meets the trajectory $\mathcal{I}^{\Z}_{\tilde{\F}}(\tilde{x}_i)$ and, for arbitrarily large $n >0$, the leaf $T_{i}^{ n}  \psi'_i$ meets the trajectory $\mathcal{I}^{\Z}_{\tilde{\F}}(\tilde{x}_i)$. Moreover,  the sequence $(T_{i}^{-n}\overline{\psi_i})_{n \geq 0}$ of compact subsets of $\overline{\Hy^2}$ converges to $\gamma_{i,-}$. Hence we can take $n_- >0$ sufficiently large so that Lemma \ref{LemDisjointleavescase1} holds with $\phi_{i}=T_{i}^{-n_{-}} \psi_i$. In the same way, take $n_{+} >0$ sufficiently large so that Lemma \ref{LemDisjointleavescase1} holds with $\phi'_{i}=T_{i}^{n_{+}} \psi'_i$.
\end{proof}

Before proving a corollary, we need a geometric lemma.
Let $R= \max(R_1,R_2)$ (see Lemma~\ref{LemBoundeddistance}) and, for $i=1,2$, (see Figure~\ref{FigAi}, see also \eqref{EqDesNbh} for the definition of $R$-neighbourhood)
\begin{align}\label{EqDefAi}
A_i & = \left(\tilde{\gamma}_{i+1}\ \cup T_{i+1}^{-1}\tilde\gamma_i\ 
\cup (T_{i+1}\gamma_{i,-},\gamma_{i,+}) \cup
\!\!\bigcup_{w\in\langle T_{i}^{-1}, T_{i+1}^{-1}\rangle_+}\!\! T_{i+1}^{-1}w \tilde{\gamma}_{i+1}\right)_R,\\
B_i & = \left(\tilde{\gamma}_{i+1}\ \cup T_{i+1}\tilde\gamma_i\ 
\cup (T_{i+1}^{-1}\gamma_{i,-},\gamma_{i,+})\ 
\cup \!\!\bigcup_{w\in\langle T_{i}, T_{i+1}\rangle_+}\!\! T_{i+1}w \tilde{\gamma}_{i+1}\right)_R.\nonumber
\end{align}

The reader will note that there is \emph{a priori} no reason for the set of geodesics $\bigcup_{j=1,2}\bigcup_{w\in\langle T_1,T_2\rangle} w\tilde\gamma_j$ to be a tree in $\Hy^2$ (cases like in Figure~\ref{FigAi} could occur). 

\begin{figure}
\begin{center}
\includegraphics[scale=1]{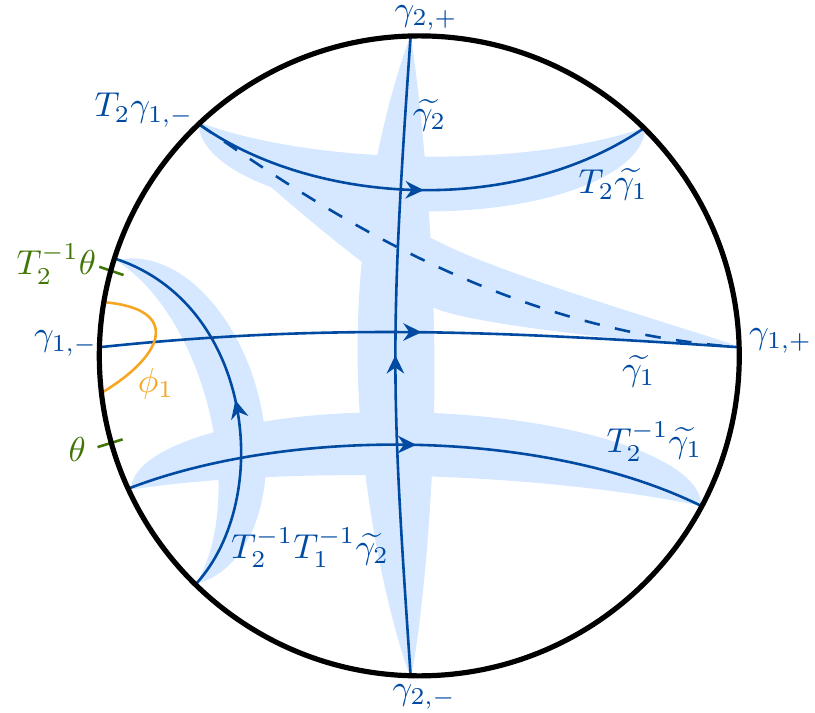}
\caption{A part of the set $A_1$ of \eqref{EqDefAi} (in light blue).}\label{FigAi}
\end{center}
\end{figure}

\begin{lemma}\label{AiBiVoisinages}
The closure of $A_i$ in $\overline{\Hy^2}$ does not meet $\gamma_{i,-}$, and the closure of $B_i$ in $\overline{\Hy^2}$ does not meet $\gamma_{i,+}$
\end{lemma}

\begin{proof}
We prove the lemma only for $A_i$, the case of $B_i$ being identical. We choose an orientation of $\partial\Hy^2$ such that the points $\gamma_{i,\pm}$ are oriented as in Figure~\ref{FigAi}.

First of all, we obviously have that $\gamma_{i,-}\notin\overline{(\tilde{\gamma}_{i+1})_R \cup (T_{i+1}^{-1}\tilde\gamma_i)_R 
\cup (T_{i+1}\gamma_{i,-},\gamma_{i,+})_R}$.

Let $\theta\in (\gamma_{i,-},\gamma_{i+1,-})_{\partial\Hy^2}$ be such that $T_{i+1}^{-1}\theta\in (\gamma_{i+1,+},\gamma_{i,-})_{\partial\Hy^2}$ (see Figure~\ref{FigAi}). Notice that the arc $[\gamma_{i,-},\gamma_{i+1,-}]_{\partial\Hy^2}$ is positively stable under both $T_i^{-1}$ and $T_{i+1}^{-1}$, so for any $w\in\langle T_{i}^{-1}, T_{i+1}^{-1}\rangle_+$, we have $w\gamma_{i+1,-} \in [\gamma_{i,-},\gamma_{i+1,-}]_{\partial\Hy^2}$ and hence $T_{i+1}^{-1}w\gamma_{i+1,-} \in [\theta,\gamma_{i+1,-}]_{\partial\Hy^2}$.

We now consider the points $T_{i+1}^{-1}w\gamma_{i+1,+}$ and prove that they stay at a uniformly positive distance from $\gamma_{i,-}$. Note that $T_{i+1}^{-1}\gamma_{i+1,+}=\gamma_{i+1,+}$, so we do not lose generality by supposing that $w$ ends with $T_i^{-1}$: we write $w=w'T_i^{-1}$. Note that the set $[\gamma_{i,-},\gamma_{i+1,-}]_{\partial\Hy^2}$ attracts all the points of $[T_i^{-1}\gamma_{i+1,+},\gamma_{i+1,-}]_{\partial\Hy^2}$: there is $\ell>0$ such that if $\len(w')\ge \ell$, then $w'T_i^{-1}\gamma_{i+1,+}\in[T_{i+1}^{-1}\theta,\gamma_{i+1,-}]_{\partial\Hy^2}$. Hence, if $\len(w')\ge \ell$, then $T_{i+1}^{-1}w'T_i^{-1}\gamma_{i+1,+}\in[\theta,\gamma_{i+1,-}]_{\partial\Hy^2}$. This proves that 
\begin{align*}
\inf_{w\in\langle T_{i}^{-1}, T_{i+1}^{-1}\rangle_+} & d_{\partial\Hy^2} \big(\gamma_{i,-},\,T_{i+1}^{-1}w\gamma_{i+1,+}\big)\\
& = \min\Big( d_{\partial\Hy^2}\big(\gamma_{i,-},\theta\big),\, 
\inf_{\substack{w'\in\langle T_{i}^{-1}, T_{i+1}^{-1}\rangle_+\\ \len(w')\le\ell}}d_{\partial\Hy^2}\big(\gamma_{i,-},\,T_{i+1}^{-1}w'T_{i}^{-1}\gamma_{i+1,+}\big)\Big).
\end{align*}
We have reduced the bounding of the distance to a finite number of cases, so it suffices to prove that for any $w'\in\langle T_{i}^{-1}, T_{i+1}^{-1}\rangle_+$ with $\len(w')\le\ell$, we have $\gamma_{i,-}\neq T_{i+1}^{-1}w'T_{i}^{-1}\gamma_{i+1,+}$. But this last statement is a consequence of Lemma~\ref{commonendpoint}: if we had a deck transformation $U$ such that $U\gamma_{2,+} = \gamma_{1,-}$, this would mean that the two geodesic arcs $\gamma_1$ and $\gamma_2$ coincide (as sets), which is a contradiction.
\end{proof} 

From now on, we take $\phi_i$ and $\phi_i'$ the leaves given by Lemma~\ref{LemDisjointleavescase1} for $U_{i,-}$ the connected component of $\overline{\Hy^2}\setminus \overline{A_i}$ containing $\gamma_{i,-}$ and $U_{i,+}$ the connected component of $\overline{\Hy^2}\setminus \overline{B_i}$ containing $\gamma_{i,+}$ (Lemma~\ref{AiBiVoisinages} ensures that these connected components are indeed neighbourhoods of $\gamma_{i,+}$ and $\gamma_{i,-}$).

In the following corollary, we identify words on elements of $\pi_1(S)$ with the deck transformations which are obtained by composing the word's letters.

\begin{corollary} \label{CorDisjointleavescase1}
Let $n \geq 0$.
\begin{enumerate}
\item For any word $w$ in $T_1$ and $T_2$, the trajectory $w \mathcal{I}^{\Z}_{\tilde{\F}}(\tilde{x}_{i+1})$ does not meet $T_{i}^{-n}\phi_{i}$ and, for any word $w$ in $T_1$ and $T_2$ containing $T_{i+1}$, the trajectory $w \mathcal{I}^{\Z}_{\tilde{\F}}(\tilde{x}_{i})$ does not meet $T_{i}^{-n}\phi_{i}$.
\item For any word $w$ in $T_1^{-1}$ and $T_2^{-1}$ starting with $T_{i+1}^{-1}$, the trajectories $w \mathcal{I}^{\Z}_{\tilde{\F}}(\tilde{x}_{i+1})$ and $w \mathcal{I}^{\Z}_{\tilde{\F}}(\tilde{x}_{i})$ do not meet $T_{i}^{-n}\phi_{i}$.
\item For any word $w$ in $T_1^{-1}$ and $T_2^{-1}$, the trajectory $w \mathcal{I}^{\Z}_{\tilde{\F}}(\tilde{x}_{i+1})$ does not meet $T_i^{n}\phi'_{i}$ and, for any word $w$ in $T_1^{-1}$ and $T_2^{-1}$ containing $T_{i+1}^{-1}$, the trajectory $w \mathcal{I}^{\Z}_{\tilde{\F}}(\tilde{x}_{i})$ does not meet $T_i^{n}\phi'_{i}$.
\item For any word $w$ in $T_1$ and $T_2$ starting with $T_{i+1}$, the trajectories $w \mathcal{I}^{\Z}_{\tilde{\F}}(\tilde{x}_{i+1})$ and $w \mathcal{I}^{\Z}_{\tilde{\F}}(\tilde{x}_{i})$ do not meet $T_{i}^{n}\phi'_{i}$.
\end{enumerate}
\end{corollary}

\begin{proof} We prove the first two points. To prove points 3. and 4., exchange the roles of $T_1$ and $T_{1}^{-1}$, of $T_2$ and $T_2^{-1}$, of $\gamma_{1,+}$ and $\gamma_{1,-}$ and of $\gamma_{2,+}$ and $\gamma_{2,-}$ and change the leaf $\phi_i$ to the leaf $\phi'_i$ in the following proof. Also, to simplify notation, we suppose $n=0$.

1. For any word $w$ in $T_1$ and $T_2$, observe that either the geodesic $w \tilde{\gamma}_{i+1}$ and the leaf $\phi_i$ are separated by the geodesic $\tilde{\gamma}_{i+1}$, or $w \tilde{\gamma}_{i+1}=\tilde{\gamma}_{i+1}$. Moreover, by definition of $R$, $w \I_{\tilde{\F}}^{\Z}(\tilde{x}_{i+1}) \subset \left( w \tilde{\gamma}_{i+1} \right)_{R}$ so that the trajectory $w \I_{\tilde{\F}}^{\Z}(\tilde{x}_{i+1})$ does not meet the connected component of $\overline{\Hy^2} \setminus \overline{(\tilde{\gamma}_{i+1})_R}$ which contains $\phi_{i}$.

Let $w$ be a word on $T_1$ and $T_2$ which contains $T_{i+1}$. Observe that the endpoints of $w \tilde{\gamma}_i$ lie between $T_{i+1} \gamma_{i,-}$ and $\gamma_{i,+}$. So the set $(w \tilde{\gamma}_i)_{R}$ does not meet the connected component of $\overline{\Hy^2} \setminus \overline{(T_{i+1}\gamma_{i,-},\gamma_{i,+})_R}$ containing $\gamma_{i,-}$ and $\phi_{i}$. Hence $w \I^{\Z}_{\tilde{\F}}(\tilde{x}_i)$ does not meet the leaf $\phi_i$.
\medskip

2. Let $w$ be a word on $T_1^{-1}$ and $T_2^{-1}$ which starts with $T_{i+1}^{-1}$. By definition of $A_i$, the set $(w \tilde{\gamma}_{i+1})_{R}$ is included in $A_i$,  and hence does not meet the leaf $\phi_i$. Therefore, the trajectory $w \I_{\tilde{\F}}^{\Z}(\tilde{x}_{i+1})$ does not meet the leaf $\phi_i$. 

Finally, let $w'\in\langle T_1^{-1},T_{2}^{-1}\rangle_+$. Observe that both ends of $w'\tilde\gamma_i$ lie in $[\gamma_{i,-},\gamma_{i,+}]_{\partial\Hy^2}$, so that both ends of $T_{i+1}^{-1}w'\tilde\gamma_i$ lie in the connected component of $\overline{\Hy^2} \setminus \overline{(T_{i+1}^{-1}\tilde{\gamma}_{i} )_R}$ that does not contain the point $\gamma_{i,-}$ and the leaf $\phi_i$. The trajectory $T_{i+1}^{-1}w' \I_{\tilde{\F}}^{\Z}(\tilde{x}_i)$ is hence disjoint from the leaf $\phi_i$.
\end{proof}

\subsubsection{Transverse intersections}

Let $P_i$ be the maximal integer $j \in \Z$ such that the trajectory $\mathcal{I}^{\Z}_{\tilde{\F}}(\tilde{x}_i)$ meets $\tilde{f}^{j}(\tilde{x}_{i})$ strictly before it meets the leaf $\phi_i$ and the trajectory $(-\infty,\tilde{f}^{j}(\tilde{x}_{i}))_i$ is disjoint from the set $A_i$ (defined in \eqref{EqDefAi} page \pageref{EqDefAi}).

Changing the point $\tilde{x}_i$ to $\tilde{f}^{P_i}(\tilde{x}_{i})$ if necessary, we can suppose that $P_1=P_2=0$.

\begin{lemma} \label{LemTransadm}
There exist integers $m_1>0$ and $m_2 >0$, which can be taken arbitrarily large, and integers $r_1>m_1 p_1$, $r_{2}>m_2 p_2$, such that
\begin{enumerate}
\item For any $i \in \Z /2$, the leaf $T_{i}^{r_i} \phi'_i$ meets  the trajectory $\mathcal{I}^{\Z}_{\tilde{\F}}(\tilde{x}_{i})$, the trajectory $[\phi_i,T_{i}^{r_i} \phi'_i]_i$ is admissible of order $m_iq_i$ and, for any $0 \leq j \leq m_i p_i$, the trajectory $[\phi_i,T_{i}^{r_i} \phi'_i]_i$ meets the trajectory $T_{i}^{j} \mathcal{I}^{\Z}_{\tilde{\F}}(\tilde{x}_{i+1})$.
\item For any $i \in \Z /2$  the trajectory $(-\infty,\phi_i]_i$ lies in the connected component of $\overline{\Hy^2} \setminus \overline{ A_i}$ which contains $\gamma_{i,-}$ and the trajectory $[T_i^{r_i} \phi'_i,+\infty)_i$ lies in the connected component of 
$\overline{\Hy^2} \setminus T_{i}^{m_i p_i} \overline{B_i}$
which contains $\gamma_{i,+}$.
\item For any $i\in \Z /2$ and any $0 \leq j \leq m_i p_i$, the transverse paths $[\phi_i, T_{i}^{r_i} \phi'_i]_i$ and $T_{i}^j[\phi_{i+1}, T_{i+1}^{r_{i+1}} \phi'_{i+1}]_{i+1}$ have an $\tilde{\F}$-transverse intersection at some point $\tilde y_j^i$.
\end{enumerate}
\end{lemma}

In the same way we proved Corollary \ref{CorDisjointleavescase1} from Lemma \ref{LemDisjointleavescase1}, it is possible to prove the following corollary by using the second point of Lemma \ref{LemTransadm}. Note that points 1. and 2. of this corollary are direct consequences of Lemma \ref{LemDisjointleavescase1}.
As the proof of points 3. and 4. are identical to the proof of Corollary \ref{CorDisjointleavescase1}, we leave it to the reader.

\begin{corollary} \label{CorTransadm}
\begin{enumerate}
\item For any word $w$ in $T_1$ and $T_2$, the trajectory $w \mathcal{I}^{\Z}_{\tilde{\F}}(\tilde{x}_{i+1})$ does not meet $(-\infty,\phi_i]_i$ and, for any word $w$ in $T_1$ and $T_2$ which contains $T_{i+1}$, the trajectory $w \mathcal{I}^{\Z}_{\tilde{\F}}(\tilde{x}_{i})$ does not meet $(-\infty,\phi_i]_i$.
\item For any word $w$ in $T_1^{-1}$ and $T_2^{-1}$ which starts with $T_{i+1}^{-1}$, the trajectory $w \mathcal{I}^{\Z}_{\tilde{\F}}(\tilde{x}_{i+1})$ does not meet $(-\infty,\phi_i]_i$ and, for any word $w$ in $T_1^{-1}$ and $T_2^{-1}$ which starts with $T_{i+1}^{-1}$, the trajectory $w \mathcal{I}^{\Z}_{\tilde{\F}}(\tilde{x}_{i})$ does not meet $(-\infty,\phi_i]_i$.
\item For any word $w$ in $T_1^{-1}$ and $T_2^{-1}$, the trajectory $T_{i}^{m_i p_i}w \mathcal{I}^{\Z}_{\tilde{\F}}(\tilde{x}_{i+1})$ does not meet $[T_i^{r_i} \phi'_i,+\infty)_i$ and, for any word $w$ in $T_1^{-1}$ and $T_2^{-1}$ which contains $T_{i+1}^{-1}$, the trajectory $T_{i}^{m_i p_i}w \mathcal{I}^{\Z}_{\tilde{\F}}(\tilde{x}_{i})$ does not meet $[T_i^{r_i} \phi'_i,+\infty)_i$.
\item For any word $w$ in $T_1$ and $T_2$ which starts with $T_{i+1}$, the trajectory $T_{i}^{m_i p_i}w \mathcal{I}^{\Z}_{\tilde{\F}}(\tilde{x}_{i+1})$ does not meet $[T_i^{r_i}\phi'_i,+\infty)_i$ and, for any word $w$ in $T_1$ and $T_2$ which starts with $T_{i+1}$, the trajectory $T_{i}^{m_ip_i} w \mathcal{I}^{\Z}_{\tilde{\F}}(\tilde{x}_{i})$ does not meet $[T_i^{r_i}\phi'_i,+\infty)_i$.
\end{enumerate}
\end{corollary}

\begin{proof} [Proof of Lemma \ref{LemTransadm}]
We will find integers $m_1$ and $m_2$ such that the first two items of the lemma are satisfied. We will then see that the third item is automatically satisfied.

We fix $i=1,2$. Observe that the first part of the second point holds by definition of $\phi_i$ and Lemma \ref{LemDisjointleavescase1}. We parametrize the geodesic $\tilde{\gamma}_i$ by arclength and identify points on $\tilde{\gamma}_i$ with their parameters. Recall that $\pi_{\tilde{\gamma}_i}$ denotes the orthogonal projection on the geodesic $\tilde{\gamma}_i$. Let $\lambda= \max \pi_{\tilde{\gamma}_i}(B_i)$ (it exists by Lemma~\ref{AiBiVoisinages}). Let $k_{i}$ be an integer such that the trajectory $[\tilde{x}_i,\tilde{f}^{k_{i}}(\tilde{x}_i)]_i$ meets the leaves $\phi_{i}$ and $\phi'_{i}$ but not at its endpoints. Finally, fix $v'_i, v''_i \in \big(\frac{p_i}{q_i} \ell(\gamma_i),v_i\big)$ with $v''_i < v'_i$.

By Lemma \ref{LemRecurrentpoint}, for any $n$ sufficiently large,
$$\pi_{\tilde{\gamma}_i}\big(\tilde{f}^n(\tilde{x}_i)\big) \geq \pi_{\tilde{\gamma}_i}(\tilde{x}_i)+n v'_i.$$
Moreover, for any $n$ sufficiently large,
$$\pi_{\tilde{\gamma}_i}(\tilde{x}_i)+n v'_i \geq \lambda+M_i+nv''_i,$$
where $M_i$ is given by Lemma \ref{LemBoundeddistance}. Take $N \in \mathbb{N}$ such that the two above properties hold for any $n \geq N$. By Lemma \ref{LemRecurrenceleaves}, there exists $n_i \geq N$, which can be taken arbitrarily large, such that the segment $[\tilde{f}^{n_i}(\tilde{x}_i),\tilde{f}^{n_i+k_i}(\tilde{x}_i)]_i$ meets $T_{i}^{r_{i}} \phi'_{i}$ for some $r_{i} \geq \frac{v'_i n_i}{\ell(\gamma_i)}$. Let $m_i$ be the smallest integer such that $n_i+k_i \leq m_iq_i$. If $n_i$ is chosen sufficiently large, then $r_i > m_ip_i$ and, for any $n \geq n_i$,
$$ \pi_{\tilde{\gamma}_i}(\tilde{f}^{n}(\tilde{x}_i)) > \lambda+M_i+m_i p_i\ell(\gamma_i).$$
Indeed, when $n_i$ is sufficiently large,
$$n_i v_i'' > \left\lfloor \frac{n_i+k_i}{q_i} \right\rfloor p_i \ell(\gamma_i)=m_i p_i \ell(\gamma_i).$$

This implies that the half-trajectory $[\tilde{f}^{n_i}(\tilde{x}_i),+\infty)_i$ is disjoint from $T_{i}^{m_ip_i}B_i$, and that the segment $[\tilde{x}_i,\tilde{f}^{m_iq_i}(\tilde{x}_i)]_i$ meets the leaf $T_i^{r_i} \phi'_i$. This proves the second point as $[T_{i}^{r_i}\phi'_{i},+\infty)_i \subset[\tilde{f}^{n_i}(\tilde{x}_i),+\infty)_i$.  Moreover, recall that $P_i=0$ so that the segment $[\tilde{x}_i,\tilde{f}^{m_iq_i}(\tilde{x}_i)]_i$ also meets the leaf $\phi_{i}$. Hence, by Proposition~\ref{PropPasFondLCT}, the segment $[\phi_{i},T_{i}^{r_i} \phi'_{i}]_i$ is admissible of order $m_iq_i$. 

Let us prove the first point now. We already saw that $T_{i}^{r_i} \phi'_i$ meets $\I^{\Z}_{\tilde\F}(\tilde{x}_i)$ and that the segment $[\phi_{i},T_{i}^{r_i} \phi'_{i}]_i$ is admissible of order $m_iq_i$, so it remains to prove that for any $0 \leq j \leq m_i p_i$, the trajectory $[\phi_i,T_{i}^{r_i} \phi'_i]_i$ meets the trajectory $T_{i}^{j} \mathcal{I}^{\Z}_{\tilde{\F}}(\tilde{x}_{i+1})$. Recall that, by definition of $R$, for any $j$ with $0 \leq j \leq m_ip_i$, $T_i^{j} \I_{\tilde{\F}}^{\Z}(\tilde{x}_{i+1}) \subset (T_i^{j} \tilde{\gamma}_{i+1})_R$ and that the set $(T_i^{j} \tilde{\gamma}_{i+1})_R$ meets neither the connected component of $\overline{\Hy^2} \setminus \overline{(\tilde{\gamma}_{i+1})_R} \supset \overline{\Hy^2} \setminus \overline{A_i}$ which contains $\gamma_{i,-}$, $\phi_{i}$ and $(-\infty,\phi_i]_i$, by definition of $\phi_i$, nor the connected component of $\overline{\Hy^2} \setminus \overline{(T_{i}^{m_ip_i}\tilde{\gamma}_{i+1})_R} \supset \overline{\Hy^2} \setminus \overline{T_{i}^{m_ip_i}B_i}$ which contains $\gamma_{i,+}$, $T_i^{r_i}\phi'_{i}$ and $[T_{i}^{r_i} \phi'_i,+\infty)_i$. As, for any $0 \leq j \leq m_ip_i$, the trajectory $T_i^{j} \I_{\tilde{\F}}^{\Z}(\tilde{x}_{i+1})$ is disjoint from $(-\infty,\phi_i]_i$ and $[T_i^{r_i} \phi'_i,+\infty)_i$ and meets $\I_{\tilde{\F}}^{\Z}(\tilde{x}_{i})$, we deduce that it meets $[\phi_i,T_i^{r_i}\phi'_i]_i$.
The first point is hence satisfied.

Let us prove now the third point. Let $0 \leq j \leq m_i p_i$. By Lemma \ref{LemDisjointleavescase1}, the set $(T_{i}^j\tilde{\gamma}_{i+1})_{R}$ does not meet the leaf $\phi_{i}$. Recall that $T_{i}^j\I^{\Z}_{\tilde{\F}}(\tilde{x}_{i+1})$ is contained in $(T_i^{j}\tilde{\gamma}_{i+1})_{R}$, by definition of $R$. Hence, the leaf $\phi_{i}$ does not meet the trajectory $T_{i}^j\I^{\Z}_{\tilde{\F}}(\tilde{x}_{i+1})$. In the same way, the trajectory $T_{i}^j\I^{\Z}_{\tilde{\F}}(\tilde{x}_{i+1})$ does not meet the leaf $T_{i}^{r_{i}}(\phi'_{i})$ either, as $j \leq m_i p_i < r_{i}$.

Similarly, by Lemma \ref{LemDisjointleavescase1}, the trajectory $\I^{\Z}_{\tilde{\F}}(\tilde{x}_{i})$ does not meet the leaves $T_{i}^j(\phi_{i+1})$ and $T_{i}^j(\phi'_{i+1})$, as $\I^{\Z}_{\tilde{\F}}(\tilde{x}_{i}) \subset (\gamma_i)_R$. Finally, by Corollary \ref{CorTransadm} (points 1. and 3.), which is deduced from the already proved second point of the lemma,
$$\Big((-\infty,\phi_i]_i \cup [T_i^{r_i} \phi'_i,+\infty)_i\Big)\ 
\cap\ T_{i}^{j}\Big((-\infty,\phi_{i+1}]_{i+1} \cup [T_{i+1}^{r_{i+1}} \phi'_{i+1},+\infty)_{i+1}\Big)=\emptyset$$
and we indeed have an $\tilde{\F}$-transverse intersection by Lemma \ref{LemEssentialpoints}.3.
\end{proof}

\subsubsection{Admissible trajectories}

Let us fix integers $m_1$, $m_2$, $r_1>m_1p_1$ and $r_2>m_2p_2$ such that Lemma \ref{LemTransadm} is satisfied. We let
$$ \alpha = [\phi_1, T_1^{r_1} \phi'_1]_1
\qquad  \text{and}\qquad
\beta = [\phi_2, T_2^{r_2} \phi'_2]_2.$$

Let $I=(i_{n},j_{n})_{n \geq 1}$ be a sequence of couples of integers. 
For any $n \geq 1$,  we let
$$ \left\{ \begin{array}{rcl}
T^{I_{n,1}} & = & T_1^{i_{1}} T_2^{j_{1}} \ldots T_1^{i_{n-1}} T_2^{j_{n-1}}T_1^{i_{n}} \\
T^{I_{n,2}} & = & T_1^{i_{1}} T_2^{j_{1}} \ldots T_1^{i_{n}} T_2^{j_{n}}.
\end{array} \right.
$$
and by convention
$$T^{I_{0,1}}=T^{I_{0,2}}=\Id_{\tilde{S}}.$$

\begin{lemma}\label{LemTransadmconc}
Let $I=(i_{n},j_{n})_{n \geq 0}$ be any sequence with $1 \leq i_{n} \leq m_1 p_1$ and $1 \leq j_{n} \leq m_2 p_2$ for any $n$. Then for any $n \geq 1$, there exists an $\tilde{\F}$-transverse path $\alpha_{n}$ with the following properties.
\begin{enumerate}
\item The path $\alpha_n$ is admissible of order $n(m_1q_1+m_2q_2)$.
\item The path $\alpha_n$ joins the leaf $\phi_1$ to the leaf $T^{I_{n,1}} T_2^{r_2} \phi'_2$.
\item The path $\alpha_n$ is contained in
$$ \bigcup_{ k \leq n} T^{I_{k-1,2}} \alpha \cup T^{I_{k,1}}\beta.$$ 
\item The path $\alpha_n$ has an $\tilde{\F}$-transverse intersection with the path $T^{I_{n,2}} \alpha_n$.
\end{enumerate}
\end{lemma}

For any $n \in \Z$, the geodesics $\gamma_1$ and $T_1^n \gamma_2$ meet so that $ \mathcal{I}_{\tilde{\F}}^{\Z}(\tilde{x}_1)$ and $T_{1}^n\mathcal{I}_{\tilde{\F}}^{\Z}(\tilde{x}_2)$ are geometrically transverse. By Lemma~\ref{LemEssentialpoints}.2., there exists $\tilde{y}_n$ an essential intersection point between those two trajectories. 

\begin{proof}
We prove the statement by induction on $n \geq 1$.

\begin{figure}
\begin{center}
\includegraphics[scale=1]{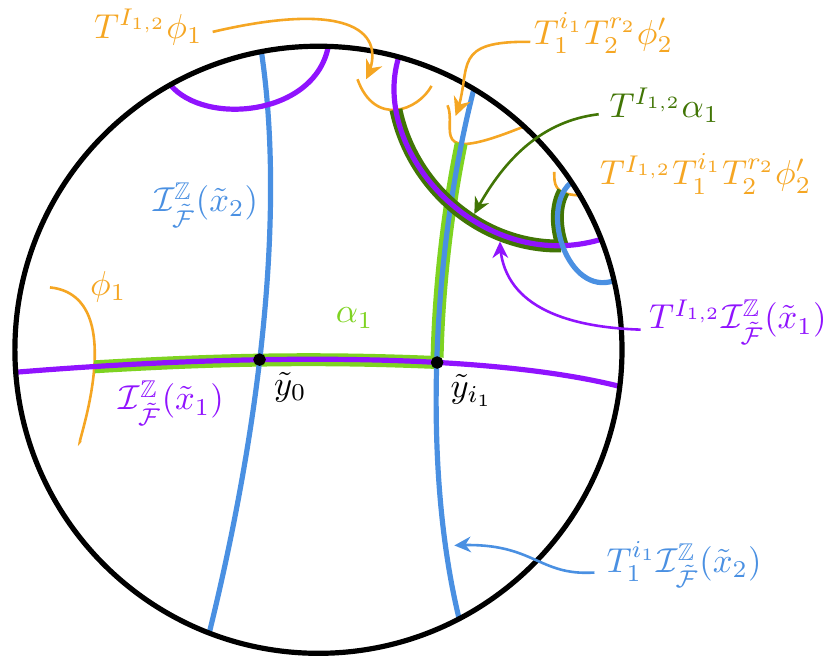}
\caption{The case $n=1$ for the proof of Lemma~\ref{LemTransadmconc}.}\label{FigTransadmconc}
\end{center}
\end{figure}

First, for $n=1$, as $i_1\le m_1p_1$, Lemma \ref{LemTransadm} ensures that the paths $\alpha$ and $T_1^{i_1} \beta$ have an $\tilde{\F}$-transverse intersection. Hence, by Proposition \ref{PropFondLCT}, the transverse path
$$\alpha_1 = [\phi_1,\tilde y_{i_1}]_1\ 
T_1^{i_1}[T_1^{-i_1}\tilde y_{i_1}, T_2^{r_2} \phi'_2]_2$$
satisfies the first three required properties.

Let us check that the fourth property is also satisfied. Notice that the transverse trajectory
$$ \alpha'_1= (-\infty,\phi_1]_1 \alpha_1 T_1^{i_1}[T_2^{r_2} \phi'_2,+\infty)_2$$
and its translate under $T^{I_{1,2}}=T_1^{i_1}T_2^{j_1}$ are geometrically transverse, as the geodesics $(\gamma_{1,-},T_1^{i_1}\gamma_{2,+})$ and its translate under $T^{I_{1,2}}$, which is $(T_1^{i_1}T_2^{j_1} \gamma_{1,-}, T_1^{i_1}T_2^{j_1} T_1^{i_1} \gamma_{2,+})$, meet: observe that the point $T_2^{j_1} \gamma_{1,-}$ lies in $(\gamma_{2,+},\gamma_{1,-})_{\partial\Hy^2}$ and that the point $T_2^{j_1} T_1^{i_1} \gamma_{2,+}$ lies in $(\gamma_{1,-}, \gamma_{2,+})_{\partial\Hy^2}$.

By Corollary \ref{CorDisjointleavescase1} (the numbers correspond to the cases of the corollary),
$$ \begin{array}{rcl}
 \phi_1 \cap T^{I_{1,2}} \alpha'_{1} & = & \emptyset\quad  (1.)\\
T_{1}^{i_1} T_2^{r_2} \phi'_{2} \cap T^{I_{1,2}} \alpha'_1 & = & \emptyset\\
 T^{I_{1,2}} \phi_1 \cap \alpha'_1 & = & \emptyset \quad  (2.)\\
 T^{I_{1,2}} T_{1}^{i_1}T_2^{r_2}\phi'_2 \cap \alpha'_{1} & = & \emptyset \quad  (3.).
 \end{array}$$
To prove the second relation, observe that the leaf $T_2^{r_2-j_1} \phi'_2$ is disjoint from the trajectories $\I^{\Z}_{\tilde{\F}}(\tilde{x}_1)$ (by 3.) and $T_{1}^{i_1} \I^{\Z}_{\tilde{\F}}(\tilde{x}_2)$ (by 4.) as $r_2-j_1 \geq r_2-m_2p_2 \geq 0$.

Moreover, 
$$ \left((-\infty, \phi_1]_1 \cup T_1^{i_1}[T_2^{r_2} \phi'_2,+\infty)_2 \right) \cap T^{I_{1,2}}\left( (-\infty, \phi_1]_1 \cup T_1^{i_1}[T_2^{r_2} \phi'_2,+\infty)_2 \right)=\emptyset,$$
because by Corollary \ref{CorTransadm} (the numbers correspond to the cases of the corollary),
\begin{align*}
(-\infty, \phi_1]_1 \cap T^{I_{1,2}}(-\infty, \phi_1]_1 & = \emptyset \quad  (1.)\\
(-\infty, \phi_1]_1 \cap T^{I_{1,2}}T_1^{i_1}[T_2^{r_2} \phi'_2,+\infty)_2 & = \emptyset \quad  (1.)\\
T_1^{i_1}[T_2^{r_2} \phi'_2,+\infty)_2 \cap T^{I_{1,2}}(-\infty, \phi_1]_1 & = \emptyset \quad  (3.)\\
T_1^{i_1}[T_2^{r_2} \phi'_2,+\infty)_2 \cap T^{I_{1,2}} T_1^{i_1}[T_2^{r_2} \phi'_2,+\infty)_2 & = \emptyset \quad  (3.)
\end{align*}
so that the paths $\alpha_1$ and $T^{I_{1,2}}\alpha_1$ intersect $\tilde{\F}$-transversally by Lemma \ref{LemEssentialpoints}.
\bigskip

Suppose the lemma holds for some integer $n \geq 1$ and let us prove it for $n+1$. The initialization of the induction proves that there exists  a transverse path $\alpha_{1,i_{n+1}}$ which is admissible of order $m_1q_1+m_2q_2$ which is contained in $\alpha \cup T_1^{i_{n+1}} \beta$ and which joins the leaf $\phi_1$ to the leaf $T_1^{i_{n+1}}T_2^{r_2} \phi'_2$.

Let us prove first that the paths $T^{I_{n,2}}\alpha_{1,i_{n+1}}$ and $\alpha_n$ have an $\tilde{\F}$-transverse intersection. Let
$$ \alpha'_n= (-\infty,\phi_1]_1 \alpha_n T^{I_{n,1}}[T_2^{r_2} \phi'_2,+\infty)_2$$
and
$$ \alpha'_{1,i_{n+1}}= (-\infty,\phi_1]_1 \alpha_1 T_{1}^{i_{n+1}}[T_2^{r_2} \phi'_2,+\infty)_2.$$
Observe that the point $T^{I_{n,2}} \gamma_{1,-}$ lies between the points $\gamma_{1,-}$ and $T^{I_{n,1}}\gamma_{2,+}$ on $\partial \Hy^2$ and that the point $T^{I_{n,2}} T_1^{i_{n+1}} \gamma_{2,+}$ lies between the points $T^{I_{n,1}}\gamma_{2,+}$ and $\gamma_{2,-}$ on $\partial \Hy^2$. Hence the trajectories $\alpha'_n$ and $T^{I_{n,2}} \alpha'_{1,i_{n+1}}$ are geometrically transverse. Corollary \ref{CorDisjointleavescase1} and the third property satisfied by $\alpha'_n$ ensure that the leaves $T^{I_{n,2}} \phi_{1}$ and $T^{I_{n,2}} T_1^{i_{n+1}} T_2^{r_2} \phi'_2$ are disjoint from $\alpha'_{n}$. It also ensures that the leaves $\phi_1$ and $T^{I_{n,1}} T_2^{r_2} \phi'_2$ are disjoint from the transverse path $T^{I_{n,2}}\alpha'_{1,i_{n+1}}$.  Moreover, by Corollary \ref{CorTransadm} and the third property from the induction hypothesis,
$$ \left\{ \begin{array}{c}
\left((-\infty,\phi_1]_1  \cup T^{I_{n,1}} [T_2^{r_2} \phi'_2,+\infty)_2 \right) \cap T^{I_{n,2}}\alpha'_{1,i_{n+1}}=\emptyset \\ 
T^{I_{n,2}}\left((-\infty,\phi_1]_1  \cup T_1^{i_{n+1}} [T_2^{r_2} \phi'_2,+\infty)_2 \right) \cap \alpha'_{n}=\emptyset
\end{array} \right.$$
Hence the paths $T^{I_{n,2}} \alpha_{1,i_{n+1}}$ and $\alpha_{n}$ intersect $\tilde{\F}$-transversally by Lemma \ref{LemEssentialpoints}. Finally, Proposition \ref{PropFondLCT} gives a path $\alpha_{n+1}$ which satisfies the three first conditions. By a proof which is similar to the initialization step, we prove that the paths $\alpha_{n+1}$ and $T^{I_{n+1,2}} \alpha_{n+1}$ intersect $\tilde{\F}$-transversally. This completes the induction.
\end{proof}

\begin{corollary} \label{CorTransperorbit}
Let $I=(i_{n},j_{n})_{n \geq 0}$ be any sequence with $1 \leq i_{n} \leq m_1 p_1$ and $1 \leq j_{n} \leq m_2 p_2$ for any $n$. Then, for any integer $n \geq 1$ and any integer $j$ with $0 \leq j \leq m_2 p_2$, there exists points $\tilde{x}_{n,I}$ and $\tilde{x}'_{n,I}$ such that
$$\tilde{f}^{n (m_1q_1+m_2q_2)}(\tilde{x}_{n,I})=T_{2}^{j}T^{I_{n,1}}\tilde{x}_{n,I} \ \text{and} \ \tilde{f}^{n (m_1q_1+m_2q_2)}(\tilde{x}'_{n,I})=T^{I_{n,2}}\tilde{x}'_{n,I}.$$
\end{corollary}

\begin{proof} The existence of the points $\tilde{x}'_{n,I}$ is a consequence of Lemma \ref{LemTransadmconc} and of Theorem~\ref{ExistPasSuperCheval}. Exchanging the roles of $T_1$ and $T_2$ in Lemma \ref{LemTransadmconc}, we obtain the existence of the points $\tilde{x}_{n,I}$.
\end{proof}

\begin{proof}[End of the proof of Theorem \ref{Th2transverse} in the first case]
Take any word $w$ in letters $T_1$ and $T_2$ which contains at least one $T_1$ letter and one $T_2$ letter. Of course, we identify such a word with a deck transformation of $\tilde{S}$.
Write
$$w=T_1^{i_1} T_{2}^{j_1} \ldots T_{1}^{i_K}T_{2}^{j_K}$$
with
$$\left\{ 
\begin{array}{l}
K \geq 0 \\
i_{n}, j_{n}>0 \ \text{if} \ 2 \leq n \leq K-1 \\
j_{1} >0 \ \text{and} \ i_{K} >0 \\
i_{1}  \geq 0 \ \text{and} \ j_{K}  \geq 0.
\end{array}
\right.$$

Take integers $m_{1}$ and $m_2$ large enough so that $\max(i_1+i_K, \max_{1 \leq n \leq K}i_{n}) \leq m_1 p_1$ and $\max(\max_{1 \leq n \leq K}j_{n}, j_{1}+j_{K}) \leq m_2 p_2$.

If $i_1>0$ and $j_{K}>0$ or $i_{1}=0$ and $j_{K}=0$, Corollary \ref{CorTransperorbit} gives directly the wanted result.
If $i_{1}= 0$ apply Corollary \ref{CorTransperorbit} to the word $ T_{1}^{i_2} T_2^{j_2} \ldots T_{1}^{i_K}T_{2}^{j_K+j_1}$ to get a periodic point $\tilde{x}$ associated to this deck transformation. The point $T_{2}^{j_{1}}(\tilde{x})$ gives us then the wanted periodic orbit. If $j_{K}=0$ apply Corollary \ref{CorTransperorbit} to the word $ T_{1}^{i_1+i_K} T_2^{j_1} \ldots T_{1}^{i_{K-1}}T_{2}^{j_{K-1}}$ to get a periodic point $\tilde{x}$ associated to this deck transformation. The point $T_{1}^{-i_{K}}(\tilde{x})$ gives us then the wanted periodic orbit.
\end{proof}

\subsection{Case 2}

In this subsection, we prove Theorem \ref{Th2transverse} in the case where one of the trajectories $\mathcal{I}_{\tilde{\F}}^{\mathbb{Z}}(\tilde{x}_i)$ satisfies condition 
\begin{enumerate}
\item[\ref{C2}] for any deck transformation $\tau \in \pi_1(S) \setminus \langle T_i\rangle $, any leaf which meets $\mathcal{I}^{\Z}_{\tilde{\F}}(\tilde{x}_i)$ does not meet $\tau \mathcal{I}^{\Z}_{\tilde{\F}}(\tilde{x}_i)$.
\end{enumerate}
Changing the roles of $x_1$ and $x_2$ if necessary, we can suppose that $\mathcal{I}_{\tilde{\F}}^{\Z}(\tilde{x}_1)$ satisfies condition \ref{C2}.

As in the previous case, we start by choosing carefully the respective lifts $\tilde{x}_1$ and $\tilde{x}_2$ of 
the points $x_1$ and $x_2$ to $\tilde{S}$. In the course of the proof, we will skip details when the arguments are similar to some arguments which were given in the first case. 

The proof in this second case is a bit more complex than in the first one. The first step, made in Paragraph~\ref{ParagChoixPoints}, is a bit less technical than in the first case. However, we will see appearing configurations in which there are no $\F$-transverse intersections between our initial trajectories. This will force us to consider other transverse trajectories, and will complicate the rest of the proof. the last part of the proof, made in Paragraph~\ref{ParagAdmiss2}, requires a very careful study of the possible intersections between leaves and trajectories.

\subsubsection{Choice of the points $\tilde{x}_1$ and $\tilde{x}_2$} \label{ParagChoixPoints}

See Figure~\ref{FigCase2} for these notations. For any $n \in \Z$, the geodesics $T_2^n \tilde\gamma_1$ and $\tilde\gamma_2$ meet so that $T_2^{n} \mathcal{I}_{\tilde{\F}}^{\Z}(\tilde{x}_1)$ and $\mathcal{I}_{\tilde{\F}}^{\Z}(\tilde{x}_2)$ are also geometrically transverse. Let $\tilde{y}'_n$ be the essential intersection point between $T_2^{n} \mathcal{I}_{\tilde{\F}}^{\Z}(\tilde{x}_1)$ and $\mathcal{I}_{\tilde{\F}}^{\Z}(\tilde{x}_2)$ that is minimal for $<_2$ (it exists by Lemma~\ref{LemEssentialpoints}.2.).  As $\mathcal{I}_{\tilde{\F}}^{\Z}(\tilde{x}_1)$ satisfies condition \ref{C2}, the lifts $T_2^{n} \mathcal{I}_{\tilde{\F}}^{\Z}(\tilde{x}_1)$ are pairwise disjoint so that the sequence $(\tilde{y}'_n)_{n \in \Z}$ is increasing for the order relation $<_2$.

For any $n \in \Z$, the geodesics $\tilde\gamma_1$ and $T_1^n \tilde\gamma_2$ meet so that $ \mathcal{I}_{\tilde{\F}}^{\Z}(\tilde{x}_1)$ and $T_{1}^n\mathcal{I}_{\tilde{\F}}^{\Z}(\tilde{x}_2)$ are geometrically transverse. Denote by $\tilde{y}_n$ an essential intersection point between those two trajectories. We also take $\tilde{y}_0=\tilde{y}'_{0}$.

In the whole proof, we orient $\partial \Hy^2$ in such a way that the point $\gamma_{2,-}$ lies on the positively oriented segment of $\partial \Hy^2$ which joins $\gamma_{1,-}$ to $\gamma_{1,+}$.

\begin{figure}
\begin{minipage}{.47\linewidth}
\begin{center}
\includegraphics[width=\linewidth]{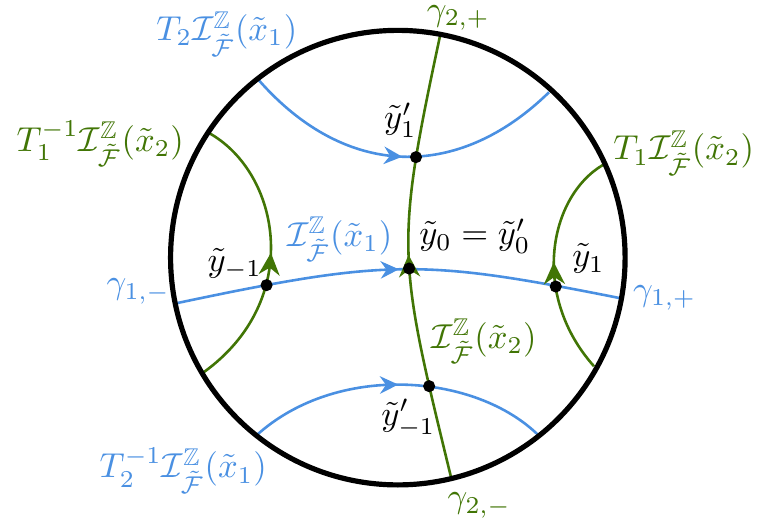}
\caption{Notation of the beginning of Case 2.}\label{FigCase2}
\end{center}
\end{minipage}
\hfill
\begin{minipage}{.47\linewidth}
\begin{center}
\includegraphics[width=\linewidth]{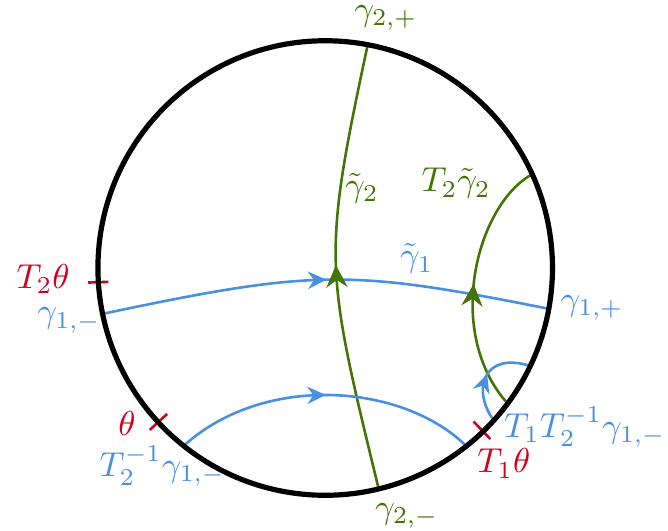}
\caption{Proof of Claim~\ref{ClaimAA}.}\label{FigClaimAA}
\end{center}
\end{minipage}
\end{figure}

The following lemma is the analogue in this case of Lemma \ref{LemDisjointleavescase1}. We need the following notation:
$$ A= \overline{\bigcup_{\epsilon \in \left\{ -1,0,1\right\}}\!\! T_2^{\epsilon}(\tilde{\gamma}_1)_{R}\,  \cup\, \bigcup_{i \neq 0} T_1^{i}(\tilde{\gamma}_2)_R\, \cup\!\!\! \bigcup_{\substack{i \geq 1\\ w \in \langle T_1^{-1},T_2^{-1}\rangle_+}}\!\!\!\! T_1^{i}T_2^{-1}w(\tilde{\gamma}_2)_R }\ \subset \ \overline{\Hy^2}.$$

\begin{align*}
(B_i)_{i \in \Z} & = \overline{\phi_{\tilde{y}'_i} \cup (-\infty,\tilde{y}'_i]_2 \cup T_2^{i}(-\infty,T_2^{-i}\tilde{y}'_i]_1} ,\\
(B'_i)_{i \in \Z} & = \overline{ \bigcup_{n \geq 0} T_2^{n}(\phi_{\tilde{y}'_i} \cup [\tilde{y}'_i,+\infty)_2) } .
\end{align*}

\begin{lemma} \label{LemDisjointleavescase2}
There exists an integer $k'_{0}$ such that the set $B_{-k'_0}$ is included in the connected component of $\overline{\Hy^{2}} \setminus A$ containing $\gamma_{2,-}$, and the set $B'_{k'_0}$ is included in the connected component of $\overline{\Hy^{2}} \setminus A$ containing $\gamma_{2,+}$
\end{lemma}

In the sequel, we will denote
\[\phi_2=\phi_{\tilde{y}'_{-k'_0}}
\qquad\text{and}\qquad
\phi'_2=\phi_{\tilde{y}'_{k'_0}}.\]

\begin{proof}
We need to prove the following claim, which is similar to Lemma~\ref{AiBiVoisinages}.

\begin{claim}\label{ClaimAA}
The set $A$ meets neither $\gamma_{2,-}$ nor $\gamma_{2,+}$.
\end{claim}

\begin{proof}
First, observe that
\begin{multline*}
\overline{\bigcup_{\epsilon \in \left\{ -1,0,1\right\}} T_2^{\epsilon}(\tilde{\gamma}_1)_{R}  \cup \bigcup_{i \neq 0} T_1^{i}(\tilde{\gamma}_2)_R} \cap \partial \Hy^2\\
= \bigcup_{\epsilon \in \left\{ -1,0,1 \right\}} \left\{ T_2^{\epsilon}\gamma_{1,-}, T_2^{\epsilon}\gamma_{1,+} \right\}
\cup \overline{\left\{ T_1^{i} \gamma_{2,+}, i \neq 0 \right\}}
\cup \overline{\left\{ T_1^{i} \gamma_{2,-}, i \neq 0 \right\}}
\end{multline*}
so that the set
$$\overline{\bigcup_{\epsilon \in \left\{ -1,0,1\right\}} T_2^{\epsilon}(\tilde{\gamma}_1)_{R}  \cup \bigcup_{i \neq 0} T_1^{i}(\tilde{\gamma}_2)_R}$$
meets neither $\gamma_{2,-}$ nor $\gamma_{2,+}$. 
\medskip

By Lemma \ref{commonendpoint}, the set
$${\bigcup_{i \geq 1}\, \bigcup_{\substack{w \in \langle T_1^{-1},T_2^{-1}\rangle_+}}  T_1^{i}T_2^{-1}w(\tilde{\gamma}_2)_R}$$
meets neither $\gamma_{2,-}$ nor $\gamma_{2,+}$ ($\langle T_1,T_2\rangle$ is a free group). So it suffices to prove that this set does not accumulate on $\gamma_{2,-}$ or $\gamma_{2,+}$ either. As the deck transformation $T_2$ preserves $\tilde{\gamma}_2$, observe that
$${\bigcup_{i \geq 1}\, \bigcup_{\substack{w \in \langle T_1^{-1},T_2^{-1}\rangle_+}}  \!\!\!T_1^{i}T_2^{-1}w(\tilde{\gamma}_2)_R}
\ =\ 
{\bigcup_{i \geq 1}\, \bigcup_{\substack{w \in \langle T_1^{-1},T_2^{-1}\rangle_+}} \!\!\! T_1^{i}T_2^{-1}w T_1^{-1}(\tilde{\gamma}_2)_R} \ \cup\ \bigcup_{i \geq 1} T_1^{i}(\tilde{\gamma}_2)_R.$$
We already saw that the set
$$\overline{\bigcup_{i \geq 1} T_1^{i}(\tilde{\gamma}_2)_R}$$
meets neither $\gamma_{2,-}$ nor $\gamma_{2,+}$. 

Condition \ref{C2} forces the geodesics $T_2^{-1} \tilde{\gamma}_{1}$ and $T_1T_2^{-1}\tilde{\gamma}_1$ to be disjoint so that $T_1 T_2^{-1} \gamma_{1,-}$ belongs to $(T_2^{-1} \gamma _{1,+},\gamma_{1,+})_{\partial \Hy^2} \subset (\gamma _{2,-},\gamma_{1,+})_{\partial \Hy^2}$.
By this property, if we take $\theta \in (\gamma_{1,-},T_2^{-1} \gamma_{1,-})_{\partial \Hy^2}$ sufficiently close to $T_2^{-1} \gamma_{1,-}$, then we have $T_1 \theta \in (\gamma_{2,-}, \gamma_{1,+})_{\partial\Hy^2}$ (see Figure~\ref{FigClaimAA}).

Observe that the attractor of the restriction of the action of the semigroup generated by $T_1^{-1}$ and $T_2^{-1}$ to $[T_1^{-1}\gamma_{2,+},\gamma_{1,+}]_{\partial \Hy^2}$ is contained in $[\gamma_{1,-},\gamma_{1,+}]_{\partial \Hy^2}$. Hence there exists an integer $N \geq 0$ such that, for any word $w$ in $T_1^{-1}$ and $T_2^{-1}$ whose length is greater than or equal to $N$, both points $w T_1^{-1} \gamma_{2,+}$ and $w T_1^{-1} \gamma_{2,-}$ belong to $(T_2 \theta,\gamma_{1,+})_{\partial \Hy^2}$. Hence 
$$\overline{\bigcup_{i \geq 1}\, \bigcup_{\substack{w \in \langle T_1^{-1},T_2^{-1}\rangle_+\\ \ell(w) \geq N}} T_1^{i}T_2^{-1}w T_1^{-1}(\tilde{\gamma}_2)_R } \cap \partial \Hy^2 \subset [ T_1 \theta, \gamma_{1,+}]_{\partial \Hy^2} \subset (\gamma_{2,-}, \gamma_{1,+}]_{\partial \Hy^2}.$$

It remains to treat the case where the length of $w$ is smaller than $N$. By Lemma~\ref{commonendpoint}, for any word $w$ in $T_1^{-1}$ and $T_2^{-1}$, neither $T_2^{-1}w T_1^{-1} \gamma_{2,-}$ nor $T_2^{-1}w T_1^{-1} \gamma_{2,+}$ meet $\gamma_{1,-}$. Hence there exists an integer $I \geq 1$ such that, for any $i \geq I$ and any word $w$ in $T_1^{-1}$ and $T_2^{-1}$ whose length is smaller than $N$, the intersection $\overline{(T_1^{i}T_2^{-1} w T_1^{-1} \tilde{\gamma_{2}})_R}$ with $\partial\Hy^2$ is included in $(T_2^{-1} \gamma_{1,+},T_2 \gamma_{1,+})_{\partial \Hy^2}$.

Finally, the set of words $T_1^{i}T_2^{-1} w T_1^{-1}$ with $i\le I$ and the length of $w$ smaller than $N$ is finite, so by Lemma~\ref{commonendpoint} we have
\[\overline{\bigcup_{\substack{i\le I\\ \len(w)\le N}}(T_1^{i}T_2^{-1} w T_1^{-1} \tilde{\gamma_{2}})_R} \, \cap\, \{\gamma_{2,-},\gamma_{2,+}\} = \emptyset.\]
This finishes the proof of the claim.
\end{proof}

Let $i \in \Z$. As the trajectory $\I_{\tilde{\F}}^{\Z}(\tilde{x}_1)$ satisfies condition \ref{C2}, the leaf $\phi_{\tilde{y}'_i}$ and the half trajectory $T_2^{i}(-\infty,T_2^{-i}\tilde{y}'_i]_1$ are disjoint from the trajectories $T_2^{i-1}\I_{\tilde{\F}}^{\Z}(\tilde{x}_1)$ and $T_2^{i+1}\I_{\tilde{\F}}^{\Z}(\tilde{x}_1)$ and lie between them. Moreover, the sequence $\left( \overline{T_2^n \I_{\tilde{\F}}^{\Z}(\tilde{x}_1)} \right)_n$ of compact subsets of $\overline{\Hy^2}$ converges to $\gamma_{2,+}$ when $n \rightarrow +\infty$ and to $\gamma_{2,-}$ when $n \rightarrow -\infty$ for the Hausdorff topology. Therefore, the sequence $(B_i)_{i \in \Z}=\overline{\phi_{\tilde{y}'_i} \cup (-\infty,\tilde{y}'_i]_2 \cup T_2^{i}(-\infty,T_2^{-i}\tilde{y}'_i]_1} $ of subsets of $\overline{\Hy^2}$ converges to $\gamma_{2,-}$ when $i \rightarrow -\infty$ for the Hausdorff topology and the sequence $(B'_i)_{i \in \Z}=\overline{ \bigcup_{n \geq 0} T_2^{n}(\phi_{\tilde{y}'_i} \cup [\tilde{y}'_i,+\infty)_2) }$ of subsets of $\overline{\Hy^2}$ converges to $\gamma_{2,+}$ when $i \rightarrow +\infty$ for the Hausdorff topology. Hence, for $i \geq 0$ sufficiently large, both sets $B_{-i}$ and $B'_i$ avoid the set $A$, which proves the lemma.
\end{proof}

The following is similar to Corollary~\ref{CorDisjointleavescase1}.

\begin{corollary} \label{CorDisjointleavescase2}
Let $n \geq 1$. Recall that $\phi_2=\phi_{\tilde{y}'_{-k'_0}}
\qquad$ and $\phi'_2=\phi_{\tilde{y}'_{k'_0}}$.
\begin{enumerate}
\item For any word $w$ in $T_1$ and $T_2$, the trajectory $w \mathcal{I}^{\Z}_{\tilde{\F}}(\tilde{x}_{2})$ does not meet $T_{1}^{-n}\phi_{2}$ nor $T_1^{-n}T_{2}^{-k'_0}(-\infty,T_2^{k'_{0}}\tilde{y}'_{-k'_0}]_1$.
\item For any word $w$ in $T_1^{-1}$ and $T_2^{-1}$ which starts with $T_{2}^{-1}$, the trajectory $w \mathcal{I}^{\Z}_{\tilde{\F}}(\tilde{x}_{2})$ does not meet $T_{1}^{-n}\phi_{2}$ nor $T_1^{-n}T_{2}^{-k'_0}(-\infty,T_2^{k'_{0}}\tilde{y}'_{-k'_0}]_1$.
\item For any word $w$ in $T_1$ and $T_2$ which starts with $T_1$, the trajectory $w \mathcal{I}^{\Z}_{\tilde{\F}}(\tilde{x}_{2})$ does not meet $T_2^{n-1}\phi'_{2}$ nor $T_2^{n-1}[\tilde{y}'_{k'_0},+\infty)_2$.
\item For any word $w$ in $T_1^{-1}$ and $T_2^{-1}$ which starts with $T_{1}^{-1}$, the trajectory $w \mathcal{I}^{\Z}_{\tilde{\F}}(\tilde{x}_{2})$ does not meet $T_2^{n-1}\phi'_{2}$ nor $T_2^{n-1}[\tilde{y}'_{k'_0},+\infty)_2$..
\end{enumerate}
\end{corollary}

\begin{proof}
\begin{enumerate}
\item Let $w$ be any word on $T_1$ and $T_2$, and $n\ge 1$. By Lemma \ref{LemDisjointleavescase2}, the leaf $T_{1}^{-n}\phi_{2}$ and the half-trajectory $T_1^{-n}T_{2}^{-k'_0}(-\infty,T_2^{k'_{0}}\tilde{y}'_{-k'_0}]_1$ lie in
$\overline{(\tilde{\gamma}_2)_R^\complement}$. So, as the $\alpha$-limit of $T_1^{-n}T_{2}^{-k'_0}(-\infty,T_2^{k'_{0}}\tilde{y}'_{-k'_0}]_1$ is $T_1^{-n}T_{2}^{-k'_0}\gamma_{1,-}$, and as the segment $[\gamma_{1,-},T_1^{-n}T_{2}^{-k'_0}\gamma_{1,-}]_{\partial \Hy^2}$ meets neither $\gamma_{2,-}$ nor $\gamma_{2,+}$, the leaf $T_{1}^{-n}\phi_{2}$ and the half-trajectory $T_1^{-n}T_{2}^{-k'_0}(-\infty,T_2^{k'_{0}}\tilde{y}'_{-k'_0}]_1$ lie in the connected component of $\overline{\Hy^2} \setminus \overline{(\tilde{\gamma}_2)_R}$ which contains $\gamma_{1,-}$. On the other hand, the geodesic $w \tilde{\gamma}_2$ either lies in the connected component of $\overline{\Hy^2} \setminus \overline{\tilde{\gamma}_2}$ which contains $\gamma_{1,+}$ or $w\tilde{\gamma}_2=\tilde{\gamma}_2$. Hence $w \mathcal{I}^{\Z}_{\tilde{\F}}(\tilde{x}_{2}) \subset (w \tilde{\gamma}_2)_R$ does not meet the connected component of $\overline{\Hy^2} \setminus \overline{(\tilde{\gamma}_2)_R}$ which contains the leaf $T_{1}^{-n}\phi_{2}$ and the half-trajectory $T_1^{-n}T_{2}^{-k'_0}(-\infty,T_2^{k'_{0}}\tilde{y}'_{-k'_0}]_1$.
\item Lemma \ref{LemDisjointleavescase2} implies that, for any word $w$ in $T_1^{-1}$ and $T_2^{-1}$ which starts with $T_2^{-1}$, the neighbourhood $w (\tilde{\gamma}_2)_R$ is disjoint from $T_{1}^{-n}\phi_{2}$ and from $T_1^{-n}T_{2}^{-k'_0}(-\infty,T_2^{k'_{0}}\tilde{y}'_{-k'_0}]_1$. The second point follows as $w \mathcal{I}^{\Z}_{\tilde{\F}}(\tilde{x}_{2}) \subset (w \tilde{\gamma}_{2})_R$. 
\item Let $w$ be any word on $T_1$ and $T_2$ which starts with $T_1$. By Lemma \ref{LemDisjointleavescase2}, the sets $T_2^{n-1}\phi'_{2}$ and $T_2^{n-1}[\tilde{y}'_{k'_0},+\infty)_2$ lie in the connected component of $\overline{\Hy^2} \setminus \overline{(T_1 \tilde{\gamma}_2)_R}$ which contains $T_2^{n-1}\gamma_{2,+} = \gamma_{2,+}$. On the other hand, the geodesic $w\tilde{\gamma}_2$ either lies in the connected component of $\overline{\Hy^2} \setminus \overline{T_1 \tilde{\gamma}_2}$ which does not contain $\gamma_{2,+}$ or $w\tilde{\gamma}_2=T_1 \tilde{\gamma}_2$. As $w \mathcal{I}^{\Z}_{\tilde{\F}}(\tilde{x}_{2}) \subset (w \tilde{\gamma}_{2})_R$, this proves this third point.
\item This last point is proved identically to the third one by changing $T_1$ to $T_1^{-1}$ and $T_2$ to $T_2^{-1}$.
\end{enumerate}

\end{proof}

Changing the point $\tilde{x}_2$ to some other point of its orbit if necessary, we further suppose that the trajectory $(-\infty,\tilde{x}_2]_2$ is disjoint from $\phi_2$, from $(T_2^{j}\tilde{\gamma}_1)_R$ for any $j \geq -2$ and from $(T_1^{i} \tilde{\gamma}_2)_R$ for any $i \neq 0$. The following is similar to Lemma~\ref{LemTransadm}.

\begin{lemma} \label{LemAdmcase2} There exists an integer $m_2 >0$, which can be taken arbitrarily large, and $r_2 > m_2p_2$ such that the two following properties are satisfied. 
\begin{enumerate}
\item The segment of trajectory $[\tilde{x}_{2},\tilde{f}^{m_{2} q_{2}}(\tilde{x}_{2})]_{2}$ meets the leaves $\phi_2$ and $T_2^{r_2} \phi'_2$.
\item The half-trajectories $(-\infty,\tilde{x}_2]_2$ and $[\tilde{f}^{m_2q_2}(\tilde{x}_2),+\infty)_2$ are disjoint from $(T_2^{i} \tilde{\gamma}_1)_R$, for any $i$ with $-2 \leq i \leq m_2p_2+2$ and from $(T_1^{j}\tilde{\gamma}_2)_R$ for any $j \neq 0$.
\end{enumerate}\end{lemma}

Observe that this lemma implies that all the points $\tilde{y}'_n$, with $-2 \leq n \leq m_2p_2+2$, belong to the segment $[\tilde{x}_2,\tilde{f}^{m_2q_2}(\tilde{x}_2)]_2$.

\begin{proof} 
The proof is similar to the proof of Lemma \ref{LemTransadm}.
\end{proof}

\begin{lemma} \label{LemDisjointtrajectories}
There exists an integer $k_0 >0$ with the following property: for any integer $k$ with $|k| \geq k_0$ and any $j \in \Z$ with $j \neq 0$, the sets $T_{1}^k (\tilde\gamma_2)_R$ and $T_1^{k} T_{2}^{j}(\tilde\gamma_1)_R$ are disjoint from $(\tilde\gamma_2)_R$, and $T_2 T_1^{-k} \tilde \gamma_2 \subset L(\tilde \gamma_1)$.
\end{lemma}

By Lemma \ref{LemRecurrentpoint}, the transverse trajectories $\mathcal{I}^{\Z}_{\tilde{\F}}(\tilde{x}_i)$ stay at distance at most $R>0$ from $\tilde{\gamma}_i$, so the conclusions of the lemma stay true when replacing $(\tilde\gamma_i)_R$ by $T_i^{i'}\mathcal{I}^{\Z}_{\tilde{\F}}(\tilde{x}_i)$ for some $i'\in\Z$.

\begin{proof}
Note that if $k$ is large enough, then the set $T_2 T_1^{-k} \tilde \gamma_2 \subset L(\tilde \gamma_1)$ is contained in an arbitrary neighbourhood of $T_2\gamma_{1,+}\subset L(\tilde \gamma_1)$. Similarly, $-k$ is large enough, then the set $T_2 T_1^{-k} \tilde \gamma_2 \subset L(\tilde \gamma_1)$ is contained in an arbitrary neighbourhood of $T_2\gamma_{1,-}\subset L(\tilde \gamma_1)$.

Now, let 
$$\mathcal A = \overline{\bigcup_{j \neq 0} T_{2}^{j} (\tilde\gamma_1)_R} \cup \overline{(\tilde\gamma_2)_R}.$$
As the sequences of points $(T_{2}^{j} \gamma_{1,+})_{ j \in \Z}$ and  $(T_{2}^{j} \gamma_{1,-})_{ j \in \Z}$ both converge to $\gamma_{2,+}$ when $j \rightarrow +\infty$ and to $\gamma_{2,-}$ when $j \rightarrow -\infty$, we have
$\mathcal A \cap \partial \Hy^2 = \big\{ T_{2}^j\gamma_{1,\pm} \ | \ j \neq 0 \big\} \cup \left\{ \gamma_{2,\pm} \right\}$.
In particular, the set $\mathcal A$ meets neither $\gamma_{1,+}$ nor $\gamma_{1,-}$.
Hence the sequence of compact sets $ \left(T_1^{k} \mathcal A \right)_{k \in \Z}$ converges for the Hausdorff topology to the point $\gamma_{1,-}$ when $k \rightarrow -\infty$ and to $\gamma_{1,+}$ when $k \rightarrow +\infty$: we can find an integer $k_0$ so that the first properties of the lemma are satisfied. The proof of the last property follows the same strategy, using the remark made at the beginning of the proof.
\end{proof}

From now on, we fix integers $m_2>0$ and $k_0>0$ such that Lemma~\ref{LemAdmcase2} and \ref{LemDisjointtrajectories} are satisfied.

We denote by $\hat{\F}$, $\hat{f}$, $\hat{x}_1$, $\hat{x}_2$  respective lifts to $\widetilde{\dom} \F$ of $\tilde{\F}$, $\tilde{f}$,  $\tilde{x}_1$, $\tilde{x}_2$ in such a way that  $\mathcal{I}_{\hat{\F}}^{\Z}(\hat{x}_1)$ and $\mathcal{I}_{\hat{\F}}^{\Z}(\hat{x}_2)$ meet at some lift $\hat{y}_0$ of $\tilde{y}_0$ and $\hat{\phi}_2$ and $\wh{T_2^{r_2}\phi'_2}$ meet $\mathcal{I}_{\hat{\F}}^{\Z}(\hat{x}_2)$. We choose the lift $\hat{f}$ of $\tilde{f}$ which is isotopic to the identity.

For any $i \in \mathbb{Z}$, $i\neq 0$, we denote by $\widehat{T_1^i}$ the lift of $T_1^i$ such that $\widehat{T_1^i}\mathcal{I}_{\hat{\F}}^{\Z}(\hat{x}_2)$ meets  $\mathcal{I}_{\hat{\F}}^{\Z}(\hat{x}_1)$ at a lift $\hat{y}_i$ of $\tilde{y}_i$ and by $\widehat{T_2^i}$ a lift of $T_2^i$ such that $ \mathcal{I}_{\hat{\F}}^{\Z}(\hat{x}_2)$ meets $\widehat{T_{2}^i}\mathcal{I}_{\hat{\F}}^{\Z}(\hat{x}_1)$ at some lift $\hat{y}'_i$ of $\tilde{y}'_i$. Note that it is possible that $\wh{T_1^i}\neq \wh{T_1}^i$.
We fix respective lifts $\hat{\phi}_2$ and $\wh{T_2^{r_2}\phi'_2}$ of $\phi_2$ and $T_2^{r_2} \phi'_2$ such that the leaf $\phi_2$  meets $ \mathcal{I}_{\hat{\F}}^{\Z}(\hat{x}_2)$ at the point $\hat{y}'_{-k'_0}$ and the leaf $\wh{T_2^{r_2}\phi'_2}$ meets $ \mathcal{I}_{\hat{\F}}^{\Z}(\hat{x}_2)$.

 To be completely rigorous and so that these objects are uniquely defined, we need to consider parameters on the trajectories instead of actual points, as in the definition of essential intersection points. However, we chose to drop the mention of those parameters to simplify notation.

Let (see Figure~\ref{FigBhat})
\begin{align}\label{EqDefBHat}
\hat{B} & = \hat{f}^{m_2 q_2} \left( \overline{L \left(\widehat{T_1^{-k_{0}}}\hat{\phi}_2\right)} \right) \cup \overline{R \left(\widehat{T_1^{-k_{0}}}\wh{T_2^{r_2}\phi'}_{2} \right)}\\
& = \widehat{T_1^{-k_{0}}}\left(\hat{f}^{m_2 q_2} \left( \overline{L (\hat{\phi}_2)} \right) \cup \overline{R (\wh{T_2^{r_2}\phi'}_{2} )}\right).\nonumber
\end{align}
Observe that the set $\hat{B}$ contains the trajectory $\widehat{T_1^{-k_0}} \I^{\Z}_{\hat{\F}}(\hat{x}_2)$ and hence meets $\I^{\Z}_{\hat{\F}}(\hat{x}_1)$.  In the same way, for any $i \in \Z$, the set $\widehat{T_1^i}\widehat{T_1^{-k_0}}^{-1} \hat{B}$ contains the trajectory $\widehat{T_1^{i}} \I^{\Z}_{\hat{\F}}(\hat{x}_2)$ and hence meets $\I^{\Z}_{\hat{\F}}(\hat{x}_1)$.

Let 
$$\hat{\mathcal C}= \hat{f}^{m_2 q_2} \left( \overline{L \left(\widehat{T_1^{-k_{0}}}\hat{\phi}_2\right)} \right) \setminus L \left(\widehat{T_1^{-k_{0}}}\hat{\phi}_2\right) \subset \hat{B}.$$
As the trajectory $\mathcal{I}_{\tilde{\F}}^{\Z}(\tilde{x}_1)$ satisfies condition \ref{C2}, the leaves $\phi_2$ and $T_2^{r_2}\phi'_2$ and their translates under iterates of $T_1$ do not touch $\mathcal{I}_{\tilde{\F}}^{\Z}(\tilde{x}_1)$. Hence, for any $i \in \Z$,
$$ \widehat{T_1^i} \widehat{T_1^{-k_{0}}}^{-1}\hat{B} \cap \I^{\Z}_{\hat{\F}}(\hat{x}_1)=\widehat{T_1^i} \widehat{T_1^{-k_{0}}}^{-1} \hat{\mathcal C} \cap \I^{\Z}_{\hat{\F}}(\hat{x}_1).$$

\begin{figure}
\begin{minipage}{.4\linewidth}
\begin{center}
\includegraphics[width=\linewidth]{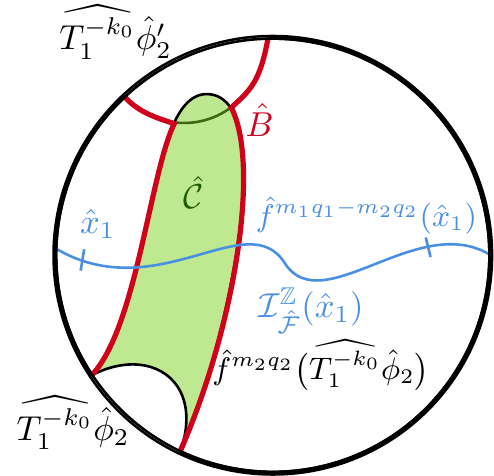}
\caption{The set $\hat B$. The ambient space here is $\wt{\dom\F}\simeq\Hy^2$.}\label{FigBhat}
\end{center}
\end{minipage}
\hfill
\begin{minipage}{.55\linewidth}
\begin{center}
\includegraphics[width=\linewidth]{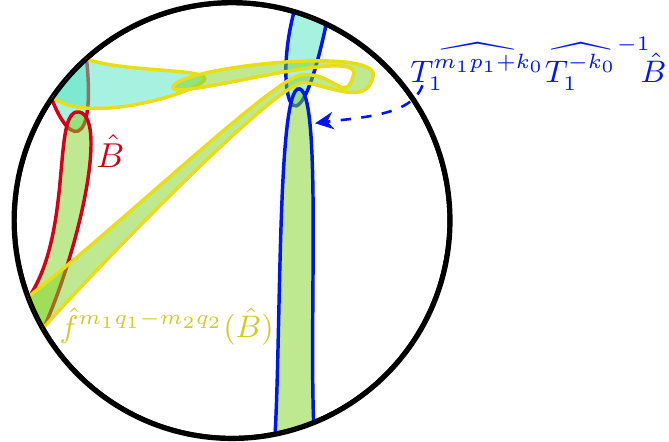}
\caption{Proof of Lemma~\ref{LemNotransautointersection}: the set $\hat B$, its translate by $\widehat{T_1^{m_1p_1+k_0}}$ and its image by $\hat f^{m_1q_1-m_2q_2}$.}\label{FigBhat2}
\end{center}
\end{minipage}
\end{figure}

Let $\tilde{\mathcal C}$ be the projection of $\hat{\mathcal C}$ on $\tilde{S}$ and recall that $\pi_{\tilde{\gamma}_1}$ denotes the orthogonal projection on $\tilde{\gamma}_1$. As before, we parametrize $\tilde{\gamma}_1$ by arclength and identify points of $\tilde{\gamma}_1$ with their corresponding parameters.

Changing $\tilde{x}_1$ to some of its iterates under $\tilde{f}$, we can suppose that 
\begin{equation}\label{Eqx1Min}
\max \left\{ n \in \Z \ | \  \tilde{f}^n( \tilde{x}_1) < \min \pi_{\tilde{\gamma_1}}(\tilde{\mathcal C}) \right\}= 0.
\end{equation}
Indeed, by Lemma~\ref{LemDisjointleavescase2}, $\min \pi_{\tilde{\gamma_1}}(\tilde{\mathcal C})\in\tilde{\gamma_1}$ (it is bigger than $\gamma_{1,-}$).

\begin{lemma}\label{LemAdmcase22}
There exists an integer $m_1 >0$, which can be taken arbitrarily large, such that, for any $i$ with $-k_0 \leq i \leq m_1p_1+k_0$,
$$ \widehat{T_1^i}\widehat{T_1^{-k_0}}^{-1} \hat{B} \cap \I^{\Z}_{\hat{\F}}(\hat{x}_1) = \widehat{T_1^i}\widehat{T_1^{-k_0}}^{-1} \hat{B} \cap (\hat{x}_1,\hat{f}^{m_1q_1-m_2q_2}(\hat{x}_1))_1.$$
In particular, the segment $[\hat{x}_1,\hat{f}^{m_1q_1-m_2q_2}(\hat{x}_1)]_1$ contains the points $\hat{y}_i$, for any $i$ with $-k_0 \leq i \leq k_0+m_1p_1$.
\end{lemma}

The last sentence of the lemma comes from the fact that the set $\widehat{T_1^i}\widehat{T_1^{-k_0}}^{-1} \hat{B}$ contains the trajectory $\widehat{T_1^{i}} \I^{\Z}_{\hat{\F}}(\hat{x}_2)$.

\begin{proof}
Fix $v'_1$ with $\frac{p_1}{q_1}<v'_1<v_1$. By Lemma \ref{LemRecurrentpoint}, for any $n$ sufficiently large,
$$ \pi_{\tilde{\gamma}_1}(\tilde{f}^{n}(\tilde{x}_1)) > nv'_1 +\pi_{\tilde{\gamma}_1}(\tilde{x}_1).$$
Moreover, for any $n$ sufficiently large,
$$
nv'_1 + \pi_{\tilde{\gamma}_1}(\tilde{x}_1)   >  \max(\pi_{\tilde{\gamma}_1}(\tilde{\mathcal C}))+(n+m_2q_2) \frac{p_1}{q_1} \ell(\gamma_1)+2k_{0} \ell(\gamma_1)+M_1,
$$
where $M_1$ is an upper bound on the diameters of the paths $\I_{\tilde{\F}}(\tilde{x})$ for $\tilde{x} \in K_1$, and $K_1$ is given by Lemma \ref{LemBoundeddistance}. 

Take an integer $m_1$ sufficiently large so that the two above properties hold for any $n \geq m_1 q_1 -m_2q_2$: note that $m_2$ does not depend on $m_1$ ($m_2$ being fixed, one can choose $m_1$ arbitrarily large). Then, for any $n \geq m_1q_1-m_2q_2$ and any $-k_0 \leq i \leq m_1p_1+k_0$,
$$ \pi_{\tilde{\gamma}_1}(\tilde{f}^n(\tilde{x}_1)) > \max \pi_{\tilde{\gamma}_1}(T_1^{i+k_0}\tilde{\mathcal C})+M_1$$
so that the half-trajectory $[\tilde{f}^{m_1 q_1-m_2q_2}(\tilde{x}_1),+\infty)_1$ does not meet $T_1^{i+k_0}\tilde{\mathcal C}$. Hence the half trajectory $[\hat{f}^{m_1 q_1-m_2 q_2}(\hat{x}_1),+\infty)_1$  does not meet $\widehat{T_1^i}\widehat{T_1^{-k_0}}^{-1} \hat{B}$. Moreover, the choice \eqref{Eqx1Min} of the point $\tilde{x}_1$ ensures that $(-\infty,\hat{x}_1]_{1}$ does not meet $\widehat{T_1^i}\widehat{T_1^{-k_0}}^{-1} \hat{B}$.
\end{proof}

From now on, we fix an integer $m_1>0$ such that Lemma~\ref{LemAdmcase22} is satisfied, and that $m_1q_1-2m_2q_2\ge 0$.

\subsubsection{Transverse intersections}

\begin{figure}
\begin{center}
\includegraphics[scale=1]{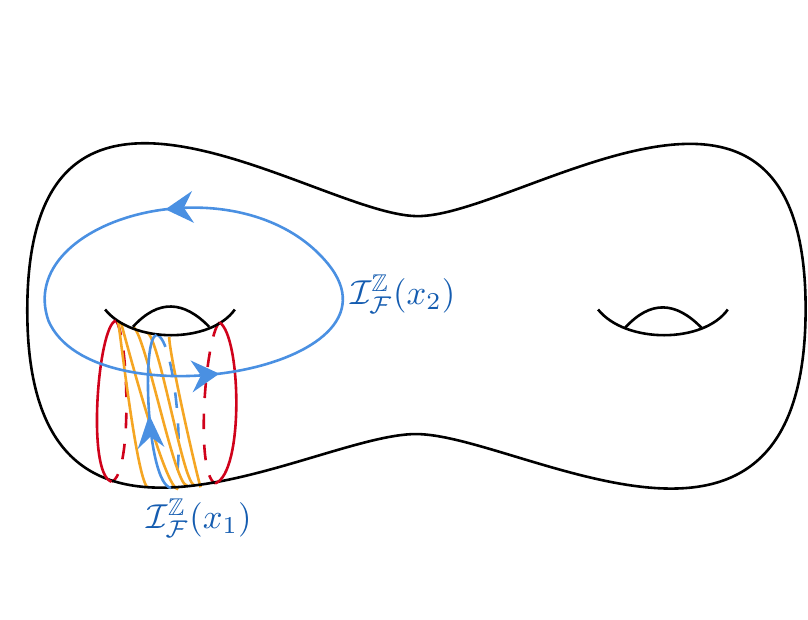}
\hfill
\includegraphics[scale=1]{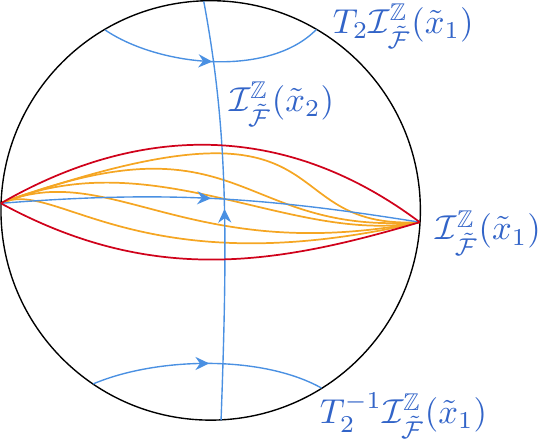}
\caption{Lellouch's example \cite{lellouch} on the genus-2 surface (left) and on the universal cover $\wt\dom(\F)$: the trajectories $\I_{\tilde \F}^\Z(\tilde x_1)$ and $\I_{\tilde \F}^\Z(\tilde x_2)$ have no $\F$-transverse intersection, as $\I_{\tilde \F}^\Z(\tilde x_1)$ is equivalent to a subpath of $\I_{\tilde \F}^\Z(\tilde x_2)$.}\label{Figlellouch}
\end{center}
\end{figure}

As we said earlier, $\mathcal{I}_{\tilde{\F}}^{\Z}(\tilde{x}_1)$ has not necessarily an $\tilde{\F}$-transverse intersection with $\mathcal{I}_{\tilde{\F}}^{\Z}(\tilde{x}_2)$ (see Figure~\ref{Figlellouch}). To overcome this problem, we will find a path $\tilde{\alpha}$ which has an $\tilde{\F}$-transverse intersection with $\mathcal{I}_{\tilde{\F}}^{\Z}(\tilde{x}_2)$ and which contains a large segment of $\mathcal{I}_{\tilde{\F}}^{\Z}(\tilde{x}_1)$. We borrow this idea and the idea of the proof of the following lemma from Lellouch's PhD thesis \cite{lellouch}.

\begin{lemma} \label{LemNotransautointersection}
The transverse path (see Figure~\ref{FigAlpha})
$$ \tilde{\alpha}= T_1^{-k_{0}}[\phi_2,T_{1}^{k_{0}}\tilde{y}_{-k_{0}}]_2 \, [\tilde{y}_{-k_{0}}, \tilde{y}_{m_1p_1+k_{0}}]_1\, T_1^{m_1 p_1+k_0}[T_1^{-m_1 p_1-k_0}\tilde{y}_{m_1p_1+k_{0}},T_2^{r_2} \phi'_2]_2$$
is admissible of order $m_1 q_1$ and has an $\tilde{\F}$-transverse intersection with the transverse path $T_{1}^{i}[\phi_2,T_2^{r_2}\phi'_2]_2$ at the point $\tilde{y}_i$, for any $0 \leq i \leq m_1 p_1$.  
\end{lemma}

\begin{proof}
Observe that, by Lemma~\ref{LemAdmcase22}, both sets $\hat{B}$ and $\widehat{T_1^{m_1p_1+k_{0}}}\widehat{T_1^{-k_0}}^{-1}\hat{B}$ separate the points $\hat{x}_1$ and $\hat{f}^{m_1q_1-m_2 q_2}(\hat{x}_1)$. Hence $\hat{f}^{m_1q_1-m_2q_2}(\hat{B}) \cap \widehat{T_1^{m_1p_1+k_0}} \widehat{T_1^{-k_0}}^{-1} (\hat{B}) \neq \emptyset$ (see Figure~\ref{FigBhat2}). By the definition 
\[\hat{B} = 
\hat{f}^{m_2 q_2} \left( \overline{L \left(\widehat{T_1^{-k_{0}}}\hat{\phi}_2\right)} \right)
\cup
\overline{R \left(\widehat{T_1^{-k_{0}}}\wh{T_2^{r_2}\phi'_2} \right)}
\]
of $\hat B$, this amounts to say that one of the intersections
\begin{align*}
\widehat{T_1^{-k_{0}}}\hat f^{m_1q_1-m_2q_2}\left(\overline{L (\hat{\phi}_2)}\right)
& \cap\, \widehat{T_1^{m_1p_1+k_0}} \overline{L (\hat{\phi}_2)}\\
\widehat{T_1^{-k_{0}}}\hat f^{m_1q_1}\left(\overline{L (\hat{\phi}_2)}\right)
& \cap\, \widehat{T_1^{m_1p_1+k_0}} \overline{R (\wh{T_2^{r_2}\phi'_{2}} )}\\
\widehat{T_1^{-k_{0}}}\hat f^{m_1q_1-2m_2q_2}\left(\overline{R (\wh{T_2^{r_2}\phi'_2})}\right)
& \cap\, \widehat{T_1^{m_1p_1+k_0}} \overline{L (\hat{\phi}_2)}\\
\widehat{T_1^{-k_{0}}}\hat f^{m_1q_1-m_2q_2}\left(\overline{R (\wh{T_2^{r_2}\phi'_2})}\right)
& \cap\, \widehat{T_1^{m_1p_1+k_0}} \overline{R (\wh{T_2^{r_2}\phi'_2})}
\end{align*}
is nonempty.

By Lemma~\ref{LemDisjointleavescase2}, the leaf $\phi_2$ is disjoint from the images of $\I_{\tilde\F}^\Z(\tilde x_2)$ by $T_1^i$ for any $i\neq 0$, hence none of the sets $L(\hat \phi_2)$ and $L(\wh{T_1^i}\hat\phi_2)$ is included in the other one. As the foliation is made of Brouwer lines, the same holds for $L(\hat \phi_2)$ and $L(\hat f^j(\wh{T_1^i}\hat\phi_2))$ for any $j\in\Z$. A similar property holds for the leaf $\hat\phi'_2$, which implies that the first and the last intersections are empty. 
 
For the third intersection, remark that as $m_1q_1-2m_2q_2\ge 0$, we have
\[\hat f^{m_1q_1-2m_2q_2}\left(\overline{R (\wh{T_2^{r_2}\phi'_2})}\right)\subset \overline{R (\wh{T_2^{r_2}\phi'_2})}.\]
Moreover, by Lemma~\ref{LemDisjointleavescase2}, the leaves $\phi_2$ and $\phi'_2$ are disjoint from the images of $\I_{\tilde\F}^\Z(\tilde x_2)$ by $T_1^i$ for any $i\neq 0$, so 
\[\widehat{T_1^{-k_{0}}}\overline{R (\wh{T_2^{r_2}\phi'_{2}})} \cap \widehat{T_1^{m_1p_1+k_0}} \overline{L (\hat{\phi}_2)} = \emptyset.\]
This implies that the third intersection is empty.

Hence the second intersection is nonempty, so
$$\hat{f}^{m_1q_1}(\widehat{T_1^{-k_{0}}}\hat{\phi}_{2}) \cap \widehat{T_1^{m_1p_1+k_0}} \wh{T_2^{r_2}\phi'_{2}} \neq \emptyset.$$
This means that the transverse path $\tilde{\alpha}$ is admissible of order $m_1 q_1$. 
\bigskip

Now, let $0 \leq i \leq m_1p_1$ and let us prove that $\tilde{\alpha}$ has an $\tilde{\F}$-transverse intersection with $T_1^i \tilde{\beta}$ where $\tilde{\beta}= [\phi_2,T_2^{r_2}\phi'_2]_2$. To do that, we want to use Lemma \ref{LemEssentialpoints}.

Let (recall that $k_0'$ comes from Lemma~\ref{LemDisjointleavescase2} and defines the leaf $\phi_2$)
$$ \tilde{\alpha}'=T_1^{-k_0} T_2^{-k'_{0}}( -\infty,T_2^{k'_0}\tilde{y}'_{-k'_0}]_1 \tilde{\alpha} T_1^{m_1p_1+k_0}T_2^{r_2}[ \phi'_2,+\infty)_2$$
and 
\[\tilde{\beta}'= (-\infty,\phi_2]_2 \tilde{\beta} T_2^{r_2}[\phi'_2,+\infty)_2 = \I_{\tilde \F}^\Z(\tilde x_2).\]

The transverse path $\tilde{\alpha}'$ joins $T_1^{-k_0} T_2^{-k'_{0}}\gamma_{1,-} \in (\gamma_{1,-},\gamma_{2,-})_{\partial \Hy^2}$ to $T_1^{m_1p_1+k_0} \gamma_{2,+} \in (\gamma_{1,+},T_1^{m_1p_1} \gamma_{2,+})_{\partial \Hy^2}$. The transverse path $T_1^i\tilde{\beta}'$ joins $T_1^{i} \gamma_{2,-} \in (\gamma_{2,-},\gamma_{1,+})_{\partial \Hy^2}$ to $T_1^{i} \gamma_{2,+} \in ( T_1^{m_1p_1} \gamma_{2,+},\gamma_{1,-})_{\partial \Hy^2}$ so that $\tilde{\alpha}'$ and $T_1^i \tilde{\beta}'$ are geometrically transverse.

To prove that $\tilde{\alpha}$ and $T_1^i \tilde{\beta}$ are $\tilde{\F}$-transverse, it suffices to use Lemma \ref{LemEssentialpoints} and to prove the following statements.
\begin{enumerate}[label=\alph*)]
\item The leaves $T_1^{-k_0} \phi_2$ and $T_1^{m_1p_1+k_0}T_2^{r_2} \phi'_2$ as well as the half-trajectories $T_1^{-k_0} T_2^{-k'_{0}}( -\infty,T_2^{k'_0}\tilde{y}'_{-k'_0}]_1$ and $T_1^{m_1p_1+k_0}T_2^{r_2}[ \phi'_2,+\infty)_2$ are disjoint from $T_1^i \tilde{\beta}'$.
\item The leaves $T_1^{i} \phi_2$ and $T_1^i T_2^{r_2} \phi'_2$ as well as the half trajectories $T_1^{i}(-\infty,\phi_2]_2$ and $T_1^{i} T_2^{r_2}[\phi'_2,+\infty)_2$ are disjoint from $\tilde{\alpha}'$.
\end{enumerate}
The first point a) is a consequence of Corollary \ref{CorDisjointleavescase2} (points 1 and 4).
Let us prove the second point b). 

We first prove that the trajectory $T_1^{i}T_2^{-1} \I^{\Z}_{\tilde{\F}}(\tilde{x}_1)$ separates $T_1^{i}\phi_2$ and $T_1^{i}(-\infty,\phi_2]_2$ from $\tilde{\alpha}'$. By Condition \ref{C2} and as the leaf $T_1^{i}\phi_2$ crosses $T_1^{i} T_2^{-k'_0}\I^{\Z}_{\tilde{\F}}(\tilde{x}_1)$, the leaf $T_1^{i} \phi_2$ does not meet $T_1^{i}T_2^{-1} \I^{\Z}_{\tilde{\F}}(\tilde{x}_1)$. By Lemma \ref{LemDisjointleavescase2}, $T_1^{i}(-\infty,\phi_2]_2$ is also disjoint from $T_1^{i}T_2^{-1} \I^{\Z}_{\tilde{\F}}(\tilde{x}_1)$. Hence the sets $T_1^{i}\phi_2$ and $T_1^{i}(-\infty,\phi_2]_2$ belong to the connected component of $\overline{\Hy^2} \setminus\overline{T_1^{i}T_2^{-1} \I^{\Z}_{\tilde{\F}}(\tilde{x}_1)}$ which contains the point $T_1^{i} \gamma_{2,-}$.  However, by Lemma \ref{LemDisjointtrajectories}, the trajectory $T_1^{i}T_2^{-1} \I^{\Z}_{\tilde{\F}}(\tilde{x}_1)$ is disjoint from $T_1^{-k_0}\I^{\Z}_{\tilde{\F}}(\tilde{x}_2)$ and from $T_1^{m_1p_1+k_0}\I^{\Z}_{\tilde{\F}}(\tilde{x}_2)$, and, by Condition \ref{C2}, it is disjoint from $T_1^{-k_0} T_2^{-k'_0}\I^{\Z}_{\tilde{\F}}(\tilde{x}_1)$ and from $\I^{\Z}_{\tilde{\F}}(\tilde{x}_1)$ so that the trajectory $T_1^{i}T_2^{-1} \I^{\Z}_{\tilde{\F}}(\tilde{x}_1)$ separates $T_1^{i}\phi_2$ and $T_1^{i}(-\infty,\phi_2]_2$ from $\tilde{\alpha}'$.

In the same way, we prove that the trajectory $T_1^{i}T_2^{r_2} \I^{\Z}_{\tilde{\F}}(\tilde{x}_1)$ separates $T_1^{i}T_2^{r_2}\phi'_2$ and $T_1^{i}T_2^{r_2}[\phi'_2,+\infty)_2$ from $\tilde{\alpha}'$.
\end{proof}

\begin{figure}
\begin{minipage}{.55\linewidth}
\begin{center}
\includegraphics[width=\linewidth]{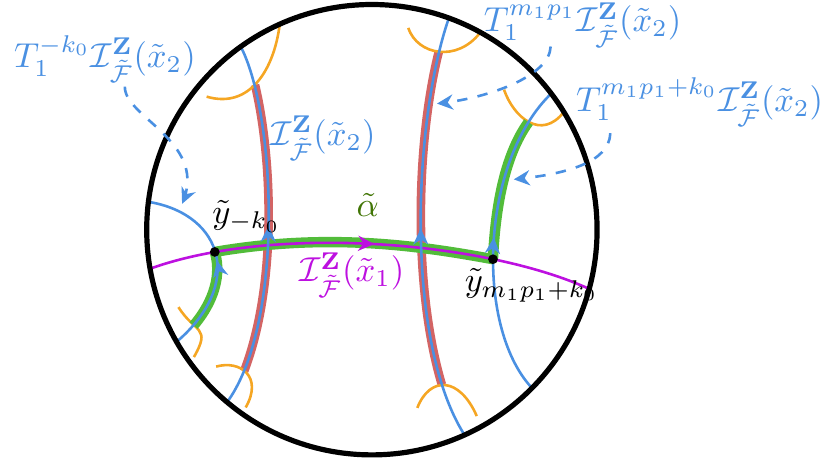}
\caption{The path $\tilde\alpha$ of Lemma~\ref{LemNotransautointersection} (green) has an $\tilde{\F}$-transverse intersection with $T_1^i\tilde\beta = T_1^i [\phi_2,T_2^{r_2}\phi'_2]_2$ (red) for any $0 \leq i \leq m_1 p_1$.}\label{FigAlpha}
\end{center}
\end{minipage}
\hfill
\begin{minipage}{.42\linewidth}
\begin{center}
\includegraphics[width=\linewidth]{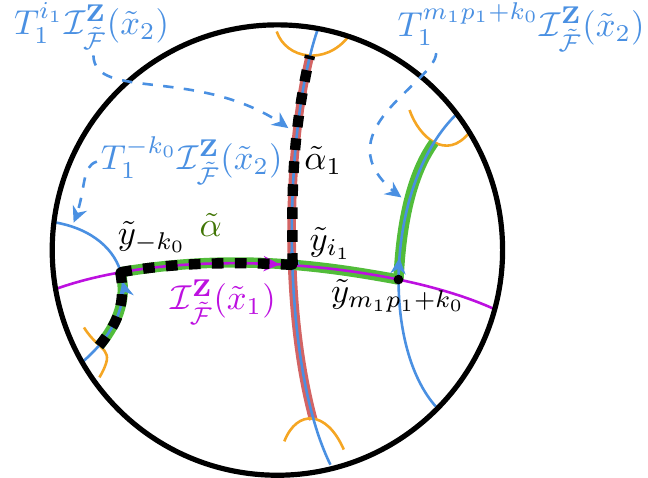}
\caption{Building the path $\tilde\alpha_1$ (dotted) of the base case of Lemma~\ref{LemNotransadmconc} from $\tilde\alpha$ (green) and $T_1^{i_1}\tilde\beta$ (red).}\label{FigAlpha1}
\end{center}
\end{minipage}
\end{figure}

\subsubsection{Admissible trajectories}\label{ParagAdmiss2}

From this point on, using the admissible path $\tilde{\alpha}$, the proof is similar to what we saw in the first case, though not identical. 
Recall that
\[\tilde{\beta}= [\phi_2,T_2^{r_2}\phi'_2]_2.\]

For any sequence $I=(i_{n},j_{n})_{n \geq 0}$ of couples of integers with $1 \leq i_{n} \leq m_1 p_1$ and $1 \leq j_{n} \leq m_2p_2$ for any $n \geq 0$, and, for any $k \geq 0$,  we let
$$T^{I_{k,1}}= T_1^{i_1} T_2^{j_1} T_1^{i_2} T_2^{j_2} \ldots T_1^{i_{k-1}} T_2^{j_{k-1}} T_1^{i_k}$$
and 
$$T^{I_{k,2}}= T_1^{i_1} T_2^{j_1} T_1^{i_2} T_2^{j_2} \ldots T_1^{i_{k}} T_2^{j_{k}}$$
with the convention that
$$T^{I_{0,1}}=T^{I_{0,2}}= \Id_{\Hy^2}.$$

\begin{lemma}\label{LemNotransadmconc}
Let $I=(i_{n},j_{n})_{n \geq 0}$ be any sequence with $1 \leq i_{n} \leq m_1 p_1$ and $1 \leq j_{n} \leq m_2 p_2$ for any $n$. Then, for any $n \geq 1$, there exists an $\tilde{\F}$-transverse path $\tilde{\alpha}_n$ such that:
\begin{enumerate}
\item The path $\tilde{\alpha}_n$ is admissible of order $n(m_1q_1+m_2q_2)$.
\item The path $\tilde{\alpha}_n$ joins the leaf $T_1^{-k_{0}}\phi_2$ to the leaf $T^{I_{n,1}} T_2^{r_2} \phi'_2$.
\item The path $\tilde{\alpha}_n$ is contained in
$ \bigcup_{ k \leq n} \big(T^{I_{k-1,2}} \tilde\alpha \cup T^{I_{k,1}}\tilde\beta\big).$
\item The path $\tilde{\alpha}_n$ has an $\tilde{\F}$-transverse intersection with the path $T^{I_{n,2}} \tilde{ \alpha}_n$.
\end{enumerate}
\end{lemma}

\begin{proof}
We will prove the lemma by induction on $n$.

\paragraph{Base case:}
Let (see Figure~\ref{FigAlpha1})
\[\tilde{\alpha}_1 = T_1^{-k_{0}}[\phi_2,T_{1}^{k_{0}}\tilde{y}_{-k_{0}}]_2 \,[\tilde{y}_{-k_{0}}, \tilde{y}_{i_1}]_1\,  T_1^{i_1}[T_1^{-i_1}\tilde{y}_{i_1}, T_2^{r_2}\phi_2']_2\]
We want to prove that:
\begin{enumerate}
\item The path $\tilde{\alpha}_1$ is admissible of order $m_1q_1+m_2q_2$.
\item The path $\tilde{\alpha}_1$ joins the leaf ${T_1^{-k_{0}}}\phi_2$ to the leaf $T_1^{i_1} T_2^{r_2} \phi'_2$.
\item The path $\tilde{\alpha}_1$ is contained in
$\tilde\alpha \cup T_1^{i_1}\tilde\beta.$
\item The path $\tilde{\alpha}_1$ has an $\tilde{\F}$-transverse intersection with the path $T_1^{i_1} T_2^{j_1} \tilde{ \alpha}_1$.
\end{enumerate}

By Lemma \ref{LemNotransautointersection}, the transverse path $\tilde{\alpha}$, which is admissible of order $m_1 q_1$, has an $\tilde{\F}$-transverse intersection with $T_1^{i_1}\tilde{\beta}$, which is admissible of order $m_2q_2$. Hence, by Proposition \ref{PropFondLCT}, the transverse path $\tilde{\alpha}_1$ satisfies the first three properties of the lemma. Let us check that the fourth property is also satisfied.

Consider the transverse path (see Figure~\ref{FigAlpha1prim})\footnote{Note that the right extension is made by $T_2^{r_2}[ \phi'_2,+\infty)_2$ and not $[ T_2^{r_2}\phi'_2,+\infty)_2$. Of course, the last point on $\tilde{\alpha}_1$ might not be the first point on $T_1^{i_1}T_2^{r_2}[ \phi'_2,+\infty)_2$ but those two points belong to the same leaf $T_1^{i_1}T_2^{r_2} \phi'_2$: it is possible to find a transverse path which meets the same leaves as this concatenation of paths.}
$$\tilde{\alpha}'_1 = T_1^{-k_0}T_2^{-k'_{0}}(-\infty,T_2^{k'_{0}}{\tilde{y}'_{-k'_{0}}}]_1 \tilde{\alpha}_1 T_1^{i_1}T_2^{r_2}[ \phi'_2,+\infty)_2.$$
This biinfinite path joins the point $T_1^{-k_0}T_2^{-k'_{0}} \gamma_{1,-}$ of $\partial \Hy^2$ to the point $T_{1}^{i_{1}} \gamma_{2,+}$ of $\partial \Hy^2$. We first prove that the transverse paths $\tilde{\alpha}'_1$ and $T^{I_{1,2}} \tilde{\alpha}'_1$ are geometrically transverse.

Remark that the transverse path $T^{I_{1,2}} \tilde{\alpha}'_1 = T_1^{i_1}T_2^{j_1} \tilde{\alpha}'_1$ joins the point the point $T_1^{i_1} T_2^{j_1}T_1^{-k_0}T_2^{-k'_{0}} \gamma_{1,-}$ of $\partial \Hy^2$ to the point $T_1^{i_1} T_{2}^{j_1}T_{1}^{i_{1}} \gamma_{2,+}$ of $\partial \Hy^2$. 

As the point $T_2^{j_1} T_{1}^{i_1} \gamma_{2,+}$ belongs to $(\gamma_{1,+},\gamma_{2,+})_{\partial \Hy^2}$, the point $T_1^{i_1} T_{2}^{j_1}T_{1}^{i_{1}} \gamma_{2,+}$ belongs to $(\gamma_{1,+},\, T_1^{i_1}\gamma_{2,+})_{\partial \Hy^2} \subset (T_1^{-k_0}T_2^{-k_0'}\gamma_{1,-},\, T_1^{i_1}\gamma_{2,+})_{\partial \Hy^2}$.

Consider the conjugates
\[T'_1=T_2^{k'_{0}} T_{1}^{k_0} T_1 T_1^{-k_0} T_2^{-k'_0}=T_2^{k'_{0}} T_1  T_2^{-k'_0}\quad \text{and}\quad T'_2=T_2^{k'_{0}} T_{1}^{k_0} T_2 T_1^{-k_0} T_2^{-k'_0}\]
of resp. $T_1$ and $T_2$ by $T_2^{k'_{0}} T_{1}^{k_0}$.
Observe that the the axis of $T'_1$ is the geodesic $T_{2}^{k'_0}\tilde\gamma_1$, and that the axis of $T'_2$ is the geodesic $T_2^{k'_{0}}T_1^{k_0} \tilde\gamma_2$ (hence, it is disjoint from $\tilde\gamma_2$ and from $\tilde\gamma_1$, by Lemma~\ref{LemDisjointtrajectories}).
Finally, observe that 
$T_1^{i_1} T_2^{j_1}T_1^{-k_0}T_2^{-k'_{0}} \gamma_{1,-}=T_1^{-k_{0}}T_{2}^{-k'_{0}}(T'_1)^{i_1}(T'_2)^{j_1} \gamma_{1,-}$.

\begin{claim}\label{LastClaim}
The endpoints of the geodesic $T_2^{k'_{0}}T_1^{k_0} T_2^{-1} \tilde\gamma_1$ satisfy 
\begin{align*}
T_2^{k'_{0}}T_1^{k_0} T_2^{-1} \gamma_{1,-} 
& \in (\gamma_{1,+},T_2^{k'_{0}}T_1^{k_0} {\gamma_{2,+}})_{\partial\Hy^2} \subset (\gamma_{1,+},\gamma_{2,+})_{\partial\Hy^2} \\
T_2^{k'_{0}}T_1^{k_0} T_2^{-1} \gamma_{1,+}
& \in (T_2^{k'_{0}}T_1^{k_0} \gamma_{2,-}, \gamma_{2,+})_{\partial\Hy^2} \subset (\gamma_{1,+},\gamma_{2,+})_{\partial\Hy^2} .
\end{align*}
\end{claim}

\begin{figure}
\begin{minipage}{.54\linewidth}
\begin{center}
\includegraphics[width=\linewidth]{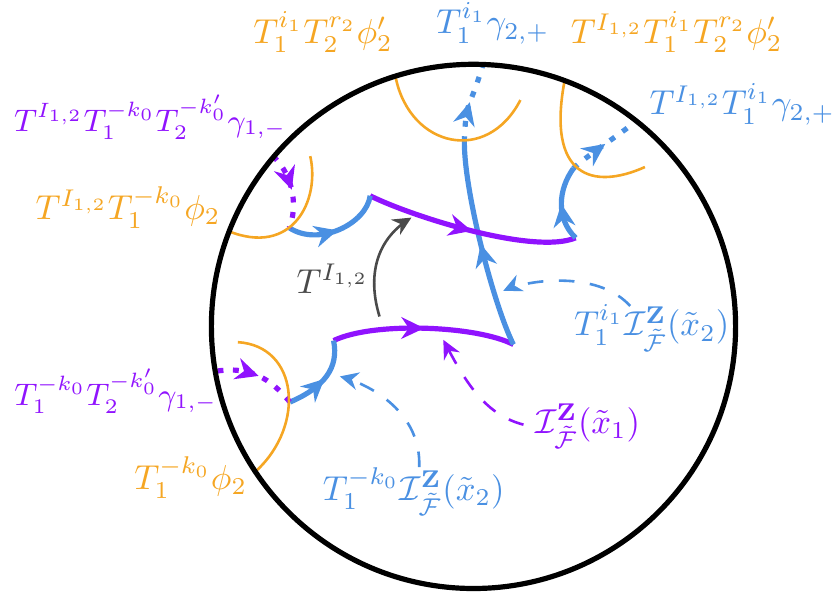}
\caption{The paths $\tilde\alpha_1$ and $T^{I_{1,2}}\tilde\alpha_1$ (thick lines) are prolonged to the paths $\tilde\alpha'_1$ and $T^{I_{1,2}}\tilde\alpha'_1$ by the dotted paths. The paths $\tilde\alpha_1$ and $T^{I_{1,2}}\tilde\alpha_1$ have an $\tilde\F$-transverse intersection.}\label{FigAlpha1prim}
\end{center}
\end{minipage}
\hfill
\begin{minipage}{.42\linewidth}
\begin{center}
\includegraphics[width=\linewidth]{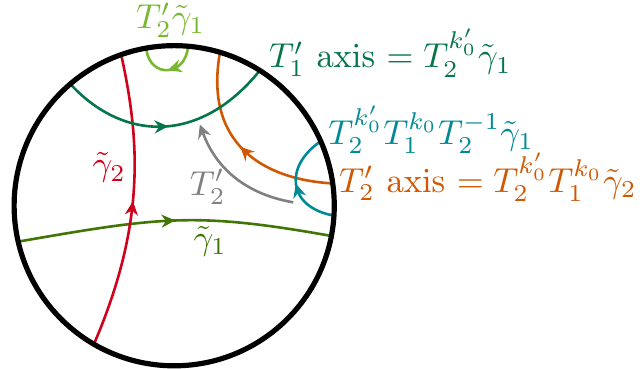}
\caption{Position of geodesics for Claim~\ref{LastClaim}.}\label{FigLastClaim}
\end{center}
\end{minipage}
\end{figure}

\begin{proof}
Consider the geodesic $T_2^{k'_{0}}T_1^{k_0} \tilde\gamma_2$ (see Figure~\ref{FigLastClaim}). We have that  $T_2^{k'_{0}}T_1^{k_0} \gamma_{2,+}\in (\gamma_{1,+},\gamma_{2,+})_{\partial\Hy^2}$. Moreover, this geodesic $T_2^{k'_{0}}T_1^{k_0} \tilde\gamma_2$ is disjoint from both $\tilde\gamma_2$ (by a trivial geometric argument) and $\tilde\gamma_1$ (by Lemma \ref{LemDisjointtrajectories}). So its other endpoint $T_2^{k'_{0}}T_1^{k_0} \gamma_{2,-}$ also lies in $(\gamma_{1,+},\gamma_{2,+})_{\partial\Hy^2}$.

The geodesic $T_2^{k'_0}T_1^{k_0} T_2^{-1}\tilde\gamma_1$ of the claim  crosses the geodesic $T_2^{k'_{0}}T_1^{k_0} \tilde\gamma_2$ of the previous paragraph. Moreover, it is disjoint from both $\tilde \gamma_1$ (by condition \ref{C1}) and $\tilde\gamma_2$ (by Lemma~\ref{LemDisjointtrajectories}). We get the claim by combining it with the orientation of the intersection between $T_2^{k'_0}T_1^{k_0} T_2^{-1}\tilde\gamma_1$ and $T_2^{k'_{0}}T_1^{k_0} \tilde\gamma_2$.
\end{proof}

As a consequence, the geodesic $T_2^{k'_{0}}T_1^{k_0} T_2^{-1} \tilde\gamma_1$ crosses the geodesic axis $T_2^{k'_{0}}T_1^{k_0} \tilde\gamma_2$ of $T'_2$, and so does its image $T_{2}^{k'_{0}}\tilde\gamma_1$ under $T'_2$. In other words, the geodesic axes of $T'_1$ and $T'_2$ cross (as in Figure~\ref{FigLastClaim}).

We also deduce from the claim that the geodesic $\tilde\gamma_1$ is strictly on the left of the geodesic $T_2^{k'_{0}}T_1^{k_0} T_2^{-1} \tilde\gamma_1$. As the image under $T'_2$ of the latter geodesic is $T_{2}^{k'_{0}}\tilde\gamma_1$, we deduce that the geodesic $T'_{2}\tilde\gamma_1$ lies strictly on the left of the geodesic $T_2^{k'_{0}} \tilde\gamma_{1}$. Because of the relative position of the geodesic axes of $T'_1$ and $T'_2$, we deduce that $(T'_1)^{i_1}(T'_{2})^{j_1}\tilde\gamma_1$ also lies strictly on the left of the geodesic axis $T_2^{k'_{0}} \tilde\gamma_{1}$ of $T'_1$. Hence the point $T_1^{i_1} T_2^{j_1}T_1^{-k_0}T_2^{-k'_{0}} \gamma_{1,-}=T_1^{-k_{0}}T_{2}^{-k'_{0}}(T'_1)^{i_1}(T'_2)^{j_1} \gamma_{1,-}$ belongs to $(\gamma_{1,+},\gamma_{1,-})_{\partial \Hy^2}$. However, as the point $T_1^{-k_0}T_2^{-k'_0} \gamma_{1,-}$ belongs to $(\gamma_{2,+},\gamma_{2,-})_{\partial \Hy^2}$, then the point $T_1^{i_1} T_2^{j_1}T_1^{-k_0}T_2^{-k'_0} \gamma_{1,-}$ belongs to $(T_1^{i_1}\gamma_{2,+},T_1^{i_1}\gamma_{2,-})_{\partial \Hy^2}$. Therefore, 
\begin{equation}\label{EqCaseN1}
T_1^{i_1} T_2^{j_1}T_1^{-k_0}T_2^{-k'_0} \gamma_{1,-}\in
(T_1^{i_1}\gamma_{2,+}, \gamma_{1,-})_{\partial \Hy^2} \subset 
(T_1^{i_1}\gamma_{2,+}, T_1^{-k_0}T_2^{-k'_0} \gamma_{1,-})_{\partial \Hy^2}
\end{equation} 
This finishes the proof that the transverse paths $\tilde{\alpha}'_1$ and $T^{I_{1,2}} \tilde{\alpha}'_1$ are geometrically transverse. 
\bigskip

To prove that  the paths $\tilde{\alpha}_1$ and $T^{I_{1,2}}\tilde{\alpha}_1$ are $\tilde{\F}$-transverse, it suffices to use Lemma~\ref{LemEssentialpoints} and to prove the following statements.
\begin{itemize}
\item The leaves $T_1^{-k_0}\phi_2$ and $T_1^{i_1}T_2^{r_2} \phi'_2$ as well as the trajectories $T_1^{-k_0}T_2^{-k'_{0}}(-\infty,T_2^{k'_{0}}\tilde{y}'_{k'_{0}}]_1$ and $T_1^{i_1}T_2^{r_2}[ \phi'_2,+\infty)_2$  do not meet the transverse path $T^{I_{1,2}}\tilde{\alpha}'_1$.
\item The leaves $T^{I_{1,2}}T_1^{-k_0}\phi_2$ and $T^{I_{1,2}}T_1^{i_1}T_2^{r_2} \phi'_2$ as well as the trajectories $T^{I_{1,2}}T_1^{-k_0}T_2^{-k'_{0}}(-\infty,T_2^{k'_{0}}\tilde{y}'_{k'_{0}}]_1$ and $T^{I_{1,2}}T_1^{i_1}T_2^{r_2}[ \phi'_2,+\infty)_2$  do not meet the transverse path $\tilde{\alpha}'_1$.
\end{itemize}

As the leaf $T_1^{-k_0} \phi_2$ meets $T_1^{-k_0}T_2^{-k'_0} \I^{\Z}_{\tilde{\F}}(\tilde{x}_1)$, then, by Condition \ref{C2}, this leaf and $T_1^{-k_0}T_2^{-k'_{0}}(-\infty,T_2^{k'_{0}}\tilde{y}'_{k'_{0}}]_1$  do not meet any other translates of $\I^{\Z}_{\tilde{\F}}(\tilde{x}_1)$. In particular, they do not meet the pieces of $T^{I_{1,2}}\tilde{\alpha}'_1$ which are contained in translates of $\mathcal{I}^{\Z}_{\tilde{\F}}(\tilde{x}_1)$, namely $T^{I_{1,2}}T_{1}^{-k_0}T_2^{-k'_0}\mathcal{I}^{\Z}_{\tilde{\F}}(\tilde{x}_1)$ and $T^{I_{1,2}}\mathcal{I}^{\Z}_{\tilde{\F}}(\tilde{x}_1)$. Moreover, the leaf $T_1^{-k_0} \phi_2$ and the piece of trajectory $T_1^{-k_0}T_2^{-k'_{0}}(-\infty,T_2^{k'_{0}}\tilde{y}'_{k'_{0}}]_1$ do not meet $T^{I_{1,2}}T_1^{i_1}  \I^{\Z}_{\tilde{\F}}(\tilde{x}_2)$ (by 1. of Corollary \ref{CorDisjointleavescase2}) nor $T^{I_{1,2}}T_1^{i_1} T_2^{r_2} \I^{\Z}_{\tilde{\F}}(\tilde{x}_2)$ (by 1. of Corollary \ref{CorDisjointleavescase2}). Finally, the leaf $T_1^{-k_0} \phi_2$ and the piece of trajectory $T_1^{-k_0}T_2^{-k'_{0}}(-\infty,T_2^{k'_{0}}\tilde{y}'_{k'_{0}}]_1$ do not meet $T^{I_{1,2}}T_1^{-k_0} \I^{\Z}_{\tilde{\F}}(\tilde{x}_2)$: indeed, on the one hand, by Lemma~\ref{LemDisjointleavescase2}, these two sets are contained in\footnote{The first $R$ stands for the right of the set while the second one denotes the $R$-neighbourhood.} $R\big((\tilde \gamma_1)_R\big)$; on the other hand, by Lemma~\ref{LemDisjointtrajectories}, we have $T_2 T_1^{-k_0} \tilde \gamma_2 \subset L(\tilde \gamma_1)$, which implies that
\[T^{I_{1,2}}T_1^{-k_0} \I^{\Z}_{\tilde{\F}}(\tilde{x}_2) \cap R\big((\tilde \gamma_1)_R\big) = \emptyset.\]
This proves that the leaf $T_1^{-k_0} \phi_2$ and the piece of trajectory $T_1^{-k_0}T_2^{-k'_{0}}(-\infty,T_2^{k'_{0}}\tilde{y}'_{k'_{0}}]_1$ are disjoint from the transverse path $T^{I_{1,2}}\tilde{\alpha}'_1$.

Let us prove that the leaf $T_1^{i_1}T_2^{r_2} \phi'_2$ and the half-trajectory $T_1^{i_1}T_2^{r_2}[ \phi'_2,+\infty)_2$ are disjoint from the transverse path $T^{I_{1,2}}\tilde{\alpha}'_1$. Observe that the leaf $T_1^{i_1}T_2^{r_2} \phi'_2$ meets $T_1^{i_1}T_2^{r_2+k'_0} \mathcal{I}^{\Z}_{\tilde{\F}}(\tilde{x}_1)$ so that, by Condition \ref{C2}, this leaf does not meet the pieces of $T^{I_{1,2}}\tilde{\alpha}'_1$ which are contained in translates of $\mathcal{I}^{\Z}_{\tilde{\F}}(\tilde{x}_1)$, namely $T^{I_{1,2}}T_{1}^{-k_0}T_2^{-k'_0}\mathcal{I}^{\Z}_{\tilde{\F}}(\tilde{x}_1)$ and $T^{I_{1,2}}\mathcal{I}^{\Z}_{\tilde{\F}}(\tilde{x}_1)$. Moreover, the half-trajectory $T_1^{i_1} T_2^{r_2}[\phi'_2,+\infty)_2$ is disjoint from the pieces of $T^{I_{1,2}}\tilde{\alpha}'_1$ which are contained in translates of $\mathcal{I}^{\Z}_{\tilde{\F}}(\tilde{x}_1)$ because $T_1^{i_1}T_2^{r_2}\I^{\Z}_{\tilde{\F}}(\tilde{x}_1)$ separates this half-trajectory from the pieces of $T^{I_{1,2}}\tilde{\alpha}'_1$ which are contained in translates of $\mathcal{I}^{\Z}_{\tilde{\F}}(\tilde{x}_1)$. Also, the leaf $T_1^{i_1}T_2^{r_2} \phi'_2$ and the half-trajectory $T_1^{i_1}T_2^{r_2}[ \phi'_2,+\infty)_2 $ are disjoint from the pieces of $T^{I_{1,2}}\tilde{\alpha}'_1$ which are contained in translates of $\mathcal{I}^{\Z}_{\tilde{\F}}(\tilde{x}_2)$, namely $T^{I_{1,2}}T_1^{-k_0}\I_{\tilde{\F}}^{\Z}(\tilde{x}_2)$ (by 4. of Corollary~\ref{CorDisjointleavescase2}) and $T^{I_{1,2}}T_1^{i_1}\I_{\tilde{\F}}^{\Z}(\tilde{x}_2)$ (by 3. of Corollary~\ref{CorDisjointleavescase2}).

Using Condition \ref{C2} and Corollary \ref{CorDisjointleavescase2}, we prove similarly that the leaves $T^{I_{1,2}}T_1^{-k_0}\phi_2$ and $T^{I_{1,2}}T_1^{i_1}T_2^{r_2} \phi'_2$ as well as the trajectories $T^{I_{1,2}}T_1^{-k_0}T_2^{-k'_{0}}(-\infty,T_2^{k'_{0}}\tilde{y}'_{k'_{0}}]_1$ and $T^{I_{1,2}}T_1^{i_1}T_2^{r_2}[ \phi'_2,+\infty)_2$  do not meet the transverse path $\tilde{\alpha}'_1$.

This completes the case $n=1$.

\paragraph{Induction:}
Now, suppose that we have constructed a transverse path $\tilde{\alpha}_n$ which satisfies the conditions of the lemma for some $n \geq 1$ and let us construct a transverse path $\tilde{\alpha}_{n+1}$ which satisfies the lemma. Using the $n=1$ case, we can construct a transverse path $\tilde{\alpha}_{1,i_{n+1}}$ with the following properties.
\begin{enumerate}
\item It is admissible of order $m_1q_1+m_2q_2$.
\item It joins $T_1^{-k_0}\phi_2$ to $T_1^{i_{n+1}}T_2^{r_2}\phi'_2$.
\item It is contained in $\tilde{\alpha} \cup T_1^{i_{n+1}} \tilde{\beta}$.
\end{enumerate}
Now, we prove that the transverse paths $\tilde{\alpha}_n$ and $T^{I_{n,2}}\tilde{\alpha}_{1,i_{n+1}}$ have an $\tilde{\F}$-transverse intersection. By Proposition \ref{PropFondLCT}, this will yield a path $\tilde{\alpha}_{n+1}$ which satisfies the first three properties of the lemma. As in the case $n=1$, the strategy is to use Lemma \ref{LemEssentialpoints} to prove that we have a transverse intersection.

Let 
$$\tilde{\alpha}'_{1,i_{n+1}}=T_1^{-k_0}T_2^{-k'_{0}}(-\infty,T_2^{k'_{0}}\tilde{y}'_{-k'_{0}}]_1 \tilde{\alpha}_{1,i_{n+1}} T_1^{i_{n+1}}{T_2^{r_2}}[\phi'_2,+\infty)_2$$
and 
$$\tilde{\alpha}'_{n} =T_1^{-k_0}T_2^{-k'_{0}}(-\infty,T_2^{k'_{0}}\tilde{y}'_{-k'_{0}}]_1 \tilde{\alpha}_{n} T^{I_{n,1}}T_2^{r_2} [\phi'_2,+\infty)_2.$$

Let us prove first that the paths $\tilde{\alpha}'_{n}$ and $T^{I_{n,2}}\tilde{\alpha}'_{1,i_{n+1}}$ are geometrically transverse. The path $\tilde{\alpha}'_{n}$ joins the point $T_1^{-k_0}T_2^{-k'_0} \gamma_{1,-}$ to the point $T^{I_{n,1}} \gamma_{2,+}$. The path $T^{I_{n,2}}\tilde{\alpha}'_{1,i_{n+1}}$ joins the point $T^{I_{n,2}}T_1^{-k_0}T_2^{-k'_0} \gamma_{1,-}$ to the point $T^{I_{n,2}}T_1^{i_{n+1}} \gamma_{2,+}$. 

The point $T_{2}^{j_{n}} T_1^{i_{n+1}} \gamma_{2,+}$ belongs to $(\gamma_{1,+}, \gamma_{2,+})_{\partial \Hy^2}$ so that the point $T^{I_{n,2}} T_1^{i_{n+1}} \gamma_{2,+}$ --- which is the left end of $T^{I_{n,2}}\tilde{\alpha}'_{1,i_{n+1}}$ --- belongs to $(\gamma_{1,+},T^{I_{n,1}}\gamma_{2,+})_{\partial \Hy^2} \subset (T_1^{-k_0}T_2^{-k_0'}\gamma_{1,-},T^{I_{n,1}}\gamma_{2,+})_{\partial \Hy^2}$ (remark that the ends of this interval are the endpoints of $\tilde\alpha'_n$).

We saw during the $n=1$ case (Equation~\eqref{EqCaseN1}) that the point $T_1^{i_{n}}T_2^{j_{n}}T_1^{-k_0}T_2^{-k'_0}\gamma_{1,-}$ belongs to $(T_1^{i_{n}}\gamma_{2,+}, \gamma_{1,-})_{\partial \Hy^2}$. Hence the point $T^{I_{n,2}}T_1^{-k_0}T_2^{-k'_0}\gamma_{1,-}$ belongs to $(T^{I_{n,1}} \gamma_{2,+}, \gamma_{1,-})_{\partial \Hy^2}$ which is contained in $(T^{I_{n,1}} \gamma_{2,+},T_1^{-k_0}T_2^{-k'_0}\gamma_{1,-})_{\partial \Hy^2}$. This proves that the paths $\tilde{\alpha}'_{n}$ and $T^{I_{n,2}}\tilde{\alpha}'_{1,i_{n+1}}$ are geometrically transverse. 
\medskip

It remains to check the following statements.
\begin{enumerate}[label=\alph*)]
\item The leaves $T_1^{-k_0}\phi_2$ and $T^{I_{n,1}}T_2^{r_2} \phi'_2$ as well as the half-trajectories $T_1^{-k_0}T_2^{-k'_{0}}(-\infty,T_2^{k'_{0}}\tilde{y}'_{-k'_{0}}]_1$ and $T^{I_{n,1}}T_2^{r_2} [\phi'_2,+\infty)_2$  do not meet the transverse path $T^{I_{n,2}}\tilde{\alpha}'_{1,i_{n+1}}$.
\item The leaves $T^{I_{n,2}}T_1^{-k_0}\phi_2$ and $T^{I_{n,2}}T_1^{i_{n+1}}T_2^{r_2} \phi'_2$ as well as the half-trajectories $T^{I_{n,2}}T_1^{-k_0}T_2^{-k'_{0}}(-\infty,T_2^{k'_{0}}\tilde{y}'_{-k'_{0}}]_1$ and $T^{I_{n,2}}T_1^{i_{n+1}}T_2^{r_2}[\phi'_2,+\infty)_2$  do not meet the transverse path $\tilde{\alpha}'_n$.
\end{enumerate}

By the $n=1$ case, the leaf $T_1^{i_{n}}T_2^{r_2} \phi'_2$ and the half-trajectory $T_1^{i_{n}}T_2^{r_2} [\phi'_2,+\infty)_2$ do not meet the transverse path $T_1^{i_{n}} T_2^{j_{n}}\tilde{\alpha}'_{1,i_{n+1}}$ so that the leaf $T^{I_{n,1}}T_2^{r_2} \phi'_2$ and the half-trajectory  $T^{I_{n,1}}T_2^{r_2} [\phi'_2,+\infty)_2$ do not meet the transverse path $T^{I_{n,2}}\tilde{\alpha}'_{1,i_{n+1}}$.

By condition \ref{C2}, the leaf $T_1^{-k_0}\phi_2$ (which crosses $T_1^{-k_0}T_2^{-k_0'} \I^{\Z}_{\tilde{\F}}(\tilde{x}_1)$) and the half-trajectory $T_1^{-k_0}T_2^{-k'_{0}}(-\infty,T_2^{k'_{0}}\tilde{y}'_{-k'_{0}}]_1$ do not meet the translates of $\I^{\Z}_{\tilde{\F}}(\tilde{x}_1)$ which are contained in $T^{I_{n,2}}\tilde{\alpha}'_{1,i_{n+1}}$ (translates by $T^{I_{n,2}}T_1^{-k_0}T_2^{-k_0'}$ and $T^{I_{n,2}}$). By Corollary \ref{CorDisjointleavescase2}.1, the leaf $T_1^{-k_0}\phi_2$ and the half-trajectory $T_1^{-k_0}T_2^{-k'_{0}}(-\infty,T_2^{k'_{0}}\tilde{y}'_{k'_{0}}]_1$ do not meet $T^{I_{n,2}} T_1^{i_{n+1}} \I^{\Z}_{\tilde{\F}}(\tilde{x}_2)$. To verify point a), it remains to check that those sets do not meet $T^{I_{n,2}} T_1^{-k_0} \I^{\Z}_{\tilde{\F}}(\tilde{x}_2)$ either, which amounts to showing that neither $\phi_2$ nor $T_2^{-k'_{0}}(-\infty,T_2^{k'_{0}}\tilde{y}'_{-k'_{0}}]_1$ meet $T_1^{k_0}T^{I_{n,2}} T_1^{-k_0} \I^{\Z}_{\tilde{\F}}(\tilde{x}_2)$. 
 By Lemma \ref{LemDisjointtrajectories}, the trajectory $T_2 T_1^{-k_{0}} \mathcal{I}^{\mathbb{Z}}_{\tilde{\F}}(\tilde{x}_2)$ is strictly on the left of $\mathcal{I}^{\mathbb{Z}}_{\tilde{\F}}(\tilde{x}_1)$ so has both endpoints in $[\gamma_{1,+},\gamma_{1,-}]_{\partial \Hy^2}$. Hence both endpoints of $T_1^{k_0}T^{I_{n,2}} T_1^{-k_0} \I^{\Z}_{\tilde{\F}}(\tilde{x}_2) \subset (T_1^{k_0}T^{I_{n,2}} T_1^{-k_0} \tilde{\gamma_{2}})_R$ are contained in $[\gamma_{1,+},\gamma_{1,-}]_{\partial \Hy^2}$. As a consequence, $T_1^{k_0}T^{I_{n,2}} T_1^{-k_0} \I^{\Z}_{\tilde{\F}}(\tilde{x}_2)$ does not meet the connected component of $\overline{\Hy^2} \setminus \overline{(\tilde{\gamma}_1)_{R}}$ which contains $\phi_2$ and $T_2^{-k'_{0}}(-\infty,T_2^{k'_{0}}\tilde{y}'_{-k'_{0}}]_1$: recall that, by Lemma~\ref{LemDisjointleavescase2}, the open set $(\tilde{\gamma}_1)_R$ is disjoint from both sets. This proves point~a).

We now prove point b).
By Condition \ref{C2}, the leaves $T^{I_{n,2}}T_1^{-k_0}\phi_2$ and $T^{I_{n,2}}T_1^{i_{n+1}}T_2^{r_2} \phi'_2$ and the half-trajectory $T^{I_{n,2}}T_1^{-k_0}T_2^{-k'_{0}}(-\infty,T_2^{k'_{0}}\tilde{y}'_{k'_{0}}]_1$ do not meet the translates of $\I^{\Z}_{\tilde{\F}}(\tilde{x}_1)$ which are contained in $\tilde{\alpha}'_n$.

By Lemma~\ref{LemDisjointleavescase2}, the trajectories $T^{I_{n,2}}T_1^{i_{n+1}}T_2^{r_2} \I^{\Z}_{\tilde{\F}}(\tilde{x}_1)$ and $T^{I_{n,2}}T_1^{i_{n+1}}T_2^{r_2} [\phi'_2,+\infty)_2$ are disjoint, and by Condition~\ref{C2}, the trajectory $T^{I_{n,2}}T_1^{i_{n+1}}T_2^{r_2} \I^{\Z}_{\tilde{\F}}(\tilde{x}_1)$ and the translates of $\I^{\Z}_{\tilde{\F}}(\tilde{x}_1)$ which are contained in $\tilde{\alpha}'_{n}$ are disjoint. By looking at the order on the limit points on $\partial\Hy^2$, we deduce that the trajectory $T^{I_{n,2}}T_1^{i_{n+1}}T_2^{r_2} \I^{\Z}_{\tilde{\F}}(\tilde{x}_1)$ separates $T^{I_{n,2}}T_1^{i_{n+1}}T_2^{r_2} [\phi'_2,+\infty)_2$ from the translates of $\I^{\Z}_{\tilde{\F}}(\tilde{x}_1)$ which are contained in $\tilde{\alpha}'_n$. We have proved all the properties of point b) that concern the intersections with translates of $\I^{\Z}_{\tilde{\F}}(\tilde{x}_1)$ which are contained in $\tilde{\alpha}'_n$. So it remains to treat the translates of $\I^{\Z}_{\tilde{\F}}(\tilde{x}_2)$ which are contained in $\tilde{\alpha}'_n$

By Corollary \ref{CorDisjointleavescase2}.2, the leaf $T^{I_{n,2}}T_1^{-k_0}\phi_2$ and the half-trajectory $T^{I_{n,2}}T_1^{-k_0}T_2^{-k'_{0}}(-\infty,T_2^{k'_{0}}\tilde{y}'_{k'_{0}}]_1$ do not meet any of the translates of $\I^{\Z}_{\tilde{\F}}(\tilde{x}_2)$ which are contained in $\tilde{\alpha}'_{n}$. Corollary \ref{CorDisjointleavescase2}.4 implies that the leaf $T^{I_{n,2}}T_1^{i_{n+1}}T_2^{r_2} \phi'_2$ and the half-trajectory $T^{I_{n,2}}T_1^{i_{n+1}}T_2^{r_2} [\phi'_2,+\infty)_2$ do not meet any of the translates of $\I^{\Z}_{\tilde{\F}}(\tilde{x}_2)$ which are contained in $\tilde{\alpha}'_{n}$. This proves point b).

This allows us to use Lemma \ref{LemEssentialpoints} and deduce that $\tilde{\alpha}_n$ and $\tilde{\alpha}_{1,i_{n+1}}$ are $\tilde{\F}$-transverse. Proposition \ref{PropFondLCT} then gives an $\tilde{\F}$ transverse path which satisifies the first three properties. Using similar techniques, we can prove that $\tilde{\alpha}_{n+1}$ and $T^{I_{n+1,2}} \tilde{\alpha}_{n+1}$ have an $\tilde{\F}$-transverse intersection, which completes the induction.
\end{proof}

Observe that the admissible paths that we obtain have an $\tilde{\F}$-transverse intersection with $\tilde{\gamma}$ on one side and an $\tilde{\F}$-transverse intersection with $\mathcal{I}_ {\tilde{\F}}^{\Z}(\tilde{x}_2)$ on the other side. As a consequence of Lemma \ref{LemTransadmconc} and of Theorem~\ref{ExistPasSuperCheval}, we obtain the following corollary, which is similar to Corollary \ref{CorTransperorbit}.

\begin{corollary} \label{CorNotransperorbit}
For any finite sequence $(i_n,j_n)_{1 \leq n \leq K}$, with $K \geq 1$, $1 \leq i_{n} \leq m_1 p_1$, $1 \leq j_{n} \leq m_2 p_2$ for any $n$, there exists points $\tilde{x}$ and $\tilde{y}$ of $\tilde{S}$ such that
$$ \tilde{f}^{K(m_1q_1+m_2q_2)}(\tilde{x})= T_{1}^{i_{1}} T_2^{j_1} \ldots T_1^{i_K}T_2^{j_K}(\tilde{x})$$
and
$$ \tilde{f}^{K(m_1q_1+m_2q_2)+m_1q_1}(\tilde{y})= T_{1}^{i_{0}} T_2^{j_0} \ldots T_1^{i_{K-1}}T_2^{j_{K-1}}T_{1}^{i_K}(\tilde{y}).$$
\end{corollary}

\begin{proof}[End of the proof of Theorem \ref{Th2transverse} in the first case]
Take any word $w$ in letters $T_1$ and $T_2$ which contains at least one $T_1$ letter and one $T_2$ letter. Of course, we identify such a word with a deck transformation of $\tilde{S}$.
Write
$$w=T_1^{i_1} T_{2}^{j_1} \ldots T_{1}^{i_K}T_{2}^{j_K}$$
with
$$\left\{ 
\begin{array}{l}
K \geq 0 \\
i_{n}, j_{n}>0 \ \text{if} \ 2 \leq n \leq K-1 \\
j_{1} >0 \ \text{and} \ i_{K} >0 \\
i_{1}  \geq 0 \ \text{and} \ j_{K}  \geq 0.
\end{array}
\right.$$

Take integers $m_{1}$ and $m_2$ large enough so that $\max(i_1+i_K, \max_{1 \leq n \leq K}i_{n}) \leq m_1 p_1$ and $\max(\max_{1 \leq n \leq K}j_{n}, j_{1}+j_{K}) \leq m_2 p_2$.

If $i_1>0$ and $j_{K}>0$, Corollary \ref{CorNotransperorbit} gives directly the wanted result. If $i_1>0$ and $j_{K}=0$, apply Corollary \ref{CorNotransperorbit} to the word $T_1^{i_1+i_K}T_2^{j_1} T_1^{i_2}T_2^{j_2} \ldots T_1^{i_{K-1}}T_2^{j_{K-1}}$ to obtain a point $\tilde{x}$ which satisfies the corollary. The point $T_1^{-i_{K}}\tilde{x}$ gives us the wanted periodic orbit.

If $i_{1}= 0$ apply Corollary \ref{CorNotransperorbit} to the word $ T_{1}^{i_2} T_2^{j_2} \ldots T_{1}^{i_K}T_{2}^{j_K+j_1}$ to get a lift $\tilde{x}$ of a periodic point associated to this deck transformation. The point $T_{2}^{j_{1}}(\tilde{x})$ gives us then the wanted periodic orbit.
\end{proof}

\small

\bibliographystyle{amsalpha}
\bibliography{Biblio}

\providecommand{\bysame}{\leavevmode\hbox to3em{\hrulefill}\thinspace}
\providecommand{\MR}{\relax\ifhmode\unskip\space\fi MR }
\providecommand{\MRhref}[2]{%
  \href{http://www.ams.org/mathscinet-getitem?mr=#1}{#2}
}
\providecommand{\href}[2]{#2}
\begin{thebibliography}{AZdPJ21}

\bibitem[ABP20]{alonso2020generic}
Juan Alonso, Joaqu{\'i}n Brum, and Alejandro Passeggi, \emph{Generic rotation
  sets in hyperbolic surfaces}, 2020.

\bibitem[AZ20]{MR4092853}
Salvador Addas-Zanata, \emph{A consequence of the growth of rotation sets for
  families of diffeomorphisms of the torus}, Ergodic Theory Dynam. Systems
  \textbf{40} (2020), no.~6, 1441--1458.

\bibitem[AZdPJ21]{MR4190050}
Salvador Addas-Zanata and Bruno de~Paula~Jacoia, \emph{A condition that implies
  full homotopical complexity of orbits for surface homeomorphisms}, Ergodic
  Theory Dynam. Systems \textbf{41} (2021), no.~1, 1--47.

\bibitem[BCLR20]{bguin2016fixed}
Fran\c{c}ois B\'{e}guin, Sylvain Crovisier, and Fr\'{e}d\'{e}ric Le~Roux,
  \emph{Fixed point sets of isotopies on surfaces}, J. Eur. Math. Soc. (JEMS)
  \textbf{22} (2020), no.~6, 1971--2046.

\bibitem[CB88]{MR964685}
Andrew~J. Casson and Steven~A. Bleiler, \emph{Automorphisms of surfaces after
  {N}ielsen and {T}hurston}, London Mathematical Society Student Texts, vol.~9,
  Cambridge University Press, Cambridge, 1988.

\bibitem[Fra88]{MR967632}
John Franks, \emph{Recurrence and fixed points of surface homeomorphisms},
  Ergodic Theory Dynam. Systems \textbf{8$^*$} (1988), no.~Charles Conley
  Memorial Issue, 99--107.

\bibitem[Fra89]{MR958891}
\bysame, \emph{Realizing rotation vectors for torus homeomorphisms}, Trans.
  Amer. Math. Soc. \textbf{311} (1989), no.~1, 107--115.

\bibitem[Gui20]{guiheneuf2020theorie}
Pierre-Antoine Guih{\'e}neuf, \emph{Th\'eorie de for\c{c}age des
  hom\'eomorphismes de surface [d'apr\`es {L}e {C}alvez et {T}al]}, 2020,
  S{\'e}minaire {B}ourbaki.

\bibitem[Han90]{MR1037109}
Michael Handel, \emph{The rotation set of a homeomorphism of the annulus is
  closed}, Comm. Math. Phys. \textbf{127} (1990), no.~2, 339--349.

\bibitem[Hay95]{MR1334719}
Eijirou Hayakawa, \emph{A sufficient condition for the existence of periodic
  points of homeomorphisms on surfaces}, Tokyo J. Math. \textbf{18} (1995),
  no.~1, 213--219.

\bibitem[Hem72]{MR295352}
John Hempel, \emph{Residual finiteness of surface groups}, Proc. Amer. Math.
  Soc. \textbf{32} (1972), 323.

\bibitem[Hom53]{MR58194}
Tatsuo Homma, \emph{An extension of the {J}ordan curve theorem}, Yokohama Math.
  J. \textbf{1} (1953), 125--129.

\bibitem[HR57]{MR89412}
Andr\'{e} Haefliger and Georges Reeb, \emph{Vari\'{e}t\'{e}s (non
  s\'{e}par\'{e}es) \`a une dimension et structures feuillet\'{e}es du plan},
  Enseign. Math. (2) \textbf{3} (1957), 107--125.

\bibitem[IIY03]{MR2003742}
Yoichi Imayoshi, Manabu Ito, and Hiroshi Yamamoto, \emph{On the
  {N}ielsen-{T}hurston-{B}ers type of some self-maps of {R}iemann surfaces with
  two specified points}, Osaka J. Math. \textbf{40} (2003), no.~3, 659--685.

\bibitem[Kat79]{MR554383}
Anatole Katok, \emph{Bernoulli diffeomorphisms on surfaces}, Ann. of Math. (2)
  \textbf{110} (1979), no.~3, 529--547.

\bibitem[Kra81]{MR611385}
Irwin Kra, \emph{On the {N}ielsen-{T}hurston-{B}ers type of some self-maps of
  {R}iemann surfaces}, Acta Math. \textbf{146} (1981), no.~3-4, 231--270.

\bibitem[KT14]{MR3238423}
Andres Koropecki and Fabio~Armando Tal, \emph{Area-preserving irrotational
  diffeomorphisms of the torus with sublinear diffusion}, Proc. Amer. Math.
  Soc. \textbf{142} (2014), no.~10, 3483--3490.

\bibitem[KT18]{MR3820002}
\bysame, \emph{Fully essential dynamics for area-preserving surface
  homeomorphisms}, Ergodic Theory Dynam. Systems \textbf{38} (2018), no.~5,
  1791--1836.

\bibitem[Kwa92]{MR1176627}
Jaroslaw Kwapisz, \emph{Every convex polygon with rational vertices is a
  rotation set}, Ergodic Theory Dynam. Systems \textbf{12} (1992), no.~2,
  333--339.

\bibitem[LC05]{MR2217051}
Patrice Le~Calvez, \emph{Une version feuillet\'ee \'equivariante du
  th\'eor\`eme de translation de {B}rouwer}, Publ. Math. Inst. Hautes \'Etudes
  Sci. (2005), no.~102, 1--98.

\bibitem[LC20]{PatriceRecent}
\bysame, \emph{Conservative surface homeomorphisms with finitely many periodic
  points}, 2020.

\bibitem[LCT18a]{MR3787834}
Patrice Le~Calvez and Fabio~Armando Tal, \emph{Forcing theory for transverse
  trajectories of surface homeomorphisms}, Invent. Math. \textbf{212} (2018),
  no.~2, 619--729.

\bibitem[LCT18b]{1803.04557}
\bysame, \emph{Topological horseshoes for surface homeomorphisms}, 2018.

\bibitem[Lel19]{lellouch}
Gabriel Lellouch, \emph{{Sur les ensembles de rotation des hom{\'e}omorphismes
  de surface en genre $\ge$ 2}}, Theses, {Sorbonne Universit{\'e}}, April 2019.

\bibitem[Les11]{MR2846925}
Pablo Lessa, \emph{Rotation vectors for homeomorphisms of non-positively curved
  manifolds}, Nonlinearity \textbf{24} (2011), no.~11, 3237--3266.

\bibitem[Les12]{PabloUnpublished}
\bysame, \emph{Two fixed points can force positive entropy of a homeomorphism
  on a hyperbolic surface}, 2012.

\bibitem[LM91]{MR1101087}
Jaume Llibre and Robert~S. MacKay, \emph{Rotation vectors and entropy for
  homeomorphisms of the torus isotopic to the identity}, Ergodic Theory Dynam.
  Systems \textbf{11} (1991), no.~1, 115--128.

\bibitem[Mat97]{MR1444450}
Shigenori Matsumoto, \emph{Rotation sets of surface homeomorphisms}, Bol. Soc.
  Brasil. Mat. (N.S.) \textbf{28} (1997), no.~1, 89--101.

\bibitem[MZ89]{MR1053617}
Micha{\l} Misiurewicz and Krystyna Ziemian, \emph{Rotation sets for maps of
  tori}, J. London Math. Soc. (2) \textbf{40} (1989), no.~3, 490--506.

\bibitem[Nie24]{MR1512188}
Jakob Nielsen, \emph{Die {I}somorphismengruppe der freien {G}ruppen}, Math.
  Ann. \textbf{91} (1924), no.~3-4, 169--209.

\bibitem[Pol92]{MR1094554}
Mark Pollicott, \emph{Rotation sets for homeomorphisms and homology}, Trans.
  Amer. Math. Soc. \textbf{331} (1992), no.~2, 881--894.

\bibitem[PPS18]{MR3784518}
Alejandro Passeggi, Rafael Potrie, and Mart\'{\i}n Sambarino, \emph{Rotation
  intervals and entropy on attracting annular continua}, Geom. Topol.
  \textbf{22} (2018), no.~4, 2145--2186.

\bibitem[Sch57]{MR88720}
Sol Schwartzman, \emph{Asymptotic cycles}, Ann. of Math. (2) \textbf{66}
  (1957), 270--284.

\bibitem[SSV10]{MR2670926}
Radu Saghin, Wenxiang Sun, and Edson Vargas, \emph{On {D}irac physical measures
  for transitive flows}, Comm. Math. Phys. \textbf{298} (2010), no.~3,
  741--756.

\bibitem[Tal12]{MR2929025}
F\'{a}bio~Armando Tal, \emph{Transitivity and rotation sets with nonempty
  interior for homeomorphisms of the {$2$}-torus}, Proc. Amer. Math. Soc.
  \textbf{140} (2012), no.~10, 3567--3579.

\bibitem[ZG04]{MR2060531}
Piotr Zgliczy\'{n}ski and Marian Gidea, \emph{Covering relations for
  multidimensional dynamical systems}, J. Differential Equations \textbf{202}
  (2004), no.~1, 32--58.

\end{thebibliography}

\end{document}